\documentclass{article}
\usepackage{amsfonts,amssymb,latexsym,amsmath,amsthm}

\newtheorem{theorem}{Theorem}[section]
\newtheorem{lemma}[theorem]{Lemma}
\newtheorem{corollary}[theorem]{Corollary}
\newtheorem{proposition}[theorem]{Proposition}

\theoremstyle{definition}
\newtheorem{definition}[theorem]{Definition}

\newtheorem{remark}[theorem]{Remark}

\newtheorem{example}[theorem]{Example}

\newtheorem{notation}[theorem]{Notation}

\newcommand{\h}{{\rm I\hskip-2pt h}}

\newcommand{\depth}{{\rm d\hskip-2pt I}}
\newcommand{\ostar}{\circledast}
\newcommand{\igual}{{=\hskip-3pt =\hskip-3pt=\hskip-3pt =\hskip-3pt}}

\newcommand{\End}{{ \rm End }}

\newcommand{\Cok}{{ \rm Cok }\,}
\newcommand{\Coker}{{ \rm Coker }}
\newcommand{\Ext}{{ \rm Ext }}
\newcommand{\Hom}{{ \rm Hom }}
\newcommand{\Ker}{{ \rm Ker }\,}

\newcommand{\Ob}{{ \rm Ob }}
\newcommand{\ad}{{ \rm ad }}

\newcommand{\cv}{{ \rm cv }}
\newcommand{\tw}{{ \rm tw }}
\newcommand{\Mod}{{ \rm Mod }}

\newcommand{\Gba}{{\rm G\overline{b}}}
\newcommand{\Gbb}{{\rm G\underline{b}}}

\newcommand{\GM}{{\rm GMod}}

\newcommand{\rad}{{ \rm rad }}
\renewcommand{\top}{{ \rm top }}

\newcommand{\proj}{{\rm proj }}

\newcommand{\g}{\hbox{-}}
\newcommand{\uddots}{\mathinner{\mkern1mu\raise1pt\vbox{\kern7pt\hbox{.}}
\mkern2mu\raise4pt\hbox{.}\mkern2mu\raise7pt\hbox{.}\mkern1mu}}

\newcommand{\hueca}[1]{\mathbb{#1}}
\newcommand{\lddots}{
\mathinner{
\mkern1mu\raise1pt}\vbox{\kern7pt\hbox{.}}
\mkern2mu\raise3pt\hbox{.}
\mkern2mu\raise7pt\hbox{.}\mkern1mu}

\renewcommand{\mod}{{\rm mod }}
\renewcommand{\Im}{{\rm Im}\,}

\newcommand{\rightdashmap}[1]{\smash{\mathop{\hbox to 
20pt{-\,-\,-\,\rightarrowfill}}\limits^{#1}}}

\newcommand{\rightmap}[1]{\smash{\mathop{\hbox to 
20pt{\rightarrowfill}}\limits^{#1}}}
\newcommand{\leftmap}[1]{\smash{\mathop{\hbox to 
20pt{\leftarrowfill}}\limits^{#1}}}
\newcommand{\lmapdown}[1]{\llap{$\vcenter{\hbox{$\scriptstyle#1$}}$}
\Bigg\downarrow}
\newcommand{\rmapdown}[1]{\Bigg\downarrow\rlap{$\vcenter{\hbox{$\scriptstyle#1$}
}$}}
\newcommand{\lmapup}[1]{\llap{$\vcenter{\hbox{$\scriptstyle#1$}}$}\Bigg\uparrow}
\newcommand{\rmapup}[1]{\Bigg\uparrow\rlap{$\vcenter{\hbox{$\scriptstyle#1$}}$}}
\newcommand{\longrightmap}[1]{\smash{\mathop{\hbox to 
4cm{\rightarrowfill}}\limits^{#1}}}
\newcommand{\longleftmap}[1]{\smash{\mathop{\hbox to 
4cm{\leftarrowfill}}\limits^{#1}}}
\newcommand{\medrightmap}[1]{\smash{\mathop{\hbox to 
2cm{\rightarrowfill}}\limits^{#1}}}
\newcommand{\medleftmap}[1]{\smash{\mathop{\hbox to 
2cm{\leftarrowfill}}\limits^{#1}}}

\newcommand{\shortlmapdown}[1]
{\downarrow\rlap{$\vcenter{\hbox{$\scriptstyle#1$}}$}}

\newcommand{\idmapdown}[1]
{\hskip-8pt\mathop{\hskip-5pt\raise6pt
\hbox{$\scriptstyle#1$}\hskip-5pt\swarrow}}

\newcommand{\ddmapdown}[1]
{\hskip-5pt\mathop{\searrow\hskip-6pt\raise5pt\hbox{$\scriptstyle#1$}}}

\newcommand{\idmapup}[1]
{\hskip-5pt\mathop{\nwarrow\hskip-6pt \raise5pt\hbox{$\scriptstyle#1$}}}

\newcommand{\ddmapup}[1]
{\hskip-8pt\mathop{\hskip-5pt\raise6pt\hbox{$\scriptstyle#1$}\hskip-5pt\nearrow}
}

\newcommand{\flechypunt}[2]{\ \smash{\mathop{
   \raise 3pt \hbox to 40pt{\rightarrow}\hskip-40pt \lower 3pt
   \hbox to 40pt{\dashrightarrow}}\limits^{#1}_{#2}}\ }
\newcommand{\dobleflechavieneva}[2]{\ \smash{\mathop{
   \raise 3pt \hbox to 40pt{\leftarrowfill}\hskip-40pt \lower 3pt
   \hbox to 40pt{\rightarrowfill}}\limits^{#1}_{#2}}\ }
\newcommand{\dobleflechavienevabis}[2]{\ \smash{\mathop{
   \raise 3pt \hbox to 20pt{\leftarrowfill}\hskip-20pt \lower 3pt
   \hbox to 40pt{\rightarrowfill}}\limits^{#1}_{#2}}\ }

\newcommand{\dobleflecha}[2]{\ \smash{\mathop{
   \raise 3pt \hbox to 40pt{\rightarrowfill}\hskip-40pt \lower 3pt
   \hbox to 40pt{\rightarrowfill}}\limits^{#1}_{#2}}\ }

\newcommand{\longequal}{\ \smash{\mathop{
   \raise 5pt \hbox to 35pt{\hrulefill}\hskip-35pt \lower 0pt
   \hbox to 35pt{\hrulefill}}}\ }

\begin{document}
 
\title{\bf Homological systems and bocses}
\vskip-1cm
\author{R. Bautista, E. P\'erez and L. Salmer\'on}
\vskip-1cm
\date{}
\maketitle

\renewcommand{\thefootnote}{}


  
  \begin{abstract}
\noindent 
We show that, up to Morita equivalence, any finite-dimensional algebra with a suitable homological system, admits an exact Borel subalgebra. This generalizes   
a theorem by Koenig, K\"ulshammer, and Ovsienko, which holds for quasi-hereditary algebras. Our proof follows the same general scheme proposed 
by these authors, in a more general context: we associate a differential graded tensor algebra with relations, using the structure of $A_\infty$-algebra of a suitable Yoneda algebra, and use its category of modules to describe the category of filtered modules associated to the given homological system. 
\end{abstract}

\section{Introduction}
 
 We denote by $k$ a fixed algebraically closed ground field. For every $k$-algebra or any bimodule over given $k$-algebras we consider, we assume that the field $k$ acts centrally on them. 
 
 Given a $k$-algebra $\Lambda$, we denote by $\Lambda\g\Mod$ the category of left $\Lambda$-modules and by $\Lambda\g \mod$ its full subcategory of finitely generated $\Lambda$-modules.

 We recall that a \emph{preordered set} $({\cal P},\leq)$ is a non-empty set ${\cal P}$ equipped with a relation $\leq$ such that $i\leq i$, for all $i\in {\cal P}$ and such that, whenever we have $i,j,s\in {\cal P}$ with  $i\leq j$ and $j\leq s$, we also have $i\leq s$.  Two elements $i,j\in {\cal P}$ are \emph{equivalent} iff $i\leq j$ and $j\leq i$. In this case, we write $i\sim j$.

 \begin{remark}\label{R: P sobre sim}
 Let $\overline{{\cal P}}={\cal P}/\sim$ be the set of equivalence 
  classes of ${\cal P}$ modulo the equivalence relation $\sim$. For any $i\in {\cal P}$, denote by $\overline{i}$ its equivalence class. Then, $\overline{\cal P}$ is a partially ordered set with the relation defined by $\overline{i}\leq \overline{j}$ iff $i\leq j$. 
 \end{remark}

 We recall the following terminology from \cite{MSX}, 
 in the particular context we are interested in. 
 
 \begin{definition}\label{D: homological system y F(Delta)}
  Given a finite-dimensional $k$-algebra $\Lambda$, a \emph{(finite) homological system $({\cal P},\leq,\{\Delta_i\}_{i\in {\cal P}})$} consists of 
  a finite preordered set $({\cal P},\leq)$ and 
  a family of pairwise non-isomorphic indecomposable finite-dimensional $\Lambda$-modules $\{\Delta_i\}_{i\in {\cal P}}$ satisfying the following two conditions:
  \begin{enumerate}
  \item $\Hom_\Lambda(\Delta_i,\Delta_j)\not=0$ implies $i\leq j$;
  \item $\Ext^1_\Lambda(\Delta_i,\Delta_j)\not=0$ implies $i\leq j$ and $i\not\sim j$.
  \end{enumerate}
  Given such an homological system, we write  $\Delta:=\bigoplus_{i\in {\cal P}}\Delta_i$ and denote by ${\cal F}(\Delta)$  the full subcategory of $\Lambda\g\mod$ consisting of the trivial module $0$ and all those $M\in \Lambda\g\mod$ which admit a \emph{$\Delta$-filtration},  that is a filtration of submodules
$$0=M_t\subseteq M_{t-1}\subseteq \cdots\subseteq M_1\subseteq M_0=M$$
such that $M_j/M_{j+1}$ is isomorphic to some module in $\{\Delta_i\mid i\in {\cal P}\}$, for each $j\in [0,t-1]$.  
 \end{definition}

 The following special homological systems will be relevant for us.

 \begin{definition}\label{D: admissible hs and strict hs}
 A finite homological system ${\cal H}=({\cal P},\leq,\{\Delta_i\}_{i\in {\cal P}})$ for a finite-dimensional $k$-algebra $\Lambda$ is called \emph{admissible} if $\Lambda\in {\cal F}(\Delta)$ and the number of isomorphism classes of indecomposable projective $\Lambda$-modules coincides with the cardinality of ${\cal P}$. 
 
  An admissible homological system  ${\cal H}=({\cal P},\leq,\{\Delta_i\}_{i\in {\cal P}})$  will be called \emph{strict} if the following condition is satisfied: 
 $$ \hbox{For each } n\geq 1, \hbox{ we have that } \Ext^n_\Lambda(\Delta_i,\Delta_j)\not=0 \hbox{ implies } i\leq j \hbox{ and } i\not\sim j.$$
 \end{definition}
 
 \begin{proposition}\label{P: admis implies H2}
 Let ${\cal H}=({\cal P},\leq,\{\Delta_i\}_{i\in {\cal P}})$ be an admissible homological system for some finite-dimensional algebra $\Lambda$. Denote by $\proj_\Lambda$ a fixed  set of representatives  of the isoclasses  of the  indecomposable projective $\Lambda$-modules. Then,  we can index $\proj_\Lambda$ with ${\cal P}$   in such a way  that each $P_i\in \proj_\Lambda$ is the projective cover of $\Delta_i$, for all $i\in {\cal P}$ and, moreover,
 there is an exact   sequence 
 $$\begin{matrix}
   0&\rightmap{}&Q_i&\rightmap{}&P_i&\rightmap{}&\Delta_i&\rightmap{}&0,
   \end{matrix}$$ 
in $\Lambda\g\mod$ such that the module $Q_i$ admits a $\Delta$-filtration with factors of the form $\Delta_j$, with $\overline{j}>\overline{i}$. 
 \end{proposition}
 
 \begin{proof} From \cite{MSX}(3.16), we know that ${\cal F}(\Delta)$ is closed under direct summands, so every indecomposable projective $\Lambda$-module $P$ belongs to ${\cal F}(\Delta)$. For each one of these indecomposable projectives $P$ we can choose a $\Delta$-filtration
 $$0=M_t\subseteq M_{t-1}\subseteq \cdots \subseteq M_1\subseteq M_0=P$$
 with $P/M_1\cong \Delta_{\sigma(P)}$, for some index $\sigma(P)\in{\cal P}$. So, there is a short exact sequence 
 $$0\rightmap{}M_1\rightmap{}P\rightmap{g}\Delta_{\sigma(P)}\rightmap{}0,$$
 so $\top(P)\cong \top(\Delta_{\sigma(P)})$ and $g$ is the projective cover of $\Delta_{\sigma(P)}$. So, given two non-isomorphic indecomposable projectives $P$ and $P'$, their tops $\top(P)$ and $\top(P')$ are not isomorphic, so the modules $\Delta_{\sigma(P)}$ and $\Delta_{\sigma(P')}$ are not isomorphic. Since the cardinality of the set of isomorphism classes of the indecomposable projectives coincides with the cardinality of ${\cal P}$, the map $P\mapsto \Delta_{\sigma(P)}$ determines a bijection between $\proj_\Lambda$ 
 and the set $\{\Delta_i\mid i\in {\cal P}\}$. Then, we can index with ${\cal P}$ the set $\proj_\Lambda$ in such a way that we have short exact sequences 
 $$0\rightmap{}M_i\rightmap{f_i}P_i\rightmap{g_i}\Delta_i\rightmap{}0,$$
 with $P_i$ indecomposable projective, for all $i\in {\cal P}$. 
 
 From \cite{MSX}(3.12) and \cite{MSX}(3.4), for each $i\in {\cal P}$, there is an exact sequence 
 $$0\rightmap{}U_i\rightmap{u_i}X_i\rightmap{v_i}\Delta_i\rightmap{}0$$
 such that:
  the module $X_i\in {\cal F}(\Delta)$ is indecomposable and satisfies  $\Ext^1_{\Lambda}(X_i,Y)=0$, for all $Y\in {\cal F}(\Delta)$; and 
   the module $U_i\in {\cal F}(\Delta)$ is filtered by modules in $\{\Delta_j \mid j\geq i \hbox{ and } j\not\sim i\}$.
 
 Since $P_i$ is projective, there is a  morphism $s_i:P_i\rightmap{}X_i$ such that $v_is_i=g_i$. Moreover, we have a pull-back diagram 
 $$\begin{matrix}
    0&\rightmap{}&M_i&\rightmap{}&E_i&\rightmap{\pi_i}&X_i&\rightmap{}&0\\
    &&\rmapdown{id_{M_i}}&&\rmapdown{h_i}&&\rmapdown{v_i}&&\\
     0&\rightmap{}&M_i&\rightmap{f_i}&P_i&\rightmap{g_i}&\Delta_i&\rightmap{}&0.\\  
   \end{matrix}$$
Since $\Ext_\Lambda^1(X_i,M_i)=0$, the exact sequence of the first row in the preceding diagram splits, and there is a morphism $\sigma_i:X_i\rightmap{}E_i$ with $\pi_i\sigma_i=id_{X_i}$. The morphism $t_i:=h_i\sigma_i:X_i\rightmap{}P_i$ satisfies $g_it_i=g_ih_i\sigma_i=v_i\pi_i\sigma_i=v_i$. Then, we have that 
$v_i=g_it_i=v_is_it_i$ and $g_i=v_i s_i=g_it_is_i$, or equivalently $v_i(id_{X_i}-s_it_i)=0$ and $g_i(id_{P_i}-t_is_i)=0$, so $id_{X_i}-s_it_i$ and $id_{P_i}-t_is_i$ are not isomorphisms. Since the endomorphism algebras of $P_i$ and $X_i$ are local, we get that $s_it_i$ and $t_is_i$ are isomorphisms. From this, we get that $s_i$ is an isomorphism and $P_i\cong X_i$. 
  \end{proof}

\begin{definition}\label{D: i-bounded} Assume that ${\cal H}=({\cal P},\leq,\{\Delta_i\}_{i\in {\cal P}})$ is an admissible homological system for a finite-dimensional $k$-algebra $\Lambda$ and keep the index assignation imposed on $\proj_\Lambda$ in (\ref{P: admis implies H2}).   
 Consider a non-zero $M\in \Lambda\g\mod$ and any projective resolution 
 $$P_M\hbox{ }:\hbox{ } \cdots\rightmap{}P_M^{-1}\rightmap{}P_M^0\rightmap{}M\rightmap{}0$$
 of $M$.  Then,  for $i\in {\cal P}$,  
 we will say that the projective resolution $P_M$ is \emph{$i$-bounded} iff it is finite and, 
 whenever $P_j$ is a direct summand of $P_M^{-t}$,  we have $\overline{j}\geq \overline{i}$ for $t\geq 0$, with  $\overline{j}> \overline{i}$ if $t>0$.
\end{definition}

 \begin{lemma}\label{L: admissible  implies P(Delta(i)) i-bounded}
Assume that ${\cal H}=({\cal P},\leq,\{\Delta_i\}_{i\in {\cal P}})$ is an admissible homological system for a finite-dimensional $k$-algebra $\Lambda$. Then, for each $i\in {\cal P}$, the $\Lambda$-module $\Delta_i$ admits an $i$-bounded minimal projective resolution $P_{\Delta_i}$ of the form 
 $$P_{\Delta_i}\hbox{ }:\hbox{ } \cdots\rightmap{}P_{\Delta_i}^{-1}\rightmap{}P_{\Delta_i}^0\rightmap{}\Delta_i\rightmap{}0.$$
\end{lemma}

\begin{proof} Notice that if $i\in {\cal P}$ is such that $\overline{i}$ is maximal in $\overline{\cal P}$, in the sequence of (\ref{P: admis implies H2}),  we get $Q_i=0$, so $\Delta_i\cong P_i$ is projective and $\Delta_i$ has a trivial $i$-bounded projective resolution $P_{\Delta_i}$ with $P_{\Delta_i}^0=P_i$. 

Observe  that given $i\in {\cal P}$ and an exact sequence $0\rightmap{}M\rightmap{}E\rightmap{}N\rightmap{}0$ in $\Lambda\g\mod$, 
where $M$ admits an $i_1$-bounded projective resolution $P_M$ and $N$ admits an $i_2$-bounded projective resolution $P_N$, with $\overline{i_1},\overline{i_2}>\overline{i}$, then the horseshoe lemma yields an $i$-bounded projective resolution  $P_E$ of $E$ 
such that: $P_j$ direct summand of $P_E^0$ implies that $\overline{j}>\overline{i}$. 

An easy induction using the last observation, shows that whenever 
$E\in {\cal  F}(\Delta)$ has a $\Delta$-filtration with factors $\{\Delta_{i_u}\}_{u\in [1,s]}$ 
 with $\overline{i_u}>\overline{i}$ for all $u\in [1,s]$, which have $i_u$-bounded  projective resolutions $\{P_{\Delta_{i_u}}\}_{u\in [1,s]}$, with $P^0_{\Delta_{i_u}}=P_{i_u}$, we have an $i$-bounded projective resolution $P_E$ of $E$ such that: $P_j$ direct summand of $P_E^0$ implies that $\overline{j}>\overline{i}$. 

Finally, fix $i\in {\cal P}$ such that $\overline{i}$ is not maximal in $\overline{\cal P}$ and assume that $\Delta_j$ admits an $i$-bounded 
 projective resolution $P_{\Delta_j}$ with $P^0_{\Delta_j}=P_j$, for all $j\in {\cal P}$ with $\overline{j}>\overline{i}$. Consider the exact sequence of (\ref{P: admis implies H2}) for $\Delta_i$. Then the preceding paragraph yields an $i$-bounded projective resolution $P_{Q_i}$ for $Q_i$, such that:  $P_j$ direct summand of $P_{Q_i}^0$ implies that $\overline{j}>\overline{i}$. Then, splicing this exact sequence with the projective resolution $P_{Q_i}$, we obtain the wanted  
 $i$-bounded projective resolution of $\Delta_i$
$$P_{\Delta_i}\hbox{ }:\hbox{ }\cdots\rightmap{}P_{Q_i}^{-1}\rightmap{}P^0_{Q_i}\rightmap{}P_i\rightmap{}\Delta_i\rightmap{}0.$$
Finally, any minimal projective resolution of $\Delta_i$ is a direct summand of $P_{\Delta_i}$ and, therefore, is also $i$-bounded.   
\end{proof}

Recall from \cite{MSX} the following.

\begin{definition} Assume that $\Lambda$ is a finite dimensional $k$-algebra, equipped with a preordered set of indexes $({\cal P},\leq)$ for the family $\{P_i\}_{i\in {\cal P}}$ of representatives of the 
 non-isomorphic indecomposable projective $\Lambda$-modules. For 
 $i\in {\cal P}$, denote by $\Delta_i$ the $i^{th}$-\emph{standard $\Lambda$-module}, that is $\Delta_i:=P_i/T_i$, where 
 $$  T_i:=     
 \sum_{\scriptsize\begin{matrix} f\in\Hom_\Lambda(P_j,P_i)\\
 \overline{j}\not\leq \overline{i} \end{matrix}}\Im(f)
 $$ 
 Then, the algebra $\Lambda$ is called a 
 \emph{prestandardly stratified algebra} if,  for each $i\in {\cal P}$, 
 there is an exact   sequence 
 $$\begin{matrix}
   0&\rightmap{}&Q_i&\rightmap{}&P_i&\rightmap{}&\Delta_i&\rightmap{}&0,
   \end{matrix}$$ 
in $\Lambda\g\mod$ such that the module $Q_i$ admits a $\Delta$-filtration with factors of the form $\Delta_j$, with $\overline{j}>\overline{i}$.
   A prestandardly stratified algebra $\Lambda$, with preordered index set $({\cal P},\leq)$ is called \emph{standardly stratified} if $({\cal P},\leq)$ is a  partial order. A standardly stratified algebra  $\Lambda$ is \emph{quasi-hereditary} iff  $({\cal P},\leq)$ is linearly ordered and $\End_\Lambda(\Delta_i)\cong k$, for each $i\in {\cal P}$. 
\end{definition}

\begin{remark} If $\Lambda$ is a prestandardly stratified algebra with preorder $({\cal P},\leq)$ and standard modules $\{\Delta_i\}_{i\in I}$, then ${\cal H}=( {\cal P},\leq,\{\Delta_i\}_{i\in I})$ is an admissible homological system, see \cite{MSX}(2.3). If, furthermore,  the algebra $\Lambda$ is quasi-hereditary, then the homological system ${\cal H}$ is strict, see \cite{M}.
\end{remark}

The following statement is interesting because it assumes no requirement on the endomorphism algebras of the standard modules. 

\begin{corollary}\label{C: pss algebra Lambda with linearly ordered poset implies H3}
 Assume that $\Lambda$ is prestandardly stratified with homological system of standard modules ${\cal H}=({\cal P},\leq,\{\Delta_i\}_{i\in {\cal P}})$, where $({\cal P},\leq)$ is linearly ordered. Then the homological system ${\cal H}$ is strict. 
\end{corollary}

\begin{proof} In the exact sequence granted by (\ref{P: admis implies H2}) for 
the standard module $\Delta_i$, we can assume that the epimorphism $P_i\rightmap{}\Delta_i$ is the canonical projection from $P_i$ onto  $\Delta_i=P_i/T_i$, where $T_i=\sum_{f \in \Hom_\Lambda(P_j,P_i); j>i}\Im(f)$.
From (\ref{L: admissible  implies P(Delta(i)) i-bounded}),  we have an 
$i$-bounded projective resolution 
$$P_{\Delta_i}\hbox{ }:\hbox{ }\cdots\rightmap{}P_{\Delta_i}^{-2}\rightmap{}P^{-1}_{\Delta_i}\rightmap{}P_i\rightmap{\nu_i}\Delta_i\rightmap{}0,$$
where $\nu_i$ is the canonical projection. 
Fix $n\in \hueca{N}$ and assume that $\Ext^n_\Lambda(\Delta_i,\Delta_j)\not=0$. This implies that   
 $\Hom_\Lambda(P_{\Delta_i}^{-n},\Delta_j)\not=0$. So, there is an indecomposable direct summand $P_s$ of $P_{\Delta_i}^{-n}$, which must satisfy $i<s$, with $\Hom_\Lambda(P_s,\Delta_j)\not=0$. Take any non-zero morphism $h:P_s\rightmap{}\Delta_j$. Since $P_s$ is projective, there is some morphism $f:P_s\rightmap{}P_j$ with $\nu_jf=h$. If $s>j$, we have $\nu_j(\Im (f))=0$, and $h=0$, which is not the case. So we have 
 $i<s\leq j$. 
\end{proof}

  The study of the category ${\cal F}(\Delta)$ in the mentioned quasi-hereditary case, and in more general settings, has been of great 
  interest in representation theory of algebras, see \cite{MSX} and its references. 
  
  In this paper we  generalize the main results of S. Koenig, J. K\"ulshammer and S. Ovsienko proved in \cite{KKO} for the quasihereditary case. Here we extend their arguments to arbitrary strict admissible homological systems. In the following lines we give the layout of the contents of this article.

   After section 2, which contains technical remarks on graded duals and tensor products, we show in section 3 that the Yoneda algebra $A$ of the $\Lambda$-module $\Delta$, associated to a strict homological system ${\cal H}$ for a finite dimensional $k$-algebra $\Lambda$, admits a strict structure of $A_\infty$-algebra, see (\ref{T: Yoneda A-infinito algebra es estricta}). 
  For this, we proceed as in \cite{KKO}, we reelaborate Keller's  argument in \cite{K1} supported by Kadeishvili Theorem (see also Merkulov \cite{Mer}), in our more general context. 
  
  In section 4, we construct the interlaced weak ditalgebra $\underline{\cal A}(\Delta)=({\cal A}(\Delta),I)$, see (\ref{D: weak ditalg interlaced an ideal}) and (\ref{D: def del bocs A(Delta)}), associated to the Yoneda $A_\infty$-algebra $A$ of $\Delta$. Throughout the whole paper we use the language of ditalgebras (diferential tensor algebras) and their categories of modules, as in \cite{BSZ}, enriched with the possibility of considering ``ditalgebras with relations'' called here interlaced weak ditalgebras and their categories of modules denoted by $({\cal A},I)\g\mod$, see (\ref{D: la cat de mod's de cal A(Delta)}) and \cite{bpsqh}. We prefer this to the alternative language of bocses, used in \cite{KKO}, because it seems to ease many computations.   
  
  In section 5, 6 and 7, we review elementary properties of $A_\infty$-categories and review the $A_\infty$-categories $\ad(A)$, $\cv(A)$, and  $\tw(A)$ with some of their relationships. All this supports the work of section 8 which exhibits an equivalence of categories ${\cal F}(\Delta)\simeq H^0(\tw(A))\simeq \underline{\cal A}(\Delta)\g\mod$. This generalizes theorem (8.2) of \cite{KKO} to the case of an arbitrary strict homological system. In their domain of study, the homological system is of standard modules over a quasi-hereditary algebra, and they obtain a directed bocs. This is no longer true in our case of a general strict homological system, but we still obtain a manageable strict interlaced weak ditalgebra $\underline{\cal A}(\Delta)$, see (\ref{D: strict weak interlaced ditalgebras}). 
  
  We study in section 10 and 11 some special kind of interlaced weak ditalgebras $\underline{\cal A}$. We equip their categories of modules $\underline{\cal A}\g\mod$  with a natural exact structure and then show that Burt-Butler theory on their right and left algebras still works, see \cite{BuBt} and \cite{Bu}.  All this is applied later to the interlaced weak ditalgebra  $\underline{\cal A}(\Delta)$. So, its right algebra $\Gamma$ admits a natural strict  homological system ${\cal H}'=({\cal P},\leq,\{\Delta'_i\}_{i\in {\cal P}})$ and, moreover, the algebra  $\Gamma$ is Morita equivalent to $\Lambda$.

  In the last section, 
  we show  that the Yoneda algebra associated to  an admissible homological system ${\cal H}$ admits a strict structure of 
   $A_\infty$-algebra, as in (\ref{R: strictness of A}),  if and only if the homological system ${\cal H}$ is strict, see (\ref{T: charact of strict Yoneda algebras}). 
   
   Moreover, generalizing the argumentation of \cite{KKO}, 
   we prove that any finite dimensional $k$-algebra $\Lambda$ with a strict homological system ${\cal H}=({\cal P},\leq,\{\Delta_i\}_{i\in {\cal P}})$ is Morita equivalent to a finite-dimensional $k$-algebra $\Gamma$ which admits a strict homological system   
  ${\cal H}'=({\cal P},\leq,\{\Delta'\}_{i\in P})$ such that $\Gamma$ has an exact Borel subalgebra $B$, which is regular and homological, as in the following definition (\ref{D: exact Borel subalgebra}). Moreover, there is an equivalence $\Omega:\Gamma\g\mod\rightmap{}\Lambda\g\mod$ such that $\Omega(\Delta'_i)\cong \Delta_i$, for all $i\in {\cal P}$. In the next lines we  give some terminology required for this definition.

  \begin{remark}\label{R: on directed algebra}
  Let $B$ be a basic finite-dimensional $k$-algebra.  
  Let us recall that \emph{the (Gabriel) quiver $Q$ of the algebra $B$} is constructed as follows. Consider a splitting $B=S\oplus J$ of the algebra $B$ over its radical $J$, then $S$ is a finite product of copies of the field $k$ and we can write $S=\bigoplus_{i\in {\cal P}}ke_i$, where $1=\sum_{i\in {\cal P}}e_i$ is a decomposition of the unit of $B$ as a finite sum of primitive orthogonal idempotents. So, we get that $\{P_i\}_{i\in {\cal P}}$, where $P_i=Be_i$, is a complete family of representatives of the isoclasses of the indecomposable projective $B$-modules, and 
  the family $\{S_i\}_{i\in {\cal P}}$, where $S_i=P_i/JP_i$ is a complete family of representatives of the isoclasses of the indecomposable simple $B$-modules. 
   Then, by definition, the set of points of $Q$ is ${\cal P}$ and, given $i,j\in {\cal P}$, there are exactly $\dim_k\Ext^1_B(S_i,S_j)$ arrows from $i$ to $j$ in $Q$.   
   
   Consider the \emph{preorder of precedence} $\preceq$ in the set  ${\cal P}$ determined by \hbox{$i\prec j$}  if  there is an arrow $i\rightmap{}j$ in $Q$ or, equivalently, if $\Ext^1_B(S_i,S_j)\not=0$. 
   
    Recall that $B$ can be identified with a quotient $k(Q)/I$ of the path algebra $k(Q)$ by some admissible ideal $I$, where the idempotents $e_i$ are identified with the classes of the trivial paths of $Q$. Then, we have for different $i,j\in {\cal P}$ that $e_jBe_i\not=0$ implies $i\prec j$.
   
    If the algebra $B$ is \emph{directed}, that is if its quiver $Q$ has no oriented cycle, then $\preceq$ is a partial order on 
    ${\cal P}$. In this case, we have $e_iBe_i=ke_i$, for all $i\in {\cal P}$.  
  \end{remark}
  
   The following lemma illustrates the preceding notions, but it is convenient to make some precisions first. 
 
 \begin{remark}\label{R: def de height map}
For any finite poset $(X,\leq)$, its \emph{height map} $h:X\rightmap{}\hueca{N}\cup \{0\}$ is defined for minimal elements $x\in X$, as $h(x)=0$, and, inductively, for any non-minimal element $x\in X$, as  
$h(x):=\hbox{max}\{h(y)\mid y<x\}+1$. 

Clearly, whenever we have $y<x$ in $X$, we get $h(y)<h(x)$. 
\end{remark}
  
  \begin{remark}\label{R: linearizations}
  A linearization $(X,\leq^L)$ of a finite poset $(X,\leq)$ is defined as follows. Consider the subsets $X_t:=\{i\in X\mid h(i)=t\}$, so $X=\biguplus_{t=1}^sX_t$, for some finite $s$. Then, consider for each $t\in [1,s]$ any linear order $\leq^t$ in the set $X_t$. Then, by definition, for any elements $i,j\in X$, we make $i\leq^Lj$ if $h(i)< h(j)$ or ($h(i)=h(j)$ and $i\leq^t j$).  
  \end{remark}

  \begin{lemma}\label{L: homological systems on directed algebras}
   Assume that $B$ is a directed finite-dimensional $k$-algebra and adopt the notations of (\ref{R: on directed algebra}). Then, the following holds.
   \begin{enumerate}
    \item The triple ${\cal H}=({\cal P},\preceq,\{S_i\}_{i\in {\cal P}})$ is a strict homological system for the directed algebra $B$.
    \item For each $i\in {\cal P}$, the $i^{th}$-standard $B$-module associated to the partially ordered set $({\cal P},\preceq)$ is $\Delta_i=S_i$. Clearly, we have $\End_B(S_i)\cong k$. 
    \item For any linearization $\preceq^L$ of ${\cal P}$ with the  partial order of precedence $\preceq$, we have that $B$ is a quasi-hereditary algebra with homological system of standard modules ${\cal H}^L=({\cal P},\preceq^L,\{S_i\}_{i\in {\cal P}})$. 
   \end{enumerate}
  \end{lemma}
  
  \begin{proof} (1): We already know that $({\cal P},\preceq)$ is a finite partially ordered set and that $\{S_i\}_{i\in {\cal P}}$ is a family of non-isomorphic indecomposable $B$-modules. Clearly, $\Hom_{B}(S_i,S_j)=0$, unless $i=j$. By definition of the order of precedence $\preceq$, we have   
that $\Ext^1_{B}(S_i,S_j)\not=0$ implies that $i\prec j$. Here, we have $\overline{\cal P}={\cal P}$ and $\Lambda\in {\cal F}(\oplus_{i\in {\cal P}}S_i)$, so  ${\cal H}$ 
 is an admissible homological system for the algebra $B$.

  Let us show that ${\cal H}$ is strict. Assume that $\Ext^n_B(S_i,S_j)\not=0$ and 
 consider the $i$-bounded projective resolution of $S_i$
$$\cdots \rightmap{} P_{S_i}^{-n-1}\rightmap{ \ d^{-n} \ }P_{S_i}^{-n}\rightmap{ \ d^{-n+1} \ } P_{S_i}^{-n+1}\rightmap{}\cdots\rightmap{}P_{S_i}^{-1}\rightmap{d^0}P_{S_i}^0\rightmap{\varepsilon}S_i\rightmap{}0$$
given by (\ref{L: admissible  implies P(Delta(i)) i-bounded}).  
Then, for each $n\geq 1$, we have 
$$P^{-n}_{S_i}\cong\bigoplus_{i\prec t}m_{i,t}P_t, \hbox{ for some } m_{i,t}\geq 0.$$
 In order to compute $\Ext^n_B(S_i,S_j)$, apply the functor $\Hom_B(-,S_j)$ to the given  projective resolution of $S_i$ and then look at the $n^{th}$-homology space, which is a quotient of the space $Z_n(P_{S_i}^{-n},S_j)\subseteq \Hom_B(P_{S_i}^{-n},S_j)$. From the given decomposition of $P_{S_i}^{-n}$, we see that this space is zero whenever $i\not\prec j$. Thus, $\Ext^n_B(S_i,S_j)\not=0$ implies $i\prec j$,  and ${\cal H}$ is a strict homological system for the algebra $B$.
   
   \medskip
   \noindent(2): Observe that if $i,j\in {\cal P}$ are incomparable
in the poset $({\cal P},\preceq)$, then we have $\Hom_B(P_j,P_i)=e_jBe_i=0$. Therefore, we get
$$ T_i=     
 \sum_{\scriptsize\begin{matrix} f\in\Hom_B(P_j,P_i)\\
 j\not\preceq i \end{matrix}}\Im(f)=\sum_{\scriptsize\begin{matrix} f\in\Hom_B(P_j,P_i)\\
 i\prec j \end{matrix}}\Im(f)=\rad(P_i),$$
 and we obtain $\Delta_i=P_i/T_i= S_i$.
 
  \medskip
   \noindent(3): By definition of the linearization $\preceq^L$, we have that $i\prec^Lj$ implies that $i\prec j$ or $i,j$ are $\preceq$-incomparable elements of ${\cal P}$.  Then, we obtain 
   $$ T_i=     
 \sum_{\scriptsize\begin{matrix} f\in\Hom_B(P_j,P_i)\\
 i\prec j \end{matrix}}\Im(f)=\sum_{\scriptsize\begin{matrix} f\in\Hom_B(P_j,P_i)\\
 i\prec^L j \end{matrix}}\Im(f)=\sum_{\scriptsize\begin{matrix} f\in\Hom_B(P_j,P_i)\\
 j\not\preceq^L i \end{matrix}}\Im(f).$$
  Then, if $\Delta'_i$ denotes the standard $B$-module associated to the linear order $\preceq^L$, we get $\Delta'_i=S_i$, 
  for all $i\in {\cal P}$.
  
  Since $i\prec j$ implies $i\prec^Lj$, the fact that ${\cal H}$ is a strict homological system for $B$ implies that ${\cal H}^L$ is a strict homological system for $B$. Hence, $B$ is a quasi-hereditary algebra. 
  \end{proof}

 The following definition adapts slightly the definition of exact Borel subalgebra given in \cite{BKK}(3.4).

  \begin{definition}\label{D: exact Borel subalgebra}
   Let $\Gamma$ be a finite-dimensional $k$-algebra with a strict homological system   
  ${\cal H}'=({\cal P},\leq,\{\Delta'_i\}_{i\in P})$. Then, a subalgebra $B$ of $\Gamma$ is an \emph{exact Borel subalgebra} if
  \begin{enumerate}
   \item The finite-dimensional $k$-algebra  $B$ is  directed  and 
   we can index a complete family $\{S_i\}_{i\in {\cal P}}$ of representatives of the isoclasses of the simple $B$-modules with the same index set ${\cal P}$. Thus (\ref{L: homological systems on directed algebras}) applies to $B$, in particular $B$ is quasi-hereditary with simple standard modules.
   \item The right $B$-module $\Gamma$ is projective.
   \item For each $i\in {\cal P}$, we have that $\Delta'_i\cong \Gamma\otimes_BS_i$. 
  \end{enumerate}
 The subalgebra $B$ is called an \emph{homological exact Borel subalgebra} if the morphisms 
 $\Ext^n_B(M,N)\rightmap{}\Ext^n_\Gamma(\Gamma\otimes_BM,\Gamma\otimes_BN)$
 induced by the tensor product functor $\Gamma\otimes_B-$ are epimorphisms for $n\geq 1$ and isomorphisms for $n\geq 2$.
 
 The subalgebra $B$ is called a  \emph{regular exact Borel subalgebra} if the  morphisms 
 $\Ext^n_B(S_i,S_j)\rightmap{}\Ext^n_\Gamma(\Gamma\otimes_BS_i,\Gamma\otimes_BS_j)$ induced by the tensor product functor $\Gamma\otimes_B-$ are isomorphisms for all $n\in \hueca{N}$ and $i,j\in {\cal P}$.
  \end{definition}

\section{Graded duals  and tensor product}  
  Let $S$ be a fixed basic finite-dimensional semisimple $k$-algebra.
  Recall the following  basic isomorphism involving left duals of $S$-$S$-bimodules. 
 
\begin{lemma}\label{L: iso de duales y tensores sin graduacion}
Given an $S\g S$-module $V$, we have the left dual $S$-$S$-bimodule $D(V):=\Hom_S(\,_SV,\,_SS)$. 
Then, given two $S$-$S$-bimodules  $V$ and $W$, where $W$ is finitely generated as a left $S$-module, 
there is a natural isomorphism of $S$-$S$-bimodules 
$$\theta_{V,W} :D(W)\otimes_S D(V)\rightmap{}D(V\otimes_SW)$$ 
satisfying that $\theta(\beta\otimes \alpha)[v\otimes w]=\alpha(v\beta(w))$,
for $\alpha\in D(V)$, $\beta\in D(W)$, $v\in V$, and $w\in W$.
Its inverse 
$\tau: D(V\otimes_S W)\rightmap{}D(W)\otimes_S D(V)$ can be described by 
$\tau(\gamma)=\sum_{i=1}^nw^*_i\otimes \gamma_i$, where $(w_i,w_i^*)_{i\in [1,n]}$ is a dual basis for the  projective left $S$-module $W$, $\gamma\in D(V\otimes_S W)$, and $\gamma_i\in D(V)$ is the morphism of left $S$-modules given by $\gamma_i(v)=\gamma(v\otimes w_i)$, for $v\in V$ 
and $i\in [1,n]$. 
\end{lemma}

\begin{proof} It is easy to see that both, $\theta$ and $\tau$ are well defined maps and that $\theta$ is a morphism of $S$-$S$-bimodules. For each  $\gamma\in D(V\otimes_S W)$ we have that 
$\theta\tau(\gamma)[v\otimes w]
=\theta[\sum_iw^*_i\otimes \gamma_i](v\otimes w)
=\sum_i\gamma_i(vw_i^*(w))
=\sum_i\gamma(vw_i^*(w)\otimes w_i)
=\gamma(v\otimes\sum_iw_i^*(w) w_i)
=\gamma(v\otimes w)$, thus $\theta\tau=id$. Moreover, we have for each $i\in [1,n]$ that 
$\theta(\beta\otimes \alpha)_i[v]=\theta(\beta\otimes \alpha)[v\otimes w_i]=\alpha(v\beta(w_i))=\beta(w_i)\alpha[v]$ and so $\theta(\beta\otimes \alpha)_i=\beta(w_i)\alpha$. Therefore, we get 
$\tau\theta[\beta\otimes \alpha]=\sum_iw_i^*\otimes \theta(\beta\otimes \alpha)_i=\sum_i w_i^*\otimes \beta(w_i)\alpha
=\sum_i w_i^*\beta(w_i)\otimes \alpha=\beta\otimes \alpha$. Then, we also have that 
$\tau\theta=id$ and $\theta$ is an isomorphism. 

For the naturality of the isomorphism, it is clear that, given morphisms of $S$-$S$-bimodules $f:V\rightmap{}V'$ and $g:W\rightmap{}W'$, the following squares commute
$$\begin{matrix}
   D(W)\otimes D(V')&\rightmap{\theta}&D(V'\otimes W)\\
   \lmapdown{id_{D(W)}\otimes D(f)}&&\lmapdown{D(f\otimes id_W)}\\
     D(W)\otimes D(V)&\rightmap{\theta}&D(V\otimes W)\\
  \end{matrix}
  \hbox{and}
  \begin{matrix}
   D(W')\otimes D(V)&\rightmap{\theta}&D(V\otimes W')\\
   \rmapdown{D(g)\otimes id_{D(V)}}&&\rmapdown{D(id_V\otimes g)}\\
     D(W)\otimes D(V)&\rightmap{\theta}&D(V\otimes W).\\
  \end{matrix}$$
\end{proof}

\begin{notation}\label{N: direcciones en bimods}
 By assumption, the $k$-algebra  $S$  is a finite product of copies of the field $k$. Consider the 
 decomposition of its unit  $1=\sum_{i\in {\cal P}} e_i$ as a sum of primitive central orthogonal idempotents. 
 
 Given any $S$-$S$-bimodule $B$, an element $z\in B$ is called \emph{directed} iff $z\in e_jBe_i$, for some $i,j\in {\cal P}$. In this case, we write $s(z)=i$ for its \emph{source point}  and $t(z)=j$ for its \emph{terminal point}. 
 \end{notation}
 
 In the following, we consider the algebra $S$ 
  as a graded $S$-$S$-bimodule concentrated at degree zero.

  \begin{definition}\label{D: D(B)} Let $B=\bigoplus_{n\in \hueca{Z}}B_n$ be a graded $S$-$S$-bimodule. 
   Then, for each $n\in \hueca{Z}$, the \emph{graded (left) dual $S$-$S$-bimodule} $D(B)$ is defined as $$D(B)=\bigoplus_{n\in \hueca{Z}}D(B)_n,$$
   where $D(B)_n:=\Hom_S(B_{-n},S)=D(B_{-n})$, for $n\in \hueca{Z}$. Thus, we have 
   $$D(B)=\bigoplus_{n\in \hueca{Z}}\Hom_S(B_{-n},S)=\bigoplus_{n\in \hueca{Z}}\Hom_{\GM\g S}^n(B,S)=\Hom_{\GM\g S}(B,S).$$
   Given a homogeneous morphism of graded $S$-$S$-bimodules $f:A\rightmap{}B$ of degree $\vert f\vert=d$, 
   we have the restriction morphism $f_n:A_{-n-d}\rightmap{}B_{-n}$ which determines a morphism of $S$-$S$-bimodules 
   $$D(f_n):\Hom_S(B_{-n},S)\rightmap{}\Hom_S(A_{-n-d},S).$$
   Hence, we have the morphism $D(f)=\bigoplus_{n\in \hueca{Z}}D(f_n)$ of graded $S$-$S$-bimodules 
   $$D(B)=\bigoplus_{n\in \hueca{Z}}\Hom_S(B_{-n},S)\rightmap{}\bigoplus_{n\in \hueca{Z}}\Hom_S(A_{-n-d},S)=D(A),$$
   which is homogeneous with degree $\vert D(f)\vert=d=\vert f\vert$. 
  \end{definition}

  \begin{remark}\label{localmente finitos y acotados}  A graded $S$-$S$-bimodule $B=\bigoplus_{n\in \hueca{Z}}B_n$ is called \emph{locally finite} iff each $B_n$ is finite-dimensional over $k$. 
    We will say that a graded $S$-$S$-bimodule $B$  is \emph{bounded below} (resp. \emph{bounded above}) if there is some $\ell_B\in \hueca{Z}$ such that $B_n=0$, for  all $n<\ell_B$ (resp. for all $n>\ell_B$).

   Notice that $D(B)$    is a graded locally finite bounded above (resp. bounded below) $S$-$S$-bimodule, whenever $B$ is a graded locally finite bounded below (resp. bounded above) $S$-$S$-bimodule. 
   
   Moreover, if $A$ and $B$ are graded locally finite bounded below (resp. bounded above) $S$-$S$-bimodules, then the graded $S$-$S$-bimodule $A\otimes_SB$ is locally finite and bounded below (resp. bounded above). Indeed, if $A$ and $B$ are bounded below by $\ell_A$ and $\ell_B$, respectively then $A\otimes_SB$ is bounded below by $\ell_A+\ell_B$: 
    Each homogeneous component $[A\otimes_S B]_n=\bigoplus_{s+t=n}A_s\otimes_S B_t$ is a finite direct sum, because $A_{n-t}\otimes_SB_t\not=0$ implies that $\ell_B\leq t\leq n-\ell_A$. If $n<\ell_A+\ell_B$ and $A_{n-t}\otimes_SB_t\not=0$ we get a contradiction. Thus, we get $[A\otimes_SB]_n=0$, for $n<\ell_A+\ell_B$, and then $A\otimes_SB$ is bounded below.
   
    We have the graded category $\Gbb\g S$ (resp. $\Gba\g S$) of graded locally finite  bounded below (resp. bounded above) $S$-$S$-bimodules, with hom set 
    $$\Hom_{\Gbb\g S}(A,B)=\bigoplus_{n\in \hueca{Z}}\Hom_{S\g S}^n(A,B),$$
    where $\Hom_{S\g S}^n(A,B)$ is the set of homogeneous morphisms of $S$-$S$-bimodules of degree $n$. Then, we have 
    the contravariant equivalence of graded categories
   $$D:\Gbb\g S\rightmap{}\Gba\g S.$$
 \end{remark}  
  
\begin{lemma}\label{L: tensor products and isomorphism tau}
Let $A,B\in \Gbb\g S$ and consider the 
 bounded below locally finite graded $S$-$S$-bimodule $A\otimes_SB$. 
 Then, we have a canonical isomorphism of graded $S$-$S$-bimodules  
$$D(B)\otimes_S D(A)\rightmap{ \ \ \theta_{A,B}  \ \ }D(A\otimes_S B)$$
 such that $\theta_{A,B}(\beta\otimes \alpha)[a\otimes b]=\alpha(a\beta(b)),$
  for any homogeneous elements $\beta\in D(B)$, 
$\alpha\in D(A)$, $a\in A$, and  $b\in B$. Its inverse is again an isomorphism of graded $S$-$S$-bimodules
$$D(A\otimes_S B)\rightmap{ \ \ \tau_{A,B}  \ \ }D(B)\otimes_S D(A).$$
\end{lemma}

\begin{proof} The morphism $\theta_{A,B}$ maps, for each $n\in \hueca{Z}$, the finite direct sum 
$$[D(B)\otimes_SD(A)]_n=\bigoplus_{r+s=n}D(B)_r\otimes_S D(A)_s=\bigoplus_{r+s=n}\Hom_S(B_{-r},S)\otimes_S\Hom_S(A_{-s},S)$$
 onto the direct sum 
 $$[D(A\otimes_SB)]_n=\Hom_S([A\otimes_SB]_{-n},S)=\bigoplus_{r+s=n}\Hom_S(A_{-s}\otimes_S B_{-r},S)$$
 with the direct sum of isomorphisms 
 $\bigoplus_{r+s=n}\theta_{A_{-s},B_{-r}}$ described in  (\ref{L: iso de duales y tensores sin graduacion}). 
The latter can be applied because $B_{-r}$ is a finitely generated left $S$-module.  
 \end{proof}

\begin{definition}\label{D: hatD}
Consider the contravariant equivalence $$\hat{D}:\Gbb\g S\rightmap{}\Gba\g S,$$
defined by $\hat{D}(A)=D(A)^{op}$, the opposite $S$-$S$-bimodule of $D(A)$. Thus, $D(A)^{op}=D(A)$ as vector spaces and, whenever $s,s'\in S$, the action of $S$ on $\hat{D}(A)$, denoted with a star $*$ is given by $s*\beta*s'=s'\beta s$, for  $\beta\in D(A)$. Moreover, $\hat{D}(f):\hat{D}(B)\rightmap{}\hat{D}(A)$ has the same recipe than $D(f)$. 

Then, for any $A,B\in \Gbb\g S$ we have an isomorphism 
of graded $S$-$S$-bimodules  
$$\hat{D}(A)\otimes_S \hat{D}(B)\rightmap{ \ \ \hat{\theta}_{A,B}  \ \ }\hat{D}(A\otimes_S B)$$
 such that $\hat{\theta}_{A,B}(\alpha\otimes \beta)[a\otimes b]=\alpha(a\beta(b)),$
  for any homogeneous elements $\beta\in \hat{D}(B)$, 
$\alpha\in \hat{D}(A)$, $a\in A$, and  $b\in B$. 

The map $\hat{\theta}_{A,B}$ is obtained as the composition
$$D(A)^{op}\otimes_S D(B)^{op}\rightmap{\zeta}(D(B)\otimes_S D(A))^{op}\rightmap{ \ \ \theta_{A,B}  \ \ }D(A\otimes_S B)^{op}$$
where $\zeta$ is the flip isomorphism $\alpha\otimes \beta\mapsto \beta\otimes \alpha$. 
\end{definition}

\begin{lemma}\label{L: naturalidad de hat(theta)}
Given two homogeneous morphisms $f:A\rightmap{}A'$ and $g:B\rightmap{}B'$  in $\Gbb\g S$, we have the homogeneous morphism $f\otimes g:A\otimes_S B\rightmap{}A'\otimes_SB'$ in $\Gbb\g S$.  
Then, the following diagram commutes in $\Gba\g S$:
$$\begin{matrix}
 \hat{D}(A')\otimes_S\hat{D}(B')&\rightmap{\hat{\theta}_{A',B'}}&\hat{D}(A'\otimes_SB')\\
 \lmapdown{\hat{D}(f)\otimes \hat{D}(g)}&&\rmapdown{(-1)^{\vert f\vert \vert g\vert}\hat{D}(f\otimes g)}\\
 \hat{D}(A)\otimes_S\hat{D}(B)&\rightmap{\hat{\theta}_{A,B}}&\hat{D}(A\otimes_SB).\\
\end{matrix}$$
\end{lemma}
 
 \begin{proof}
 Take a typical homogeneous generator $\alpha\otimes \beta$ of  $\hat{D}(A')\otimes_S\hat{D}(B')$, so $\alpha\in \hat{D}(A')$ and $\beta\in \hat{D}(B')$ are homogeneous, and take $a\otimes b\in A\otimes_S B$ a typical homogeneous generator, so $a\in A$ and $b\in B$ are homogeneous. 
 Make $\Delta_1= (-1)^{\vert f\vert \vert g\vert}[\hat{D}(f\otimes g)\hat{\theta}_{A',B'}(\alpha\otimes \beta)](a\otimes b)$, then we have 
 $$\begin{matrix}
 \Delta_1&=&(-1)^{\vert f\vert \vert g\vert}[\hat{\theta}_{A',B'}(\alpha\otimes \beta)(f\otimes g)](a\otimes b)\hfill\\
 &=&
 (-1)^{\vert f\vert \vert g\vert+\vert g\vert \vert a\vert}\hat{\theta}_{A',B'}(\alpha\otimes \beta)(f(a)\otimes g(b))\hfill\\
 &=&
 (-1)^{\vert f\vert \vert g\vert+\vert g\vert \vert a\vert}\alpha[f(a)\beta(g(b))].\hfill\\
 \end{matrix}$$
 Now make $\Delta_2:=[\hat{\theta}_{A,B}(\hat{D}(f)\otimes \hat{D}(g))(\alpha\otimes \beta)](a\otimes b)$, then we get
 $$\begin{matrix}
   \Delta_2&=&  [\hat{\theta}_{A,B}(\hat{D}(f)\otimes \hat{D}(g))(\alpha\otimes \beta)](a\otimes b)\hfill\\
   &=&
   (-1)^{\vert D(g)\vert \vert \alpha\vert}\hat{\theta}_{A,B}(\hat{D}(f)(\alpha)\otimes \hat{D}(g)(\beta))(a\otimes b)\hfill\\
   &=&
   (-1)^{\vert g\vert \vert \alpha\vert}\hat{\theta}_{A,B}(\alpha f\otimes \beta g)(a\otimes b)\hfill\\
   &=&
    (-1)^{\vert g\vert \vert \alpha\vert}\alpha f(a \beta g(b))
    =
     (-1)^{\vert g\vert \vert \alpha\vert}\alpha[ f(a \beta(g(b))].
    \hfill\\
   \end{matrix}$$

If $-\vert \alpha\vert \not= \vert f(a)\vert$, since 
$\alpha\in \Hom_S^{\vert \alpha\vert}(A',S)= \Hom_S(A'_{-\vert \alpha\vert},S)$, we obtain that $\alpha[ f(a \beta(g(b))]=0$ and,
hence, $\Delta_1=\Delta_2$.   Otherwise, we have  $-\vert \alpha\vert=
\vert f\vert +\vert a\vert$ and hence 
$$(-1)^{\vert f\vert \vert g\vert+\vert g\vert \vert a\vert}=(-1)^{\vert g\vert \vert \alpha\vert}.$$
 \end{proof}

\begin{definition}\label{L: def recursiva de hat(theta)}
 Given $A_1,\ldots,A_n\in \Gbb\g S$, the isomorphism 
 $$\hat{D}(A_1)\otimes_S\cdots \otimes_S \hat{D}(A_n)\rightmap{ \ \ \hat{\theta}_{A_1,\ldots,A_n} \ \ }\hat{D}(A_1\otimes_S \cdots \otimes_S A_n)$$
 is defined recursively, for $n\geq 3$, as the composition 
 of 
 $$\hat{D}(A_1)\otimes_S\cdots \otimes_S \hat{D}(A_n)
 \rightmap{ \ \ id_{\hat{D}(A_1)}\otimes \hat{\theta}_{A_2,\ldots,A_n} \ \ }
 \hat{D}(A_1)\otimes_S \hat{D}(A_2\otimes_S \cdots \otimes_S A_n)$$
 with 
 $$\hat{D}(A_1)\otimes_S \hat{D}(A_2\otimes_S \cdots \otimes_S A_n)
 \rightmap{ \ \ \hat{\theta}_{A_1,A_2\otimes\cdots\otimes A_n} \ \ }
 \hat{D}(A_1\otimes_S \cdots \otimes_S A_n).$$
 Moreover, for any $A\in  \Gbb\g S$, we make $\hat{\theta}_{A}:=id_{\hat{D}(A)}$, by definition. 
 
 Then, given $n\geq 1$ and a typical homogeneous generator $\alpha_1\otimes\cdots\otimes \alpha_n$ of the graded tensor product $\hat{D}(A_1)\otimes_S\cdots \otimes_S \hat{D}(A_n)$ and a typical homogeneous generator $a_1\otimes\cdots\otimes a_n$ of the graded tensor product $A_1\otimes_S \cdots \otimes_S A_n$, we have 
 $$\hat{\theta}_{A_1,\ldots,A_n}(\alpha_1\otimes\cdots\otimes \alpha_n)(a_1\otimes\cdots\otimes a_n)=\alpha_1(a_1\alpha_2(a_2(\alpha_3\cdots a_{n-1}\alpha_n(a_n))\cdots).$$
 The inverse 
 $$\hat{\tau}_{A_1,\ldots,A_n}: \hat{D}(A_1\otimes_S\cdots\otimes_S A_n)\rightmap{}\hat{D}(A_1)\otimes_S\cdots\otimes_S\hat{D}(A_n)$$
 of $\hat{\theta}_{A_1,\ldots,A_n}$ will be relevant in the following. 
\end{definition}

\begin{lemma}\label{L: asociatividad de la theta}
Given $A_1,\ldots,A_n\in  \Gbb\g S$ and an increasing sequence natural numbers $0=i_0<i_1<\cdots <i_{\ell-1}<i_\ell=n$, $\ell\geq 2$, we have the following associativity property for the isomorphisms $\hat{\theta}$. 
Make $I_{j}:=(A_{i_{j-1}+1},A_{i_{j-1}+2},\ldots,A_{i_j})$ and 
$A^{\otimes I_j}:=A_{i_{j-1}+1}\otimes A_{i_{j-1}+2}\otimes\cdots\otimes A_{i_{j}}$, for $j\in [1,\ell]$. 
Then, we have 
$$\hat{\theta}_{A^{\otimes I_1},\ldots,A^{\otimes I_\ell}}(\hat{\theta}_{I_1}\otimes\hat{\theta}_{I_2}\otimes \cdots\otimes \hat{\theta}_{I_\ell})=\hat{\theta}_{A_1,\ldots,A_n}.$$
Taking inverses the formula becomes
$$(\hat{\tau}_{I_1}\otimes\hat{\tau}_{I_2}\otimes \cdots\otimes \hat{\tau}_{I_\ell})\hat{\tau}_{A^{\otimes I_1},\ldots,A^{\otimes I_\ell}}=\hat{\tau}_{A_1,\ldots,A_n}.$$
\end{lemma}

\begin{proof} For $\ell=2$, if  $\alpha_1\otimes\cdots\otimes\alpha_n\in \hat{D}(A_1)\otimes\cdots\otimes\hat{D}(A_n)$ is a typical homogeneous generator, we have 
$$\hat{\theta}_{A^{\otimes I_1},A^{\otimes I_2}}(\hat{\theta}_{I_1}\otimes \hat{\theta}_{I_2})(\alpha_1\otimes\cdots\otimes\alpha_n)=
\hat{\theta}_{A^{\otimes I_1},A^{\otimes I_2}}(\hat{\theta}_{I_1}(\alpha_1\otimes\cdots\otimes\alpha_{i_1})\otimes \hat{\theta}_{I_2}(\alpha_{i_1+1}\otimes\cdots\otimes\alpha_n)).
$$
Evaluating at a typical homogeneous generator $a_1\otimes\cdots\otimes a_n$ we get 
$$\hat{\theta}_{I_1}(\alpha_1\otimes\cdots\otimes\alpha_{i_1})[a_1\otimes\cdots\otimes a_{i_1}
\hat{\theta}_{I_2}(\alpha_{i_1+1}\otimes\cdots\otimes\alpha_n)(a_{i_1+1}\otimes\cdots\otimes a_n)]$$
which clearly coincides with $\hat{\theta}_{A_1,\ldots,A_n}(\alpha_1\otimes\cdots\otimes\alpha_n)(a_1\otimes\cdots\otimes a_n)$.
So we can proceed by induction. Assume that the lemma holds for sequences 
with  length $\ell-1$ and let us show it for a sequence of length $\ell$.

We have that 
$\hat{\theta}_{A^{\otimes I_1},\ldots,A^{\otimes I_\ell}}(\hat{\theta}_{I_1}\otimes\hat{\theta}_{I_2}\otimes \cdots\otimes \hat{\theta}_{I_\ell})(\alpha_1\otimes\cdots\otimes\alpha_n)$ 
equals 
$$\hat{\theta}_{A^{\otimes I_1},\ldots,A^{\otimes I_\ell}}[\hat{\theta}_{I_1}(\alpha_1\otimes\cdots\otimes\alpha_{i_1})\otimes\hat{\theta}_{I_2}(\alpha_{i_1+1}\otimes\cdots\otimes\alpha_{i_2})\otimes \cdots\otimes \hat{\theta}_{I_\ell}(\alpha_{i_{\ell-1}+1}\otimes\cdots\otimes\alpha_n)].$$
Evaluating at a typical homogeneous generator $a_1\otimes\cdots\otimes a_n$, we obtain 
$$\hat{\theta}_{I_1}(\alpha_1\otimes\cdots\otimes\alpha_{i_1})[a_1\otimes\cdots\otimes a_{i_1}H]$$
where 
$H=\hat{\theta}_{A^{\otimes I_2},\ldots,A^{\otimes I_\ell}}(\hat{\theta}_{I_2}\otimes \cdots\otimes \hat{\theta}_{I_\ell})(\alpha_{i_1+1}\otimes\cdots\otimes\alpha_n)(a_{i_1+1}\otimes\cdots\otimes a_n)$. By induction hypothesis, we have   
$$H=\hat{\theta}_{A_{i_1+1},\ldots,A_n}(\alpha_{i_1+1}\otimes\cdots\otimes\alpha_n)(a_{i_1+1}\otimes\cdots\otimes a_n).$$
Then, we have the wanted formula. 
\end{proof}

\begin{remark}\label{D: de lambda n} For any natural number $n\geq 2$, and $d_1,\ldots,d_n\in \hueca{Z}$, we can define  
$$\lambda_n(d_1,\ldots,d_n)=\sum_{1\leq i< j\leq n}d_id_j.$$
So, we have $\lambda_n(d_1,\ldots,d_n)=\lambda_{n-1}(d_2,\ldots,d_n)+d_1(d_2+\cdots+d_n)$. Then, from (\ref{L: naturalidad de hat(theta)}), we obtain the following lemma by induction on the natural number $n$. 
\end{remark}

\begin{lemma}\label{L: naturalidad iterada de theta}
Given homogeneous morphisms $A_1\rightmap{f_1}B_1,\ldots,A_n\rightmap{f_n}B_n$ in $\Gbb\g S$ we have a commutative diagram 
$$\begin{matrix}
 \hat{D}(B_1)\otimes_S\cdots \otimes_S \hat{D}(B_n)&\rightmap{ \ \ \hat{\theta}_{B_1,\ldots,B_n} \ \ }&\hat{D}(B_1\otimes_S \cdots \otimes_S B_n)\\
 \lmapdown{\hat{D}(f_1)\otimes\cdots \otimes \hat{D}(f_n)}&&\lmapdown{(-1)^{\lambda_n(\vert f_1\vert,\ldots,\vert f_n\vert)}\hat{D}(f_1\otimes\cdots\otimes f_n)}\\
 \hat{D}(A_1)\otimes_S\cdots \otimes_S \hat{D}(A_n)&\rightmap{ \ \ \hat{\theta}_{A_1,\ldots,A_n} \ \ }&\hat{D}(A_1\otimes_S \cdots \otimes_S A_n).\\
\end{matrix}$$
Taking inverses, we obtain the formula:
$$(\hat{D}(f_1)\otimes\cdots \otimes \hat{D}(f_n))\hat{\tau}_{B_1,\ldots,B_n}=
(-1)^{\lambda_n(\vert f_1\vert,\ldots,\vert f_n\vert)}\hat{\tau}_{A_1,\ldots,A_n}\hat{D}(f_1\otimes\cdots\otimes f_n).$$
\end{lemma}

\begin{definition}\label{D: tau iterada n veces}
Once we fix some  $B\in  \Gbb\g S$, we can consider the following sequence of 
canonical isomorphisms:
\begin{enumerate}
\item For $n=1$, we make  $\hat{\theta}_1:=id_{\hat{D}(B)}:\hat{D}(B)\rightmap{}\hat{D}(B)$; and, 
for $n\geq 2$, we make $\hat{\theta}_n:=\hat{\theta}_{B,\ldots,B}:\hat{D}(B)^{\otimes n}\rightmap{}\hat{D}(B^{\otimes n})$.

and their inverses:

\item For $n=1$, we make  $\hat{\tau}_1:=id_{\hat{D}(B)}:\hat{D}(B)\rightmap{}\hat{D}(B)$; and, 
for $n\geq 2$, we make $\hat{\tau}_n:=\hat{\tau}_{B,\ldots,B}:\hat{D}(B^{\otimes n})\rightmap{}\hat{D}(B)^{\otimes n}$.
\end{enumerate}
\end{definition}

\begin{lemma}\label{L: la formula para manipular la diferencial hat d} 
Given  $B\in  \Gbb\g S$, non-negative integers $r,s,t,n$ with $r+s+t=n$ and $s\geq 1$, and any homogeneous morphism of graded $S$-$S$-bimodules $b_s: B^{\otimes s}\rightmap{}B$ of degree $\vert b_s\vert=1$, we have  
 the following equality of morphisms from $\hat{D}(B^{\otimes (r+1+t)})$ to $\hat{D}(B)^{\otimes n}$ 
$$(id_{\hat{D}(B)}^{\otimes r}\otimes  \hat{\tau}_s\hat{D}(b_s)\otimes id^{\otimes t}_{\hat{D}(B)})\hat{\tau}_{r+1+t}
=
\hat{\tau}_n\hat{D}(id_B^{\otimes r}\otimes b_s\otimes id_B^{\otimes t}).$$
\end{lemma}

\begin{proof} Make $A_i=B$, for all $i\in [1,n]$; let $\ell:=r+1+t$ and consider the sequence $0=i_0<i_1<\cdots <i_\ell=n$ defined by $i_j:=j$, for $j\in [1,r]$; $i_{r+1}:=r+s$; and $i_j:=j$, for $j\in [r+2,\ell]$. Then, from the associativity formula in 
(\ref{L: asociatividad de la theta}), we have 
$$(id_{\hat{D}(B)}^{\otimes r}\otimes\hat{\tau}_s\otimes  
id_{\hat{D}(B)}^{\otimes t})\hat{\tau}_{J_r,B^{\otimes s},J_t}
=\hat{\tau}_n,$$
where, for $m\geq 0$, we denote by $J_m:=(B,B,\ldots,B)$, a sequence with $m$ terms.

Now, apply the formula in (\ref{L: naturalidad iterada de theta}) to the family $f_1,\ldots,f_{r+1+t}$ defined by $f_i:=id_B$, for $i\in [1,r]\cup [r+2,r+1+t]$, and $f_{r+1}:=b_s:B^{\otimes s}\rightmap{}B$.
Thus $\vert f_{r+1}\vert=1$ 
 and $\vert f_i\vert=0$ for $i\not=r+1$. Hence $\lambda_{r+1+t}(\vert f_1\vert,\ldots,\vert f_{r+1+t}\vert)=0$, and we get 
 the formula
 $$(id_{\hat{D}(B)}^{\otimes r}\otimes\hat{D}(b_s)\otimes 
id^{\otimes t}_{\hat{D}(B)})
\hat{\tau}_{r+1+t}
=
\hat{\tau}_{J_r,B^{\otimes s},J_t}
\hat{D}(id_B^{\otimes r}\otimes b_s\otimes id_B^{\otimes t}).$$

Then, if $\Delta=(id_{\hat{D}(B)}^{\otimes r}\otimes  \hat{\tau}_s\hat{D}(b_s)\otimes id^{\otimes t}_{\hat{D}(B)})\hat{\tau}_{r+1+t}$, 
we have 
$$\begin{matrix}
\Delta
&=&
(id_{\hat{D}(B)}^{\otimes r}\otimes  \hat{\tau}_s
\otimes id^{\otimes t}_{\hat{D}(B)})  
(id_{\hat{D}(B)}^{\otimes r}\otimes \hat{D}(b_s)\otimes id^{\otimes t}_{\hat{D}(B)})
\hat{\tau}_{r+1+t}\hfill\\
&=&
 (id_{\hat{D}(B)}^{\otimes r}\otimes  \hat{\tau}_s
\otimes id^{\otimes t}_{\hat{D}(B)})  
\hat{\tau}_{J_r,B^{\otimes s},J_t}
\hat{D}(id_B^{\otimes r}\otimes b_s\otimes id_B^{\otimes t})
\hfill\\
&=&  
\hat{\tau}_n
\hat{D}(id_B^{\otimes r}\otimes b_s\otimes id_B^{\otimes t}).
\hfill\\
  \end{matrix}$$
\end{proof}

\section{The Yoneda $A_\infty$-algebra}

Assume that $S$ is a finite product of copies of the field $k$ 
and $1=\sum_{i\in {\cal P}}e_i$ is a decomposition of the unit element of $S$ as a sum of orthogonal idempotents.
We recall the following basic definition. 

 \begin{definition}\label{A-infinite algebra y idempes estrictos} 
    An \emph{$A_{\infty}$-algebra} $A$ is a graded $S$-$S$-bimodule  
   $A$, equipped with a sequence of homogeneous 
   morphisms of $S$-$S$-bimodules  
   $$\{m_n:A^{\otimes n}\rightmap{}A\}_{n\in \hueca{N}},$$
   where each $m_n$ has degree $\vert m_n\vert=2-n$, such that, for each 
   $n\in \hueca{N}$, the following \emph{Stasheff identity} holds.
   $$\begin{matrix}S_n:=\sum_{\scriptsize\begin{matrix}r+s+t=n\\ s\geq 1;r,t\geq 0\end{matrix}} 
   (-1)^{r+st}m_{r+1+t}(id^{\otimes r}\otimes m_s\otimes id^{\otimes t})=0.\end{matrix}$$  
   \end{definition}
   
   \begin{lemma}\label{L: extension  A-infinite algebra structure}
   Let $\widehat{A}$ be an \emph{augmented graded $S$-algebra}, that is the $S$-$S$-bimodule $\widehat{A}$ admits a decomposition of graded $S$-algebras $\widehat{A}=S\oplus A$ and  has $1_S$ as unit element; thus the product of $\widehat{A}$ restricts to a product on the graded $S$-algebra $A$. Then, whenever $A$ admits a structure of $A_\infty$-algebra 
   $\{m_n:A^{\otimes n}\rightmap{}A\}_{n\in \hueca{N}}$ such that $m_1=0$ and $m_2$ is induced by the product of $A$, we have that $\widehat{A}$ admits a unique structure of $A_\infty$-algebra 
   $\{\widehat{m}_n:\widehat{A}^{\otimes n}\rightmap{}\widehat{A}\}_{n\in \hueca{N}}$ such that $\widehat{m}_1=0$ and $\widehat{m}_2$ is induced by the product of $\widehat{A}$ and, moreover:   $\widehat{m}_n(z_1\otimes z_2\otimes\cdots \otimes z_n)=0$, for all $n\geq 3$ and $z_1,\ldots,z_n\in \widehat{A}$ with $z_u\in \{e_i\mid i\in {\cal P}\}$, for some $u\in [1,n]$. 
   \end{lemma}
 
   \begin{proof}  Make  $\widehat{m}_1=0$ and let  $\widehat{m}_2:\widehat{A}\otimes_S\widehat{A}\rightmap{}\widehat{A}$ be the morphism of $S$-$S$-bimodules induced by the product of $\widehat{A}$. For $n\geq 3$, we have  an $S$-$S$-bimodule decomposition $\widehat{A}^{\otimes n}=A^{\otimes n}\oplus E_n$, where $E_n$ denotes the $S$-$S$-subbimodule generated by tensors with some factor in $\{e_i\mid i\in {\cal P}\}$. Then, consider the extension  $\widehat{m}_n:\widehat{A}^{\otimes n}\rightmap{}\widehat{A}$ of $m_n$ which maps $E_n$ to zero.  Thus $\vert m_n\vert=\vert\widehat{m}_n\vert$ for all $n$.

We claim that 
$\widehat{S}_n=0$, for all $n\geq 1$, where 
$$\widehat{S}_n=\sum_{\scriptsize\begin{matrix}r+s+t=n\\ s\geq 1;r,t\geq 0\end{matrix}} 
   (-1)^{r+st}\widehat{m}_{r+1+t}(id^{\otimes r}\otimes \widehat{m}_s\otimes id^{\otimes t}).$$

Since $\widehat{m}_1=0$, we clearly have that $\widehat{S}_1=\widehat{m}_1^2=0$ and  $$\widehat{S}_2=-\widehat{m}_2(\widehat{m}_1\otimes id)-\widehat{m}_2(id\otimes \widehat{m}_1)+\widehat{m}_1\widehat{m_2}=0.$$ 

For $n\geq 3$, we consider a typical non-zero generator 
$z=z_1\otimes\cdots\otimes z_n\in \widehat{A}^{\otimes n}$, where we assume that all the elements 
$z_1,\ldots,z_n\in \widehat{A}$ are directed, see (\ref{N: direcciones en bimods}). 
Let us examine $\widehat{S}_n(z)$. If all $z_i\in A$, we have $\widehat{S}_n(z)=S_n(z)=0$. So, assume that $z_u\in \{e_i\mid i\in {\cal P}\}$, for some $u\in [1,n]$. 
Consider the index set $I=\{(r,s,t)\mid r+s+t=n; s\geq 1; r,t\geq 0\}$. Then, $\widehat{S}_n=\sum_{(r,s,t)\in I}\widehat{S}_n(r,s,t)$, where 
 $$\widehat{S}_n(r,s,t):=(-1)^{r+st}\widehat{m}_{r+1+t}(id^{\otimes r}\otimes \widehat{m}_s\otimes id^{\otimes t}).$$

We prove that $\widehat{S}_n(z)=0$ in each one of the following three possible cases. 

 \medskip
\noindent\emph{Case 1:} $u=1$. 
 \medskip
 
 In each one of the following situations, we have that $\widehat{S}_n(r,s,t)(z)=0$.
 \begin{enumerate}
  \item[1.1] If we have  $r>1$;
  \item[1.2] If we have $r=1$ and $t\not=0$;
  \item[1.3] If we have $r=0$ and ($s=1$ or $s\geq 3)$;  
 \end{enumerate}
 There are only two other possible situations:
 \begin{enumerate}
  \item[1.4] If we have $r=0$ and $s=2$. Thus, $t=n-2$ and 
 \end{enumerate}
$$\begin{matrix}
\widehat{S}_n(r,s,t)(z)&=&(-1)^{r+st}\widehat{m}_{r+1+t}(\widehat{m}_2(e_i\otimes z_2)\otimes z_3\otimes\cdots\otimes z_n)\hfill\\   
&=&
\widehat{m}_{n-1}(z_2\otimes z_3\otimes\cdots\otimes z_n).\hfill\\
 \end{matrix}$$
 \begin{enumerate}
  \item[1.5] If we have $r=1$ and $t=0$. Thus, $s=n-1$ and 
  $$\begin{matrix}
\widehat{S}_n(r,s,t)(z)&=&(-1)^{r+st}\widehat{m}_{r+1+t}(e_i\otimes\widehat{m}_s(z_2\otimes\cdots \otimes z_n))\hfill\\   
&=&
-\widehat{m}_2(e_i\otimes\widehat{m}_s(z_2\otimes z_3\otimes\cdots\otimes z_n))\hfill\\
&=&
-\widehat{m}_{n-1}(z_2\otimes z_3\otimes\cdots\otimes z_n)\hfill\\
 \end{matrix}$$
 \end{enumerate}
Therefore, we get $\widehat{S}_n(z)=\widehat{S}_n(0,2,n-2)(z)+\widehat{S}_n(1,n-1,0)(z)=0$.

 \medskip
\noindent\emph{Case 2:} $u=n$. 
 \medskip

  This case is similar to the preceding one. Indeed, in each one of the following situations, we have that $\widehat{S}_n(r,s,t)(z)=0$:
 \begin{enumerate}
  \item[2.1] If we have  $t>1$;
  \item[2.2] If we have $t=1$ and $r\not=0$;
  \item[2.3] If we have $t=0$ and ($s=1$ or $s\geq 3)$;  
 \end{enumerate}
 There are only two other possible situations:
 \begin{enumerate}
  \item[2.4] If we have $t=0$ and $s=2$. Thus, $r=n-2$ and 
 \end{enumerate}
$$\begin{matrix}
\widehat{S}_n(r,s,t)(z)&=&(-1)^{r+st}\widehat{m}_{r+1+t}(z_1\otimes \cdots\otimes z_r\otimes \widehat{m}_2(z_{n-1}\otimes e_i))\hfill\\   
&=&
(-1)^{n-2}\widehat{m}_{n-1}(z_1\otimes z_2\otimes\cdots\otimes z_{n-1}).\hfill\\
 \end{matrix}$$
 \begin{enumerate}
  \item[2.5] If we have $t=1$ and $r=0$. Thus, $s=n-1$ and 
  $$\begin{matrix}
\widehat{S}_n(r,s,t)(z)&=&(-1)^{r+st}\widehat{m}_{r+1+t}(\widehat{m}_s(z_1\otimes\cdots \otimes z_{n-1})\otimes e_i)\hfill\\   
&=&
(-1)^{n-1}\widehat{m}_2(\widehat{m}_{n-1}(z_1\otimes\cdots\otimes z_{n-1})\otimes e_i)\hfill\\
&=&
(-1)^{n-1}\widehat{m}_{n-1}(z_1\otimes z_2\otimes\cdots\otimes z_{n-1}).\hfill\\
 \end{matrix}$$
 \end{enumerate}
Therefore, we get $\widehat{S}_n(z)=\widehat{S}_n(n-2,2,0)(z)+\widehat{S}_n(0,n-1,1)(z)=0$. 

 \medskip
\noindent\emph{Case 3:} $1<u<n$. 
 \medskip
 
  Again, in each one of the following situations, we have that $\widehat{S}_n(r,s,t)(z)=0$:
 \begin{enumerate}
  \item[3.1] If we have  $r\geq u$;
  \item[3.2] If we have $t\geq n-u+1$ (so $u\in [n-t+1,n-1]$);
  \item[3.3] If we have $r<u$ and $t<n-u+1$, and ($s=1$ or $s\geq 3)$;  
 \end{enumerate}
 If we are not in these situations, we have $r<u$, $t<n-u+1$ and $s=2$. Thus, we are left with two cases: 
 \begin{enumerate}
  \item[3.4] We have $r=u-1$, and $s=2$, thus $t=n-u-1$; or  
  \item[3.5] We have $r=u-2$, and $s=2$, thus $t=n-u$.
 \end{enumerate}
 
 In case 3.4, we have that $\widehat{S}_n(r,s,t)(z)=\widehat{S}_n(u-1,2,n-u-1)(z)$ equals 
 $$(-1)^{u-1}\widehat{m}_{n-1}(z_1\otimes\cdots\otimes z_{u-1}\otimes \widehat{m}_2(e_i\otimes z_{u+1})\otimes\cdots\otimes z_n)$$
 that is 
 $(-1)^{u-1}\widehat{m}_{n-1}(z_1\otimes\cdots\otimes z_{u-1}\otimes z_{u+1}\otimes\cdots\otimes z_n)$. 

 In case 3.5, we have that $\widehat{S}_n(r,s,t)(z)=\widehat{S}_n(u-2,2,n-u)(z)$ equals 
$$
(-1)^{u-2}\widehat{m}_{n-1}(z_1\otimes \cdots\otimes z_r\otimes \widehat{m}_2(z_{u-1}\otimes e_i)\otimes z_{n-t+1}\otimes\cdots\otimes z_n)$$
that is 
$(-1)^{u-2}\widehat{m}_{n-1}(z_1\otimes\cdots\otimes z_{u-1}\otimes z_{u+1}\otimes\cdots\otimes z_n)$.
Therefore, we get 
$$\widehat{S}_n(z)=\widehat{S}_n(u-1,2,n-u-1)(z)+\widehat{S}_n(u-2,2,n-u)(z)=0.$$ 
\end{proof}

\begin{definition}\label{D: el algebra de Yoneda y el algebra cal E}
 Assume that $\{\Delta_i\}_{i\in {\cal P}}$ is a family of non-isomorphic indecomposable $\Lambda$-modules, where $\Lambda$ is a finite-dimensional $k$-algebra. Make $\Delta=\bigoplus_{i\in {\cal P}}\Delta_i$. Then,  
since $k$ is algebraically closed, there is a splitting of the endomorphism algebra $\End_\Lambda(\Delta)=S\bigoplus J$ over its radical $J$. Moreover, the semisimple subalgebra $S$ of  $\End_\Lambda(\Delta)$ can be described as $S=\bigoplus_{i\in {\cal P}} ke_i$, where  $e_i=\sigma_i\pi_i\in \End_\Lambda(\Delta)$ is the idempotent defined by the composition of the canonical projection $\pi_i:\Delta\rightmap{}\Delta_i$ and the canonical injection $\sigma_i:\Delta_i\rightmap{}\Delta$.

\emph{The Yoneda algebra of $\Delta$} is the graded $S$-algebra $$A=\bigoplus_{n=0}^\infty A_n, \hbox{ where } A_n=\Ext_\Lambda^n(\Delta,\Delta), \hbox{ for }n\geq 0.$$ 
\end{definition}

\begin{definition}\label{D: the graded endomorphism algebra of PDelta}
Assume, as before, that $\{\Delta_i\}_{i\in {\cal P}}$ is a family of non-isomorphic indecomposable $\Lambda$-modules, where $\Lambda$ is a finite-dimensional $k$-algebra. Make $\Delta=\bigoplus_{i\in {\cal P}}\Delta_i$.
Let $P_{\Delta_i}$ be a fixed projective resolution of $\Delta_i$, for each $i\in {\cal P}$, and consider them as graded differential $\Lambda$-modules, where $\Lambda$ is considered as a differential graded $k$-algebra concentrated at degree 0. Make $P_\Delta:=\bigoplus_{i\in {\cal P}}P_{\Delta_i}$ and 
$${\cal E}={\cal E}(P_\Delta):=\bigoplus_{n\in \hueca{Z}}\End^n_\Lambda(P_\Delta),$$
where $\End_\Lambda^n(P_\Delta)$ is the space of homogeneous morphisms $f:P_\Delta\rightmap{}P_\Delta$ of graded $\Lambda$-modules with degree $\vert f\vert=n$; it can be identified with $\Hom^0_\Lambda(P_\Delta,P_\Delta[n])$. 

The $k$-algebra ${\cal E}$ admits a differential $d:{\cal E}\rightmap{}{\cal E}$ given on homogeneous elements $f\in {\cal E}$ by the formula $d(f)=d_{P_\Delta}f-(-1)^{\vert f\vert}fd_{P_\Delta}$. So we have a differential graded $k$-algebra ${\cal E}$. 

For each $i\in {\cal P}$, denote by $\pi^{*}_i:P_\Delta\rightmap{}P_{\Delta_i}$ and 
$\sigma^{*}_i:P_{\Delta_i}\rightmap{}P_\Delta$ the projection and the injection, respectively, associated to the direct sum $P_\Delta=\bigoplus_{i\in {\cal P}}P_{\Delta_i}$, and make $e^{*}_i:=\sigma^{*}_i\pi_i^{*}$. Then, we obtain a decomposition $id_{P_\Delta}=\sum_{i\in {\cal P}}e_i^{*}$ of the unit of the algebra ${\cal E}$ as a sum of orthogonal idempotents. 

There is a canonical isomorphism from $S$ onto  
$\bigoplus_{i\in {\cal P}}ke_i^*$, and since $d(e_i^*)=0$, for all $i\in {\cal P}$, the \emph{differential graded endomorphism algebra ${\cal E}$ of $P_\Delta$} is in fact a differential graded  $S$-algebra. 
\end{definition}

\begin{remark}\label{R: el iso canónico cal E --> Yoneda} 
For any family $\{\Delta_i\}_{i\in {\cal P}}$ as above, 
 a well known result of homological algebra states that there is an isomorphism of graded $k$-algebras 
$$\Phi:H({\cal E}(P_\Delta))\rightmap{} \bigoplus_{n\in \hueca{Z}}\Ext_\Lambda^n(\Delta,\Delta),$$
where we take $\Ext^n_\Lambda(\Delta,\Delta):=0$, for $n<0$. 

Moreover, it is easy to see that $\Phi(\overline{e_i^{*}})=e_i$, for all $i\in {\cal P}$, where $\overline{e_i^{*}}$ denotes the class of the $0$-cycle $e_i^{*}$ in the homology graded $S$-algebra $H({\cal E}(P_\Delta))$.  
\end{remark}

\begin{proposition}\label{D: descompos homologica para cal E}
Assume that ${\cal H}=({\cal P},\leq,\{\Delta_i\}_{i\in {\cal P}})$ is a strict homological system for the $k$-algebra $\Lambda$. 
Consider a fixed $i$-bounded projective resolution $P_{\Delta_i}$ of $\Delta_i$, for each $i\in {\cal P}$, as in (\ref{L: admissible  implies P(Delta(i)) i-bounded}), and adopt the notation of the preceding definition (\ref{D: the graded endomorphism algebra of PDelta}). The following holds.
\begin{enumerate}
 \item We have a $d$-invariant decomposition ${\cal E}={\cal E}'\oplus {\cal E}''$ of graded $S$-$S$-bimodules, 
  where 
$${\cal E}'=\bigoplus_{\scriptsize\begin{matrix}i,j\in {\cal P}\\
\overline{i}\leq \overline{j}\\ \end{matrix}}
 e_j^*{\cal E}e_i^*\hbox{  \ and  \ } 
{\cal E}''=\bigoplus_{\scriptsize\begin{matrix}i,j\in {\cal P}\\
\overline{i}\not\leq \overline{j}\\ \end{matrix}}e_j^*{\cal E}e_i^*.$$
 Moreover, ${\cal E}'$ is a differential graded $S$-subalgebra of ${\cal E}$ and $H({\cal E}'')=0$. 
\item We have a $d$-invariant decomposition
${\cal E}'=[\bigoplus_{i\in {\cal P}}ke_i^{*}]\oplus {\cal R}'$ of graded $S$-$S$-bimodules, 
where ${\cal R}'$ is a differential graded $S$-subalgebra of  ${\cal E}'$. Therefore, $$H({\cal E})=H({\cal E}')=[\bigoplus_{i\in {\cal P}}k\overline{e_i^*}]\oplus H({\cal R}').$$
\end{enumerate}
\end{proposition}

\begin{proof}(1): We will use the graded isomorphism $\Phi$ of (\ref{R: el iso canónico cal E --> Yoneda}). Clearly, there is a decomposition as in (1)  of graded vector spaces. 
Since $e_j\Ext_\Lambda^n(\Delta,\Delta)e_i=\Ext_\Lambda^n(\Delta_i,\Delta_j)=0$, for all $n\geq 0$ and all $\overline{i}\not\leq\overline{j}$, the restriction of the isomorphism $\Phi$ gives  
$$H({\cal E}'')\cong \bigoplus_{\scriptsize\begin{matrix}i,j\in {\cal P}\\
\overline{i}\not\leq \overline{j}\\ \end{matrix}}e_j\Ext^n_\Lambda(\Delta,\Delta)e_i=0.$$
Moreover, the spaces ${\cal E}'$ and ${\cal E}''$ are invariant under the differential, because $d(e_i^{*})=0$, for all $i\in {\cal P}$.

(2) 
In order to describe the ideal ${\cal R}'$ of ${\cal E}'$, 
observe that we have the following isomorphisms for $i,j\in {\cal P}$:
$$ \phi_{i,j}^n:e^*_j\Hom^n_\Lambda(P_\Delta,P_\Delta)e^*_i\rightmap{}  \Hom^n_\Lambda(P_{\Delta_i},P_{\Delta_j}),\hbox{ given by }f\mapsto f_{i,j}:=\pi_j^*f\sigma^*_i,$$
and
$$\psi_{i,j}:\bigoplus_{n\in \hueca{Z}}\Hom^n_\Lambda(P_{\Delta_i},P_{\Delta_j})\cong \bigoplus_{s,n\in \hueca{Z}}\Hom_\Lambda(P^s_{\Delta_i},P^{s+n}_{\Delta_j})=
\bigoplus_{s,t\in \hueca{Z}}\Hom_\Lambda(P^s_{\Delta_i},P^t_{\Delta_j})$$
given, for $f_{i,j}\in \Hom^n_\Lambda(P_{\Delta_i},P_{\Delta_j})$, by $\psi_{i,j}(f_{i,j})=\sum_{s}f_{i,j,s,s+n}$,   
where $f_{i,j,s,s+n}:P^s_{\Delta_i}\rightmap{}P^{s+n}_{\Delta_j}$ is the restriction of $f_{i,j}$. Then, we obtain the following
sequence of isomorphisms
$$\begin{matrix}
  {\cal E}'&=&
  \bigoplus_{\scriptsize\begin{matrix}i,j\in {\cal P}\\
\overline{i}\leq \overline{j}\\ \end{matrix}}
  e_j^*[ \bigoplus_{n\in \hueca{Z}}\Hom^n_\Lambda(P_\Delta,P_\Delta)]e_i^*\hfill\\
  &=&
  \bigoplus_{\scriptsize\begin{matrix}i,j\in {\cal P}\\
\overline{i}\leq \overline{j}\\ \end{matrix}}
  \bigoplus_{n\in \hueca{Z}}e_j^*\Hom^n_\Lambda(P_\Delta,P_\Delta)e_i^*\hfill\\
   &\cong&
  \bigoplus_{\scriptsize\begin{matrix}i,j\in {\cal P}\\
\overline{i}\leq \overline{j}\\ \end{matrix}}
  \bigoplus_{n\in \hueca{Z}}\Hom^n_\Lambda(P_{\Delta_i},P_{\Delta_j}),\hfill\\
  &\cong&
  \bigoplus_{\scriptsize\begin{matrix}i,j\in {\cal P}\\
\overline{i}\leq \overline{j}\\ \end{matrix}}
  \bigoplus_{s,t\in \hueca{Z}}\Hom_\Lambda(P^s_{\Delta_i},P^t_{\Delta_j})=:\underline{\cal E}'.\hfill\\
  \end{matrix}$$
 Denote by $\phi=\bigoplus_{i,j,n} \phi^n_{i,j}$ the first isomorphism and by $\psi=\bigoplus_{i,j}\psi_{i,j}$ the second one. Then, we have that 
 $\phi(e_i^*)=id_{P_{\Delta_i}}\in \Hom^0_\Lambda(P_{\Delta_i},P_{\Delta_i})$, thus we have 
 $\psi\phi(e_i^*)=\sum_s id_{P^s_{\Delta_i}}\in \bigoplus_{s\in \hueca{Z}}\Hom_\Lambda(P^s_{\Delta_i},P^s_{\Delta_i})$. 
 
 The space $\underline{\cal E}'$ has a natural structure of $k$-algebra, given typical generators 
 $$P_{\Delta_i}^s\rightmap{ \ f_{i,u,s,r} \ }P_{\Delta_u}^r 
 \hbox{ and } 
 P^l_{\Delta_v}\rightmap{ \ g_{v,j,l,t} \ }P^t_{\Delta_j},$$
 the product $g_{v,j,l,t}\cdot f_{i,u,s,r}$ is $g_{u,j,r,t}\circ f_{i,u,s,r}$ if $v=u$ and $l=r$, and it is zero otherwise. 
 The morphism $\psi\phi:{\cal E}'\rightmap{} \underline{\cal E}'$ is in fact an isomorphism of $k$-algebras. Indeed, 
 given $f\in e^*_t\Hom^n_\Lambda(P_\Delta,P_\Delta)e^*_i$
and $g\in e^*_j \Hom^m_\Lambda(P_\Delta,P_\Delta)e^*_t$, we have 
$(gf)_{i,j}=\sum_{t'}g_{t',j}f_{i,t'}=g_{t,j}f_{i,t}\in \Hom^{n+m}_\Lambda(P_{\Delta_i},P_{\Delta_j})$, so  
$$\begin{matrix}
   \psi\phi(gf)=\psi((gf)_{i,j})
   &=&
   \psi(g_{t,j}f_{i,t})\hfill\\
   &=&
   \sum_s(g_{t,j}f_{i,t})_{s,s+n+m}\hfill\\
   &=&
   \sum_sg_{t,j,s+n,s+n+m}f_{i,t,s,s+n}.\hfill\\
  \end{matrix}$$
Therefore, 
$$\begin{matrix}\psi\phi(g)\psi\phi(f)
  &=&
  \psi(g_{t,j})\psi(f_{i,t})\hfill\\
  &=&
  (\sum_rg_{t,j,r,r+m})(\sum_sf_{i,t,s,s+n})\hfill\\
  &=&
  \sum_sg_{t,j,s+n,s+n+m}f_{i,t,s,s+n}
  =
  \psi\phi(gf).\hfill\\
  
  \end{matrix}$$

 Since $k$ is algebraically closed and $P^0_{\Delta_i}\cong P_i$, its endomorphism algebra splits over its radical, so we obtain  $$\Hom_\Lambda(P^0_{\Delta_i},P^0_{\Delta_i})=kid_{P^0_{\Delta_i}}\oplus \rad_\Lambda(P^0_{\Delta_i},P^0_{\Delta_i}),$$ where $\rad_\Lambda(P^0_{\Delta_i},P^0_{\Delta_i})$ consists of the non-isomorphisms of 
 $\Hom_\Lambda(P^0_{\Delta_i},P^0_{\Delta_i})$. 

  Now consider the linear subspace ${\cal R}'=(\psi\phi)^{-1}(\underline{\cal R}')$ of ${\cal E}'$, where 
  $\underline{\cal R}'$ is the linear subspace of $\underline{\cal E}'$ generated by the morphisms of $\Lambda$-modules  
  $f:P^s_{\Delta_i}\rightmap{}P^t_{\Delta_j}$ such that $s\not=0$ or $t\not=0$, or, if $s=0=t$, then $f$ is not an isomorphism. 
  
 In order to show that ${\cal R}'$ is an ideal of ${\cal E}'$, take any typical linear generator $g:P^s_{\Delta_i}\rightmap{}P^t_{\Delta_j}$ of $\underline{\cal R}'$ and any typical generator $f:P^t_{\Delta_j}\rightmap{}P_{\Delta_u}^r$ of $\underline{\cal E}'$ with possible non-zero product $fg:P_{\Delta_i}^s\rightmap{}P_{\Delta_u}^r$.  If $s\not=0$ or $r\not=0$, we get $fg\in \underline{\cal R}'$. So assume that $s=0=r$ and that $fg$ is an isomorphism. Then, $P_i=P^0_{\Delta_i}$ is a direct summand of 
 $P^t_{\Delta_j}$. Since $P_{\Delta_j}$ is a $j$-bounded minimal projective resolution, if $t<0$, we get that 
 $\overline{i}>\overline{j}$, which is not the case because $g\in \underline{\cal E}'$ implies that $\overline{i}\leq \overline{j}$. Thus $t=0$, so 
 $P_i\cong P_j$, and we obtain $i=j$. Then, the composition $fg$ belongs to the radical 
 $\rad_\Lambda(P_i,P_i)$, because $g$ does, a contradiction.  
  Hence $fg$ is not an isomorphism and, so, $fg\in \underline{\cal R}'$. 
 This shows that ${\cal R}'$ is a left ideal of ${\cal E}'$. Similarly, one shows that ${\cal R}'$ is a right ideal of ${\cal E}'$. 
 
 Let us show that ${\cal E}'=[\sum_{i\in {\cal P}}ke^*_i]+{\cal R}'$. 
 Any $f=\sum_{i,j,s,t}f_{i,j,s,t}\in \underline{\cal E}'$, can be written 
 as follows:
 $$f=\sum_{s\not=0 \hbox{ \scriptsize{or} } t\not=0}f_{i,j,s,t}+\sum_{i\not=j}f_{i,j,0,0}+\sum_ic_i^fid_{P^0_{\Delta_i}}+\sum_i\rho_i^f,$$
 where $f_{i,i,0,0}=c^f_iid_{P_{\Delta_i}^0}+\rho_i^f\in kid_{P^0_{\Delta_i}}\oplus\rad_\Lambda(P^0_{\Delta_i},P^0_{\Delta_i})$. Moreover, we have 
 $$\psi\phi(e_i^*)=\sum_sid_{P^s_{\Delta_i}}=id_{P^0_{\Delta_i}}-\rho_i,\hbox{ \ where \  }\rho_i=\sum_{s<0}id_{P^s_{\Delta_i}}.$$
 Then, in the preceding expression for $f$, we can replace 
 $$\sum_ic_i^fid_{P^0_{\Delta_i}} \hbox{ \ \ by \ \ } \sum_ic_i^f \psi\phi(e^*_i)+\sum_ic_i^f\rho_i,$$
 to obtain that $f=\sum_ic_i^f \psi\phi(e^*_i)+\sum_ic_i^f\widehat{\rho}_i$, for some $\widehat{\rho}_i\in \underline{\cal R}'$. Therefore, we obtain ${\cal E}'=[\sum_{i\in {\cal P}}ke^*_i]+{\cal R}'$. 
 
 If $\sum_ic_ie_i^*=0$, we get $\psi\phi(\sum_ic_ie_i^*)=\sum_i\sum_sc_iid_{P^s_{\Delta_i}}$, a  combination of linearly independent elements in $\underline{\cal E}'$, thus all the coefficient are zero and 
 $\sum_{i\in {\cal P}}ke^*_i=\bigoplus_{i\in {\cal P}}ke^*_i$. 
 Moreover, it is clear that $[\bigoplus_{i\in {\cal P}}kid_{P^0_{\Delta_i}}]\cap \underline{\cal R}'=0$, thus if $\sum_ic_ie_i^*\in {\cal R}'$, we obtain that 
 $\sum_ic_iid_{P^0_{\Delta_i}}-\sum_ic_i\rho_i=\psi\phi(\sum_ic_ie_i^*)\in \underline{\cal R}'$, thus $\sum_ic_iid_{P^0_{\Delta_i}}\in \underline{\cal R}'$ and all the coefficients are zero, so we finally get ${\cal E}'=[\bigoplus_{i\in {\cal P}}ke^*_i]\oplus {\cal R}'$ as we wanted. 
 
 We claim that ${\cal R}'$ is a $d$-invariant subalgebra of ${\cal E}'$. Notice first that $\phi(d_{P_\Delta})=\sum_id_{P_{\Delta_i}}$, thus 
 $\psi\phi(d_{P_\Delta})=\sum_i\psi(d_{P_{\Delta_i}})=\sum_{i,s}d^s_{P_{\Delta_i}}$. Now, take any typical generator $f=f_{i,j,r,s}\in 
 \underline{\cal E}'$, so $f'=(\psi\phi)^{-1}(f)$ is a typical generator of ${\cal E}'$, which has degree $n:=s-r$. Then we have 
 $$\begin{matrix}
  \psi\phi(d(f'))&=&
  \psi\phi(d_{P_\Delta}f'-(-1)^nf'd_{P_\Delta})\hfill\\
  &=&
   \psi\phi(d_{P_\Delta})f-(-1)^nf\psi\phi(d_{P_\Delta})\hfill\\
  &=&
  \sum_{u,a}d^a_{P_{\Delta_u}}f-(-1)^n\sum_{v,b}fd^b_{P_{\Delta_v}}\hfill\\
  &=&
  d^s_{P_{\Delta_j}}f-(-1)^nfd^{r-1}_{P_{\Delta_i}}.\hfill\\
 \end{matrix}$$
 These maps are disposed as follows
 $$\begin{matrix}P_{\Delta_i}^{r-1}&
\rightmap{ \ d^{r-1}_{P_{\Delta_i}} \ }&P^r_{\Delta_i}&\rightmap{ \ f \ }&P_{\Delta_j}^s&
\rightmap{ \ d^s_{P_{\Delta_j}} \ }&P^{s+1}_{\Delta_j}.\end{matrix}$$
Let us examine $d^s_{P_{\Delta_j}}f$. If $r\not=0$ or $s+1\not=0$, then $d^s_{P_{\Delta_j}}f\in \underline{\cal R}'$. If $r=0$ and $s=-1$, and $d^s_{P_{\Delta_j}}f$ is an isomorphism, we get that $P_i$ is a direct summand of $P^{-1}_{\Delta_j}$ and, since $P_{\Delta_j}$ is a $j$-bounded 
 projective resolution, we obtain $\overline{i}>\overline{j}$, which is not the case. Thus $d^s_{P_{\Delta_j}}f$ is not an isomorphism and belongs to $\underline{\cal R}'$. Similarly, one shows that $ f d^{r-1}_{P_{\Delta_i}}\in \underline{\cal R}'$. Therefore, we get $d(f')\in {\cal R}'$. So, we finally obtain that $d({\cal R}')\subseteq {\cal R}'$.  

Then, we have a $d$-invariant decomposition ${\cal E}'=[\bigoplus_{i\in {\cal P}}ke_i^*]\oplus {\cal R}'$, where ${\cal R'}$ is a differential graded $k$-algebra and then $Z({\cal E}')=[\bigoplus_{i\in {\cal P}}ke_i^*]\oplus Z({\cal R}')$ and $B({\cal E}')=B({\cal R'})$, thus 
$$H({\cal E})=H({\cal E}')=[\bigoplus_{i\in {\cal P}}k\overline{e_i^*}]\oplus H({\cal R}').$$
Notice that if we identify $S$ with $\bigoplus_{i\in {\cal P}}ke_i^*$, we obtain that ${\cal R}'$ is differential graded (without unit) $S$-algebra. 
\end{proof}

The following theorem generalizes the argument of \cite{K1}(3.5), which relies on a theorem by Kadeishvili, see  \cite{K1}(3.3), \cite{Kad} and \cite{Mer}.

\begin{theorem}\label{T: Yoneda A-infinito algebra es estricta}
 Assume that ${\cal H}=({\cal P},\leq,\{\Delta_i\}_{i\in {\cal P}})$ is a strict homological system for the $k$-algebra $\Lambda$. Then, the Yoneda graded $S$-algebra $A$ 
 of the $\Lambda$-module $\Delta=\bigoplus_{i\in {\cal P}}\Delta_i$ admits a structure of $A_\infty$-algebra with $S$-$S$-bimodule higher multiplications $\{m_n:A^{\otimes n}\rightmap{}A\}_{n\in \hueca{N}}$ such that $m_1=0$, $m_2$ is induced by the product of the Yoneda algebra and, if $n\geq 3$, they satisfy that $m_n(z_1\otimes z_2\otimes\cdots \otimes z_n)=0$, whenever  $z_i\in \{e_j\mid j\in {\cal P}\}$, for some $i\in [1,n]$. 
\end{theorem}

\begin{proof} Adopt the notation of the preceding Proposition. Apply Kadeishvili theorem to the differential graded $S$-algebra ${\cal R}'$, to obtain a structure $\{m_n\}_{n\in \hueca{N}}$ of $A_\infty$-algebra on the graded homology $S$-algebra $H({\cal R}')$, such that $m_1=0$ and $m_2$ is induced by the product of $H({\cal R}')$. 

Then, apply (\ref{L: extension  A-infinite algebra structure}) to the augmented graded $S$-algebra 
$H({\cal E}')=[\oplus_{i\in {\cal P}}k\overline{e_i^*}]\oplus H({\cal R}')$, to obtain a structure  of $A_\infty$-algebra $\{\widehat{m}_n\}_{n\in \hueca{N}}$ on $H({\cal E}')$ such that $\widehat{m}_1=0$, $\widehat{m}_2$ is induced by the product of $H({\cal E}')$, and $\widehat{m}_n(z_1\otimes\cdots\otimes z_n)=0$, for $n\geq 3$ and $z_1,\ldots,z_n\in H({\cal E}')$, whenever 
$z_i\in \{\overline{e_j^*}\mid j\in {\cal P}\}$, for some $i\in [1,n]$. 

  Finally, we can derive the statement of the theorem by transfering the structure of $A_\infty$-algebra of $H({\cal E})=H({\cal E}')$  onto the Yoneda 
   algebra $A$ with the help of the graded isomorphism $\Phi:H({\cal E})\rightmap{}A$ mentioned in (\ref{R: el iso canónico cal E --> Yoneda}). 
\end{proof}

 \begin{remark}\label{R: strictness of A}
We can consider the Yoneda algebra $A$ associated to any admissible homological system  ${\cal H}=({\cal P},\leq,\{\Delta_i\}_{i\in {\cal P}})$  for a finite-dimensional $k$-algebra $\Lambda$. Assume that $A$ admits a structure of $A_\infty$-algebra with higher multiplications $\{m_n:A^{\otimes n}\rightmap{}A\}_{n\in \hueca{N}}$. 
 Notice that, for $n\geq 3$, the following are equivalent:
\begin{enumerate}
 \item $m_n(z_1\otimes z_2\otimes\cdots \otimes z_n)=0$, whenever  $z_i\in \{e_j\mid j\in {\cal P}\}$, for some $i\in [1,n]$;
 \item $m_n(z_1\otimes z_2\otimes\cdots \otimes z_n)=0$, whenever  $z_i=id_\Delta$, for some $i\in [1,n]$.
\end{enumerate}
We will say that the $A_\infty$-algebra $A$ \emph{has a strict unit} or that it \emph{has a strict structure of $A_\infty$-algebra} (or even that $A$ is \emph{strict}) if $m_1=0$, $m_2$ is induced by the product of the Yoneda algebra, and  the preceding property of  the higher multiplications of $A$ is satisfied. 
\end{remark}

\section{Algebras linked to strict Yoneda $A_\infty$-algebras}
From now on, unless otherwise specified, we will keep the notation of the following remarks.

\begin{remark}\label{R: merkulov}
 We denote by  ${\cal H}=({\cal P},\leq,\{\Delta_i\}_{i\in {\cal P}})$
 a fixed admissible homological system for  a given finite-dimensional $k$-algebra $\Lambda$, as in (\ref{D: admissible hs and strict hs}). Consider the $\Lambda$-module $\Delta=\bigoplus_{i\in {\cal P}}\Delta_i$. Then, we have a decomposition $\End_\Lambda(\Delta)
 =S\bigoplus J$, where $J$ is the Jacobson radical of $\End_{\Lambda}(\Delta)$ and $S=\bigoplus_{i\in {\cal P}}ke_i$, where $\{e_i\}_{i\in {\cal P}}$ are the canonical central idempotents of $S$ considered in (\ref{D: el algebra de Yoneda y el algebra cal E}).
 
 For the study of the category ${\cal F}(\Delta)$, the Yoneda algebra  
 $$A=\bigoplus_{n\geq 0}A_n,\hbox{  \ where  \ } A_n=\Ext_\Lambda^n(\Delta,\Delta), \hbox{ for } n\geq 0,$$ 
 of $\Delta$ is an important tool.
 From now on, we assume that the Yoneda algebra $A$  
  admits a strict structure $\{m_n:A^{\otimes n}\rightmap{}A\}_{n\in \hueca{N}}$ of $A_\infty$-algebra as in (\ref{R: strictness of A}), thus $m_1=0$, $m_2$ is induced by the product of $A$, and, for $n\geq 3$, we have $m_n(z_1\otimes\cdots\otimes z_n)=0$, whenever $z_u\in \{e_i\mid i\in {\cal P}\}$ for some $u\in [1,n]$. 
  
  From (\ref{T: Yoneda A-infinito algebra es estricta}), we know that this is the case whenever ${\cal H}$ is a strict homological system. We will prove in  (\ref{T: charact of strict Yoneda algebras}) that this is the only possible case. 
\end{remark}

\begin{definition}\label{D: construccion barra} We refer to the preceding $A_\infty$-algebra $A$ as \emph{the Yoneda $A_\infty$-algebra of $\Delta$}.  
Now, we consider the \emph{bar construction $B$} of this $A_\infty$-algebra. That is we consider  the shifted graded $S$-$S$-bimodule $B:=A[1]$,  endowed with the sequence $\{b_n:B^{\otimes n}\rightmap{}B\}_{n\in \hueca{N}}$ of $S$-$S$-bimodule morphisms given by the following commutative squares
 $$\begin{matrix}A^{\otimes n}&\rightmap{\sigma^{\otimes n}}&B^{\otimes n}\\
     \lmapdown{m_n}&&\rmapdown{b_n}\\
  A&\rightmap{\sigma}&B,\\
  \end{matrix}$$
 where $\sigma_n:A\rightmap{}A[1]=B$ is the identity map, considered as a homogeneous morphism of graded $S$-$S$-bimodules of degree $-1$. Thus, each $b_n:B^{\otimes n}\rightmap{}B$ is a homogeneous morphism of degree $\vert b_n\vert=1$, and $b_1=0$. Then, the  
 family $\{b_n:B^{\otimes n}\rightmap{}B\}_{n\in \hueca{N}}$
satisfies the basic relations 
 $$\sum_{\scriptsize\begin{matrix}r+s+t=n\\ r,t\geq 0; s\geq 1\end{matrix}}
 b_{r+1+t}(id^{\otimes r}\otimes b_{s}\otimes id^{\otimes t})=0, 
 \hbox{ for all } n\in  \hueca{N}.$$
\end{definition}

\begin{remark}\label{R: B y D(B) homogeneos} Adopt the preceding notation.
\begin{enumerate}
 \item The homogeneous summands of the graded $S$-$S$-bimodule $B\in  \Gbb\g S$  are
 $$\begin{matrix}
  B_j=A_{j+1}=0, \hbox{ for } j<-1;\hfill\\
  B_{-1}=A_0=\End_\Lambda(\Delta);\hfill\\
  B_0=A_1=\Ext^1_\Lambda(\Delta,\Delta);\hfill\\
  B_j=A_{j+1}=\Ext^{j+1}_\Lambda(\Delta,\Delta), \hbox{ for } j\geq 1.\hfill\\
 \end{matrix}$$
 \item For the graded dual $S$-$S$-bimodule $\hat{D}(B)\in \Gba\g S$ we have: 
 $$\begin{matrix}
  \hat{D}(B)_j=\hat{D}(B_{-j})=0, \hbox{ for } j>1;\hfill\\
  \hat{D}(B)_1=\hat{D}(B_{-1})=\hat{D}(\End_\Lambda(\Delta));\hfill\\
  \hat{D}(B)_0=\hat{D}(B_0)=\hat{D}(\Ext^1_\Lambda(\Delta,\Delta));\hfill\\
  \hat{D}(B)_{-j}=\hat{D}(B_j)=\hat{D}(\Ext^{j+1}_\Lambda(\Delta,\Delta)), \hbox{ for } j> 0.\hfill\\
 \end{matrix}$$
 \item For $i,j\in {\cal P}$, we have  
 $e_jB_{-1}e_i=e_j\Hom_\Lambda(\Delta,\Delta)e_i\cong\Hom_\Lambda(\Delta_i,\Delta_j)$, hence 
 $$e_j*\hat{D}(B)_1*e_i=e_iD(B)_1e_j=e_iD(B_{-1})e_j\cong D(e_jB_{-1}e_i)=D\Hom_\Lambda(\Delta_i,\Delta_j).$$
 Similarly, we have 
 $$\begin{matrix}e_j*\hat{D}(B)_0*e_i&=&e_iD(B)_0e_j\hfill
 &=&D(e_jB_0e_i)\hfill\\
 &=& D(e_j\Ext^1_\Lambda(\Delta,\Delta)e_i)&\cong& D(\Ext^1_\Lambda(\Delta_i,\Delta_j)).\hfill\\
\end{matrix}$$  
 From the definition of homological system (\ref{D: homological system y F(Delta)}), we have 
 $$\begin{cases}0\not=z\in e_j*\hat{D}(B)_1*e_i\hbox{ implies }s(z)=i\leq j=t(z),\hbox{ and }&\\
0\not=z\in e_j*\hat{D}(B)_0*e_i\hbox{ implies }s(z)=i\leq j=t(z) \hbox{ and } s(z)=i\not\sim j=t(z).
\end{cases}$$
\end{enumerate}
\end{remark}

\begin{definition}\label{L: completion of tensor algebra}
Given any $S$-$S$-bimodule $B$, the 
\emph{formal tensor series $S$-algebra $\widehat{T}_S(B)$ over $B$} is defined as an $S$-$S$-bimodule by 
$$\widehat{T}_S(B)=\prod_{i=0}^\infty B^{\otimes i}.$$
We write the elements $f\in \widehat{T}_S(B)$ as formal series 
$f=\sum_{i=0}^\infty f_i$, where $f_i\in B^{\otimes i}$.
The product of the $S$-algebra $\widehat{T}_S(B)$ is given by the formula:
$$\left(\sum_{i=0}^\infty f_i\right)\left(\sum_{j=0}^\infty g_j\right)=\sum_{t=0}^\infty \left(\sum_{i+j=t} f_ig_j\right),$$
where $f_ig_j$ is the usual product in the tensor algebra $T_S(B)$. 
\end{definition}
 
 \begin{remark}\label{R: sobre el TS[[B]]}
 The algebra defined above is indeed a unital associative $S$-algebra. Moreover:
 \begin{enumerate}
  \item If, for $0\not=f\in \widehat{T}_S(B)$,    
 we denote by $\nu(f)$   
 the minimal integer $i\geq 0$ such that $f_i\not=0$, then 
the map $\nu$ determines a metric 
$$\depth{}:\widehat{T}_S(B)\times \widehat{T}_S(B)\rightmap{}\hueca{R}$$
 such that  $\depth{(f,g)}=2^{-\nu(f-g)}$,  if  $f\not=g$ and 
 $\depth(f,f)=0$. With the metric $\depth{}$, the algebra  $\widehat{T}_S(B)$ becomes a topological algebra 
\item If we denote by $\langle B\rangle$ the ideal of $T_S(B)$ generated by $B$, then $\widehat{T}_S(B)$ is the $\langle B\rangle$-adic completion of the tensor algebra $T_S(B)$. 
 \end{enumerate}
 \end{remark}
 
When $B$ is a graded $S$-$S$-bimodule, although the algebra $\widehat{T}_S(B)$ is not in general, in a natural way, a graded $S$-algebra, we can introduce the following terminology.

 \begin{definition}\label{D: homog y diferential en T_S[[B]]}
  Let $B$ be a graded $S$-$S$-bimodule. Given $n\in \hueca{Z}$, an element $f=\sum_{i=0}^\infty f_i\in \widehat{T}_S(B)$ is called \emph{homogeneous of degree $n$} iff  
 $\vert f_i\vert=n$, for all $i\geq 0$. Then, the degree of such an homogeneous $f$ is denoted by $\vert f\vert$. Clearly, the product of two homogeneous elements $f,g\in \widehat{T}_S(B)$ is again homogeneous and $\vert fg\vert=\vert f\vert +\vert g\vert$. 
 
 A morphism of $S$-$S$-bimodules 
 $d:\widehat{T}_S(B)\rightmap{}\widehat{T}_S(B)$ will be called \emph{a differential on $\widehat{T}_S(B)$} if the following holds:
 \begin{enumerate}
  \item Whenever $f\in \widehat{T}_S(B)$ is homogeneus, we have that $d(f)$ is homogeneous of degree $\vert f\vert+1$.
  \item Given homogeneous elements $f,g\in \widehat{T}_S(B)$, the Leibniz formula holds:
  $$ d(fg)=d(f)g+(-1)^{\vert f\vert}fd(g).$$
  \item We have $d^2=0$.
 \end{enumerate}
 \end{definition}

 Let us come back now to the notation used before (\ref{L: completion of tensor algebra}). We want to construct a differential on $\widehat{T}_S(\hat{D}(B))$ using the family of maps $\{b_n:B^{\otimes n}\rightmap{}B\}_{n\in \hueca{N}}$.  
 
 \begin{proposition}\label{P: diferencial en T[[B]]}
 With the preceding notation, consider for $s \geq 1$, the homogeneous morphism of $S$-$S$-bimodules of degree $1$
 $$\hat{d}(b_s):=(\hat{D}(B)\rightmap{ \ \hat{D}(b_s) \ }\hat{D}(B^{\otimes s})\rightmap{ \ \hat{\tau}_s \ }\hat{D}(B)^{\otimes s}).$$
 The family $\{\hat{d}(b_j)\}_{j\geq 1}$ determines a  
 morphism of $S$-$S$-bimodules 
 $$d_1:\hat{D}(B)\rightmap{}\widehat{T}_S(\hat{D}(B))\hbox{ such that }d_1(h)=\sum_{s=1}^\infty \hat{d}(b_s)(h), \hbox{ for } h\in \hat{D}(B).$$
 For each $j\geq 2$, using Leibniz formula for homogeneous elements,  the morphism $d_1$ induces a morphism of $S$-$S$-bimodules 
 $$d_j:\hat{D}(B)^{\otimes j}\rightmap{}\widehat{T}_S(\hat{D}(B)).$$ 
 Thus, $d_j$ has the following association recipe: 
 $$d_j=\sum_{\scriptsize\begin{matrix} r+1+t=j\\
 r,t\geq 0 \end{matrix}}id^{\otimes r}\otimes d_1\otimes id^{\otimes t}
 =
 \sum_{s=1}^\infty\hbox{ }
 \sum_{\scriptsize\begin{matrix} r+1+t=j\\
 r,t\geq 0 \end{matrix}}id^{\otimes r}\otimes \hat{d}(b_s)\otimes id^{\otimes t}.$$
 More explicitely,  given $j\geq 1$ and a homogeneous generator $h_1\otimes\cdots\otimes h_j\in \hat{D}(B)_{i_1}\otimes\cdots\otimes \hat{D}(B)_{i_j}\subseteq \hat{D}(B)^{\otimes j}$ of degree $m$, with $i_1+\cdots+ i_j=m$, the recipe for computing $d_j(h_1\otimes\cdots \otimes h_j)$ 
 is 
 $$\sum_{r=1}^j(-1)^{\sum_{u=1}^{r-1}\vert h_u\vert}
 h_1\otimes\cdots\otimes h_{r-1}\otimes d_1(h_r)\otimes h_{r+1}\otimes \cdots\otimes h_j,$$
 which coincides with
 $$\sum_{s=1}^{\infty}\sum_{r=1}^j(-1)^{\sum_{u=1}^{r-1}\vert h_u\vert}
 h_1\otimes\cdots\otimes h_{r-1}\otimes \hat{d}(b_s)(h_r)\otimes h_{r+1}\otimes \cdots\otimes h_j.$$
 The family $\{d_j:\hat{D}(B)^{\otimes j}\rightmap{}\widehat{T}_S(\hat{D}(B))\}_{j\geq 0}$, where $d_0=0:S\rightmap{}\widehat{T}_S(\hat{D}(B))$,  determines a morphism of $S$-$S$-bimodules $d:T_S(\hat{D}(B))\rightmap{}\widehat{T}_S(\hat{D}(B))$ by the universal property of the direct sums. Then,  we can extend $d$ to a new 
  map $d:\widehat{T}_S(\hat{D}(B))\rightmap{}\widehat{T}_S(\hat{D}(B))$, which we denote with the same symbol $d$, given by 
 $$d(\sum_{j=0}^\infty f_j)=\sum_{j=0}^\infty d_j(f_j).$$
 Moreover, we have that $d(S)=0$ and  $d$ is a differential on $\widehat{T}_S(\hat{D}(B))$. Thus, given $f=\sum_{j=0}^\infty f_j\in \widehat{T}_S(\hat{D}(B))$, the element $d(f)=\sum_{n=0}^\infty d(f)_n\in \widehat{T}_S(\hat{D}(B))$ has components $d(f)_n\in \hat{D}(B)^{\otimes n}$, for each $n\geq 0$,  given by
 $$d(f)_n=\sum_{j=1}^n\sum_{r+1+t=j}(id^{\otimes r}\otimes\hat{d}(b_{n-j+1})\otimes id^{\otimes t})(f_j).$$
 \end{proposition}

\begin{proof} Since each $b_i$ is homogeneous with $\vert b_i\vert=1$, as remarked in (\ref{D: D(B)}), we have $\hat{d}(b_i)=\hat{\tau}_i \hat{D}(b_i)$ is homogeneous of degree $\vert \hat{d}(b_i)\vert=\vert \hat{\tau}_i\vert+\vert \hat{D}(b_i)\vert=\vert b_i\vert=1$.
It follows that $d(f_i)$ is homogeneous of degree $\vert f_i\vert+1$, for all $i\geq 0$ and $f_i\in \hat{D}(B)^{\otimes i}$ homogeneous. Hence we obtain that $d(f)$ is homogeneous with $\vert d(f)\vert =\vert f\vert+1$, for all homogeneous $f\in \widehat{T}_S(\hat{D}(B))$. 

\medskip
\noindent\emph{Step 1: The morphism  $d$ satisfies Leibniz formula.}
\medskip

Assume that $f,g\in \widehat{T}_S(\hat{D}(B))$ are homogeneous. We want to show that $$d(fg)=d(f)g+(-1)^{\vert f\vert}fd(g).$$
Fix $n\geq 0$, and let us show that their $n$-components coincide 
$$d(fg)_n=[d(f)g+(-1)^{\vert f\vert}fd(g)]_n.$$
So, we have to show that 
$$d(fg)_n=\sum_{u+v=n}d(f)_ug_v+(-1)^{\vert f\vert}f_ud(g)_v.$$

If $n=0$, then $u=0=v$ and the equality is clear. So, from now on, we assume that $n\geq 1$. 
By definition, for $j\geq 0$, we have 
$$(fg)_j=\sum_{a+c=j}f_ag_c=\sum_{\scriptsize\begin{matrix} a+c=j\\
 a,c\geq 1 \end{matrix}}f_a\otimes g_c+ f_0g_j+f_jg_0,$$
then $d(fg)_n=S+S_0+S^0$, where 
$$S=\sum_{j=1}^n \sum_{r+1+t=j}\sum_{\scriptsize\begin{matrix} a+c=j\\
 a,c\geq 1 \end{matrix}}(id^{\otimes r}\otimes\hat{d}(b_{n-j+1})\otimes id^{\otimes t})(f_a\otimes g_c)$$
$$S_0=
\sum_{j=1}^n \sum_{r+1+t=j}(id^{\otimes r}\otimes\hat{d}(b_{n-j+1})\otimes id^{\otimes t})(f_0g_j)$$
and
$$S^0=\sum_{j=1}^n \sum_{r+1+t=j}(id^{\otimes r}\otimes\hat{d}(b_{n-j+1})\otimes id^{\otimes t})(f_jg_0).$$
Moreover, we have that 
$$\sum_{u+v=n}d(f)_ug_v+(-1)^{\vert f\vert}f_ud(g)_v= 
S'+(-1)^{\vert f\vert}f_0d(g)_n+d(f)_ng_0,$$
where 
$S'=\sum_{\scriptsize\begin{matrix} u+v=n\\
 u,v\geq 1 \end{matrix}}d(f)_ug_v+(-1)^{\vert f\vert}f_ud(g)_v$.
 
Since $id^{\otimes r}\otimes\hat{d}(b_{n-j+1})\otimes id^{\otimes t}$ is a morphism of $S$-$S$-bimodules, we have $d(f)_ng_0=S^0$. Similarly, we have that $(-1)^{\vert f\vert}f_0d(g)_n=S_0$: indeed, if $f_0\not=0$, then $\vert f\vert=\vert f_0\vert=0$, thus we can use that $id^{\otimes r}\otimes\hat{d}(b_{n-j+1})\otimes id^{\otimes t}$ is a morphism of $S$-$S$-bimodules; if $f_0=0$, the equality is clear. 

So, we are reduced to prove that $S=S'$. We have $S'=S'_g+S'_f$, where 
$$ S_g'=
   \sum_{\scriptsize\begin{matrix} u+v=n\\
 u,v\geq 1 \end{matrix}}\sum_{j_1=1}^u \sum_{r_1+1+t_1=j_1}(id^{\otimes r_1}\otimes\hat{d}(b_{n-j_1+1})\otimes id^{\otimes t_1})(f_{j_1})\otimes g_v$$
 and
 $$S'_f= \sum_{\scriptsize\begin{matrix} u+v=n\\
 u,v\geq 1 \end{matrix}}\sum_{j_2=1}^v \sum_{r_2+1+t_2=j_2}(-1)^{\vert f_u\vert}f_u\otimes(id^{\otimes r_2}\otimes\hat{d}(b_{n-j_2+1})\otimes id^{\otimes t_2})(g_{j_2}),$$
 because $\vert f\vert=\vert f_u\vert$, for each $u\geq 1$.

 Moreover, we have  $S=S_g+S_f$, where 
$$S_g=\sum_{j=1}^n \sum_{r+1+t=j}\sum_{\scriptsize\begin{matrix} a+c=j\\
 a,c\geq 1; c\leq t \end{matrix}}(id^{\otimes r}\otimes\hat{d}(b_{n-j+1})\otimes id^{\otimes t})(f_a\otimes g_c)$$
 and
 $$S_f=\sum_{j=1}^n \sum_{r+1+t=j}\sum_{\scriptsize\begin{matrix} a+c=j\\
 a,c\geq 1; a\leq r \end{matrix}}(id^{\otimes r}\otimes\hat{d}(b_{n-j+1})\otimes id^{\otimes t})(f_a\otimes g_c).$$
Indeed, we have $a\leq r$ iff $j-c=a\leq r=j-1-t$ iff $c\geq t+1$. So, $a\not\leq r$ iff $c\leq t$. It will be enough to show that
$S_f=S'_f$ and $S_g=S'_g$. Consider the sets of indices of the sum $S_g$ 
$$I_g=\left\{ (j,r,t,a,c)\in \hueca{Z}^5\mid \begin{matrix}1\leq j\leq n; r+1+t=j; a+c=j;\\  a,c\geq 1; r,t\geq 0; c\leq t\end{matrix}\right\},$$
and the set of indices of the sum $S'_g$ 
$$I'_g=\left\{ (j_1,r_1,t_1,u,v)\in \hueca{Z}^5\mid \begin{matrix}u+v=n; 1\leq j_1\leq u; r_1+1+t_1=j_1;\\  u,v\geq 1; r_1,t_1\geq 0\end{matrix}\right\}.$$
Since we have 
$$S_g=\sum_{j=1}^n \sum_{r+1+t=j}\sum_{\scriptsize\begin{matrix} a+c=j\\
 a,c\geq 1; c\leq t \end{matrix}}(id^{\otimes r}\otimes\hat{d}(b_{n-j+1})\otimes id^{\otimes (t-c)})(f_a)\otimes g_c,$$
we obtain $S_g=S'_g$, because we have the bijection $\phi_g:I_g\rightmap{}I'_g$ of indices given by $(j,r,t,a,c)\mapsto(j_1,r_1,t_1,u,v)=(j-c,r,t-c,n-c,c)$. Similarly, if we consider the set of indices of the sum $S_f$
$$I_f=\left\{ (j,r,t,a,c)\in \hueca{Z}^5\mid \begin{matrix}1\leq j\leq n; r+1+t=j; a+c=j;\\  a,c\geq 1; r,t\geq 0; a\leq r\end{matrix}\right\},$$
and the set of indices of the sum $S'_f$ 
$$I'_f=\left\{ (j_2,r_2,t_2,u,v)\in \hueca{Z}^5\mid \begin{matrix}u+v=n; 1\leq j_2\leq v; r_2+1+t_2=j_2;\\  u,v\geq 1; r_2,t_2\geq 0\end{matrix}\right\}.$$
Since we have 
$$S_f=\sum_{j=1}^n \sum_{r+1+t=j}\sum_{\scriptsize\begin{matrix} a+c=j\\
 a,c\geq 1; a\leq r \end{matrix}}(-1)^{\vert f_a\vert}f_a\otimes(id^{\otimes (r-a)}\otimes\hat{d}(b_{n-j+1})\otimes id^{\otimes t})(g_c),$$
we obtain $S_f=S'_f$, because we have the bijection $\phi_f:I_f\rightmap{}I'_f$ of indices given by $(j,r,t,a,c)\mapsto(j_2,r_2,t_2,u,v)=(j-a,r-a,t,a,n-a)$.

\medskip
\noindent\emph{Step 2: We have $d^2=0$.}
\medskip

Since $d^2:\widehat{T}_S(\hat{D}(B))\rightmap{}  \widehat{T}_S(\hat{D}(B))$ is an $S$-$S$-bimodule morphism satisfying  $d^2(fg)=d^2(f)g+fd^2(g)$, for any homogeneous $f,g\in \widehat{T}_S(\hat{D}(B))$, it is determined by its values on $\hat{D}(B)$. So, it will be enough to show that its restriction to $\hat{D}(B)$ is zero. This restriction is the composition 
$$\hat{D}(B)\rightmap{d_1}\widehat{T}_S(\hat{D}(B))\rightmap{d}\widehat{T}_S(\hat{D}(B)).$$
From (\ref{L: la formula para manipular la diferencial hat d}) and (\ref{D: construccion barra}), we have 
$$\begin{matrix}dd_1
&=&
\sum_{j=1}^\infty d_j\hat{d}(b_j)\hfill\\
&=&
\sum_{j=1}^\infty \sum_{s=1}^\infty\hbox{ }
 \sum_{\scriptsize\begin{matrix} r+1+t=j\\
 r,t\geq 0 \end{matrix}}(id^{\otimes r}\otimes \hat{d}(b_s)\otimes id^{\otimes t})\hat{d}(b_j)\hfill\\
 &=&
 \sum_{j,s=1}^\infty \sum_{\scriptsize\begin{matrix} r+1+t=j\\
 r,t\geq 0 \end{matrix}}(id_{\hat{D}(B)}^{\otimes r}\otimes \hat{\tau}_s\hat{D}(b_s)\otimes id_{\hat{D}(B)}^{\otimes t})\hat{\tau}_j\hat{D}(b_j)\hfill\\
 &=&
 \sum_{n=1}^\infty\sum_{\scriptsize\begin{matrix} r+s+t=n\\
 r,t\geq 0;s\geq 1 \end{matrix}}(id_{\hat{D}(B)}^{\otimes r}\otimes \hat{\tau}_s\hat{D}(b_s)\otimes id_{\hat{D}(B)}^{\otimes t})\hat{\tau}_{r+1+t}\hat{D}(b_{r+1+t})
 \hfill\\
 &=&
 \sum_{n=1}^\infty\sum_{\scriptsize\begin{matrix} r+s+t=n\\
 r,t\geq 0;s\geq 1 \end{matrix}}
 \hat{\tau}_n\hat{D}(id_B^{\otimes r}\otimes b_s\otimes id_B^{\otimes t})\hat{D}(b_{r+1+t})\hfill\\
 &=&
 \sum_{n=1}^\infty\hat{\tau}_n\hat{D}[\sum_{\scriptsize\begin{matrix} r+s+t=n\\
 r,t\geq 0;s\geq 1 \end{matrix}}b_{r+1+t}(id_B^{\otimes r}\otimes b_s\otimes id_B^{\otimes t})]=0.\hfill\\
\end{matrix}$$ 
\end{proof}

\begin{lemma}\label{L: sobre elementos homogeneos de TS(DB)} Consider the quotient poset $\overline{\cal P}$ of the preordered set ${\cal P}$, as in (\ref{R: P sobre sim}), and make 
$\ell:=\vert \overline{\cal P}\vert$. Then, with the preceding notation,  we have:
\begin{enumerate}
 \item If $0\not= h$ is a homogeneous element in $(\hat{D}(B)_0\oplus \hat{D}(B)_1)^{\otimes n}$, for some $n\geq 0$, then 
 $\vert h\vert\geq n-\ell.$
 \item If $h=\sum_{n\geq 0}h_n\in \widehat{T}_S(\hat{D}(B)_0\oplus \hat{D}(B)_1)$ is a homogeneous element, then $h_n=0$ for 
 $n> \vert h\vert+\ell$, thus 
 $h\in T_S(\hat{D}(B)_0\oplus \hat{D}(B)_1)$. 
\end{enumerate}
\end{lemma}

\begin{proof} (1): If $r:=\vert h\vert$, we can assume that 
$$h=\alpha_{t_r}^r\otimes\cdots\otimes\alpha_{1}^r\otimes x_r\otimes\cdots \otimes x_2\otimes \alpha_{t_1}^1\otimes\cdots\otimes\alpha_1^1\otimes x_1\otimes \alpha_{t_0}^0\otimes\cdots\otimes \alpha_1^0,$$
 where  $\alpha_i^j$ are directed elements in $\hat{D}(B)_0$ and $x_1,\ldots,x_r$ are directed elements in $\hat{D}(B)_1$;  
 $t_0,t_1,\ldots,t_r\geq 0$ and $r+\sum_{j=0}^rt_j=n$. 
 Then, from (\ref{R: B y D(B) homogeneos})(3), we have in the poset $\overline{\cal P}$
  $$\overline{s(\alpha_1^0)}<\overline{t(\alpha_1^0)}=\overline{s(\alpha_2^0)}<\cdots <\overline{t(\alpha_{t_0}^0)}=\overline{s(x_1)}\leq \overline{t(x_1)}=\overline{s(\alpha_1^1)},$$
  $$\overline{s(\alpha_1^1)}<\overline{t(\alpha_1^1)}=\overline{s(\alpha_2^1)}<\cdots <\overline{t(\alpha_{t_1}^1)}=\overline{s(x_2)}\leq \overline{t(x_2)}=\overline{s(\alpha_1^2)},$$
  $$\cdots$$
   $$\overline{s(\alpha_1^r)}<\overline{t(\alpha_1^r)}=\overline{s(\alpha_2^r)}<\cdots <\overline{t(\alpha_{t_r}^r)}.$$
   Thus, 
  $\overline{s(\alpha_1^0)}<\overline{s(\alpha_2^0)}<\cdots<\overline{s(\alpha_{t_0}^0)}<\cdots<\overline{s(\alpha^r_1)}<
  \overline{s(\alpha_2^r)}<\cdots<\overline{s(\alpha_{t_r}^r)}$. This implies that 
  $\sum_{j=0}^rt_j\leq  \ell$. Therefore, we have $n=r+\sum_{j=0}^rt_j\leq r+\ell$, and (1) follows. (2) is clearly follows from (1). 
\end{proof}

\begin{remark}\label{N: la base de B}
It is convenient to fix a special \emph{directed basis} for the graded $S$-$S$-bimodules $B$ and $\hat{D}(B)$. They are chosen as follows. For each $t\geq 0$ and  $i,j\in {\cal P}$, we choose a $k$-basis $\hueca{B}_t(i,j)$ for the space $e_jB_te_i$; then, we consider the basis  $\hueca{B}_t=\bigcup_{i,j}\hueca{B}_t(i,j)$ of $B_t$; for $t=-1$, we have $B_{-1}=\End_{\Lambda}(\Delta)=S\bigoplus J=\bigoplus_{i\in {\cal P}} ke_i\bigoplus J$, so we choose, for any indices $i,j\in {\cal P}$, a $k$-basis  $\hueca{J}(i,j)$ for $e_jJe_i$; then, make $\hueca{J}=\bigcup_{i,j\in {\cal P}}\hueca{J}(i,j)$; for $i\in {\cal P}$, we consider the $k$-basis $\hueca{B}_{-1}(i,i)=\hueca{J}(i,i)\cup\{e_i\}$ of $e_iB_{-1}e_i$; then take $\hueca{B}_{-1}=\{e_i\mid i\in {\cal P}\}\cup \hueca{J}$. Finally, we can consider the $k$-basis $\hueca{B}=\bigcup_{t\geq -1}\hueca{B}_t$ of $B$.  

Then, we can consider for each finite-dimensional $B_t$, the dual basis $(\hueca{B}^*)_{-t}:=(\hueca{B}_t)^*$ of $\hat{D}(B_t)=\hat{D}(B)_{-t}\subseteq \hat{D}(B)$, with elements $x^*\in \hat{D}(B_t)=\hat{D}(B)_{-t}$, and then consider the union $\hueca{B}^*:=\bigcup_{t\leq 1}(\hueca{B}^*)_t$.  
By definition, if $x\in \hueca{B}_t(i,j)$, the directed element $x^*\in (\hueca{B}^*)_{-t}(i,j)\subseteq e_j*\hat{D}(B_t)*e_i\subseteq\hat{D}(B)_{-t}$ satisfies, for $y\in \hueca{B}$, that $x^*(y)=0$, unless $y\in \hueca{B}_t(i,j)$ and, in this case, we have 
$x^*(y)=\delta_{x,y}e_{j}=\delta_{x,y}e_{t(y)}$. 

For instance, we have 
$e_i^*\in (\hueca{B}^*)_1\subseteq \hat{D}(B)_1\subseteq\hat{D}(B)$ such that $e_i^*(B_t)=0$, for $t\not= -1$, $e_i^*(J)=0$, and $e_i^*(e_j)=\delta_{i,j}e_j$. 

Notice that each tensor power $\hat{D}(B)^{\otimes n}$ admits as $k$-basis the set of elements of the form $z=x^*_{i_n}\otimes\cdots\otimes x^*_{i_2}\otimes x^*_{i_1}$, where  $x_{i_1},\ldots,x_{i_n}$ are directed basic elements in $\hueca{B}_{i_1},\ldots,\hueca{B}_{i_n}$, respectively, such that 
$t(x^*_{i_1})=s(x^*_{i_2}),\ldots, t(x^*_{i_{n-1}})=s(x^*_{i_n})$, or equivalently, such that 
$t(x_{i_1})=s(x_{i_2}),\ldots, t(x_{i_{n-1}})=s(x_{i_n})$
. We denote this basis by $\widehat{\hueca{T}}_n$. Therefore, for $\gamma\in\hat{D}(B)^{\otimes n}$, we can write 
$$\gamma=\sum_{z\in \widehat{\hueca{T}}_n} c_z^\gamma z, \hbox{ for some   scalars } c_z^\gamma\in k.$$
\end{remark}

 \begin{lemma}\label{L: desc del algebra completada} Consider the following linear spaces:
 \begin{enumerate}
  \item ${\cal N}$ is the linear subspace of $\widehat{T}_S(\hat{D}(B))$ generated by the elements $g=\sum_{n\geq 1}g_n$, such that each $g_n\in \hat{D}(B)_{i_1}\otimes\cdots\otimes\hat{D}(B)_{i_n}$ with
  $i_j<0$, for some $j\in [1,n]$;
 \item ${\cal E}$ is the linear subspace of $\widehat{T}_S(\hat{D}(B)_0\oplus \hat{D}(B)_1)\subseteq \widehat{T}_S(\hat{D}(B))$ generated by the elements $g=\sum_{n\geq 1}g_n$, such that each $g_n$ is a sum of elements of the form $\alpha_n\otimes\cdots\otimes\alpha_1$ with $\alpha_i\in \hat{D}(B)_0\cup \hat{D}(B)_1$ and $\alpha_i=e_s^*$, for some $i\in [1,n]$ and $s\in {\cal P}$. 
 \end{enumerate}
 Then, we have the vector space direct sum decomposition 
 $$\widehat{T}_S(\hat{D}(B))=\widehat{T}_S(\hat{D}(B)_0\oplus \hat{D}(J))\oplus {\cal E}\oplus {\cal N}.$$  
 \end{lemma}

 \begin{proof} We first remark that 
 $\widehat{T}_S(\hat{D}(B))=\widehat{T}_S(\hat{D}(B)_0\oplus \hat{D}(B)_1)\oplus {\cal N}$. Indeed,   
 given $f=\sum_{n\geq 0}f_n\in\widehat{T}_S(\hat{D}(B))$,
 so $f_n\in \hat{D}(B)^{\otimes n}$, for each $n$. Then, $f_n=\sum_j h_j$ is a finite sum with $h_j\in [\hat{D}(B)^{\otimes n}]_j=\bigoplus_{i_1+\cdots +i_n=j}\hat{D}(B)_{i_1}\otimes\cdots\otimes\hat{D}(B)_{i_n}$ homogeneous for each $j$. Then, $h_j=h'_j+h''_j$, where 
 $$h'_j\in\bigoplus_{\scriptsize\begin{matrix} i_1+\cdots +i_n=j\\
 0\leq i_1,\ldots, i_n\leq 1\end{matrix}}\hat{D}(B)_{i_1}\otimes\cdots\otimes\hat{D}(B)_{i_n}\subseteq \widehat{T}_S(\hat{D}(B)_0\oplus \hat{D}(B)_1)$$
 $$h''_j\in\bigoplus_{\scriptsize\begin{matrix} i_1+\cdots +i_n=j\\
 \hbox{ some } i_t< 0\end{matrix}}\hat{D}(B)_{i_1}\otimes\cdots\otimes\hat{D}(B)_{i_n}\subseteq {\cal N}.$$
 Thus, we get $f=f'+f''$, with $f'_n=\sum_jh'_j$ and $f''_n=\sum_jh''_j$, thus $f'\in \widehat{T}_S(\hat{D}(B)_0\oplus \hat{D}(B)_1)$ and 
 $f''\in {\cal N}$. 
 
 Now, we remark that $\widehat{T}_S(\hat{D}(B)_0\oplus \hat{D}(B)_1)=\widehat{T}_S(\hat{D}(B)_0\oplus \hat{D}(J))\oplus {\cal E}$. Indeed, given $f=\sum_{n\geq 0}f_n\in\widehat{T}_S(\hat{D}(B)_0\oplus\hat{D}(B)_1)$,
 so $f_n\in (\hat{D}(B)_0\oplus \hat{D}(B)_1)^{\otimes n}$, for each $n$. 
  Hence, we have a finite sum of homogeneous components $f_n=\sum_jh_j$, with 
 $$h_j\in\bigoplus_{\scriptsize\begin{matrix} i_1+\cdots +i_n=j\\
 0\leq i_1,\ldots, i_n\leq 1\end{matrix}}\hat{D}(B)_{i_1}\otimes\cdots\otimes\hat{D}(B)_{i_n}.$$
 But $\hat{D}(B)_1=\bigoplus_{s\in {\cal P}} ke_s^*\oplus \hat{D}(J)$. Thus, each component decomposes as $h_j=h'_j+h''_j$, with 
 $h'_j\in \widehat{T}_S(\hat{D}(B)_0\oplus \hat{D}(J))$ and $h''_j\in {\cal E}$.
 \end{proof}
 
In the following lemmas, given a homogeneous morphism $f:M\rightmap{}N$ with degree $m$ of graded $S$-$S$-bimodules, we denote by $f^t:M_t\rightmap{}N_{t+m}$ the restriction of the morphism $f$ to the homogeneous components of $M$ and $N$, thus $f=\bigoplus_{t\in \hueca{Z}}f^t:M\rightmap{}N$. 

 \begin{lemma}\label{L: theta vs componentes homogeneas}
  Given $n\geq 2$ and $t\in \hueca{Z}$,  consider the homogeneous component of degree $t$ of $B^{\otimes n}$ thus, we have a finite direct sum 
 $$(B^{\otimes n})_t=\bigoplus_{i_1+\cdots+ i_n=t}B_{i_1}\otimes B_{i_2}\otimes\cdots\otimes B_{i_n}.$$ 
 We have canonical projections  
 $\pi_{i_1,\ldots,i_n}:(B^{\otimes n})_t\rightmap{}B_{i_1}\otimes B_{i_2}\otimes\cdots\otimes B_{i_n}$ and the injections 
 $\sigma_{i_1,\ldots,i_n}:B_{i_1}\otimes B_{i_2}\otimes\cdots\otimes B_{i_n}\rightmap{}(B^{\otimes n})_t$. Similarly, we have the finite direct sum 
 $$\begin{matrix}(\hat{D}(B)^{\otimes n})_{-t}
 &=&
 \bigoplus_{-i_1-\cdots- i_n=-t}\hat{D}(B)_{-i_1}\otimes \hat{D}(B)_{-i_2}\otimes\cdots\otimes \hat{D}(B)_{-i_n}\hfill\\
 &=&
 \bigoplus_{i_1+\cdots+ i_n=t}\hat{D}(B_{i_1})\otimes \hat{D}(B_{i_2})\otimes\cdots\otimes \hat{D}(B_{i_n})\hfill\\   
   \end{matrix}$$ 
 and the injections 
 $\hat{\sigma}_{i_1,\ldots,i_n}:\hat{D}(B_{i_1})\otimes \hat{D}(B_{i_2})\otimes\cdots\otimes \hat{D}(B_{i_n}) \rightmap{}(\hat{D}(B)^{\otimes n})_{-t}$ and projections $\hat{\pi}_{i_1,\ldots,i_n}: (\hat{D}(B)^{\otimes n})_{-t}\rightmap{}\hat{D}(B_{i_1})\otimes \hat{D}(B_{i_2})\otimes\cdots\otimes \hat{D}(B_{i_n})$.
 Moreover, we have 
 $$\hat{D}((B^{\otimes n})_t)=\bigoplus_{j_1+\cdots+ j_n=t}\hat{D}(B_{j_1}\otimes B_{j_2}\otimes\cdots\otimes B_{j_n}).$$ 
 Consider $\hat{\tau}_{i_1,\ldots,i_n}:=\hat{\tau}_{B_{i_1},\ldots,B_{i_n}}$ and the restriction of 
 $\hat{\tau}_n:\hat{D}(B^{\otimes n})\rightmap{}\hat{D}(B)^{\otimes n}$ to the components of degree $-t$ 
 $$\hat{\tau}_n^{-t}:\hat{D}(B^{\otimes n})_{-t}\rightmap{}(\hat{D}(B)^{\otimes n})_{-t}.$$
 Then, the following squares commute
 $$\begin{matrix}
 \hat{D}(B_{i_1}\otimes\cdots\otimes B_{i_n})&\rightmap{ \ \hat{D}(\pi_{i_1,\ldots,i_n}) \ }&\hat{D}(\bigoplus_{j_1+\cdots+j_n=t}B_{j_1}\otimes\cdots\otimes B_{j_n})\\
\lmapdown{ \hat{\tau}_{i_1,\ldots,i_n} }&&\rmapdown{ \hat{\tau}_n^{-t} }\\
\hat{D}(B_{i_1})\otimes\cdots\otimes \hat{D}(B_{i_n}) &
 \rightmap{ \ \hat{\sigma}_{i_1,\ldots,i_n} \ }
 &\bigoplus_{j_1+\cdots+j_n=t}\hat{D}(B_{j_1})\otimes\cdots\otimes \hat{D}(B_{j_n})\\\end{matrix}$$
 $$\begin{matrix}
 \hat{D}(B_{i_1}\otimes\cdots\otimes B_{i_n})&\leftmap{ \ \hat{D}(\sigma_{i_1,\ldots,i_n}) \ }&\hat{D}(\bigoplus_{j_1+\cdots+j_n=t}B_{j_1}\otimes\cdots\otimes B_{j_n})\\
\lmapdown{ \hat{\tau}_{i_1,\ldots,i_n} }&&\rmapdown{ \hat{\tau}_n^{-t} }\\
\hat{D}(B_{i_1})\otimes\cdots\otimes \hat{D}(B_{i_n}) &
 \leftmap{ \ \hat{\pi}_{i_1,\ldots,i_n} \ }
 &\bigoplus_{j_1+\cdots+j_n=t}\hat{D}(B_{j_1})\otimes\cdots\otimes \hat{D}(B_{j_n})\\\end{matrix}$$
 \end{lemma}
 
 \begin{proof} In order to show the commutativity of the first square, it will be enough to show that the following square involving the inverses commutes
 $$\begin{matrix}
 \hat{D}(B_{i_1})\otimes\cdots\otimes \hat{D}(B_{i_n}) &
 \rightmap{ \ \hat{\sigma}_{i_1,\ldots,i_n} \ }
 &\bigoplus_{j_1+\cdots+j_n=t}\hat{D}(B_{j_1})\otimes\cdots\otimes \hat{D}(B_{j_n})\\
\lmapdown{ \hat{\theta}_{i_1,\ldots,i_n} }&&\rmapdown{ \hat{\theta}_n^{-t} }\\
\hat{D}(B_{i_1}\otimes\cdots\otimes B_{i_n})&\rightmap{ \ \hat{D}(\pi_{i_1,\ldots,i_n}) \ }&\hat{D}(\bigoplus_{j_1+\cdots+j_n=t}B_{j_1}\otimes\cdots\otimes B_{j_n}).\\
\end{matrix}$$
For this, take a typical generator $\alpha_{i_1}\otimes\cdots\otimes\alpha_{i_n}\in \hat{D}(B_{i_1})\otimes\cdots\otimes \hat{D}(B_{i_n})$ and a typical generator  
$$\sum_{j_1+\cdots+j_n=t}a_{j_1}\otimes\cdots\otimes a_{j_n}\in\bigoplus_{j_1+\cdots+j_n=t}B_{j_1}\otimes\cdots\otimes B_{j_n}.$$
Then, if we make $\Delta= \hat{D}(\pi_{i_1,\ldots,i_n})\hat{\theta}_{i_1,\ldots,i_n}(\alpha_{i_1}\otimes\cdots\otimes\alpha_{i_n})(\sum a_{j_1}\otimes\cdots\otimes a_{j_n})$, we get
$$\begin{matrix}\Delta&=&\hat{D}(\pi_{i_1,\ldots,i_n})\hat{\theta}_{i_1,\ldots,i_n}(\alpha_{i_1}\otimes\cdots\otimes\alpha_{i_n})(\sum a_{j_1}\otimes\cdots\otimes a_{j_n})\hfill\\
&=&
[\hat{\theta}_{i_1,\ldots,i_n}(\alpha_{i_1}\otimes\cdots\otimes\alpha_{i_n})\pi_{i_1,\ldots,i_n}](\sum a_{j_1}\otimes\cdots\otimes a_{j_n})\hfill\\
&=&\hat{\theta}_{i_1,\ldots,i_n}(\alpha_{i_1}\otimes\cdots\otimes\alpha_{i_n})(a_{i_1}\otimes\cdots\otimes a_{i_n})\hfill\\
&=&
\hat{\theta}_n(\alpha_{i_1}\otimes\cdots\otimes\alpha_{i_n})(\sum a_{j_1}\otimes\cdots\otimes a_{j_n})\hfill\\
&=&
\hat{\theta}_n^{-t}\hat{\sigma}_{i_1,\ldots,i_n}(\alpha_{i_1}\otimes\cdots\otimes\alpha_{i_n})(\sum a_{j_1}\otimes\cdots\otimes a_{j_n}).\hfill\\
\end{matrix}$$
Recall that in the definition of $\hat{D}(B)$, we identified 
$\hat{D}(B_{i_u})$ with $\hat{D}(B)_{-i_u}$ mapping each $\alpha_{i_u}:B_{i_u}\rightmap{}S$ onto its extension $\alpha_{i_u}:B\rightmap{}S$ such that 
 $\alpha_{i_u}(B_{j_v})=0$ whenever $i_u\not=j_v$. 

 The commutativity of the second square of the lemma is equivalent to the commutativity of the square
 $$\begin{matrix}
 \hat{D}(B_{i_1})\otimes\cdots\otimes \hat{D}(B_{i_n}) &
 \leftmap{ \ \hat{\pi}_{i_1,\ldots,i_n} \ }
 &\bigoplus_{j_1+\cdots+j_n=t}\hat{D}(B_{j_1})\otimes\cdots\otimes \hat{D}(B_{j_n})\\
\lmapdown{ \hat{\theta}_{i_1,\ldots,i_n} }&&\rmapdown{ \hat{\theta}_n^{-t} }\\
 \hat{D}(B_{i_1}\otimes\cdots\otimes B_{i_n})&\leftmap{ \ \hat{D}(\sigma_{i_1,\ldots,i_n}) \ }&\hat{D}(\bigoplus_{j_1+\cdots+j_n=t}B_{j_1}\otimes\cdots\otimes B_{j_n}).\\
 \end{matrix}$$
 In order to show that this square commutes, take a typical generator  $$\sum_{j_1+\cdots+j_n=t}\alpha_{j_1}\otimes\cdots\otimes \alpha_{j_n}\in 
 \bigoplus_{j_1+\cdots+j_n=t}\hat{D}(B_{j_1})\otimes\cdots\otimes \hat{D}(B_{j_n})$$ 
 and a typical generator $a_{i_1}\otimes\cdots\otimes a_{i_n}\in B_{i_1}\otimes\cdots\otimes B_{i_n}$. Then, if we make 
 $\Delta'=\hat{D}(\sigma_{i_1,\ldots,i_n}) 
 \hat{\theta}_n^{-t}
 (\sum \alpha_{j_1}\otimes\cdots\otimes \alpha_{j_n})
 (a_{i_1}\otimes\cdots\otimes a_{i_n})$, we obtain 
 $$\begin{matrix}
    \Delta'&=&\hat{D}(\sigma_{i_1,\ldots,i_n}) 
 \hat{\theta}_n
 (\sum \alpha_{j_1}\otimes\cdots\otimes \alpha_{j_n})
 (a_{i_1}\otimes\cdots\otimes a_{i_n})\hfill\\
 &=&
 \hat{\theta}_n
 (\sum \alpha_{j_1}\otimes\cdots\otimes \alpha_{j_n})\sigma_{i_1,\ldots,i_n}
 (a_{i_1}\otimes\cdots\otimes a_{i_n})\hfill\\
 &=&
  \hat{\theta}_n
 (\sum \alpha_{j_1}\otimes\cdots\otimes \alpha_{j_n})
 (a_{i_1}\otimes\cdots\otimes a_{i_n})\hfill\\
 &=&
 \hat{\theta}_{i_1,\ldots,i_n}(\alpha_{i_1}\otimes\cdots\otimes \alpha_{i_n})(a_{i_1}\otimes\cdots\otimes a_{i_n})\hfill\\
 &=&
 \hat{\theta}_{i_1,\ldots,i_n}\hat{\pi}_{i_1,\ldots,i_n}(\sum \alpha_{j_1}\otimes\cdots\otimes \alpha_{j_n})(a_{i_1}\otimes\cdots\otimes a_{i_n}).\hfill\\
   \end{matrix}$$
 \end{proof}

 \begin{lemma}\label{L: formula para hat(d)(bn)} 
 Given $n\geq 2$ and $t\in \hueca{Z}$. With the notation of
 (\ref{L: theta vs componentes homogeneas}), consider the compositions 
 $$b^t_{i_1,\ldots,i_n}=(B_{i_1}\otimes B_{i_2}\otimes\cdots\otimes B_{i_n}\rightmap{ \ \sigma_{i_1,\ldots,i_n} \ }(B^{\otimes n})_t\rightmap{ \ b_n^t \ }B_{t+1}).$$ 
 Then, the components of the morphism 
 $$\hat{d}(b_n)=\bigoplus_{t\in \hueca{Z}}\hat{d}(b_n)^{-t-1}:\bigoplus_{t\in \hueca{Z}}\hat{D}(B)_{-t-1}\rightmap{}\bigoplus_{t\in \hueca{Z}}(\hat{D}(B)^{\otimes n})_{-t}$$ satisfies the following  formula for $t\in \hueca{Z}$,
 $$\hat{d}(b_n)^{-t-1}=\sum_{i_1+\cdots +i_n=t}\hat{\sigma}_{i_1,\ldots,i_n}\hat{d}(b_{i_1,\ldots,i_n}^t),$$
 where $\hat{d}(b_{i_1,\ldots,i_n}^t)$ is by definition the composition $\hat{\tau}_{i_1,\ldots,i_n}\hat{D}(b^t_{i_1,\ldots,i_n})$.
 \end{lemma}

 \begin{proof} We have the morphisms $b_n:B^{\otimes n}\rightmap{}B$ and $\hat{D}(b_n):\hat{D}(B)\rightmap{}\hat{D}(B^{\otimes n})$. Therefore, we have the restriction morphism  
 $\hat{D}(b_n)^{-t}:\hat{D}(B)_{-t}\rightmap{}\hat{D}(B^{\otimes n})_{-t+1}.$
 
 We also have  $\hat{\tau}_n:\hat{D}(B^{\otimes n})\rightmap{}\hat{D}(B)^{\otimes n}$, so we have the restriction morphism 
 $\hat{\tau}^{-t}_n:\hat{D}(B^{\otimes n})_{-t}\rightmap{}(\hat{D}(B)^{\otimes n})_{-t}.$
 
 Recall that $\hat{d}(b_n)=\hat{\tau}_n\hat{D}(b_n)$, hence 
 $\hat{d}(b_n)^{-t-1}=\hat{\tau}_n^{-t}\hat{D}(b_n)^{-t-1}$. 
We have the following diagram
  $$\begin{matrix}
     \hat{D}(B)_{-t-1}&\rightmap{ \ \hat{D}(b_n)^{-t-1} \ }&
     \hat{D}(B^{\otimes n})_{-t}&\rightmap{ \ \hat{D}(\sigma_{i_1,\ldots,i_n})^{-t} \ }&\hat{D}(B_{i_1}\otimes\cdots\otimes B_{i_n})\\
     \rmapdown{\hat{d}(b_n)^{-t-1}}&&\rmapdown{\hat{\tau}_n^{-t}}&&\lmapdown{\hat{\tau}_{i_1,\ldots,i_n}}\\
     (\hat{D}(B)^{\otimes n})_{-t}&\longequal&(\hat{D}(B)^{\otimes n})_{-t}&\leftmap{ \ \hat{\sigma}_{i_1,\ldots,i_n} \ }&\hat{D}(B_{i_1})\otimes\cdots\otimes \hat{D}(B_{i_n}).\\
    \end{matrix}$$
Since $id_{(B^{\otimes n})_t}=\sum_{j_1+\cdots+j_n=t}\sigma_{j_1,\ldots,j_n}\pi_{j_1,\ldots,j_n}$, we get $id_{\hat{D}((B^{\otimes n})_t)}=id_{\hat{D}(B^{\otimes n})_{-t}}=\sum_{j_1+\cdots+j_n=t}\hat{D}(\pi_{j_1,\ldots,j_n})\hat{D}(\sigma_{j_1,\ldots,j_n})$. Therefore, from (\ref{L: theta vs componentes homogeneas}),
we obtain,
$$\begin{matrix}
  \hat{d}(b_n)^{-t-1}
  &=&
  \hat{\tau}_n^{-t}\hat{D}(b_n)^{-t-1}\hfill\\
  &=&
  \sum_{j_1+\cdots+j_n=t}\hat{\tau}_n^{-t}\hat{D}(\pi_{j_1,\ldots,j_n})\hat{D}(\sigma_{j_1,\ldots,j_n})\hat{D}(b_n)^{-t-1}\hfill\\
  &=&
  \sum_{j_1+\cdots+j_n=t}\hat{\sigma}_{j_1,\ldots,j_n}\hat{\tau}_{j_1,\ldots,j_n}\hat{D}(\sigma_{j_1,\ldots,j_n})\hat{D}(b_n)^{-t-1}\hfill\\
  &=&
  \sum_{j_1+\cdots+j_n=t}\hat{\sigma}_{j_1,\ldots,j_n}\hat{\tau}_{j_1,\ldots,j_n}\hat{D}(b^t_{j_1,\ldots,j_n}),\hfill\\
  \end{matrix}$$
because $\hat{D}(b_n^t)=\hat{D}(b_n)^{-t-1}$. 
  \end{proof}
 
\begin{proposition}\label{P: donde se define beta} 
For any $n\geq 2$,  consider the restriction 
$$b_{n,0}=(B_0^{\otimes n}\subseteq (B^{\otimes n})_0\rightmap{b_n^0}B_1)$$
of the morphism 
 $b_n:B^{\otimes n}\rightmap{}B$. 
 So, $b_{n,0}=b^0_{i_1,\ldots,i_n}$, where $i_1=\cdots=i_n=0$ and, by definition, we have 
 $\hat{d}(b_{n,0})=\hat{\tau}_{i_1,\ldots,i_n}\hat{D}(b_{n,0}):\hat{D}(B_1)\rightmap{}\hat{D}(B_0)^{\otimes n}$. Then, we have  a morphism of $S$-$S$-bimodules 
$$\beta:\hat{D}(B_1)\rightmap{}T_S(\hat{D}(B)_0)\subseteq \widehat{T}_S(\hat{D}(B)),$$
 defined by $\beta(x)=\sum_{n\geq 2}\hat{d}(b_{n,0})(x)$,  for $x\in \hat{D}(B_1)$. Moreover, for $x\in \hat{D}(B_1)$, we have 
 $$d(x)=\beta(x)+h(x), \hbox{ with } h(x)\in {\cal N}.$$
\end{proposition}

\begin{proof}  For each $x\in \hat{D}(B_1)$, we have that $\hat{d}(b_{n,0})(x)\in \hat{D}(B_0)^{\otimes n}$ is homogeneous of degree $0$. Moreover, from (\ref{L: sobre elementos homogeneos de TS(DB)})(1), we get that $\hat{d}(b_{n,0})(x)=0$, for all $n>\ell$. Thus we get $\beta(x)=\sum_{n=2}^\ell \hat{d}(b_{n,0})(x)\in T_S(\hat{D}(B)_0)$. 
 
From (\ref{L: formula para hat(d)(bn)})  we have 
$$d(x)=\sum_{n\geq 2}\hat{d}(b_n)(x)=\sum_{n\geq 2}\sum_{i_1+\cdots+i_n=0}\hat{\sigma}_{i_1,\ldots,i_n}\hat{d}(b^0_{i_1,\ldots,i_n})(x).$$
Here, for each index $(i_1,\ldots,i_n)$ with $i_1+\cdots+i_n=0$, we have   $\hat{d}(b^0_{i_1,\ldots,i_n})(x)\in \hat{D}(B_{i_1})\otimes\cdots\otimes\hat{D}(B_{i_n})=\hat{D}(B)_{-i_1}\otimes\cdots\otimes\hat{D}(B)_{-i_n}$. If some $i_s>0$, then 
$\hat{d}(b_{i_1,\ldots,i_n}^0)(x)\in {\cal N}$. If all $i_1,\ldots,i_n\leq 0$, having  zero sum implies that they are all zero and $\hat{d}(b^0_{i_1,\ldots,i_n})(x)=\hat{d}(b_{n,0})(x)$. Thus, we get the desired formula.
\end{proof}

\begin{lemma}\label{L: coefs de hat(d)(bn)} 
From (\ref{N: la base de B}),  for $\alpha\in \hat{D}(B)$ and $n\geq 2$, we have 
  $$\hat{d}(b_n)(\alpha)=\sum_{z\in \widehat{\hueca{T}}_n} c_z^\alpha z\in \hat{D}(B)^{\otimes n}, \hbox{ for some   scalars } c_z^\alpha\in k,$$
  where the elements  $z\in\widehat{\hueca{T}}_n$ have the form $z=x_{i_n}^*\otimes\cdots\otimes x_{i_1}^*$,  for some directed basic elements $x_{i_1},\ldots,x_{i_n}\in \hueca{B}$, with each $x_{i_j}\in \hueca{B}_{-i_j}$. Moreover, we  have that
  $$e_{t(x_{i_n})}c^\alpha_{x_{i_n}^*\otimes\cdots\otimes x_{i_1}^*}=\alpha(b_n(x_{i_n}\otimes\cdots\otimes x_{i_1}))$$
  for each $x_{i_n}^*\otimes\cdots\otimes x_{i_1}^*\in \widehat{\hueca{T}}_n$. 
\end{lemma}

\begin{proof} From (\ref{N: la base de B}), we have the expression described for $\hat{d}(b_n)(\alpha)$ in the statement of our lemma. Thus, 
applying the inverse $\hat{\theta}$ of $\hat{\tau}$, we obtain 
$$\hat{D}(b_n)(\alpha)=\sum_{z\in \widehat{\hueca{T}}_n} \hat{\theta}(c_z^\alpha z)\in \hat{D}(B^{\otimes n}).$$
Evaluating both sides in the preceding equality at a fixed homogeneous basic element $x_{j_n}\otimes\cdots\otimes x_{j_1}\in B^{\otimes n}$,   we get
$$\begin{matrix}\alpha(b_n(x_{j_n}\otimes\cdots\otimes x_{j_1}))
&=&
\sum_z c^\alpha_zx^*_{i_n}(x_{j_n}x^*_{i_{n-1}}(x_{j_{n-1}}\cdots x^*_{i_1}(x_{j_1})\cdots )\hfill\\
&=&
c^\alpha_{z_0}x^*_{j_n}(x_{j_n}x^*_{j_{n-1}}(x_{j_{n-1}}\cdots x^*_{j_1}(x_{j_1})\cdots )\hfill\\
&=& c_{z_0}^\alpha e_{t(x_{j_n})},\hfill\\\end{matrix}$$
where $z_0=x^*_{j_n}\otimes\cdots\otimes x^*_{j_1}\in \widehat{\hueca{T}}_n$.
Indeed, we have  $x_{i_1}^*(x_{j_1})=\delta_{i_1,j_1}e_{t(x_{j_1})}$; $x_{i_2}^*(x_{j_2}e_{t(x_{j_1})})=x_{i_2}^*(x_{j_2})=\delta_{i_2,j_2}e_{t(x_{j_2})}$ and, in a finite number of steps, we have the last preceding equality.  
\end{proof}

 \begin{proposition}\label{P: diferencial de basicos}
  There are morphisms of $S$-$S$-bimodules 
  $$\delta_0:\hat{D}(B)_0\rightmap{}[T_S(\hat{D}(B)_0\oplus \hat{D}(J))]_1$$ 
 $$\delta_1:\hat{D}(J)\rightmap{}[T_S(\hat{D}(B)_0\oplus \hat{D}(J))]_2$$
 where $[T_S(\hat{D}(B)_0\oplus \hat{D}(J))]_1$ and 
 $[T_S(\hat{D}(B)_0\oplus \hat{D}(J))]_2$ denote the homogeneous components of the graded tensor algebra $T_S(\hat{D}(B)_0\oplus \hat{D}(J))$ of degree $1$ and $2$, respectively, 
 such that the differential $d$ of the algebra $\widehat{T}_S(\hat{D}(B))$ satisfies the following.
 \begin{enumerate}
  \item For any $i\in {\cal P}$, we have 
  $$d(e_i^*)=e_i^*\otimes e_i^*+h(e_i^*), 
  \hbox{ \ with \ } h(e_i^*)\in {\cal N}.$$
  \item For any directed basic element $x\in e_j\hueca{B}_0e_i$, hence with $i\not=j$, we have 
  $$d(x^*)=e_j^*\otimes x^*-x^*\otimes e_i^*+\delta_0(x^*)+h(x^*), \hbox{ \ with \ } h(x^*)\in {\cal N}.$$ 
  \item For any directed basic element $x\in e_j\hueca{J}e_i$, we have 
  $$d(x^*)=e_j^*\otimes x^*+x^*\otimes e^*_i+\delta_1(x^*)+h(x^*), \hbox{ \ with \ } h(x^*)\in {\cal N}.$$ 
 \end{enumerate}
 \end{proposition}

 \begin{proof}   By definition of the differential $d$, for any homogeneous element $\alpha\in \hat{D}(B)$,  we have $d(\alpha)=\sum_{n\geq 2}\hat{d}(b_n)(\alpha)$. 
 If  $\vert\alpha\vert\in \{ 0,1\}$,  from (\ref{L: coefs de hat(d)(bn)}), we have
 $$\hat{d}(b_n)(\alpha)=\sum_{z\in \widehat{\hueca{T}}_n}c_z^{\alpha}z.$$
 Let $\widehat{\hueca{T}}_n(\alpha)$ be the subset of $\widehat{\hueca{T}}_n$ formed by the basic elements $z=x^*_{i_n}\otimes\cdots\otimes x^*_{i_1}\in \widehat{\hueca{T}}_n$ with $\sum_{s=1}^n\vert x_{i_s}^*\vert=\vert x^*_{i_n}\otimes\cdots\otimes x_{i_1}^*\vert=\vert \alpha\vert+1$ and 
 $\vert x_{i_1}^*\vert,\ldots,\vert x^*_{i_n}\vert\in \{0,1\}$. 
 Then, we have 
 $$\hat{d}(b_n)(\alpha)=g_n(\alpha)+h_n(\alpha),$$
 where 
 $h_n(\alpha)=\sum_{z\in \widehat{\hueca{T}}_n\setminus{\widehat{\hueca{T}}_n(\alpha)}}c_z^{\alpha}z\in {\cal N}$ and $g_n(\alpha)=\sum_{z\in \widehat{\hueca{T}}_n(\alpha)}c_z^{\alpha}z$.
 If we make $h(\alpha)=\sum_{n\geq 2}h_n(\alpha)$ and $g(\alpha)=\sum_{n\geq 2}g_n(\alpha)$, we get $h(\alpha)\in {\cal N}$ and we are reduced to the  computation of $g(\alpha)=d(\alpha)-h(\alpha)$, that is of  $g_n(\alpha)$, for $n\geq 2$, 
 in each one of the  items of the proposition.

\medskip
 \noindent{(1):} Assume that $\alpha=e_i^*$. 
 \medskip

 We have an element with $\vert e_i^*\vert=1$. The subset $\widehat{\hueca{T}}_n(e_i^*)$ of $\widehat{\hueca{T}}_n$ consists of the basic elements $z=x^*_{i_n}\otimes\cdots\otimes x^*_{i_1}\in \widehat{\hueca{T}}_n$ with $\sum_{s=1}^n\vert x_{i_s}^*\vert=\vert x^*_{i_n}\otimes\cdots\otimes x_{i_1}^*\vert=2$ and 
 $\vert x_{i_1}^*\vert,\ldots,\vert x^*_{i_n}\vert\in \{0,1\}$. 
 
 For $n=2$,  we want to compute 
 $g_2(e_i^*)=\sum_{z\in \widehat{\hueca{T}}_2(e_i^*)}c_z^{e_i^*}z$. If $z=x^*_{i_2}\otimes x^*_{i_1}\in \widehat{\hueca{T}}_2(e_i^*)$, we have $\vert x_{i_1}^*\vert=1=\vert x^*_{i_2}\vert$. So, $x_{i_1},x_{i_2}\in B_{-1}=A_0=\End_\Lambda(\Delta)$. So they 
 belong to  $\{e_j\mid j\in {\cal P}\}$  or they belong to the radical
 $J$. Having in mind (\ref{L: coefs de hat(d)(bn)}) and the definition of $b_2$ in (\ref{D: construccion barra}), we have 
 $$e_{t(x_{i_2})}c^{e_i^*}_{z}=e_i^*(b_2(x_{i_2}\otimes x_{i_1}))=-(-1)^{\vert x_{i_2}\vert}e_i^*\sigma(\sigma^{-1}(x_{i_2})\sigma^{-1}(x_{i_1})).$$ 
 The element $e_i^*\sigma(\sigma^{-1}(x_{i_2})\sigma^{-1}(x_{i_1}))$ is not zero iff $x_{i_2}=e_i=x_{i_1}$. Therefore, we obtain 
 $g_2(e_i^*)=\sum_{z\in \widehat{\hueca{T}}_2(e_i^*)}c_z^{e_i^*}z=e_i^*\otimes e_i^*$. 
 
 For $n\geq 3$, we want to compute 
 $g_n(e_i^*)=\sum_{z\in \widehat{\hueca{T}}_n(e_i^*)}c_z^{e_i^*}z$.
 If we consider an element   $z=x_{i_n}^*\otimes\cdots\otimes x^*_{i_1}\in \widehat{\hueca{T}}_n(e_i^*)$, we must have some  $\vert x_{i_s}\vert=0$. This implies, by (\ref{R: B y D(B) homogeneos}), that $\overline{s(x_{i_s}^*)}<\overline{t(x^*_{i_s})}$. Then, $\overline{s(x^*_{i_1})}<\overline{t(x^*_{i_n})}$,  and $\overline{s(b_n(x_{i_n}\otimes\cdots\otimes x_{i_1}))}<\overline{t(b_n(x_{i_n}\otimes\cdots\otimes x_{i_1}))}$.  Hence, $e_{t(x_{i_n})}c_z^{e_i^*}=e_i^*(b_n(x_{i_n}\otimes\cdots\otimes x_{i_1}))=0$, and $c_z^{e_i^*}=0$. 
 Thus, $g_n(e_i^*)=\sum_{z\in \widehat{\hueca{T}}_n(e_i^*)}c_z^{e_i^*}z=0$.
 
 As a consequence, we obtain $d(e_i^*)-h(e_i^*)=\sum_{n\geq 2}g_n(e_i^*)=e_i^*\otimes e_i^*$. 
 
 \medskip
 \noindent{(2):} Suppose $x$ is a directed basic element in $e_j\hueca{B}_0e_i$. 
 \medskip
  
  We know that $i\not=j$, due to (\ref{R: B y D(B) homogeneos})(3).
  
  Here $\alpha=x^*$ has degree $\vert x^*\vert=0$ and $\widehat{\hueca{T}}_n(x^*)$  consists   
 of the basic elements $z=x^*_{i_n}\otimes\cdots\otimes x^*_{i_1}\in \widehat{\hueca{T}}_n$ with $\sum_{s=1}^n\vert x_{i_s}^*\vert=\vert x^*_{i_n}\otimes\cdots\otimes x_{i_1}^*\vert=1$ and 
 $\vert x_{i_1}^*\vert,\ldots,\vert x^*_{i_n}\vert\in \{0,1\}$.
 
 For $n=2$,  we want to compute 
 $g_2(x^*)=\sum_{z\in \widehat{\hueca{T}}_2(x^*)}c_z^{x^*}z$. Given any element  $z=x^*_{i_2}\otimes x^*_{i_1}\in \widehat{\hueca{T}}_2(x^*)$, we have 
  $$e_{t(x_{i_2})}c^{x^*}_{z}=
  x^*(b_2(x_{i_2}\otimes x_{i_1}))=-(-1)^{\vert x_{i_2}\vert}x^*(\sigma(\sigma^{-1}(x_{i_2})\sigma^{-1}(x_{i_1}))).$$
 There are the following two possibilities: 
 \begin{enumerate}
  \item[(a)] If $\vert x_{i_2}^*\vert=1$ and $\vert x^*_{i_1}\vert=0$. 
  Then, we have $x_{i_2}\in B_{-1}=A_0=\End_\Lambda(\Delta)$. Thus, 
  if moreover, $x_{i_2}\in J$ we have $z=x^*_{i_2}\otimes x^*_{i_1}\in T_S(\hat{D}(B)_0\oplus \hat{D}(J))$. If $x_{i_2}\not\in J$, then $x_{i_2}=e_s$, for some $s\in {\cal P}$. Then, 
  $$e_{t(x_{i_2})}c^{x^*}_{z}=
  x^*(\sigma(\sigma^{-1}(e_s)\sigma^{-1}(x_{i_1}))).$$
  As before, if this expression is not zero, we must have 
  $x_{i_1}=x$ and $e_s=e_j$. Therefore, $z=e^*_j\otimes x^*$ and   $c^{x^*}_{z}=1$. 
 \item[(b)] If $\vert x_{i_2}^*\vert=0$ and $\vert x^*_{i_1}\vert=1$. Here we have $x_{i_1}\in B_{-1}=A_0=\End_\Lambda(\Delta)$. Thus, 
  if moreover, $x_{i_1}\in J$ we have $z=x^*_{i_2}\otimes x^*_{i_1}\in T_S(\hat{D}(B)_0\oplus \hat{D}(J))$. If $x_{i_1}\not\in J$, then $x_{i_1}=e_s$, for some $s\in {\cal P}$. Then, 
  $$e_{t(x_{i_2})}c^{x^*}_{z}=-x^*(\sigma(\sigma^{-1}(x_{i_2})\sigma^{-1}(e_s))).$$
  As before, if this element is not zero, we must have 
  $x_{i_2}=x$ and $e_s=e_i$. Therefore, $z=x^*\otimes e^*_i$ and   $c^{x^*}_{z}=-1$. 
 \end{enumerate}
 From the preceding argument, we obtain 
 $$g_2(x^*)=e_j^*\otimes x^*-x^*\otimes e_i^*+\delta_0(x^*),
 \hbox{ with }\delta_0(x^*)\in [T_S(\hat{D}(B)_0\oplus \hat{D}(J))]_1.$$ 
 
 For $n\geq 3$,  we want to compute 
 $g_n(x^*)=\sum_{z\in \widehat{\hueca{T}}_n(x^*)}c_z^{x^*}z$. Here, the basic elements  $z=x^*_{i_n}\otimes \cdots\otimes x^*_{i_1}\in \widehat{\hueca{T}}_n(x^*)$  have degree $\vert z\vert=1$, thus only one of $x^*_{i_1},\ldots,x^*_{i_n}$ has degree 1 and the others have degree 0. If $\vert x_{i_s}\vert=1$, then  $x_{i_s}=e_r$, for some $r\in {\cal P}$, or $x_{i_s}\in J$. The first possibility implies that we have $b_n(x_{i_n}\otimes\cdots\otimes x_{i_1})=0$, see (\ref{R: merkulov}); therefore, $c_z^{x^*}=0$. Thus, if $c_z^{x^*}\not=0$, 
 we get $g_n(x^*)\in [T_S(\hat{D}(B)_0\oplus \hat{D}(J))]_1$.
  It follows that
   $$d(x^*)-h(x^*)=\sum_{n\geq 2}g_n(x)=e_j^*\otimes x^*-x^*\otimes e_i^*+\delta_0(x^*),$$
  where $\delta_0(x^*)\in [T_S(\hat{D}(B)_0\oplus \hat{D}(J))]_1$. 
 
 \medskip
 \noindent{(3):} Take any directed basic element $x\in e_j\hueca{J}e_i$. 
 \medskip
 
  Here $\alpha=x^*$ has degree $\vert x^*\vert=1$ and the set  $\widehat{\hueca{T}}_n(x^*)$  consists   
 of the basic elements $z=x^*_{i_n}\otimes\cdots\otimes x^*_{i_1}\in \widehat{\hueca{T}}_n$ with $\sum_{s=1}^n\vert x_{i_s}^*\vert=\vert x^*_{i_n}\otimes\cdots\otimes x_{i_1}^*\vert=2$ and 
 $\vert x_{i_1}^*\vert,\ldots,\vert x^*_{i_n}\vert\in \{0,1\}$.
 
 For $n=2$,  we want to compute 
 $g_2(x^*)=\sum_{z\in \widehat{\hueca{T}}_2(x^*)}c_z^{x^*}z$. Given any element  $z=x^*_{i_2}\otimes x^*_{i_1}\in \widehat{\hueca{T}}_2(x^*)$, we have that $x_{i_1}$ and $x_{i_2}$ both have degree 1 and 
  $$e_{t(x_{i_2})}c^{x^*}_{z}=
  x^*(b_2(x_{i_2}\otimes x_{i_1}))=x^*(\sigma(\sigma^{-1}(x_{i_2})\sigma^{-1}(x_{i_1}))).$$
 There are the following three possibilities: 
 \begin{enumerate}
  \item[(a)] We have $x_{i_2}=e_s$, for some $s\in {\cal P}$.   
   Then, 
  $$e_sc^{x^*}_{z}=
  x^*(\sigma(\sigma^{-1}(e_s)\sigma^{-1}(x_{i_1}))).$$
  As before, if this expression is not zero, we must have 
  $x_{i_1}=x$ and $e_s=e_j$. Therefore, $z=e^*_j\otimes x^*$ has coefficient    $c^{x^*}_{z}=1$. 
 \item[(b)] We have $x_{i_1}=e_s$, for some $s\in {\cal P}$. Then, 
  $$e_{t(x_{i_2})}c^{x^*}_{z}=x^*(\sigma(\sigma^{-1}(x_{i_2})\sigma^{-1}(e_s))),$$
  and, if this element is not zero, we must have 
  $x_{i_2}=x$ and $e_s=e_i$. Therefore, $z=x^*\otimes e^*_i$ has coefficient    $c^{x^*}_{z}=1$. 
  \item[(c)] We have  $x_{i_1},x_{i_2}\in J$. Here, we get $z=x^*_{i_2}\otimes x^*_{i_1}\in \hat{D}(J)^{\otimes 2}$. 
 \end{enumerate}
 From the preceding argument, we obtain 
 $$g_2(x^*)=e_j^*\otimes x^*+x^*\otimes e_i+u_2(x^*),
 \hbox{ with }u_2(x^*)\in [T_S(\hat{D}(B)_0\oplus \hat{D}(J))]_2.$$ 
 
 For $n\geq 3$,  we want to compute 
 $g_n(x^*)=\sum_{z\in \widehat{\hueca{T}}_n(x^*)}c_z^{x^*}z$. Here, the basic elements  $z=x^*_{i_n}\otimes \cdots\otimes x^*_{i_1}\in \widehat{\hueca{T}}_n(x^*)$  have degree $\vert z\vert=2$, thus only two of $x^*_{i_1},\ldots,x^*_{i_n}$ have  degree 1, say $x_{i_u}$ and $x_{i_v}$ have degree 1, and the others have degree 0. From (\ref{R: merkulov}), we know that if  $e_{t(x_{i_n})}c_z^{x^*}=x^*(b_n(x_{i_n}\otimes\cdots\otimes x_{i_1}))$   is not zero, we have  $x_{i_u},x_{i_v}\in J$ and, hence, 
 $z\in [T_S(\hat{D}(B)_0\oplus \hat{D}(J))]_2$. 
 Therefore,  
 $g_n(x^*)\in [T_S(\hat{D}(B)_0\oplus \hat{D}(J))]_2$.
 
 It follows that
   $$d(x^*)-h(x^*)=\sum_{n\geq 2}g_n(x^*)=e_j^*\otimes x^*+x^*\otimes e_i^*+\delta_1(x^*),$$
  where $\delta_1(x^*)\in [T_S(\hat{D}(B)_0\oplus \hat{D}(J))]_2$. 
 \end{proof}

 \begin{corollary}\label{C: calculo de dif d en grados 0 y 1 usando la  diferencial delta}
 Consider the differential $\delta$, with $\delta(S)=0$, on the graded tensor algebra $T_S(\hat{D}(B)_0\oplus \hat{D}(J))$ induced by the morphisms $\delta_0$ and $\delta_1$ of the preceding proposition, see (\cite{BSZ}(4.4)). The following holds. 
 
 \begin{enumerate}
  \item For any $i\in {\cal P}$, we have 
  $$d(e_i^*)=e_i^*\otimes e_i^*+h(e_i^*), \hbox{ where } h(e^*_i)\in {\cal N}.$$
  \item For any directed element $\gamma\in e_j*T_S(\hat{D}(B)_0)*e_i\setminus S$, we have 
  $$d(\gamma)=e_j^*\otimes \gamma-\gamma\otimes e_i^*+\delta(\gamma)+h(\gamma), \hbox{ where }  h(\gamma)\in {\cal N}.$$
  \item For any directed  element $\gamma\in e_j*[T_S(\hat{D}(B)_0\oplus \hat{D}(J))]_1*e_i$, we have 
  $$d(\gamma)=e_j^*\otimes \gamma+\gamma\otimes e^*_i+\delta(\gamma)+h(\gamma), \hbox{ where }
   h(\gamma)\in {\cal N}.$$
 \end{enumerate}
 \end{corollary}

 \begin{proof} (1):  This is clear. 
 
 (2): If the statement (2) holds for $\gamma_1=e_j*\gamma_1*e_i$ and 
 $\gamma_2=e_t*\gamma_2*e_j$ in $T_S(\hat{D}(B)_0)$, we have 
 $$\begin{matrix}
   d(\gamma_2\otimes\gamma_1)&=&d(\gamma_2)\otimes\gamma_1+ \gamma_2\otimes d(\gamma_1)\hfill\\
   &=&
   e_t^*\otimes \gamma_2\otimes \gamma_1-\gamma_2\otimes e_j^*\otimes\gamma_1+\delta(\gamma_2)\otimes\gamma_1+h(\gamma_2)\otimes\gamma_1\hfill\\
   &&+\,
   \gamma_2\otimes e_j^*\otimes \gamma_1-\gamma_2\otimes \gamma_1\otimes e_i^*+\gamma_2\otimes \delta(\gamma_1)+\gamma_2\otimes h(\gamma_1)\hfill\\
   &=&
   e_t^*\otimes (\gamma_2\otimes \gamma_1)-
   (\gamma_2\otimes \gamma_1)\otimes e_i^* +
   \delta(\gamma_2\otimes\gamma_1)+h(\gamma_2\otimes\gamma_1),\hfill\\
   \end{matrix}$$
where $h(\gamma_2\otimes\gamma_1)=\gamma_2\otimes h(\gamma_1)+h(\gamma_2)\otimes\gamma_1\in {\cal N}$. 
The elements in  the vector space $e_j*T_S(\hat{D}(B)_0)*e_i\setminus S$ are linear combinations of tensors of the form $\gamma=x_{i_n}^*\otimes\cdots\otimes x_{i_1}^*$, where $n\geq 1$ and each $x_{i_1},\ldots,x_{i_n}\in \hueca{B}_0$ are directed elements with distinct source and target points by (\ref{R: B y D(B) homogeneos}). Hence, by (\ref{P: diferencial de basicos})(2), we have  $d(x^*)=e_j^*\otimes x-x\otimes e_i^*+\delta(x^*)+h(x^*)$, with $h(x^*)\in {\cal N}$, for each 
$x\in \{x_{i_1},\ldots,x_{i_n}\}$. Then, we can use induction and the preceding argument to obtain (2).  

\medskip
(3) If $\gamma=\gamma_2\otimes\gamma_1$ has degree $\vert\gamma\vert=1$, 
with $\gamma_1=e_j*\gamma_1*e_i$ and 
 $\gamma_2=e_t*\gamma_2*e_j$  in $T_S(\hat{D}(B)_0\oplus \hat{D}(J))$, one of them has degree 0, so item (2) holds for it and we assume that (3) holds for the other one. We have the following two cases.
 
 \medskip
 \noindent\emph{Case 1:  $\vert \gamma_2\vert=1$ and $\vert \gamma_1\vert=0$.}
 \medskip
 
 We have 
 $$\begin{matrix}
   d(\gamma_2\otimes\gamma_1)&=&d(\gamma_2)\otimes\gamma_1- \gamma_2\otimes d(\gamma_1)\hfill\\
   &=&
   e_t^*\otimes \gamma_2\otimes \gamma_1+\gamma_2\otimes e^*_j\otimes \gamma_1+\delta(\gamma_2)\otimes \gamma_1+h(\gamma_2)\otimes \gamma_1\hfill\\
   &&-\,
   \gamma_2\otimes e_j^*\otimes \gamma_1+\gamma_2\otimes \gamma_1\otimes e_i^*-\gamma_2\otimes \delta(\gamma_1)-\gamma_2\otimes h(\gamma_1)\hfill\\
   &=&
   e_t^*\otimes (\gamma_2\otimes \gamma_1)+
   (\gamma_2\otimes \gamma_1)\otimes e_i^* +\delta(\gamma_2\otimes\gamma_1)+h(\gamma_2\otimes\gamma_1),\hfill\\
   \end{matrix}$$
 where $h(\gamma_2\otimes\gamma_1)=-\gamma_2\otimes h(\gamma_1)+h(\gamma_2)\otimes\gamma_1\in {\cal N}$.  
 
 \medskip 
 \noindent\emph{Case 2:  $\vert \gamma_2\vert=0$ and $\vert \gamma_1\vert=1$.}
 \medskip
 
We have 
 $$\begin{matrix}
   d(\gamma_2\otimes\gamma_1)&=&d(\gamma_2)\otimes\gamma_1+ \gamma_2\otimes d(\gamma_1)\hfill\\
   &=&
   e_t^*\otimes \gamma_2\otimes \gamma_1-\gamma_2\otimes e^*_j\otimes \gamma_1+\delta(\gamma_2)\otimes \gamma_1+h(\gamma_2)\otimes \gamma_1\hfill\\
   &&+\,
   \gamma_2\otimes e_j^*\otimes \gamma_1+\gamma_2\otimes \gamma_1\otimes e_i^*+\gamma_2\otimes \delta(\gamma_1)+\gamma_2\otimes h(\gamma_1)\hfill\\
   &=&
   e_t^*\otimes (\gamma_2\otimes \gamma_1)+
   (\gamma_2\otimes \gamma_1)\otimes e_i^* +\delta(\gamma_2\otimes\gamma_1)+h(\gamma_2\otimes\gamma_1),\hfill\\
   \end{matrix}$$
 where $h(\gamma_2\otimes\gamma_1)=\gamma_2\otimes h(\gamma_1)+h(\gamma_2)\otimes\gamma_1\in {\cal N}$.  

 As before, the elements in  the vector space $e_j*[T_S(\hat{D}(B)_0\oplus \hat{D}(J))]_1*e_i$ are linear combinations of tensors of the form $\gamma=x_{i_n}^*\otimes\cdots\otimes x_{i_1}^*$, where $n\geq 1$ and, with only one exception, each $x_{i_1},\ldots,x_{i_n}\in \hueca{B}_0$ are directed elements with distinct source and target points by (\ref{R: B y D(B) homogeneos}). If $x=x_{i_u}$ is the exception, 
 we have $x_{i_u}\in e_tJe_s$ and hence, by (\ref{P: diferencial de basicos})(3), we get  $d(x^*)=e_t^*\otimes x^*+x^*\otimes e_s^*+\delta(x^*)+h(x^*)$, with $h(x^*)\in {\cal N}$. Then, we can use induction and the preceding argument to obtain (3).  
 \end{proof}

 \begin{proposition}\label{P: I entrelazado}
 Consider the ideal $I$ of the algebra $T_S(\hat{D}(B)_0)$ generated by $\Im\beta$, where $\beta$ is the map introduced in (\ref{P: donde se define beta}). Consider also the $S$-$S$-subbimodule $V:= [T_S(\hat{D}(B)_0\oplus \hat{D}(J))]_1$ of $T_S(\hat{D}(B)_0\oplus \hat{D}(J))$ and the differential $\delta$ on this algebra introduced in (\ref{C: calculo de dif d en grados 0 y 1 usando la  diferencial delta}). Then the ideal $I$ has  the following properties. 
 \begin{enumerate}
  \item $\delta(I)\subseteq IV+VI$
  \item $\delta^2(T_S(\hat{D}(B)_0))\subseteq IV^2+VIV+V^2I$, and
  \item $\delta^2(V)\subseteq IV^3+VIV^2+V^2IV+V^3I$.
 \end{enumerate}
 \end{proposition}

 \begin{proof} Consider the linear projection 
 $$\pi:\widehat{T}_S(\hat{D}(B))\rightmap{}\widehat{T}_S(\hat{D}(B)_0\oplus \hat{D}(J))$$
 associated to the vector space decomposition of (\ref{L: desc del algebra completada}). We proceed in four  steps.
 
 \medskip
 \noindent\emph{Step 1: For any homogeneous element $h\in {\cal N}$ we have 
 $$\pi d(h)\in \sum_{\scriptsize\begin{matrix}r+s=\vert h\vert+1\\ r,s\geq 0\end{matrix}}V^rIV^s.$$}
  \medskip
  
  We have $h=\sum_{s\geq 1}h_s$, with $h_s\in \hat{D}(B)^{\otimes s}$ and $\vert h_s\vert=\vert h\vert$, for all $s$. 
 By definition, $d(h)=\sum_{s\geq 1} d(h_s)$. Each $h_s$ is a finite $k$-linear combination  
 of elements of the form $x^*_{i_s}\otimes\cdots\otimes x_{i_1}^*$, with $x_{i_1},\ldots,x_{i_s}\in \hueca{B}$. Then, $d(h_s)$ is a finite $k$-linear combination of elements of the form 
 $$x^*_{i_s}\otimes\cdots\otimes x^*_{i_{r+1}}\otimes d(x^*_{i_r})\otimes x^*_{i_{r-1}}\otimes\cdots\otimes x^*_{i_1}, \hbox{ with } r\in [1,s].$$
 If $\vert x^*_{i_r}\vert<-1$, then $\vert d(x^*_{i_r})\vert<0$ and this element belongs to ${\cal N}$. 
 If $\vert x^*_{i_r}\vert=-1$, then, by (\ref{P: donde se define beta}), we have  $d(x^*_{i_r})=\beta(x^*_{i_r})+h(x^*_{i_r})$, with $h(x^*_{i_r})\in {\cal N}$. This implies that $\pi d(h)$ is a finite linear combination of elements with degree $\vert h\vert +1$ of the form
  $$x^*_{i_s}\otimes\cdots\otimes x^*_{i_{r+1}}\otimes \beta(x^*_{i_r})\otimes x^*_{i_{r-1}}\otimes\cdots\otimes x^*_{i_1},$$
 such that, for $j\not=r$, we have $\vert x_{i_j}\vert\geq 0$ and if $\vert x_{i_j}\vert =1$, we have $x_{i_j}\in J$.

  \medskip
 \noindent\emph{Step 2: We have $\delta(I)\subseteq VI+IV$.}
  \medskip
  
 By Leibniz rule, it will be enough to show that $\delta(\Im\beta)\subseteq VI+IV$. 
 Take any $x^*\in e_j*\hat{D}(B_1)*e_i$, then from (\ref{P: donde se define beta}), we have  
 $$d(x^*)=\beta(x^*)+h(x^*), \hbox{ with } h(x^*)\in {\cal N}.$$
 Then, from (\ref{C: calculo de dif d en grados 0 y 1 usando la  diferencial delta}), we have  
 $$\begin{matrix}
    0=d^2(x^*)&=&d(\beta(x^*))+d(h(x^*))\hfill\\                                                                                              
&=&
e_j^*\otimes \beta(x^*)-\beta(x^*)\otimes e_i^*+\delta(\beta(x^*))+h(\beta(x^*))+d(h(x^*)).\hfill\\
    \end{matrix}$$
Applying $\pi$ to this equality we obtain $0=\delta(\beta(x^*))+\pi d(h(x^*))$. Therefore, we obtain 
$\delta(\beta(x^*))=-\pi d(h(x^*))\in VI+IV$, by the first step of this proof.

\medskip
\noindent\emph{Step 3: We have $\delta^2(T_S(\hat{D}(B)_0))\subseteq IV^2+VIV+V^2I$.}
\medskip
 
 It will be enough to prove that $\delta^2(\gamma)\in IV^2+VIV+V^2I$, for any directed element $\gamma\in e_j*\hat{D}(B)_0*e_i$. 
 From (\ref{C: calculo de dif d en grados 0 y 1 usando la  diferencial delta}), we have  
  $$d(\gamma)=e_j^*\otimes \gamma-\gamma\otimes e_i^*+\delta(\gamma)+h(\gamma), \hbox{ where }  h(\gamma)\in {\cal N}.$$
  Then, we get 
 $0=d^2(\gamma)=
d(e_j^*\otimes \gamma)-d(\gamma\otimes e_i^*)+d(\delta(\gamma))+d(h(\gamma)).$
 Using Leibniz rule on the first two summands, we obtain 
 $$\begin{matrix}d(e_j^*\otimes \gamma)&=&
    d(e^*_j)\otimes \gamma-e^*_j\otimes d(\gamma)\hfill\\
    &=&
    e^*_j\otimes e_j^*\otimes \gamma+h(e_j^*)\otimes\gamma-e_j^*\otimes d(\gamma), 
    \hbox{ and }\hfill\\
   \end{matrix}$$
 $$\begin{matrix}
 d(\gamma\otimes e_i^*)&=&
 d(\gamma)\otimes e_i^*+\gamma\otimes d(e_i^*)\hfill\\
 &=&
 d(\gamma)\otimes e_i^*+\gamma\otimes e_i^*\otimes e_i^*+\gamma\otimes h(e_i^*), \hfill\\
 \end{matrix}$$
 where $h(e_i^*), h(e_j^*)\in {\cal N}$. It follows that 
 $\pi d(\delta(\gamma))=-\pi d(h(\gamma))$. 
 
 We can write $\delta(\gamma)=\sum_s\gamma_s^2\otimes x_s\otimes \gamma^1_s$, where $\gamma_s^1,\gamma_s^2$
are directed elements in $T_S(\hat{D}(B)_0)$ and $x_s$ are directed elements in $\hat{D}(J)$,  then, 
$$d(\delta(\gamma))=
\sum_sd(\gamma_s^2)\otimes x_s\otimes \gamma_s^1+\gamma_s^2\otimes d(x_s)\otimes \gamma_s^1-\gamma_s^2\otimes x_s\otimes d(\gamma_s^1).$$
If we substitute in this formula the expressions for $d(\gamma_s^2)$, $d(x_s)$ and $d(\gamma_s^1)$, described in (\ref{C: calculo de dif d en grados 0 y 1 usando la  diferencial delta}), and then apply the morphism $\pi$, we obtain  
$$\pi d(\delta(\gamma))=\sum_s\delta(\gamma_s^2)\otimes x_s\otimes \gamma_s^1+\gamma_s^2\otimes\delta(x_s)\otimes \gamma_s^1-\gamma_s^2\otimes x_s\otimes \delta(\gamma_s^1)=\delta^2(\gamma).$$
Then, we get $\delta^2(\gamma)=\pi d(\delta(\gamma))=-\pi d(h(\gamma))\in V^2I+VIV+IV^2$, 
because, $h(\gamma)\in {\cal N}$ is homogeneous with degree 
$\vert h(\gamma)\vert=1$ and we can apply Step 1.

\medskip
\noindent\emph{Step 4: We have $\delta^2(V)\subseteq IV^3+VIV^2+V^2IV+V^3I$.}
\medskip

Having in mind Step 3, we see that it will be enough to show that 
$\delta^2(\xi)\in IV^3+VIV^2+V^2IV+V^3I$, for any directed element $\xi\in e_j*\hat{D}(J)*e_i$.  From (\ref{C: calculo de dif d en grados 0 y 1 usando la  diferencial delta}), we have  
  $$d(\xi)=e_j^*\otimes \xi+\xi\otimes e_i^*+\delta(\xi)+h(\xi), \hbox{ where }  h(\xi)\in {\cal N}.$$
  As before, this implies that $0=\pi d^2(\xi)=\pi d(\delta(\xi))+\pi d(h(\xi))$. We can write 
  $\delta(\xi)=\sum_t\gamma_t^3\otimes z^2_t\otimes \gamma_t^2\otimes z^1_t\otimes\gamma_t^1$, 
  where $\gamma_t^1,\gamma_t^2,\gamma_t^3$ are directed elements in $T_S(\hat{D}(B)_0)$ and $z_t^1,z_t^2$ are directed elements in $\hat{D}(J)$. Then, (\ref{C: calculo de dif d en grados 0 y 1 usando la  diferencial delta}) implies that 
  $$\begin{matrix}
  \pi d \delta(\xi)&=&\sum_t \delta(\gamma_t^3)\otimes z_t^2\otimes \gamma_t^2\otimes z_t^1\otimes\gamma_t^1
  +
  \gamma_t^3\otimes \delta(z_t^2)\otimes \gamma_t^2\otimes z_t^1\otimes\gamma_t^1
  \hfill\\   
  &&\,
  \sum_t -\gamma_t^3\otimes z_t^2\otimes \delta(\gamma_t^2)\otimes z_t^1\otimes\gamma_t^1
  -
  \gamma_t^3\otimes z_t^2\otimes \gamma_t^2\otimes \delta(z_t^1)\otimes\gamma_t^1
  \hfill\\   
  &&\,
  \sum_t \gamma_t^3\otimes z_t^2\otimes \gamma_t^2\otimes z_t^1\otimes\delta(\gamma_t^1)=\delta^2(\xi).\hfill\\
    \end{matrix}$$
Then, we get $\delta^2(\xi)=\pi d(\delta(\xi))=-\pi d(h(\xi))\in IV^3+VIV^2+V^2IV+V^3I$, 
because, $h(\xi)\in {\cal N}$ is homogeneous with degree 
$\vert h(\xi)\vert=2$ and we can apply Step 1.  
\end{proof}

The following statement  gives an explicit description of the differential $\delta$ of the tensor algebra $T_S(\hat{D}(B)_0\oplus \hat{D}(J))$ in terms of the operations $b_1,b_2,\ldots $ of $B$. 

\begin{lemma}\label{L: coeficientes de delta(x*)}
 Given any directed basic element $x\in \hueca{B}_0\cup \hueca{J}$, the value of the differential $\delta(x^*)$ is a finite sum of the form 
 $$\delta(x^*)=\sum_{n=2}^{\ell+1}\sum_{z\in \hueca{T}_n} c_z^{x^*} z\in T_S(\hat{D}(B)_0\oplus \hat{D}(J)), \hbox{ for some   scalars } c_z^{x^*}\in k,$$
  where the elements  $z\in\hueca{T}_n$ have the form $z=x_{i_n}^*\otimes\cdots\otimes x_{i_1}^*$,  for some directed basic elements $x_{i_1},\ldots,x_{i_n}\in \hueca{B}_0\cup \hueca{J}$, and $\ell=\vert \overline{\cal P}\vert$. Moreover, we  have that
  $$e_{t(x_{i_n})}c^{x^*}_{x_{i_n}^*\otimes\cdots\otimes x_{i_1}^*}=x^*(b_n(x_{i_n}\otimes\cdots\otimes x_{i_1}))$$
  for each $x_{i_n}^*\otimes\cdots\otimes x_{i_1}^*\in \hueca{T}_n$. 
\end{lemma}

\begin{proof} We have $\delta(x^*)=\sum_{n\geq 1}\delta_n(x^*)$, where $\delta_n(x^*)$ denotes the component of $\delta(x^*)$ in $ (\hat{D}(B)_0\oplus\hat{D}(J))^{\otimes n}$, which is zero for all $n> \vert x^*\vert+\ell$, by (\ref{L: sobre elementos homogeneos de TS(DB)}). From (\ref{L: coefs de hat(d)(bn)}), we also know that $d(x^*)=\sum_{n\geq 1}\hat{d}(b_n)(x^*)$ with $\hat{d}(b_1)=0$, and  
$$\hat{d}(b_n)(x^*)=\sum_{z\in \widehat{\hueca{T}}_n} c_z^{x^*} z\in \hat{D}(B)^{\otimes n}, \hbox{ for some   scalars } c_z^{x^*}\in k,$$
  where the elements  $z\in\widehat{\hueca{T}}_n$ have the form $z=x_{i_n}^*\otimes\cdots\otimes x_{i_1}^*$,  for some directed basic elements $x_{i_1},\ldots,x_{i_n}\in \hueca{B}$. Moreover, we  have that
  $$e_{t(x_{i_n})}c^{x^*}_{x_{i_n}^*\otimes\cdots\otimes x_{i_1}^*}=x^*(b_n(x_{i_n}\otimes\cdots\otimes x_{i_1}))$$
  for each $x_{i_n}^*\otimes\cdots\otimes x_{i_1}^*\in \widehat{\hueca{T}}_n$. 
  
  By (\ref{C: calculo de dif d en grados 0 y 1 usando la  diferencial delta}),  the projection  $\pi:\hat{T}_S(\hat{D}(B))\rightmap{}\widehat{T}_S(\hat{D}(B)_0\oplus \hat{D}(J))$ maps $\hat{d}(b_n)(x^*)$ on $\delta_n(x^*)$ and it maps to zero every element in $\widehat{\hueca{T}}_n\setminus\hueca{T}_n$.  
\end{proof}

 \begin{definition}\label{D: def del bocs A(Delta)}
 Consider the graded tensor  algebras $A(\Delta):=T_S(\hat{D}(B)_0)$ and  $T(\Delta):=T_S(\hat{D}(B)_0\oplus \hat{D}(J))$. 
 There is a canonical isomorphism of graded $S$-algebras  
 $$T(\Delta)\cong T_{A(\Delta)}(V(\Delta)),$$
 where $V(\Delta):=A(\Delta)\otimes_S\hat{D}(J)\otimes_SA(\Delta)$.  
 Consider the homogeneous ideal $I(\Delta)$ of $T(\Delta)$ generated by its subsets $I$ and $I_V:=IV+VI$. From (\ref{P: I entrelazado})(1) and \cite{bpsqh}\S3 we know that there is a canonical isomorphism 
 $$T(\Delta)/I(\Delta)\cong T_{A(\Delta)/I}(V/I_V).$$ 
 Moreover, from \cite{bpsqh}(3.7), we know that the differential 
 $\delta$ of $T(\Delta)$, see (\ref{C: calculo de dif d en grados 0 y 1 usando la  diferencial delta}), induces a differential $\overline{\delta}$ on the quotient tensor algebra $T(\Delta)/I(\Delta)$ with $\overline{\delta}^2=0$. Hence, we have a ditalgebra $(T(\Delta)/I(\Delta),\overline{\delta})$ and we can consider its category of modules as in  \cite{bpsqh}\S3.  We use the notation ${\cal A}(\Delta):=(T(\Delta),\delta)$ and 
 recall that there is an equivalence of categories 
 $$\begin{matrix}(T(\Delta)/I(\Delta),\overline{\delta})\g\Mod&\rightmap{\simeq}&({\cal A}(\Delta),I)\g\Mod\end{matrix},$$
 see (\ref{D: la cat de mod's de cal A(Delta)}). 
  We refer to the properties relating the differential $\delta$ with the ideal $I$ described in (\ref{P: I entrelazado}), by saying that the ditalgebra ${\cal A}(\Delta)$   is \emph{interlaced with the ideal $I$}.  
 \end{definition}

 In the following sections, we review the construction of an equivalence from the category ${\cal F}(\Delta)$ to the preceding module category. 
 
 \section{$A_\infty$-categories and $b$-categories}
 
 We need to recall the notion of an $A_\infty$-category
and a $b$-category, which are (in some sense equivalent) generalizations of the notion of category. 

 \begin{definition}\label{D: A-infinito category}
  An \emph{$A_\infty$-category} ${\cal A}$ consists of the following: a class of objects $\Ob({\cal A})$;
  for each pair of objects $X,Y\in \Ob({\cal A})$, there is a graded $k$-vector space 
   $${\cal A}(X,Y)=\bigoplus_{s\in \hueca{Z}}{\cal A}(X,Y)_s;$$ 
  and for each finite sequence $X_0,X_1,\ldots,X_n\in \Ob({\cal A})$, with $n\geq 1$, there is a homogeneous morphism $m_n$  of graded $k$-vector spaces with degree $\vert m_n\vert=2-n$: 
   $$\begin{matrix}{\cal A}(X_{n-1},X_n)\otimes_k\cdots \otimes_k{\cal A}(X_1,X_2)\otimes_k{\cal A}(X_0,X_1)&\rightmap{m_n}&{\cal A}(X_0,X_n)\end{matrix}.$$
  It is required that the maps $m_n$, called the \emph{higher composition maps}, satisfy the following relation for each $n\in \hueca{N}$ and each finite sequence of objects $X_0,X_1,\ldots,X_n$ of ${\cal A}$:
   $$S^\infty_n:\sum_{\scriptsize\begin{matrix}r+s+t=n\\ s\geq 1;r,t\geq 0\end{matrix}} 
   (-1)^{r+st}m_{r+1+t}(id^{\otimes r}\otimes m_s\otimes id^{\otimes t})=0.$$
   These are called the \emph{Stasheff identities}.  
 \end{definition}

 \begin{definition}\label{D: b-category}
 A \emph{$b$-category ${\cal B}$} consists  of the following: 
 a class of objects $\Ob({\cal B})$; for each pair of objects $X,Y\in \Ob({\cal B})$ a graded $k$-vector space 
 $${\cal B}(X,Y)=\bigoplus_{s\in \hueca{Z}}{\cal B}(X,Y)_s;$$ 
  and for each finite sequence $X_0,X_1,\ldots,X_n\in \Ob({\cal B})$, with $n\geq 1$, there is a homogeneous morphism $b_n$  of graded $k$-vector spaces with degree $\vert b_n\vert=1$: 
   $$\begin{matrix}{\cal B}(X_{n-1},X_n)\otimes_k\cdots \otimes_k{\cal B}(X_1,X_2)\otimes_k{\cal B}(X_0,X_1)&\rightmap{b_n}&{\cal B}(X_0,X_n)\end{matrix}.$$
  It is required that the maps $b_n$ satisfy the following relation for each $n\in \hueca{N}$ and each finite sequence of objects $X_0,X_1,\ldots,X_n$ of ${\cal B}$:
   $$S_n^b:\sum_{\scriptsize\begin{matrix}r+s+t=n\\ s\geq 1;r,t\geq 0\end{matrix}} 
   b_{r+1+t}(id^{\otimes r}\otimes b_s\otimes id^{\otimes t})=0.$$
 \end{definition}

These notions are related by the \emph{bar construction} detailed  in the following statement.

\begin{proposition}\label{P: A-infinite cats vs b-cats}
Given an $A_\infty$-category ${\cal A}$, a $b$-category ${\cal B}$ associated to ${\cal A}$ is constructed as follows. 
By definition, $\Ob({\cal B})=\Ob({\cal A})$. For $X,Y\in \Ob({\cal B})$, we consider ${\cal B}(X,Y):={\cal A}(X,Y)[1]$, the shifted graded $k$-vector space of ${\cal A}(X,Y)$, and the canonical isomorphism 
$\sigma_{X,Y}:{\cal A}(X,Y)\rightmap{}{\cal B}(X,Y)$ determined by the identity map. Given a finite sequence of objects $X_0,X_1\ldots,X_n$ of ${\cal B}$, the morphism $b_n$ is defined by the commutativity of the following diagram
$$\begin{matrix}{\cal A}(X_{n-1},X_n)\otimes_k\cdots \otimes_k{\cal A}(X_1,X_2)\otimes_k{\cal A}(X_0,X_1)&\rightmap{m_n}&{\cal A}(X_0,X_n)\\
\lmapdown{\sigma_{X_{n-1},X_n}\otimes\cdots\otimes\sigma_{X_0,X_1}}&&\rmapdown{\sigma_{X_0,X_n}}\\   
 {\cal B}(X_{n-1},X_n)\otimes_k\cdots \otimes_k{\cal B}(X_1,X_2)\otimes_k{\cal B}(X_0,X_1)&\rightmap{b_n}&{\cal B}(X_0,X_n).  
  \end{matrix}$$
The preceding construction is reversible, given a $b$-category ${\cal B}$, we can construct an $A_\infty$-category with the same objects and hom graded vector spaces with higher composition maps determined by the preceding commutative square.
\end{proposition}

\begin{proof} It is enough to show that for any $n\in \hueca{N}$ and for any sequence of objects $X_0,X_1,\ldots,X_n$ we have 
$$\sigma_{X_0,X_n}S_n^{\infty}=S_n^b(\sigma_{X_{n-1},X_n}\otimes\cdots\otimes\sigma_{X_0,X_1}).$$
This holds because for each summand, whenever $n=r+s+t$ with $s\geq 1$ and $s,t\geq 0$, we have 
$$(-1)^{r+st}\sigma m_{r+1+t}(id^{\otimes r}\otimes m_s\otimes id^{\otimes t})=b_{r+1+t}(id^{\otimes r}\otimes b_s\otimes id^{\otimes t})\sigma^{\otimes n},$$
where we avoided subindices for notational simplicity. 
\end{proof}
 
 The advantage of working with $b$-categories instead of $A_\infty$-categories is that we do not have to deal with the signs appearing in Stasheff identities. 
 
 \begin{proposition}\label{P: ad(A)}
  Consider the Yoneda algebra $A$ associated to the $\Lambda$-module $\Delta$, 
  with a strict structure of $A_\infty$-algebra, as specified in (\ref{R: merkulov}). 
  Then an $A_\infty$-category, denoted by $\ad(A)$, is defined by the following. The objects of $\ad(A)$ are the right $S$-modules; the spaces of morphisms are given by 
  $$\ad(A)(X,Y):=\bigoplus_{i,j\in {\cal P}}\Hom_k(Xe_i,Ye_j)\otimes_ke_jAe_i,$$
  with the canonical grading of the tensor product where $\Hom_k(Xe_i,Xe_j)$ is considered as a graded vector space concentrated at degree 0; finally the higher multiplications $m_n^{ad}$ are defined, for $n\in \hueca{N}$ and a sequence of right $S$-modules $X_0,X_1,\ldots,X_n$, on typical generators by 
  $$\begin{matrix}\ad(A)(X_{n-1},X_n)\otimes_k\cdots \otimes_k\ad(A)(X_1,X_2)\otimes_k\ad(A)(X_0,X_1)\hbox{\,}\rightmap{m^{ad}_n}\hbox{\,}\ad(A)(X_0,X_n)\end{matrix}$$ 
  $$\begin{matrix}(f_n\otimes a_n)\otimes\cdots\otimes(f_2\otimes a_2)\otimes (f_1\otimes a_1)&\longmapsto&
  f_n\cdots f_2f_1\otimes m_n(a_n\otimes\cdots\otimes a_1).\end{matrix}$$ 
   In the preceding recipe, since for given right $S$-modules $X$ and $Y$, we have  $$\Hom_k(X,Y)=\bigoplus_{i,j\in {\cal P}}e_i\Hom_k(X,Y)e_j=\bigoplus_{i,j\in{\cal P}}\Hom_k(Xe_i,Ye_j),$$
  we identify the elements of $\Hom_k(Xe_i,Ye_j)$ with the corresponding elements in $\Hom_k(X,Y)$, so the composition $f_n\cdots f_2f_1$ makes sense.  
 \end{proposition}

 The proof  of this proposition follows from the next result, where we consider the  bar category of $\ad(A)$.

 \begin{proposition}\label{P: ad(B)}
  Consider the Yoneda algebra $A$ associated to the $\Lambda$-module $\Delta$, 
  with the structure of $A_\infty$-algebra specified in (\ref{R: merkulov}). Then, consider its bar construction $B$, as in (\ref{D: construccion barra}).  
  Then a $b$-category, denoted by $\ad(B)$, is defined by the following. The objects of $\ad(B)$ are the right $S$-modules; the spaces of morphisms are given by 
  $$\ad(B)(X,Y):=\bigoplus_{i,j\in {\cal P}}\Hom_k(Xe_i,Ye_j)\otimes_ke_jBe_i,$$
  with the canonical grading of the tensor product where $\Hom_k(Xe_i,Xe_j)$ is considered as a graded vector space concentrated at degree 0; finally the morphisms $b^{ad}_n$ are defined, for $n\in \hueca{N}$ and a sequence of right $S$-modules $X_0,X_1,\ldots,X_n$, on typical generators by 
  $$\begin{matrix}\ad(B)(X_{n-1},X_n)\otimes_k\cdots \otimes_k\ad(B)(X_1,X_2)\otimes_k\ad(B)(X_0,X_1)\hbox{\,}\rightmap{b^{ad}_n}\hbox{\,}\ad(B)(X_0,X_n)\end{matrix}$$ 
  $$\begin{matrix}(f_n\otimes a_n)\otimes\cdots\otimes(f_2\otimes a_2)\otimes (f_1\otimes a_1)&\longmapsto&
  f_n\cdots f_2f_1\otimes b_n(a_n\otimes\cdots\otimes a_1).\end{matrix}$$
  Moreover, the $b$-category $\ad(B)$ is the bar construction of $\ad(A)$. 
 \end{proposition}
  
 \begin{proof} For $n\in \hueca{N}$ and a sequence of right $S$-modules $X_0,X_1,\ldots,X_n$, consider a typical homogeneous generator  
  $(f_n\otimes a_n)\otimes\cdots\otimes(f_2\otimes a_2)\otimes (f_1\otimes a_1)$ in the space $\ad(B)(X_{n-1},X_n)\otimes_k\cdots \otimes_k\ad(B)(X_1,X_2)\otimes_k\ad(B)(X_0,X_1)$. We want to show that 
  $$ 0=\sum_{\scriptsize\begin{matrix}r+s+t=n\\ s\geq 1;r,t\geq 0\end{matrix}} 
   b^{ad}_{r+1+t}(id^{\otimes r}\otimes b^{ad}_s\otimes id^{\otimes t})[(f_n\otimes a_n)\otimes\cdots\otimes (f_1\otimes a_1)].$$
   But this equals 
 $$(f_n\cdots f_2f_1)\otimes \sum_{\scriptsize\begin{matrix}r+s+t=n\\ s\geq 1;r,t\geq 0\end{matrix}} 
   b_{r+1+t}(id^{\otimes r}\otimes b_s\otimes id^{\otimes t})[a_n\otimes\cdots \otimes a_1]=0.$$
Hence, we get that $\ad(B)$ is a $b$-category. 

If we denote by ${\cal A}$  the $A_\infty$-category $\ad(A)$ and by ${\cal B}$ its associated $b$-category, we have, for any right $S$-modules $X$ and $Y$, that 
$$\ad(B)(X,Y)=\bigoplus_{i,j\in {\cal P}}\Hom_k(Xe_i,Ye_j)\otimes e_jBe_i
=
{\cal A}(X,Y)[1]= {\cal B}(X,Y).$$ 
For each $n\geq 1$, we have  $b^{\cal B}_n=\sigma m_n^{ad}(\sigma^{\otimes n})^{-1}$. Take a typical homogeneous generator 
$(f_n\otimes a_n)\otimes\cdots\otimes (f_1\otimes a_1)$ in $\ad(B)(X_{n-1},X_n)\otimes_k\cdots \otimes_k\ad(B)(X_0,X_1)$ and make 
$E_n=b_n^{\cal B}[(f_n\otimes a_n)\otimes\cdots\otimes (f_1\otimes a_1)]$, then 
$$\begin{matrix}
   E_n&=&\sigma m_n^{ad}(\sigma^{\otimes n})^{-1}[(f_n\otimes a_n)\otimes\cdots\otimes (f_1\otimes a_1)]\hfill\\
   &=&
   (-1)^{sgn}\sigma m_n^{ad}[(f_n\otimes \sigma^{-1}(a_n))\otimes\cdots\otimes (f_1\otimes \sigma^{-1}(a_1))]\hfill\\
   &=&
   (-1)^{sgn}\sigma[f_n\cdots f_2f_1\otimes m_n(\sigma^{-1}(a_n)\otimes\cdots\otimes\sigma^{-1}(a_1))]\hfill\\
   &=&
   f_n\cdots f_2f_1\otimes \sigma m_n(\sigma^{\otimes n})^{-1}(a_n\otimes\cdots\otimes a_1)\hfill\\
   &=&
   f_n\cdots f_2f_1\otimes b_n(a_n\otimes\cdots\otimes a_1)\hfill\\
   &=&
   b_n^{ad}[(f_n\otimes a_n)\otimes\cdots\otimes (f_1\otimes a_1)],\hfill\\
  \end{matrix}$$
where the number $sgn$ is determined by the degrees of the homogeneous elements $a_1,\ldots,a_n$. Thus, we get $b_n^{\cal B}=b_n^{ad}$, for each $n\geq 1$. 
So $\ad(B)$ is the $b$-category of the $A_\infty$-category $\ad(A)$. 
 \end{proof}

 \begin{definition}\label{D: H0(cal A)} 
  Given an $A_\infty$-category ${\cal A}$, for each pair of objects $X,Y$
  of ${\cal A}$, the homogeneous map $m_1:{\cal A}(X,Y)\rightmap{}{\cal A}(X,Y)$ has degree $1$ and satisfies $m_1^2=0$.
  Then, we have a cochain complex ${\cal A}(X,Y)$ with boundary operator $m_1$ and the space of $0$-cocycles $Z_0({\cal A})(X,Y)\subseteq {\cal A}(X,Y)_0$ and $0$-coboundaries $I_0({\cal A})(X,Y)\subseteq Z_0({\cal A})(X,Y)$. 
  Then, we can consider the cohomology space $$H^0({\cal A})(X,Y)=Z_0({\cal A})(X,Y)/I_0({\cal A})(X,Y).$$ 
  The morphism $m_2$ induces composition maps  
  $$H^0({\cal A})(Y,Z)\times H^0({\cal A})(X,Y)\rightmap{}H^0({\cal A})(X,Z),$$
  If $\zeta_1\in Z_0({\cal A})(X,Y)$ and $\zeta_2\in Z_0({\cal A})(Y,Z)$, by definition, $\overline{\zeta_2}\,\overline{\zeta_1}=\overline{m_2(\zeta_2\otimes \zeta_1)}$.  The preceding composition is well defined because $S_2^\infty=0$, and it is associative because $S^\infty_3=0$.  
  
  The $A_\infty$-category ${\cal A}$ is called \emph{homologically unitary}  if, for each object $X$ in ${\cal A}$, there is $1_X\in  H^0({\cal A})(X,X)$, such that $1_Y\overline{\zeta}=\overline{\zeta}$ and $\overline{\zeta}1_X=\overline{\zeta}$, for any $\zeta\in  
  Z_0({\cal A})(X,Y)$. For such homologically unitary  $A_\infty$-category ${\cal A}$, we can consider the category $H^0({\cal A})$  with the same objects than ${\cal A}$ and the hom spaces  $H^0({\cal A})(X,Y)$, for any pair of objects $X, Y$. 
  \end{definition}
  
  Similarly, we have the following. 
  
  \begin{definition}\label{D: H-1(cal B)}
  Given a $b$-category ${\cal B}$, for each pair of objects $X,Y$
  of ${\cal B}$, the homogeneous map $b_1:{\cal B}(X,Y)\rightmap{}{\cal B}(X,Y)$ has degree $1$ and satisfies $b_1^2=0$.
  Then, we have a cochain complex ${\cal B}(X,Y)$ with boundary operator $b_1$ and the space of $-1$-cocycles $Z_{-1}({\cal B})(X,Y)\subseteq {\cal B}(X,Y)_{-1}$ and $-1$-coboundaries $I_{-1}({\cal B})(X,Y)$. 
  Then, we can consider the cohomology space 
  $$H^{-1}({\cal B})(X,Y)=Z_{-1}({\cal B})(X,Y)/I_{-1}({\cal B})(X,Y).$$ 
  The morphism $b_2$ induces composition maps  
  $$H^{-1}({\cal B})(Y,Z)\times H^{-1}({\cal B})(X,Y)\rightmap{}H^{-1}({\cal B})(X,Z),$$
  If $\zeta_1\in Z_{-1}({\cal B})(X,Y)$ and $\zeta_2\in Z_{-1}({\cal B})(Y,Z)$, by definition, $\overline{\zeta_2}\,\overline{\zeta_1}=\overline{b_2(\zeta_2\otimes \zeta_1)}$.  
  
  The $b$-category ${\cal B}$ is called \emph{homologically unitary}  if, for each object $X$ in ${\cal B}$, there is $1_X\in  H^{-1}({\cal B})(X,X)$, such that $1_Y\overline{\zeta}=\overline{\zeta}$ and $\overline{\zeta}1_X=\overline{\zeta}$, for any $\zeta\in  
  Z_{-1}({\cal B})(X,Y)$. For such an homologically unitary  $b$-category ${\cal B}$, we can consider the category $H^{-1}({\cal B})$  with the same objects than ${\cal B}$ and the hom spaces  $H^{-1}({\cal B})(X,Y)$, for any pair of objects $X, Y$. 
  \end{definition}
  
  \begin{lemma}\label{L: H0(cal A) equiv H-1(cal B) si cal B es la barra de cal A}
  Let  ${\cal A}$ be an $A_\infty$-category and ${\cal B}$ the $b$-category obtained from ${\cal A}$ by the bar construction. Then, ${\cal A}$ is homologically unitary iff so is ${\cal B}$ and we have an isomorphism of categories  
  $$H^0({\cal A})\simeq H^{-1}({\cal B}).$$
 \end{lemma}

 \begin{proof} The maps $\sigma=\sigma_{X,Y}:{\cal A}(X,Y)\rightmap{}{\cal B}(X,Y)$ are isomorphisms of complexes of degree $-1$,
 so they  induce linear isomorphisms on the cohomologies 
 $$\overline{\sigma}:H^0({\cal A})(X,Y)\rightmap{}H^{-1}({\cal B})(X,Y).$$  
 We consider the associaton rule $E:H^0({\cal A})\rightmap{}H^{-1}({\cal B})$ which maps every object $X$ to the same object $X$ and the action on morphisms is given by these $\overline{\sigma}$'s. 
 Then,  for $\zeta_1\in Z_0({\cal A})(X,Y)$ and  $\zeta_2\in Z_0({\cal A})(Y,Z)$, 
 we have 
 $$\sigma m_2(\zeta_2\otimes \zeta_1)=b_2\sigma^{\otimes 2}(\zeta_2\otimes\zeta_1)= b_2(\sigma(\zeta_2)\otimes\sigma(\zeta_1)).$$
 Taking classes in the homology and applying $\overline{\sigma}$, we obtain 
 $$E(\overline{\zeta_1}\,\overline{\zeta_2})=
 \overline{\sigma(m_2(\zeta_2\otimes \zeta_1))}=
 \overline{b_2(\sigma(\zeta_2)\otimes\sigma(\zeta_1))}=
 E(\overline{\zeta_2})E(\overline{\zeta_1}).$$
 Thus, $E$ preserves composition. 
 
 Since $E$ is surjective on hom spaces and preserves composition, we obtain that $1_X$ is the identity of the object $X$ in $H^0({\cal A})$ iff $E(1_X)$ is the identity of the object $X$ in $H^{-1}({\cal B})$. 
  Hence $E$ is an isomorphism of categories. 
 \end{proof}

The preceding statement allows us to work with the $b$-category ${\cal B}$ associated to an $A_\infty$-category ${\cal A}$, instead of working directly with ${\cal A}$, if we are ultimately interested in $H^0({\cal A})$.

\section{The $b$-category $\tw(B)$ and ${\cal F}(\Delta)$}

From now on, we retake our preceding terminology where $A$ is the Yoneda $A_\infty$-algebra associated to $\Delta$ and $B$ is its bar construction, as in (\ref{R: merkulov}) and (\ref{D: construccion barra}). Moreover, we keep the directed vector space basis $\hueca{B}$ for $B$ chosen  in   (\ref{N: la base de B}).
 
 \begin{proposition}\label{P: tw(B)}
 There is a $b$-category $\tw(B)$ described by the following. The objects of $\tw(B)$ are  the pairs $\underline{X}=(X,\delta_X)$ where $X$ is a right $S$-module and 
 $\delta_X\in \ad(B)(X,X)_0$. Moreover we ask the following conditions on $\delta_X$: 
 \begin{enumerate}
  \item There is a finite filtration   
  $0=X_0\subseteq X_1\subseteq\cdots\subseteq X_{\ell(X)}=X$ of right $S$-modules
  such that if we express $\delta_X=\sum_{x\in \hueca{B}}f_x\otimes x$, where the maps $f_x\in \Hom_k(X,X)$ are uniquely determined, we have $f_x(X_r)\subseteq X_{r-1}$, for all $r\in [1,\ell(X)]$.
  \item We have $\sum_{s=1}^\infty b_s^{ad}((\delta_X)^{\otimes s})=0$, where we notice that the preceding condition $\it{1}$ implies that $b_s^{ad}((\delta_X)^{\otimes s})=0$ for $s\geq \ell(X)$, so we are dealing with a finite sum.  
 \end{enumerate}
 Given  $\underline{X},\underline{Y}\in\Ob(\tw(B))$, we have the hom 
 graded $k$-vector space 
 $$\tw(B)(\underline{X},\underline{Y})=\ad(B)(X,Y)=\bigoplus_{i,j\in {\cal P}}\Hom_k(Xe_i,Ye_j)\otimes_k e_jBe_i.$$
 If $n\geq 1$ and $\underline{X}_0,\underline{X}_1,\ldots,\underline{X}_n\in \Ob(\tw(B))$,  
 we have the following homogeneous linear map of degree 1  
 $$\begin{matrix}\tw(B)(\underline{X}_{n-1},\underline{X}_n)\otimes_k\cdots \otimes_k\tw(B)(\underline{X}_1,\underline{X}_2)\otimes_k\tw(B)(\underline{X}_0,\underline{X}_1)\hbox{\,}\rightmap{b^{tw}_n}\hbox{\,}\tw(B)(\underline{X}_0,\underline{X}_n)\end{matrix}$$ 
 which maps each homogeneous generator  $t_n\otimes\cdots\otimes t_2\otimes t_1$ on 
  $$\sum_{\scriptsize\begin{matrix}
  i_0,\ldots,i_n\geq 0\end{matrix}} 
   b^{ad}_{i_0+\cdots+i_n+n}(\delta_{X_n}^{\otimes i_n}\otimes t_n\otimes \delta_{X_{n-1}}^{\otimes i_{n-1}}\otimes t_{n-1}\otimes \cdots\otimes \delta_{X_1}^{\otimes i_1}\otimes t_1\otimes\delta_{X_0}^{\otimes i_0}),$$
   which is a finite sum. 
 \end{proposition}

 \begin{proof} Given any object $(X,\delta_X)$ of $\tw(B)$, we can write 
 $\delta_X=\sum_{x_i\in \hueca{B}}f_{x_i}\otimes x_i$, for any mute variable $x_i$. Then, $$\delta_X^{\otimes s}=\sum_{x_1,\ldots,x_s\in \hueca{B}}
  (f_{x_s}\otimes x_s)\otimes\cdots\otimes(f_{x_2}\otimes x_2)\otimes(f_{x_1}\otimes x_1)$$
  and
  $$b_s^{ad}(\delta_X^{\otimes s})=\sum_{x_1,\ldots,x_s\in \hueca{B}}
  f_{x_s}\cdots f_{x_2}f_{x_1}\otimes b_s(x_s\otimes\cdots\otimes x_2\otimes x_1).$$
  By condition (1), we see that any composition of more than $\ell(X)$ linear operators $f_{x_i}$ of $X$ is zero, thus for $s\geq \ell(X)$, we get 
  $f_{x_s}\cdots f_{x_2}f_{x_1}=0$, hence $b_s^{ad}(\delta_X^{\otimes s})=0$. Hence, the sum in (2) is finite. For the same reason, the sum defining $b_n^{tw}(t_n\otimes\cdots\otimes t_1)$ is finite. 
  
  In order to show that 
   $$\sum_{\scriptsize\begin{matrix}r+s+t=n\\ s\geq 1;r,t\geq 0\end{matrix}} 
   b^{tw}_{r+1+t}(id^{\otimes r}\otimes b^{tw}_s\otimes id^{\otimes t})(t_n\otimes\cdots\otimes t_1)=0,$$
   we compute each summand $S_{r,s,t}^{tw}:=b^{tw}_{r+1+t}(id^{\otimes r}\otimes b^{tw}_s\otimes id^{\otimes t})(t_n\otimes\cdots\otimes t_1),$ where $n=r+s+t$ with $s\geq 1$ and $t,r\geq 0$. We have 
   $$S_{r,s,t}^{tw}=(-1)^{\vert \tau_3^r\vert}b_{r+1+t}^{tw}(\tau_3^r\otimes b_s^{tw}(\tau_2^s)\otimes \tau_1^t),$$
  where 
  $$\begin{matrix}
\tau_1^t&=&t_t\otimes\cdots\otimes t_1,\hfill\\  
\tau_2^s&=&t_{n-r}\otimes\cdots\otimes t_{t+1}, \hfill\\
\tau^r_3&=&t_n\otimes\cdots\otimes t_{n-r+1}.\hfill\\ 
 \end{matrix}$$
  Consider sequences of non-negative integers 
  $$\begin{matrix}
     {\bf{j}}_t^1&=&(j_t,\ldots,j_0)&\hbox{ and } &
     {\bf{j}}_r^3&=&(j_n,\ldots,j_{t+s}).\hfill\\ \end{matrix}$$
  and make 
  $$\begin{matrix}
     \widehat{\tau}_1^{{\bf{j}}_t^1}&=&\delta_{X_t}^{\otimes j_t}\otimes t_t\otimes\cdots\otimes t_2\otimes \delta_{X_1}^{\otimes j_1}\otimes t_1\otimes \delta_{X_0}^{\otimes j_0},\hfill\\     
       \widehat{\tau}_3^{{\bf{j}}_r^3}&=&\delta_{X_n}^{\otimes j_n}\otimes t_n\otimes\cdots\otimes t_{t+s+2}\otimes 
       \delta_{X_{t+s+1}}^{\otimes j_{t+s+1}}\otimes t_{t+s+1}\otimes \delta_{X_{t+s}}^{\otimes j_{t+s}}.\hfill\\
    \end{matrix}$$
    Then, we obtain 
    $$S_{r,s,t}^{tw}=\sum_{ {\bf{j}}_t^1,\hbox{\,}{\bf{j}}_r^3}(-1)^{\vert \tau_3^r\vert}
     b^{ad}_{\vert {\bf{j}}_t^1\vert +\vert{\bf{j}}_r^3 \vert +r+t+1}(\widehat{\tau}_3^{{\bf{j}}_r^3}\otimes b_s^{tw}(\tau_2^s) \otimes\widehat{\tau}_1^{{\bf{j}}_t^1}),$$
    where $\vert {\bf{j}}_t^1\vert$ and $\vert{\bf{j}}_r^3 \vert$ denote the sum of their components respectively. Now, make    
    $$\begin{matrix}{\bf{j}}_s^2&=&(j'_{t+s},j_{t+s-1},\ldots,j_{t+1},j'_t), \hfill\\
    \widehat{\tau}_2^{{\bf{j}}_s^2}&=&\delta_{X_{t+s}}^{\otimes j'_{t+s}}\otimes t_{t+s}\otimes\delta_{X_{t+s-1}}^{\otimes j_{t+s-1}}
    \otimes\cdots\otimes 
    t_{t+2}\otimes \delta_{X_{t+1}}^{\otimes j_{t+1}}
    \otimes t_{t+1}
    \otimes \delta_{X_t}^{\otimes j'_t}.\hfill\\\end{matrix}$$
    Then, we have
     $$S_{r,s,t}^{tw}=
     \sum_{ {\bf{j}}_t^1,\hbox{\,}{\bf{j}}_r^3}(-1)^{\vert \tau_3^r\vert}
     \sum_{ {\bf{j}}_s^2}
     b^{ad}_{\vert {\bf{j}}_t^1\vert +\vert{\bf{j}}_r^3 \vert +r+t+1}(\widehat{\tau}_3^{{\bf{j}}_r^3}\otimes 
     b^{ad}_{\vert {\bf{j}}_s^2\vert +s}(\widehat{\tau}_2^{{\bf{j}}_s^2}) \otimes\widehat{\tau}_1^{{\bf{j}}_t^1}),$$
    where $\vert {\bf{j}}_s^2\vert$ denotes the sum of its components. 
    Then, we have 
      $$S_{r,s,t}^{tw}=
     \sum_{ {\bf{j}}_t^1,\hbox{\,}{\bf{j}}_s^2,\hbox{\,}{\bf{j}}_r^3}
     b^{ad}_{\vert {\bf{j}}_t^1\vert +\vert{\bf{j}}_r^3 \vert +r+t+1}
     (id^{\otimes \vert {\bf{j}}_r^3\vert+r}
     \otimes  b^{ad}_{\vert {\bf{j}}_s^2\vert +s}\otimes 
     id^{\otimes \vert{\bf{j}}_t^1\vert+t})
     (\widehat{\tau}_3^{{\bf{j}}_r^3}\otimes 
    \widehat{\tau}_2^{{\bf{j}}_s^2}
    \otimes\widehat{\tau}_1^{{\bf{j}}_t^1}).$$
  Now, consider the following sequence of non-negative integers
  $${\bf j}_r^3*{\bf j}_s^2*{\bf j}_t^1=(j_n,\ldots,j_{t+s+1},j_{t+s}+j'_{t+s},j_{t+s-1},\ldots,j_{t+1},j'_t+j_t,j_{t-1},\ldots,j_0).$$
  Then, we have 
    $$S_{r,s,t}^{tw}=
     \sum_{ {\bf{j}}_t^1,{\bf{j}}_s^2,{\bf{j}}_r^3}
     b^{ad}_{\vert {\bf{j}}_t^1\vert +\vert{\bf{j}}_r^3 \vert +r+t+1}
     (id^{\otimes \vert {\bf{j}}_r^3\vert+r}
     \otimes  b^{ad}_{\vert {\bf{j}}_s^2\vert +s}\otimes 
     id^{\otimes \vert{\bf{j}}_t^1\vert+t})
     (\widehat{\tau}^{{\bf j}_r^3*{\bf j}_s^2*{\bf j}_t^1}),$$
     where $\widehat{\tau}^{{\bf j}_r^3*{\bf j}_s^2*{\bf j}_t^1}$
     is defined using the sequence on non-negative integers ${\bf j}_r^3*{\bf j}_s^2*{\bf j}_t^1$ in a similar way as 
     $\widehat{\tau}_1^{\bf{j}_t^1}$ is defined using the sequence  ${\bf j}_t^1$. 
     
     For any sequence ${\bf i}_n=(i_n,\ldots,i_0)$ of non-negative integers and any integers $r',s',t'$ such that $n+\vert {\bf i}_n\vert=r'+s'+t'$ with $r',t'\geq 0$ and $s'\geq 1$, make  
     $$S_{r',s',t'}^{{\bf i}_n}=b^{ad}_{r'+1+t'}
     (id^{\otimes r'}\otimes b_{s'}^{ad}\otimes id^{\otimes t'})
     (\widehat{\tau}^{{\bf i}_n}),$$
     where $\widehat{\tau}^{{\bf i}_n}=\delta_{X_n}^{\otimes i_n}\otimes t_n\otimes\cdots\otimes \delta_{X_1}^{\otimes i_1}\otimes t_1\otimes \delta_{X_0}^{\otimes i_0}$. Hence, we have 
       $$S_{r',s',t'}^{{\bf i}_n}=
     b^{ad}_{\vert {\bf{j}}_t^1\vert +\vert{\bf{j}}_r^3 \vert +r+t+1}
     (id^{\otimes \vert {\bf{j}}_r^3\vert+r}
     \otimes  b^{ad}_{\vert {\bf{j}}_s^2\vert +s}\otimes 
     id^{\otimes \vert{\bf{j}}_t^1\vert+t})
     (\widehat{\tau}^{{\bf j}_r^3*{\bf j}_s^2*{\bf j}_t^1}),$$
     for unique integers $r,s,t$, such that $r,t\geq 0$, $s\geq 1$, and $r+s+t=n$, and unique sequences 
     ${\bf{j}}_r^3$, ${\bf{j}}_s^2$, ${\bf{j}}_t^1$ with $r'=\vert {\bf{j}}_r^3\vert +r$, $s'=\vert {\bf{j}}_s^2\vert +s$ and 
     $t'=\vert {\bf{j}}_t^1\vert +t$, thus $\widehat{\tau}^{{\bf i}_n}=\widehat{\tau}^{{\bf j}_r^3*{\bf j}_s^2*{\bf j}_t^1}$. 
     In fact,         we have the following identity 
   $$\sum_{\scriptsize\begin{matrix}r+s+t=n\\ s\geq 1;r,t\geq 0\end{matrix}} 
   S^{tw}_{r,s,t}=\sum_{{\bf i}_n}
    \sum_{\scriptsize\begin{matrix}r'+s'+t'=n+\vert {\bf i}_n\vert\\ s'\geq 1;r',t'\geq 0\end{matrix}} 
  S_{r',s',t'}^{{\bf i}_n}.  
  $$
  Now, since $\ad(B)$ is a $b$-category, for fixed $n$ and $i_n$, we have 
   $$\sum_{\scriptsize\begin{matrix}r'+s'+t'=n+\vert {\bf i}_n\vert\\ s'\geq 1;r',t'\geq 0\end{matrix}} 
  S_{r',s',t'}^{{\bf i}_n}=\sum_{\scriptsize\begin{matrix}r'+s'+t'=n+\vert {\bf i}_n\vert\\ s'\geq 1;r',t'\geq 0\end{matrix}} b^{ad}_{r'+1+t'}
     (id^{\otimes r'}\otimes b_{s'}^{ad}\otimes id^{\otimes t'})
     (\widehat{\tau}^{{\bf i}_n})=0.$$
     Hence, we finally obtain that 
     $\sum_{\scriptsize\begin{matrix}r+s+t=n\\ s\geq 1;r,t\geq 0\end{matrix}} S^{tw}_{r,s,t}=0$. 
 \end{proof}
 
 \begin{remark} The $A_\infty$-category $\tw(A)$ can be defined naturally in such a way that $\tw(A)$ is the $A_\infty$-category of $\tw(B)$. Thus, 
 $H^0(\tw(A))\simeq H^{-1}(\tw(B))$. 
 \end{remark}

   The following theorem of Keller and Lef\`evre-Hasegawa, see {\cite{LH}}\S7, plays an essential role in our later argumentation,  which follows closely  that of \cite{KKO}. 
  
  \begin{theorem}\label{T: Keller-Lefevre} If $A$ is the Yoneda $A_\infty$-algebra associated to the $\Lambda$-module $\Delta$, then there is an equivalence of categories 
  $${\cal F}(\Delta)\rightmap{\simeq} H^0(\tw(A)).$$
  \end{theorem}

 \section{The $b$-category $\cv(B)$}
 
 In this section $A$ is the Yoneda $A_\infty$-algebra associated to 
 $\Delta$, the algebra  $B$ denotes its bar construction, and $\hat{D}(B)$  its graded dual $S$-$S$-bimodule. The following construction, which was introduced in \cite{KKO}, plays an important role in the passage from $H^{-1}(\tw(B))$ to the module category of a special bocs, see \S\ref{de H(tw(B)) en A(Delta)-mod}.  
 
 \begin{proposition}\label{P: conv(B)}
 There is a $b$-category $\cv(B)$ described by the following. The objects of $\cv(B)$ are  the left $S$-modules, given two left $S$-modules $X$ and $Y$, we have the graded hom space 
 $$\begin{matrix}\cv(B)(X,Y)
 &=&
\bigoplus_{n\in \hueca{Z}}\Hom^n_{\GM\g S\g S}(\hat{D}(B),\Hom_k(X,Y))\hfill\\
&=&
\bigoplus_{n\in \hueca{Z}}\Hom_{\Mod\g S\g S}(\hat{D}(B)_{-n},\Hom_k(X,Y))\hfill\\
&=&
\bigoplus_{n\geq -1}\Hom_{\Mod\g S\g S}(\hat{D}(B_n),\Hom_k(X,Y)).\hfill\\
\end{matrix}$$

 If $n\geq 1$ and $X_0,X_1,\ldots,X_n\in \Ob(\cv(B))$,  
 we have the following homogeneous linear map of degree 1  
 $$\begin{matrix}\cv(B)(X_{n-1},X_n)\otimes_k\cdots \otimes_k\cv(B)(X_1,X_2)\otimes_k\cv(B)(X_0,X_1)\hbox{\,}\rightmap{b^{cv}_n}\hbox{\,}\cv(B)(X_0,X_n)\end{matrix}$$ 
 which maps each homogeneous generator  $F_n\otimes\cdots\otimes F_2\otimes F_1$ on 
  $$
   b_n^{cv}(F_n\otimes\cdots\otimes F_2\otimes F_1)=(-1)^{\lambda_n(\vert F_n\vert,\ldots,\vert F_1\vert)}\nu_n (F_n\otimes\cdots\otimes F_1)\hat{d}(b_n),$$
  here we have each $F_i:\hat{D}(B)\rightmap{}\Hom_k(X_{i-1},X_i)$, so we have 
  $$\hat{D}(B)^{\otimes n}\rightmap{ \ F_n\otimes\cdots\otimes F_1 \ }\Hom_k(X_{n-1},X_n)\otimes_S\cdots\otimes_S\Hom_k(X_0,X_1)\rightmap{\nu_n}\Hom_k(X_0,X_n),$$
  where $\nu_n$ denotes the map induced by composition, $\hat{d}(b_n)$ is defined in (\ref{P: diferencial en T[[B]]}), and $\lambda_n$ is defined in (\ref{D: de lambda n}). 
 \end{proposition}

 \begin{proof}  In order to show that 
   $$\sum_{\scriptsize\begin{matrix}r+s+t=n\\ s\geq 1;r,t\geq 0\end{matrix}} 
   b^{cv}_{r+1+t}(id^{\otimes r}\otimes b^{cv}_s\otimes id^{\otimes t})(F_n\otimes\cdots\otimes F_1)=0,$$
   we compute each summand $S_{r,s,t}^{cv}:=b^{cv}_{r+1+t}(id^{\otimes r}\otimes b^{cv}_s\otimes id^{\otimes t})(F_n\otimes\cdots\otimes F_1),$ where $n=r+s+t$ with $s\geq 1$ and $t,r\geq 0$. We proceed as we did in the proof that $\tw(B)$ is a $b$-category. Make 
  $$\begin{matrix}
\phi_1^t&=&F_t\otimes\cdots\otimes F_1,\hfill\\  
\phi_2^s&=&F_{n-r}\otimes\cdots\otimes F_{t+1}, \hfill\\
\phi^r_3&=&F_n\otimes\cdots\otimes F_{n-r+1}.\hfill\\ 
 \end{matrix}$$
  Then, we have  
 $S^{cv}_{r,s,t}=
 (-1)^{\vert \phi_3^r\vert}b_{r+1+t}^{cv}(\phi_3^r\otimes b_s^{cv}(\phi_2^s)\otimes \phi_1^t)$.
From the equality $\vert b_s^{cv}(\phi_2^s)\vert=\vert F_{n-r}\vert+\cdots+\vert F_{t+1}\vert+1$, we obtain 
$$S^{cv}_{r,s,t}=
 (-1)^{\vert \phi_3^r\vert+u_1}\nu_{r+1+t}(\phi_3^r\otimes b^{cv}_s(\phi_2^s)\otimes \phi_1^t)\hat{d}(b_{r+1+t}),$$
 where $u_1=\lambda_{r+1+t}(\vert F_n\vert,\ldots,\vert F_{n-r+1}\vert,\vert F_{n-r}\vert+\cdots+\vert F_{t+1}\vert-1, \vert F_t\vert,\ldots,\vert F_1\vert)$. Now, make $u_2=\lambda_s(\vert F_{n-r}\vert,\ldots,\vert F_{t+1}\vert)$, then we have 
 $$S^{cv}_{r,s,t}=
 (-1)^{\vert \phi_3^r\vert+u_1+u_2}\nu_{r+1+t}(\phi_3^r\otimes \nu_s(\phi_2^s)\hat{d}(b_s)\otimes \phi_1^t)\hat{d}(b_{r+1+t}),$$
 which coincides with  
 $$(-1)^u\nu_{r+1+t}(id^{\otimes r}\otimes\nu_s\otimes id^{\otimes t})(\phi_3^r\otimes \phi_2^s
 \otimes \phi_1^t)(id^{\otimes r}\otimes\hat{d}(b_s)\otimes id^{\otimes t})\hat{d}(b_{r+1+t}),$$
 where $u= \vert \phi_3^r\vert+u_1+u_2+\vert \phi_1^t\vert$. 
 
 Notice that the map $\lambda_n:\hueca{Z}^n\rightmap{}\hueca{Z}$, given by 
 $\lambda_n(d_n,\ldots,d_1)=\sum_{n\geq i> j\geq 1}d_id_j$, always satisfies that $\lambda_n(d_n,\ldots,d_1)$ coincides with 
$$\begin{matrix}
\lambda_{r+1+t}(d_n,\ldots,d_{n-r+1},(d_{n-r}+\cdots +d_{n-r-s+1})-1,d_{n-r-s},\ldots,d_{1})\\
+\,\lambda_s(d_{n-r},\ldots,d_{n-r-s+1})
+\,d_n+\cdots+d_{n-r+1}+d_{n-r-s}+\cdots+d_1.\hfill\\
  \end{matrix}$$
Therefore, we obtain that  $u=\lambda_n(\vert F_n\vert,\ldots,\vert F_1\vert)$.  From (\ref{L: la formula para manipular la diferencial hat d}), we see that the last description of $S_{s,r,t}^{cv}$ coincides with 
$$(-1)^u\nu_{r+1+t}(id^{\otimes r}\otimes\nu_s\otimes id^{\otimes t})(\phi_3^r\otimes \phi_2^s
 \otimes \phi_1^t)(id^{\otimes r}\otimes\hat{\tau}_s\hat{D}(b_s)\otimes id^{\otimes t})\hat{\tau}_{r+1+t}\hat{D}(b_{r+1+t}),$$
 and with 
 $$(-1)^u\nu_{r+1+t}(id^{\otimes r}\otimes\nu_s\otimes id^{\otimes t})(\phi_3^r\otimes \phi_2^s
 \otimes \phi_1^t)\hat{\tau}_n\hat{D}(id^{\otimes r}\otimes b_s\otimes id^{\otimes t})\hat{D}(b_{r+1+t}).$$
 Then, we have 
 $$\begin{matrix}S_{s,r,t}^{cv}
 &=&
 (-1)^u\nu_n(\phi_3^r\otimes \phi_2^s
 \otimes \phi_1^t)\hat{\tau}_n\hat{D}(id^{\otimes r}\otimes b_s\otimes id^{\otimes t})\hat{D}(b_{r+1+t})\hfill\\
 &=&
 (-1)^u\nu_n(F_n\otimes\cdots\otimes F_1)
 \hat{\tau}_n\hat{D}(b_{r+1+t}(id^{\otimes r}\otimes b_s\otimes id^{\otimes t})).\hfill\\
 \end{matrix}$$
Then, adding up, we obtain 
$$\sum_{\scriptsize\begin{matrix}r+s+t=n\\ s\geq 1;r,t\geq 0\end{matrix}} 
   S_{s,r,t}^{cv}=\sum_{\scriptsize\begin{matrix}r+s+t=n\\ s\geq 1;r,t\geq 0\end{matrix}}(-1)^u\nu_n(F_n\otimes\cdots\otimes F_1)
 \hat{\tau}_n\hat{D}(b_{r+1+t}(id^{\otimes r}\otimes b_s\otimes id^{\otimes t})),$$
 where the sum at the right can be rewritten as
 $$(-1)^u\nu_n(F_n\otimes\cdots\otimes F_1)
 \hat{\tau}_n\hat{D}(\sum_{\scriptsize\begin{matrix}r+s+t=n\\ s\geq 1;r,t\geq 0\end{matrix}}b_{r+1+t}(id^{\otimes r}\otimes b_s\otimes id^{\otimes t}))=0,$$
 so $\cv(B)$ is a $b$-category. 
 \end{proof}

Now, we describe an  equivalence between $\ad(B)$ and $\cv(B)$, 
which plays an important role in the next section. 
 
\begin{proposition}\label{P: ad(B) equiv ac(B)}
For each pair of right $S$-modules $X$ and $Y$, there is an isomorphism of graded $k$-vector spaces 
$$\Psi_{X,Y}:\ad(B)(X,Y)\rightmap{}\cv(B)(X^{op},Y^{op}).$$
The recipe of $\Psi$ on each homogeneous component of degree $n\in \hueca{Z}$ is the following. 
Given 
$$f=\sum_{x\in \hueca{B}_n}f_x\otimes x\in \ad(B)(X,Y)_n=\bigoplus_{i,j\in {\cal P}}\Hom_k(Xe_i,Ye_j)\otimes_k e_jB_ne_i$$
we have 
$\Psi(f)\in \cv(B)(X^{op},Y^{op})_n=\Hom_{S\g S}(\hat{D}(B_n),\Hom_k(X^{op},Y^{op}))$, defined by 
$$\Psi(f)(\zeta)=\sum_{x\in \hueca{B}_n} f_x\zeta(x),$$
where $\zeta\in \hat{D}(B_n)$.  As before, we consider $\zeta\in \hat{D}(B_n)\subseteq \hat{D}(B)$ 
and $f_x\in \Hom_k(Xe_{s(x)},Ye_{t(x)})\subseteq \Hom_k(X,Y)$.  
Moreover, for each sequence of right $S$-modules $X_0,X_1,\ldots,X_n$, we have the following commutative diagram
$$\begin{matrix}\ad(B)(X_{n-1},X_n)\otimes_k\cdots \otimes_k\ad(B)(X_0,X_1)&\rightmap{b^{ad}_n}&\ad(B)(X_0,X_n)\hfill\\
\rmapdown{\Psi_{X_{n-1},X_n}\otimes\cdots\otimes\Psi_{X_0,X_1}}&&\rmapdown{\Psi_{X_0,X_n}}\\   
\cv(B)(X^{op}_{n-1},X^{op}_n)\otimes_k\cdots \otimes_k\cv(B)(X^{op}_0,X^{op}_1)&\rightmap{b^{cv}_n}&\cv(B)(X^{op}_0,X^{op}_n).\hfill\\   
  \end{matrix}$$ 
\end{proposition}

\begin{proof} Given any elements $f=\sum_{x\in \hueca{B}_n}f_x\otimes x\in \ad(B)(X,Y)_n$ and $\zeta\in \hat{D}(B_n)$, the product $f_x\zeta(x)$ is taken in the $S$-$S$-bimodule $\Hom_k(X,Y)$, but belongs to $\Hom_k(X^{op},Y^{op})=\Hom_k(X,Y)^{op}$. In fact, we have 
$$\begin{matrix}
   \cv(B)(X^{op},Y^{op})_n&=&\Hom_{S\g S}(\hat{D}(B_n),\Hom_k(X^{op},Y^{op}))\hfill\\
   &=&
   \Hom_{S\g S}(\hat{D}(B_n),\Hom_k(X,Y)^{op})\hfill\\
&=&
   \Hom_{S\g S}(D(B_n),\Hom_k(X,Y)).\hfill\\
   \end{matrix}$$
In order to show that $\Psi(f)\in\cv(B)(X^{op},Y^{op})_n$, we have to verify that 
$$\Psi(f)\in \Hom_{S\g S}(D(B_n),\Hom_k(X,Y)),$$
so take any 
$\zeta\in D(B_n)=\Hom_S(\,_SB_n,\,_SS)$ and let us show that, for any $s,s'\in S$, we have  
$\Psi(f)(s\zeta s')=s[\Psi(f)(\zeta)] s'$. 
It is enough to verify this on any generator $f=f_x\otimes x$ of $\ad(B)(X,Y)_n$. We have 
$$\begin{matrix}
  \Psi(f_x\otimes x)(s\zeta s')
  &=&
  f_x[(s\zeta s')(x)]\hfill
  &=&
  f_x[\zeta(xs)s']\hfill\\
  &=&
  [f_x\zeta(xs)]s'\hfill 
  &=&
  \Psi(f_x\otimes xs)(\zeta)s' \hfill\\ 
  &=&
   \Psi(sf_x\otimes x)(\zeta)s'\hfill
  &=&
 sf_x\zeta(x)s'\hfill\\ 
  &=&
  s[\Psi(f_x\otimes x)(\zeta)]s',&\hfill\\
  \end{matrix}$$
  because $x\in e_{t(x)}Be_{s(x)}$ and $f_x\in e_{s(x)}\Hom_k(X,Y)e_{t(x)}$. 
So $\Psi_{X,Y}$ is a well defined linear map, which restricts to linear maps
$$\Hom_k(Xe_i,Ye_j)\otimes_k e_jB_ne_i\rightmap{\Psi_n^{i,j}} \Hom_k(D(e_jB_ne_i),\Hom_k(Xe_i,Ye_j)).$$
 So, in order to verify that $\Psi_{X,Y}$ is an isomorphism, we verify that each $\Psi_n^{i,j}$ is an isomorphism. For this notice that $D(e_jB_ne_i)\cong \Hom_k(e_jB_ne_i,k)=D_k(e_jB_ne_i)$ and then we have an isomorphism 
 $$\Hom_k(D(e_jB_ne_i),\Hom_k(Xe_i,Ye_j))\rightmap{\phi} \Hom_k(D_k(e_jB_ne_i),\Hom_k(Xe_i,Ye_j)).$$
 
 Recall that, if $V$, $W$ are $k$-vector spaces, there is a natural morphism $$\psi:W\otimes_k V\rightmap{}\Hom_k(V^*,W),$$ where $V^*=\Hom_k(V,k)$ and  $\psi(w\otimes v)(\zeta)=w\zeta(v)$, for $w\in W$, $v\in V$, $\zeta\in V^*$. The map $\psi$ is an isomorphism when $V$ is finite-dimensional. If we take $W=\Hom_k(Xe_i,Ye_j)$ and $V=e_jB_ne_i$, we have  that $\phi\Psi^{i,j}_n=\psi$ is an isomorphism. So $\Psi_{X,Y}$ is an isomorphism.  
 
 Now, take a sequence of right $S$-modules $X_0,X_1,\ldots,X_n$. 
 We want to show that  
 $$b_n^{cv}(\Psi_{X_{n-1},X_n}\otimes\cdots\otimes\Psi_{X_0,X_1})
 =\Psi_{X_0,X_n}b_n^{ad}.$$
 Consider  typical homogeneous generators 
 $$f_n\otimes x_n\in \ad(B)(X_{n-1},X_n),\ldots, f_1\otimes x_1\in \ad(B)(X_0,X_1),$$
 $\zeta\in D(B_n)$, and let us compute 
 $$F_n:=b_n^{cv}(
 \Psi_{X_{n-1},X_n}(f_n\otimes x_n)
 \otimes\cdots\otimes
 \Psi_{X_0,X_1}(f_1\otimes x_1))(\zeta).$$
 From (\ref{L: coefs de hat(d)(bn)}),  we have 
 $\hat{d}(b_n)(\zeta)=\sum_{z\in \widehat{\hueca{T}}_n} c_z z$, where $\widehat{\hueca{T}}_n$ consists of products of the form $z=x^*_{i_n}\otimes\cdots\otimes x^*_{i_1}$ and $e_{t(x_{i_n})}c_z=\zeta(b_n(x_{i_n}\otimes\cdots\otimes x_{i_1}))$.
 Since  
 $$F_n= (-1)^{\lambda_n(\vert x_n\vert,\ldots,\vert x_1\vert)}\nu_n(\Psi_{X_{n-1},X_n}(f_n\otimes x_n)\otimes\cdots\otimes\Psi_{X_0,X_1}(f_1\otimes x_1))\hat{d}(b_n)(\zeta),$$
 we obtain  
 $$F_n=(-1)^{\lambda_n(\vert x_n\vert,\ldots,\vert x_1\vert)}
 \sum_{z}c_z\nu_n(\Psi_{X_{n-1},X_n}(f_n\otimes x_n)\otimes\cdots\otimes\Psi_{X_0,X_1}(f_1\otimes x_1))[z].$$
 Hence, we have 
 $$\begin{matrix}
  F_n
  &=&
  \sum_{x^*_{i_n}\otimes\cdots\otimes x^*_{i_1}}c_{x^*_{i_n}\otimes\cdots\otimes x^*_{i_1}}f_nx_{i_n}^*(x_n)\cdots f_1x_{i_1}^*(x_1)\hfill\\
  &=&
   c_{x^*_n\otimes\cdots\otimes x^*_1}f_n\cdots f_1\hfill\\
   &=&
   (f_n\cdots f_1)e_{t(x_n)}c_{x^*_n\otimes\cdots\otimes x^*_1}\hfill\\  
   &=&
   (f_n\cdots f_1)\zeta(b_n(x_n\otimes\cdots\otimes x_1))\hfill\\ 
   &=&
  \Psi_{X_0,X_n}(f_n\cdots f_1\otimes b_n(x_n\otimes\cdots\otimes x_1))(\zeta)\hfill\\ 
  &=&
  \Psi_{X_0,X_n}b_n^{ad}((f_n\otimes x_n)\otimes\cdots\otimes(f_1\otimes x_1))(\zeta).\hfill\\
   \end{matrix}$$
This finishes the proof. 
 \end{proof}

 \begin{lemma}\label{L: caracterizacion de los ciclos de en tw(B)}
 Assume that $t\in \tw(B)(\underline{X},\underline{Y})_{-1}$, then $t$ is a cocycle of the cochain complex $\tw(B)(\underline{X},\underline{Y})$ if and only if, for any $\alpha\in e_j\hat{D}(B_0)e_i$,   
 we have 
 
$$\Psi(\delta_Y)(\alpha) \Psi(t)(e^*_i)=\Psi(t)(e^*_j) \Psi(\delta_X)(\alpha)+S_{\delta(\alpha)}(t),$$  
where 
$$S_{\delta(\alpha)}(t)=\sum_{s=2}^\infty\sum_{\scriptsize\begin{matrix}i_0+i_1+1=s\\ 
i_0,i_1\geq 0\end{matrix}}
\nu_s(\Psi(\delta_Y)^{\otimes i_1}\otimes \Psi(t)\otimes\Psi(\delta_X)^{\otimes i_0})[\delta(\alpha)].$$
 \end{lemma}

 \begin{proof} We have $t\in \ad(B)(X,Y)_{-1}$. By definition, the 
 $\tw(B)$-morphism $t$ is a $-1$-cocycle iff we have 
  $$b_1^{tw}(t)=\sum_{i_0,i_1\geq 0}
b^{ad}_{i_0+i_1+1}(\delta_Y^{\otimes i_1}\otimes t\otimes\delta_X^{\otimes i_0})=0.$$
 From (\ref{P: ad(B) equiv ac(B)}), the preceding formula holds iff the following holds 
$$\begin{matrix}0&=&\sum_{i_0,i_1\geq 0}
\Psi b^{ad}_{i_0+i_1+1}(\delta_Y^{\otimes i_1}\otimes t\otimes\delta_X^{\otimes i_0})\hfill\\
&=&
\sum_{i_0,i_1\geq 0}
b^{cv}_{i_0+i_1+1}(\Psi(\delta_Y)^{\otimes i_1}\otimes \Psi(t)\otimes\Psi(\delta_X)^{\otimes i_0}).\hfill\\
\hfill\\
\end{matrix}$$
For any directed element $\alpha\in e_j\hat{D}(B_0)e_i$, we have 
 $$
\sum_{i_0,i_1\geq 0}
\nu_{i_0+i_1+1}(\Psi(\delta_Y)^{\otimes i_1}\otimes \Psi(t)\otimes\Psi(\delta_X)^{\otimes i_0})\hat{d}(b_{i_0+i_1+1})(\alpha)=0,$$
or, equivalently, 
$$
\sum_{s=2}^\infty\sum_{\scriptsize\begin{matrix}i_0+i_1+1=s\\ 
i_0,i_1\geq 0\end{matrix}}
\nu_s(\Psi(\delta_Y)^{\otimes i_1}\otimes \Psi(t)\otimes\Psi(\delta_X)^{\otimes i_0})d(\alpha)=0.$$
From (\ref{C: calculo de dif d en grados 0 y 1 usando la  diferencial delta}), we have $d(\alpha)=e^*_j\otimes \alpha-\alpha\otimes e_i^*+\delta(\alpha)+h(\alpha)$, with $h(\alpha)\in {\cal N}$. Evaluating at this expression of $d(\alpha)$ the preceding equation, we obtain 
$$\begin{matrix}
 0&=& \nu_2(\Psi(t)\otimes \Psi(\delta_X))(e_j^*\otimes \alpha)-\nu_2(\Psi(\delta_Y)\otimes \Psi(t))(\alpha\otimes e_i^*)+
  S_{\delta(\alpha)}(t)\hfill\\
  &=&
  \nu_2(\Psi(t)(e_j^*)\otimes \Psi(\delta_X)(\alpha)) -\nu_2(\Psi(\delta_Y)(\alpha)\otimes \Psi(t)(e_i^*))+
  S_{\delta(\alpha)}(t),\hfill\\
  \end{matrix}$$
where $S_{\delta(\alpha)}$ is defined in the statement of this lemma. 
 Then, we get 
$$\Psi(\delta_Y)(\alpha) \Psi(t)(e_i^*)=\Psi(t)(e_j^*) \Psi(\delta_X)(\alpha)+S_{\delta(\alpha)}(t),$$
as we wanted to show. 
 \end{proof}

 \begin{lemma}\label{L: tw(B) es homologicamente unitaria}
  The $b$-category $\tw(B)$ is cohomologically unitary, as in (\ref{D: H-1(cal B)}), so $H^{-1}(\tw(B))=Z_{-1}(\tw(B))$ is a category. Moreover, for each object $\underline{X}$ the identity morphism 
 $1_{\underline{X}}$ in $Z_{-1}(\tw(B))(\underline{X},\underline{X})$ of $\underline{X}$ satisfies $\Psi(1_{\underline{X}})(e_i^*)=id_{Xe_i}$, for $i\in {\cal P}$, and $\Psi(1_{\underline{X}})(\xi)=0$, for  $\xi\in\hat{D}(J)$. 
 \end{lemma}
 
 \begin{proof} Take $\underline{X}=(X,\delta_X)\in \Ob(\tw(B))$ and consider $\tau_{X}:=\sum_{j\in {\cal P}} id_{Xe_j}\otimes e_j$, thus $\tau_{X}\in \tw(B)(\underline{X},\underline{X})_{-1}=\ad(B)(X,X)_{-1}$. From (\ref{L: caracterizacion de los ciclos de en tw(B)}),   in order to prove that $b_1^{tw}(\tau_X)=0$, we have to show, for  
 $\alpha\in e_j\hat{D}(B_0)e_i$, the following equality  
 $$\Psi(\delta_X)(\alpha) \Psi(\tau_X)(e_i^*)=\Psi(\tau_X)(e_j^*) \Psi(\delta_X)(\alpha)+S_{\delta(\alpha)}(\tau_X).$$
 Recall that $\Psi(\tau_X)\in \Hom_{S\g S}(\hat{D}(B_{-1}),\Hom_k(X^{op},X^{op}))$. For $\xi\in \hat{D}(J)$, we get  
 $$\Psi(\tau_X)(\xi)=\Psi(\sum_{j\in {\cal P}}id_{Xe_j}\otimes e_j)(\xi)=\sum_{j\in {\cal P}}id_{Xe_j}\xi(e_i)=0,$$ and, for any $i\in {\cal P}$, we have   
$\Psi(\tau_X)(e_i^*)= \Psi(\sum_jid_{Xe_j}\otimes e_j)(e_i^*)=id_{Xe_i}$.
 These equalities imply the wanted equality, because, from (\ref{L: coeficientes de delta(x*)}), it follows that $S_{\delta(\alpha)}(\tau_X)=0$. Hence, we have $\tau_X\in Z_{-1}(\tw(B))(\underline{X},\underline{X})$.

 Now, given $t\in \tw(B)(\underline{X},\underline{Y})_{-1}=\ad(B)(X,Y)_{-1}$, we  want to show that $\tau_Y t=t$ and $t \tau_X=t$ in $H^{-1}(\tw(B))=Z_{-1}(\tw(B))$. That is $b^{tw}(\tau_Y\otimes t)=t$ and $b^{tw}(t\otimes \tau_X)=t$. 
 In order to prove that $t \tau_X=t$, we have to show that 
 $$\sum_{i_0,i_1,i_2\geq 0}b^{ad}_{i_0+i_1+i_2+2}(\delta_Y^{\otimes i_2}\otimes t\otimes\delta_X^{\otimes i_1}\otimes\tau_X\otimes\delta_X^{\otimes i_0})=t.$$
 Recall that, for $n\geq 3$, we have $b_n(z_n\otimes\cdots\otimes z_2\otimes z_1)=0$ when some factor $z_i\in \{e_j\mid j\in {\cal P}\}$, see (\ref{R: merkulov}). 
 Then we are reduced to show that 
 $b_2^{ad}(t\otimes\tau_X)=t$. The last equality holds for the typical generators $f_s\otimes x_s\in \ad(B)(X,Y)_{-1}$, 
 where $x_s\in e_vB_{-1}e_u$ and 
 $f_s\in \Hom_k(Xe_u,Ye_v)$, for some $u,v\in {\cal P}$, because 
 $$\begin{matrix}
 b^{ad}_2((f_s\otimes x_s)\otimes \tau_X)
 &=&
 \sum_jb^{ad}_2((f_s\otimes x_s)\otimes  (id_{Xe_j}\otimes e_j))\hfill\\
 &=&
 \sum_jf_s id_{Xe_j}\otimes b_2(x_s\otimes e_j)\hfill\\
 &=&
 f_s\otimes ((-1)^{1+\vert x_s\vert}\sigma m_2(\sigma^{-1}(x_s)\otimes\sigma^{-1}(e_u))\hfill\\
 &=&f_s\otimes x_se_u=f_s\otimes x_s.\hfill\\
 \end{matrix}$$
 Hence $b_2^{ad}(t\otimes \tau_X)=t$, and $t\tau_X=t$ in $Z_{-1}(\tw(B))$. Similarly, we have that 
 $b_2^{ad}(\tau_Y\otimes t)=t$, and then $\tau_Yt=t$ in $Z_{-1}(\tw(B))$. 
 \end{proof}

\section{The category of $({\cal A}(\Delta),I)$-modules}\label{de H(tw(B)) en A(Delta)-mod}
 
 In this section we describe an equivalence of categories from $H^{-1}(\tw(B))$ onto $({\cal A}(\Delta),I)$-Mod. 
  We recall from \cite{bpsqh}\S 2, the following.
 
 \begin{definition}\label{D: la cat de mod's de cal A(Delta)}
  Assume that ${\cal A}={\cal A}(\Delta)$ is the weak ditalgebra associated to $\Delta$, which is interlaced with the ideal $I$, as in (\ref{D: def del bocs A(Delta)}), see (\ref{D: weak ditalgebra}) and (\ref{D: weak ditalg interlaced an ideal}).   
The category $({\cal A}(\Delta),I)\g\Mod$ is the following. 
The class of objects of $({\cal A},I)\g\Mod$ is the class of left $A(\Delta)$-modules 
$M$ 
such that $IM=0$; given two $A(\Delta)$-modules $M$ and $N$ annihilated by $I$,
the set of morphisms $\Hom_{({\cal A},I)}(M,N)$ from $M$ to $N$ in $({\cal A},I)\g\Mod$ is, by definition, 
the 
collection of pairs $f=(f^0,f^1)$, with $f^0\in \Hom_k(M,N)$ and 
$f^1\in \Hom_{A(\Delta)\g A(\Delta)}(V,\Hom_k(M,N))$, where $V$ denotes the bimodule 
$V=A(\Delta)\otimes_S \hat{D}(J)\otimes_S A(\Delta)$, such that, for any $a\in A(\Delta)$ and $m\in M$, 
the following holds
$$af^0[m]=f^0[am]+f^1(\delta(a))[m].$$
If $f^1\in \Hom_{A(\Delta)\g A(\Delta)}(V,\Hom_k(M,N))$  and 
$g^1\in \Hom_{A(\Delta)\g A(\Delta)}(V,\Hom_k(N,L))$, we consider the morphism 
$g^1*f^1\in \Hom_{A(\Delta)\g A(\Delta)}(V^{\otimes 2},\Hom_k(M,L))$ defined, for any 
$\sum_ju_j\otimes v_j\in V\otimes_AV$, by 
$$(g^1*f^1)(\sum_ju_j\otimes v_j)=\sum_j g^1(u_j)f^1(v_j):M\rightmap{}L.$$

Given   $f=(f^0,f^1)\in \Hom_{({\cal A},I)}(M,N)$ and 
$g=(g^0,g^1)\in \Hom_{({\cal A},I)}(N,L)$, by definition 
$gf:=(g^0f^0,(gf)^1)\in \Hom_{({\cal A},I)}(M,L)$, where 
$$(gf)^1(v):=g^0f^1(v)+g^1(v)f^0+(g^1*f^1)(\delta(v)),\hbox{ for } v\in V.$$
\end{definition}

 It is convenient to think an $A(\Delta)$-module $M$ as a pair $(M,\rho_M)$, where $M$ is a left $S$-module and $\rho_M:A(\Delta)\rightmap{}\End_k(M)$ is a morphism of $S$-algebras. Moreover, such $\rho_M$ is uniquely determined by a morphism of $S$-$S$-bimodules, which we denote with the same symbol, 
 $\rho_M:\hat{D}(B_0)\rightmap{}\Hom_k(\,_SM,\,_SM)$. In the next statement we reformulate the conditions defining $({\cal A},I)\g\Mod$.
 
 \begin{lemma}\label{L: (cal A,I) reformulada}
 We have the following equivalences.
 \begin{enumerate}
  \item  An $A(\Delta)$-module $M$ satisfies $IM=0$ 
  iff  $$\sum_{n=2}^\infty\nu_n(\rho_M^{\otimes n})\hat{d}(b_{n,0})(\zeta)=0,\hbox{  for all }\zeta\in \hat{D}(B)_{-1}.$$ 
  \item Given $M,N\in ({\cal A},I)\g\Mod$, a pair $(f^0,f^1)$, where $f^0\in \Hom_k(M,N)$ and 
$f^1\in \Hom_{A(\Delta)\g A(\Delta)}(V,\Hom_k(M,N))$, is a morphism in 
$({\cal A},I)\g\Mod$ iff for any directed element $\alpha\in \hat{D}(B_0)$, we have  
$$\rho_N(\alpha)f_{s(\alpha)}^0=f_{t(\alpha)}^0\rho_M(\alpha)+\sum_{n=2}^\infty\sum_{i_0+i_1+1=n}\nu_n(\rho_N^{\otimes i_1}\otimes f^1\otimes \rho_M^{\otimes i_0})(\delta(\alpha)),$$
where $f^0_i=e_if^0e_i:e_iM\rightmap{}e_iN$ denotes the restriction of $f^0$, for  $i\in {\cal P}$. 
\item Given morphisms $f=(f^0,f^1):(M,\rho_M)\rightmap{}(N,\rho_N)$ and $g=(g^0,g^1):(N,\rho_N)\rightmap{}(L,\rho_L)$ in $({\cal A},I)\g\Mod$, then 
$$(gf)^1=g^0f^1+g^1f^0-
\sum_{n=2}^\infty\sum_{\scriptsize\begin{matrix}i_0+i_1+i_2+2=n\\ 
i_0,i_1,i_2\geq 0\end{matrix}}\nu_n(\rho_L^{\otimes i_2}\otimes g^1\otimes \rho_N^{\otimes i_1}\otimes f^1\otimes\rho_M^{\otimes i_0})\delta.$$ 
 \end{enumerate}
 \end{lemma}

 \begin{proof} (1): Given an $A(\Delta)$-module $M$, the action of some element $\alpha\in\hat{D}(B_0)$ on $m\in M$ is given by $\alpha m=\rho_M(\alpha)(m)$. For an element $\gamma_s\in \hat{D}(B_0)^{\otimes s}$ of the form $\gamma_s=\alpha_s\otimes\cdots\otimes \alpha_1$, with $\alpha_1,\ldots,\alpha_s\in \hat{D}(B_0)$, we have 
 $$\gamma_s m=[\rho_M(\alpha_s)\cdots\rho_M(\alpha_1)](m)=
 \nu_s\rho_M^{\otimes s}(\alpha_s\otimes\cdots\otimes \alpha_1)(m).$$
 Thus, if $\gamma=\sum_s\gamma_s$, we have $\gamma m=\sum_s\nu_s\rho_M^{\otimes s}(\gamma_s)[m]$. Clearly, $IM=0$ iff 
 $\beta(\zeta)M=0$, for all $\zeta\in \hat{D}(B)_{-1}$. This is equivalent to 
 $\rho_M(\beta(\zeta))=0$. Recall that 
 $\beta(\zeta)=\sum_{n=2}^\infty \hat{d}(b_{n,0})(\zeta)$, a finite sum with 
 $\hat{d}(b_{n,0})(\zeta)\in \hat{D}(B_0)^{\otimes n}$. So, 
 $\rho_M(\beta(\zeta))=\sum_{n=2}^\infty\nu_n\rho_M^{\otimes n}\hat{d}(b_{n,0})(\zeta)$ and (1) holds. 
 
 \medskip
 (2): Take $f^0\in \Hom_k(M,N)$ and 
$f^1\in \Hom_{A(\Delta)\g A(\Delta)}(V,\Hom_k(M,N))$. The morphism $f^1$ is uniquely determined by its restriction $f^1:\hat{D}(J)\rightmap{}\Hom_k(M,N)$, by the formula 
$$f^1(\gamma_r\otimes \xi\otimes \gamma_t)
=
(\nu_r\rho_N^{\otimes r}\gamma_r) f^1(\xi)(\nu_t\rho_M^{\otimes t}\gamma_t)=\nu_{r+1+t}(\rho_N^{\otimes r}\otimes f^1\otimes \rho_M^{\otimes t})(\gamma_r\otimes\xi\otimes\gamma_t),$$
for $\gamma_r\in \hat{D}(B_0)^{\otimes r}$, $\xi\in \hat{D}(J)$ and $\gamma_t\in\hat{D}(B_0)^{\otimes t}$.  The condition for $(f^0,f^1)$ to be a morphism from $M$ to $N$ in $({\cal A},I)\g\Mod$ is equivalent to 
$$\alpha f^0[m]=f^0[\alpha m]+f^1(\delta(\alpha ))[m],$$
for any $\alpha\in \hat{D}(B_0)$ and $m\in M$. We can write  $\delta(\alpha)=\sum_n\delta(\alpha)_n$, where 
$$\delta(\alpha)_n=
\sum_i\gamma_i^2\otimes \xi_i\otimes \gamma_i^1,$$
with $\gamma_i^2\in \hat{D}(B_0)^{\otimes r}$, $\xi_i\in \hat{D}(J)$,  $\gamma_i^1\in \hat{D}(B_0)^{\otimes t}$ satisfying that 
$r+1+t=n$. Then, we have 
$$f^1(\delta(\alpha))=\sum_{n=2}^{\infty}\sum_{r+1+t=n}\nu_n
(\rho_N^{\otimes r}\otimes f^1\otimes \rho_M^{\otimes t})(\delta(\alpha)).$$
Then, (2) holds.

\medskip
(3): Consider morphisms $f=(f^0,f^1):(M,\rho_M)\rightmap{}(N,\rho_N)$ and $g=(g^0,g^1):(N,\rho_N)\rightmap{}(L,\rho_L)$ in $({\cal A},I)\g\Mod$. Given $v\in V$, we can write $\delta(v)=\sum_n\delta(v)_n$, where 
$$\delta(v)_n=
\sum_i
\gamma_i^3\otimes \xi_i^2\otimes\gamma_i^2\otimes \xi_i^1\otimes\gamma_i^1,$$
with $\gamma_i^3\in \hat{D}(B_0)^{\otimes r}$, $\gamma_i^2\in \hat{D}(B_0)^{\otimes s}$, $\gamma_i^1\in \hat{D}(B_0)^{\otimes t}$, and $\xi^1_i,\xi_i^2\in \hat{D}(J)$, satisfying that $r+s+t+2=n$. Then, since the morphism $f^1\in \Hom_{S\g S}(\hat{D}(J),\Hom_k(M,N))$ has degree 1, we obtain 
$$(g^1*f^1)(\delta(v))=-\sum_{n=2}^{\infty}\sum_{r+s+t+2=n}\nu_n
(\rho_L^{\otimes r}\otimes g^1\otimes \rho_N^{\otimes s}\otimes f^1\otimes \rho_M^{\otimes t})(\delta(v)).$$
Then, we get the formula in (3).
 \end{proof}

\begin{theorem}\label{T: tw(B) equiv (cal A,I)-Mod}
 Let $A$ be the Yoneda $A_\infty$-algebra associated to $\Delta$, $B$ its bar construction, and ${\cal A}=({\cal A}(\Delta),\delta)$ the corresponding weak ditalgebra with the ideal $I$ fixed in  (\ref{P: I entrelazado}). Then, there is an equivalence of categories 
 $$H^{-1}(\tw(B))\simeq ({\cal A},I)\g\Mod.$$
 Since $\tw(B)(\underline{X},\underline{Y})_n=\ad(B)(X,Y)_n=0$, for $n<-1$, the category $H^{-1}(\tw(B))$ coincides with the category 
 $Z_{-1}(\tw(B))$ which has the same objects than $\tw(B)$, and the morphisms $t:\underline{X}\rightmap{}\underline{Y}$ 
 of $Z_{-1}(\tw(B))$ are the elements in $\ad(B)(X,Y)_{-1}$ such that 
 $$b_1^{tw}(t)=\sum_{i_0,i_1\geq 0}
b^{ad}_{i_0+i_1+1}(\delta_Y^{\otimes i_1}\otimes t\otimes\delta_X^{\otimes i_0})=0.$$
The composition of two morphisms $t_1:\underline{X}_0\rightmap{}\underline{X}_1$ and $t_2:\underline{X}_1\rightmap{}\underline{X}_2$ 
 in $Z_{-1}(\tw(B))$ is given by 
 $$t_2t_1=b_2^{tw}(t_2\otimes t_1)=\sum_{i_0,i_1,i_2\geq 0}b^{ad}_{i_0+i_1+i_2+2}(\delta_{X_2}^{\otimes i_2}\otimes t_2\otimes \delta_{X_1}^{\otimes i_1}\otimes t_1\otimes\delta_{X_0}^{\otimes i_0}).$$
 
 The equivalence functor $M:Z_{-1}(\tw(B))\rightmap{}({\cal A},I)\g\Mod$ is defined as follows. Given $\underline{X}\in Z_{-1}(\tw(B))$, its image is the $A(\Delta)$-module $M(\underline{X})=(X^{op},\Psi(\delta_X))$, where $\Psi$ is the isomorphism introduced in (\ref{P: ad(B) equiv ac(B)}). Given a morphism  $t:\underline{X}\rightmap{}\underline{Y}$ 
 in $Z_{-1}(\tw(B))$, its image $M(t)=(M(t)^0,M(t)^1):M(\underline{X})\rightmap{}M(\underline{Y})$ is given by  
 $$M(t)^0=\Psi(t)(\sum_{j\in {\cal P}}  e^*_j) \hbox{ \ and  \ } M(t)^1(\xi)=\Psi(t)(\xi), \hbox{ for } \xi\in \hat{D}(J).$$
\end{theorem}

\begin{proof} \emph{Step 1:} According to (\ref{L: (cal A,I) reformulada}), in order to show that $IM(\underline{X})=0$, we have to prove that, for any $\zeta\in \hat{D}(B)_{-1}$, we have $\sum_{s=2}^\infty\nu_s(\rho_M^{\otimes s})\hat{d}(b_{s,0})(\zeta)=0$. But, since $\underline{X}$ is an object of $\tw(B)$, it satisfies $\sum_{s=1}^\infty b_s^{ad}(\delta_X^{\otimes s})=0$. From the commutativity of the square in (\ref{P: ad(B) equiv ac(B)}), using (\ref{L: formula para hat(d)(bn)}) and the fact that $\vert \delta_X\vert=0$, we obtain 
$$\begin{matrix}0&=&
\sum_{s=1}^\infty \Psi b_s^{ad}(\delta_X^{\otimes s})(\zeta)\hfill\\
&=&
\sum_{s=1}^\infty b^{cv}_s(\Psi(\delta_X)^{\otimes s})(\zeta)\hfill\\
&=&
\sum_{s=2}^\infty\nu_s(\rho_M^{\otimes s})\hat{d}(b_s)(\zeta)\hfill\\
&=&
\sum_{s=2}^\infty\nu_s(\rho_M^{\otimes s})\hat{d}(b_{s,0})(\zeta).\hfill\\
  \end{matrix}$$
So, indeed we have $M(\underline{X})\in ({\cal A},I)\g\Mod$. 

\medskip
\emph{Step 2:} Take a morphism  $t:\underline{X}\rightmap{}\underline{Y}$ 
 of $Z_{-1}(\tw(B))$. 
 From (\ref{L: caracterizacion de los ciclos de en tw(B)}), we have 
$$\Psi(\delta_Y)(\alpha) \Psi(t)(e^*_i)=\Psi(t)(e^*_j) \Psi(\delta_X)(\alpha)+S_{\delta(\alpha)}(t),$$  
for any $\alpha\in e_j\hat{D}(B_0)e_i$, where 
$$S_{\delta(\alpha)}(t)=\sum_{s=2}^\infty\sum_{\scriptsize\begin{matrix}i_0+i_1+1=s\\ 
i_0,i_1\geq 0\end{matrix}}
\nu_s(\Psi(\delta_Y)^{\otimes i_1}\otimes \Psi(t)\otimes\Psi(\delta_X)^{\otimes i_0})(\delta(\alpha)).$$

We want to show that the pair $M(t)=(M(t)^0,M(t)^1)$ satisfies the equation of (\ref{L: (cal A,I) reformulada})(2), 
where $M(t)^1$ denotes also the extension of the morphism 
$M(t)^1:\hat{D}(J)\rightmap{}\Hom_k(M(\underline{X}),M(\underline{Y}))$ to $V$.

We have that $\Psi(t)\in \Hom_{S\g S}(\hat{D}(B_{-1}),\Hom_k(X^{op},Y^{op}))$. For each $i\in {\cal P}$, we have $M(t)^0_i=e_iM(t)^0e_i=e_i\Psi(t)(\sum_je_j^*)e_i=\Psi(t)(e_i^*)$. Then, we get 
$$\Psi(\delta_Y)(\alpha) M(t)^0_i=M(t)_j^0 \Psi(\delta_X)(\alpha)+S_{\delta(\alpha)}(t),$$ 
which means that $M(t):M(\underline{X})\rightmap{}M(\underline{Y})$ is a morphism in $({\cal A},I)\g \Mod$, according to 
(\ref{L: (cal A,I) reformulada})(2).

\medskip
\emph{Step 3:} Take morphisms $t_1:\underline{X}_0\rightmap{}\underline{X}_1$ and $t_2:\underline{X}_1\rightmap{}\underline{X}_2$ 
 in $Z_{-1}(\tw(B))$, and let us show that the given association rule $M$ preserves their composition: $M(t_2t_1)=M(t_2)M(t_1)$. 
 Applying $\Psi$ to the formula of the composition $t_2t_1$ in $Z_{-1}(\tw(B))$ recalled in the statement of this theorem, we have 
 $$
 \Psi(t_2t_1)=\sum_{i_0,i_1,i_2\geq 0}\Psi(b^{ad}_{i_0+i_1+i_2+2}(\delta_{X_2}^{\otimes i_2}\otimes t_2\otimes \delta_{X_1}^{\otimes i_1}\otimes t_1\otimes\delta_{X_0}^{\otimes i_0})),$$
 which coincides with 
$$\sum_{i_0,i_1,i_2\geq 0}b^{cv}_{i_0+i_1+i_2+2}(\Psi(\delta_{X_2})^{\otimes i_2}\otimes \Psi(t_2)\otimes \Psi(\delta_{X_1})^{\otimes i_1}\otimes \Psi(t_1)\otimes\Psi(\delta_{X_0})^{\otimes i_0}).
$$
For simplicity, we write $M_u$ and $\rho_u$ instead of $M(\underline{X}_u)$ and $\Psi(\delta_{X_u})$, respectively, for $u\in \{0,1,2\}$. Then, for a basic directed element $x\in e_j\hueca{B}_{-1}e_i$, we have that $\Psi(t_2t_1)(x^*)$ coincides with 
$$\sum_{s=2}^\infty\sum_{\scriptsize\begin{matrix}i_0+i_1+i_2+2=s\\ 
i_0,i_1,i_2\geq 0\end{matrix}}(-1)^{u(i_2,i_1,i_0)}\nu_s(\rho_2^{\otimes i_2}\otimes \Psi(t_2)\otimes \rho_1^{\otimes i_1}\otimes \Psi(t_1)\otimes\rho_0^{\otimes i_0})d(x^*),$$
with  $u(i_2,i_1,i_0)=\lambda_s(0,\ldots,0,1,0,\ldots,0,1,0,\ldots,0)=1$, where the first sequence of ceros, from left to right,  has length $i_2$, the second has length $i_1$, and the third one has length $i_0$. 

If $x=e_i$, then $d(e_i^*)=e^*_i\otimes e^*_i+h(e^*_i)$, with $h(e_i^*)\in {\cal N}$. It follows that 
$M(t_2t_1)^0_i=\Psi(t_2t_1)(e_i^*)=-\nu_2(\Psi(t_2)\otimes \Psi(t_1))(e^*_i\otimes e^*_i)=
\Psi(t_2)(e_i^*)\Psi(t_1)(e^*_i)=M(t_2)^0_iM(t_1)_i^0$, for all $i\in {\cal P}$, so $M(t_2t_1)^0=M(t_2)^0M(t_1)^0$. 

If $x\in J$, we have 
$d(x^*)=e_j^*\otimes x^*+x^*\otimes e_i^*+\delta(x^*)+h(x^*)$, with $h(x^*)\in {\cal N}$. Then, we have 
that $M(t_2t_1)^1(x^*)=\Psi(t_2t_1)(x^*)=S_0+S_{\delta(x^*)}$, where 
$$S_{0}=-\nu_2(\Psi(t_2)\otimes \Psi(t_1))(e_j^*\otimes x^*+x^*\otimes e_i^*), \hbox{ and }$$ 
$$S_{\delta(x^*)}=-\sum_{s=2}^\infty\sum_{\scriptsize\begin{matrix}i_0+i_1+i_2+2=s\\ 
i_0,i_1,i_2\geq 0\end{matrix}}\nu_s(\rho_2^{\otimes i_2}\otimes \Psi(t_2)\otimes \rho_1^{\otimes i_1}\otimes \Psi(t_1)\otimes\rho_0^{\otimes i_0})(\delta(x^*)).$$
Moreover, we have 
$$\begin{matrix}
   S_0&=&
   \Psi(t_2)(e_j^*)\Psi(t_1)(x^*)+\Psi(t_2)(x^*)\Psi(t_1)(e_i^*)\hfill\\
   &=&
   M(t_2)^0_jM(t_1)^1(x^*)+M(t_2)^1(x^*)M(t_1)^0_i\hfill\\
   &=&
   [M(t_2)^0M(t_1)^1+M(t_2)^1M(t_1)^0](x^*).\hfill\\
  \end{matrix}$$
By the definition of the composition in $({\cal A},I)\g\Mod$, we have  $$(M(t_2)M(t_1))^1=M(t_2)^0M(t_1)^1+M(t_2)^1M(t_1)^0+(M(t_2)^1*M(t_1)^1)\delta.$$
So, we only have to verify that $(M(t_2)^1*M(t_1)^1)(\delta(x^*))=S_{\delta(x^*)}$. 
But, according with (\ref{L: (cal A,I) reformulada})(3), we have this equality. Hence, we get that $M$ preserves composition of morphisms. 

 According with (\ref{L: tw(B) es homologicamente unitaria}), the identity morphism 
 $1_{\underline{X}}$ in $Z_{-1}(\tw(B))(\underline{X},\underline{X})$ of the object $\underline{X}$ satisfies $\Psi(1_{\underline{X}})(e_i^*)=id_{Xe_i}$ and $\Psi(1_{\underline{X}})(\xi)=0$, for all $\xi\in\hat{D}(J)$. Then, we have  
  $M(1_{\underline{X}})^0_i=\Psi(1_{\underline{X}})(e_i^*)=id_{Xe_i}$ and,  
 if $\xi\in \hat{D}(J)$, we have $M(1_{\underline{X}})^1(\xi)=\Psi(1_{\underline{X}})(\xi)=0$. 
So $M$ preserves identities  and we have a 
functor $M:Z_{-1}(\tw(B))\rightmap{}({\cal A},I)\g\Mod$. 

\medskip
\emph{Step 4:} Let us show that the functor $M$ is dense. So, consider an object  in $(N,\rho_N)\in ({\cal A},I)\g\Mod$. Then, we can write it as $N=X^{op}$, where $X$ is a right $S$-module.    
We have 
$\rho_{N}\in \Hom_{S\g S}(\hat{D}(B)_0,\Hom_k(X^{op},X^{op}))$, but we know that 
$\Psi:\ad(B)(X,X)\rightmap{}\cv(B)(X^{op},X^{op})$ is an isomorphism, so there is some $\delta_{X}\in \ad(B)(X,X)$ homogeneous with degree 0 such that $\Psi(\delta_{X})=\rho_N$. 

By (\ref{R: B y D(B) homogeneos}), we have that $\hat{D}(B)_0$ is finite dimensional and $\hat{D}(B)_0^{\otimes \ell}=0$, then the $k$-algebra  $\Gamma=A(\Delta)=T_S(\hat{D}(B_0))$ is  finite dimensional. Consider the Jacobson radical $\rad(\Gamma)$ of $\Gamma$, so 
we have  $\rad(\Gamma)^n=0$ for its nilpotence index $n$. Consider the filtration of left $S$-modules   
$$0=\rad(\Gamma)^nN\subseteq \cdots\subseteq \rad(\Gamma)^2N\subseteq \rad(\Gamma)N\subseteq N$$
or, equivalently, the filtration of right $S$-modules
$$0=X\rad(\Gamma^{op})^n\subseteq \cdots\subseteq X\rad(\Gamma^{op})^2\subseteq X\rad(\Gamma^{op})\subseteq X.$$
If we write $\delta_X=\sum_{x\in \hueca{B}_0}f_x\otimes x$ and take a directed basic element $y\in \hueca{B}_0$, we obtain 
$\rho_N(y^*)=\Psi(\delta_X)(y^*)=\sum_{x\in \hueca{B}_0}f_xy^*(x)=f_ye_{t(y)}=f_y$. Since $y^*\in \hat{D}(B)_0\subseteq \rad(\Gamma)$, it annihilates the semisimple right $\Gamma^{op}$-module $X\rad(\Gamma^{op})^{s}/X\rad(\Gamma^{op})^{s+1}$, for all $s\in [0,n-1]$. Hence, we have 
$$f_y(X\rad(\Gamma^{op})^{s})=\rho_N(y^*)X\rad(\Gamma^{op})^{s}\subseteq X\rad(\Gamma^{op})^{s+1},$$
and the pair $(X,\delta_X)$ satisfies the first condition of (\ref{P: tw(B)}). 
For the second one, since $(N,\rho_N)\in ({\cal A},I)\g\Mod$, the $A(\Delta)$-module $N$ satisfies 
 $IN=0$. As we saw in (\ref{L: (cal A,I) reformulada}), this means that 
  $\sum_{n=2}^\infty\nu_n(\rho_N^{\otimes n})\hat{d}(b_{n,0})=0$. Then, we have 
  $$
  0=
  \sum_{n= 2}^\infty b^{cv}_n(\Psi(\delta_X)^{\otimes n})
  =
  \Psi(\sum_{n= 2}^{\infty} b_n^{ad}(\delta_X^{\otimes n})).$$
Since $\Psi$ is an isomorphism, we obtain that $\sum_{n= 2}^{\infty} b_n^{ad}(\delta_X^{\otimes n})=0$. Therefore, the pair $(X,\delta_X)$ is an object of $\tw(B)$

\medskip
\emph{Step 5:} Let us show that $M$ is a faithful and full functor. Assume that $t:\underline{X}\rightmap{}\underline{Y}$ is a morphism in $Z_{-1}(\tw(B))$ such that $M(t)=0$. Thus, we have $\Psi(t)(e_i^*)=M(t)^0_i=0$, for $i\in {\cal P}$, and  $\Psi(t)(\xi)=M(t)^1(\xi)=0$, for  $\xi\in \hat{D}(J)$. Here,  the domain of $\Psi(t)$ is $\hat{D}(B)_1=\bigoplus_{j\in {\cal P}} ke_j^*\oplus \hat{D}(J)$, so we get $\Psi(t)=0$. Since $\Psi$ is an isomorphism we obtain $t=0$, and $M$ is a faithful functor. 

Let $f=(f^0,f^1):M(X,\delta_X)\rightmap{}M(Y,\delta_Y)$ be a morphism in $({\cal A},I)\g\Mod$. Consider the morphism of $S$-$S$-bimodules $\phi:\hat{D}(B)_1\rightmap{}\Hom_k(X^{op},Y^{op})$ defined by 
$$\phi(e_i^*)=e_if^0e_i\in \Hom_k(Xe_i,Ye_i) \hbox{  \ and  \ } \phi(\xi)=e_jf^1(\xi)e_i\in \Hom_k(Xe_i,Ye_j),$$
for $i\in {\cal P}$ and $\xi\in e_jJe_i\cap \hueca{B}$. Since $\Psi$ is surjective, there is some morphism $t\in \ad(B)(X^{op},Y^{op})_{-1}$ such that $\Psi(t)=\phi$. Since $f=(f^0,f^1)$ is a morphism in $({\cal A},I)\g\Mod$, we have the equation 
(\ref{L: (cal A,I) reformulada})(2), for all directed elements $\alpha\in \hat{D}(B)_0$. Then, we can reverse the argument given in \emph{Step 2} to obtain that $t$  satisfies 
$$b_1^{tw}(t)=\sum_{i_0,i_1\geq 0}
b^{ad}_{i_0+i_1+1}(\delta_Y^{\otimes i_1}\otimes t\otimes\delta_X^{\otimes i_0})=0.$$
So, $t:(X,\delta_X)\rightmap{}(Y,\delta_Y)$ is a morphism in $Z_{-1}(\tw(B))$. Then, we have that $M$ is a full functor. 
\end{proof}

\section{The weak ditalgebra ${\cal A}(\Delta)$ and the ideal $I$}

We recall some terminology from \cite{BSZ} and \cite{bpsqh}. 

\begin{definition}\label{D: weak ditalgebra}
 A \emph{weak ditalgebra} ${\cal A}$ is a pair ${\cal A}=(T,\delta)$, where $T$ is a graded tensor algebra and $\delta$ is a \emph{differential} on $T$, that is $\delta:T \rightmap{}T$ is a homogeneous linear map of degree 1 and satisfies Leibniz rule 
  $\delta(ab)=\delta(a)b+(-1)^{\vert a\vert}a\delta(b)$ 
 on homogeneous elements $a,b\in T$. 
 The weak ditalgebra ${\cal A}$ is called a \emph{ditalgebra} if, furthermore, we have $\delta^2=0$. 
 
 A weak ditalgebra ${\cal A}=(T,\delta)$ has \emph{layer $(R,W)$} iff $R$ is a $k$-algebra and $W$ is an $R$-$R$-bimodules equipped with an $R$-$R$-bimodule decomposition $W_0\oplus W_1$ such that $W_0\subseteq T_0$, $W_1\subseteq T_1$, the algebra  $T$ is freely generated by the pair $(R,W)$, and $\delta(R)=0$, see \cite{BSZ}(4.1). When there is no danger of confusion, we write $A:=T_0$ and $V:=T_1$. Thus, we have $A\cong T_R(W_0)$ and $V\cong A\otimes_RW_1\otimes_R A$. 
 
A layer $(R,W)$ of a weak ditalgebra ${\cal A}=(T,\delta)$ is called 
 \emph{triangular} if
\begin{enumerate}
 \item There is a filtration of $R$-$R$-subbimodules 
$0=W_0^0\subseteq W_0^1\subseteq \cdots \subseteq W_0^r=W_0$ such that 
$\delta(W_0^{i+1})\subseteq A_iW_1A_i$, for all $i\in [0,r-1]$, 
where $A_i$ denotes the $R$-subalgebra $A$ generated by $W_0^i$.
\item There is a filtration of $R$-$R$-subbimodules 
$0=W_1^0\subseteq W_1^1\subseteq \cdots \subseteq W_1^s=W_1$ such that 
$\delta(W_1^{i+1})\subseteq AW_1^iAW_1^iA$, for all $i\in [0,s-1]$.
\end{enumerate}
\end{definition}

\begin{definition}\label{D: weak ditalg interlaced an ideal}
 Let ${\cal A}=(T,\delta)$ be a weak ditalgebra, make $A=T_0$ and $V=T_1$. 
 Assume that $I$ is an ideal of  $A$, satisfying that $\delta(I)\subseteq IV+VI$. 
 Then, the weak ditalgebra ${\cal A}$ is called  \emph{interlaced with the ideal $I$} 
iff 
 \begin{enumerate}
  \item $\delta^2(A)\subseteq IV^2+VIV+V^2I$, and
  \item $\delta^2(V)\subseteq IV^3+VIV^2+V^2IV+V^3I$.
 \end{enumerate}
 The pair $({\cal A},I)$ is called then an \emph{interlaced weak 
ditalgebra}.
\end{definition}

\begin{definition}\label{D: triangular ideal and balanced ideal}
 Let ${\cal A}=(T,\delta)$ be a weak ditalgebra, make $A=T_0$ and $V=T_1$, and assume that 
$I$ is an ideal of $A$. Then, 
 we say that $I$ is an  
 \emph{${\cal A}$-triangular ideal of $A$} iff there is a sequence of 
$k$-subspaces
 $$0=H_0\subseteq H_1\subseteq\cdots\subseteq H_t=I$$
 such that
 $\delta(H_i)\subseteq AH_{i-1}V+VH_{i-1}A, \hbox{ for all } i\in [1,t].$
\end{definition}

\begin{definition}\label{D: triangular layered weak ditalg}
A pair $({\cal A},I)$ is called a \emph{triangular interlaced  weak ditalgebra} if 
${\cal A}=(T,\delta)$ is a weak ditalgebra interlaced with the ideal $I$ of $T_0$, where ${\cal A}$ has a triangular layer 
and $I$ is an ${\cal A}$-triangular ideal. 
\end{definition}

The importance of these triangular interlaced weak ditalgebras $({\cal A},I)$ is that the reduction procedures introduced by the Kiev school of representation theory can be applied to the study of their module  categories $({\cal A},I)\g\Mod$, see \cite{bpsqh}.
 In order to show that the pair $({\cal A}(\Delta),I)$ constructed in \S4, see (\ref{D: def del bocs A(Delta)}), is a triangular interlaced  weak ditalgebra, we need some preliminary facts.
 
 \begin{remark}\label{R: special basis required for triangularity}
  We will use a basis $\hueca{B}$ of $B$ as constructed in (\ref{N: la base de B}), but we refine the choice of the basis of the vector space $J$ as follows. Sine $J$ is nilpotent, we have a filtration 
  $0=J^{\ell_0}\subset \cdots\subset J^2\subset J$, and an $S$-$S$-bimodule decomposition $J=L_1\oplus L_2\oplus \cdots\oplus L_{\ell_0-1}$, such that 
$J^u=J^{(u+1)}\oplus L_u$, for all $u\in [1,\ell_0-1]$. Then, we can choose directed $k$-vector space basis $\hueca{J}_u(i,j)\subseteq e_jL_ue_i$, for $i,j\in {\cal P}$ and $u\in [1,\ell_0-1]$. Then, we make $\hueca{J}_u=\bigcup_{i,j}\hueca{J}_u(i,j)$ and $\hueca{J}=\bigcup_u\hueca{J}_u$. For the construction of $\hueca{B}$ in (\ref{N: la base de B}), we use this special choice for the basis of $J$. 

As usual,  the \emph{depth $\nu_1(x)$ of an element $x\in J$} is the number such that $x\in J^{\nu_1(x)}\setminus J^{\nu_1(x)+1}$.   Here, the set of basic elements in $x\in \hueca{J}$ with $\nu_1(x)=u$ is precisely $\hueca{J}_u$. Thus, for $x\in \hueca{J}$, we have that $x^*(L_v)=0$, for all $v\not=\nu_1(x)$. 
 If $x\in \hueca{J}$, we make $\nu_1(x^*)=\nu_1(x)$.

We will also refine the choice of the basis $\hueca{B}_0$ as follows. For each $t\in \hueca{N}$, consider the $S$-$S$-subbimodule $U_t:=\sum_{r+s=t}J^rB_0J^s$ of $B_0$. Then, we have an $S$-$S$-bimodule filtration 
$0=U_{\ell_1}\subset \cdots\subset U_1\subset U_0=B_0$, and an $S$-$S$-bimodule decomposition 
$B_0=G_0\oplus G_1\oplus\cdots \oplus G_{\ell_1-1}$, with $U_v=U_{v+1}\oplus G_v$, for all $v\in [1,\ell_1-1]$. 
As before,  we  choose directed $k$-vector space basis $\hueca{G}_v(i,j)\subseteq e_jG_ve_i$, for $i,j\in {\cal P}$ 
and $v\in [1,\ell_1-1]$. Then, we make $\hueca{G}_v=\bigcup_{i,j}\hueca{G}_v(i,j)$ and $\hueca{B}_0=\bigcup_v\hueca{G}_v$. For the construction of $\hueca{B}$ in (\ref{N: la base de B}), we use this special choice for the basis of $B_0$.

As before,  the \emph{depth $\nu_0(x)$ of an element $x\in B_0$} is the number such that $x\in U_{\nu_0(x)}\setminus U_{\nu_0(x)+1}$. Thus,  for $x\in \hueca{B}_0$, we have that $x^*(G_v)=0$ if $v\not=\nu_0(x)$. Also,   for $x\in \hueca{B}_0$, we make $\nu_0(x^*)=\nu_0(x)$.
 \end{remark}

\begin{remark}\label{R: la bigrafica de (cal A)(Delta)}
 Notice that the pair $(S,W)$, where $W_0=\hat{D}(B)_0$ and $W_1=\hat{D}(J)$, is a layer for  the weak ditalgebra ${\cal A}(\Delta)=(T(\Delta),\delta)$, see (\ref{D: def del bocs A(Delta)}). 

Let us consider \emph{the bigraph ${\cal B}$ of the layered ditalgebra ${\cal A}(\Delta)$}, 
it has set of points ${\cal P}$; given $i,j\in {\cal P}$,  
the set of \emph{solid arrows} from $i$ to $j$ is $\hueca{B}^*_0(i,j)$ and 
the set of \emph{dashed arrows} from $i$ to $j$ is $\hueca{J}^*(i,j)$. 
The starting point of any path $\gamma$ is denoted by $s(\gamma)$ and its terminal point by $t(\gamma)$. A path $\gamma$ is called a \emph{cycle} iff $s(\gamma)=t(\gamma)$. A cycle $\gamma$ is called a \emph{loop} iff $\gamma$ has only one arrow. 

Notice that the graded tensor algebra $T(\Delta)=T_S(W_0\oplus W_1)$ can be identified with the \emph{path algebra $k({\cal B})$  of 
the bigraph ${\cal B}$}, whose underlying vector space has as basis the set of 
paths (of any kind of arrows) of ${\cal B}$ (including one trivial path for 
each point $i\in {\cal P}$) and the product is induced by the concatenation of paths. Each idempotent $e_i$  of $S$ is 
identified with the trivial path at the point $i$. 
The homogeneous elements $h\in k({\cal B})_u$ of 
degree 
$u$ are the linear combinations of paths containing exactly $u$ dashed arrows.

We visualize the products of basic directed elements of $\hueca{B}^*_0\cup \hueca{J}^*$ in $T(\Delta)$ as paths of the bigraph ${\cal B}$. 
An important observation is the following. 
\begin{enumerate}
 \item If $\alpha:i\rightarrow j$ is solid arrow of ${\cal B}$, 
 we have $\overline{i}<\overline{j}$ in the quotient poset $\overline{\cal P}$.
 \item If $\xi:i\rightarrow j$ is a dashed arrow  of ${\cal B}$, 
 we have $\overline{i}\leq\overline{j}$ in the quotient poset $\overline{\cal P}$. 
\end{enumerate}
\end{remark}

\begin{lemma}\label{L: precycles}
A path $\gamma$ from $i$ to $j$ in the bigraph ${\cal B}$ is called a \emph{precycle} iff $i\sim j$. A precycle of length 1 is called a \emph{preloop}. Then, any precycle $\gamma$ of length $\geq 1$ consists of dashed arrows only.
\end{lemma}

\begin{proof} Assume that $\gamma=\alpha_n\alpha_{n-1}\cdots\alpha_2\alpha_1$ is a precycle from $i$ to $j$, where $\alpha_1,\ldots,\alpha_n$ are (solid or dashed) arrows of ${\cal B}$, with $t(\alpha_u)=i_u$, for $u\in [1,n]$. Then, by the last observation in (\ref{R: la bigrafica de (cal A)(Delta)}), we have 
$\overline{i}\leq \overline{i_1}\leq \cdots 
\leq \overline{i_n}=\overline{j}$ with $\overline{i}=\overline{j}$, thus $\overline{i}=\overline{i_1}=\cdots= \overline{i_n}$. Then all the arrows $\alpha_1,\ldots,\alpha_n$ must be dashed arrows.  
\end{proof}

In the following we use the height maps $h:X\rightmap{}\hueca{N}\cup \{0\}$ associated to various  finite posets $(X,\leq)$, as defined in (\ref{R: def de height map}). Recall that every height map $h$ satisfies: $y<x$ implies $h(y)<h(x)$, for all $x,y\in X$. 

\begin{lemma}\label{L map d det by height of quotient poset}
Consider the quotient poset $\overline{\cal P}$ associated to ${\cal P}$, and its height map $h:\overline{\cal P}\rightmap{}\hueca{N}\cup \{0\}$.
Consider the set of paths $\Gamma$ of ${\cal B}$ and the map $d:\Gamma\rightmap{}\hueca{N}\cup\{0\}$ defined, for each path $\gamma$ in ${\cal B}$, as $d(\gamma)=h(\overline{t(\gamma)})- h(\overline{s(\gamma)})$.
We have the following. 
\begin{enumerate}
 \item If $\alpha$ is a solid arrow, then $d(\alpha)>0$.
 \item If $\xi$ is a dashed arrow, then $d(\xi)\geq 0$ and ($d(\xi)=0$ iff $\xi$ is a preloop).
\item Whenever $\gamma=\gamma'\gamma''$ is a 
factorization of the path $\gamma$ as a product of non-trivial subpaths, where  
exactly one of them is a precycle, we have 
$$d(\gamma)=d(\eta), 
\hbox{ where } \eta\in \{\gamma',\gamma''\} \hbox{ is not a precycle.} $$
 \item  Whenever $\gamma=\gamma'\gamma''\gamma'''$ is a 
factorization of the path $\gamma$ as a product of non-trivial subpaths, where  
at most one of them is a precycle, we have
$$d(\gamma)>d(\eta), 
\hbox{ for each } \eta\in \{\gamma',\gamma'',\gamma'''\} \hbox{ not a precycle.} $$
\end{enumerate}
\end{lemma}

\begin{proof} (1), (2) and (3) are clear. In (4), we have three cases. Assume, for instance, that $\gamma'''$ is the only precycle. Then, we have 
$$\overline{s(\gamma)}=
\overline{s(\gamma''')}=
\overline{t(\gamma''')}=
 \overline{s(\gamma'')}<
 \overline{t(\gamma'')}=
 \overline{s(\gamma')}<\overline{t(\gamma')}=
 \overline{t(\gamma)},$$
 thus $d(\gamma')<d(\gamma)$ and $d(\gamma'')<d(\gamma)$. The other cases are similar. 
\end{proof}

\begin{proposition}\label{P: triangularidad de I}
The pair $({\cal A}(\Delta),I)$ constructed in \S4, see (\ref{D: def del bocs A(Delta)}), is a triangular interlaced  weak ditalgebra.
\end{proposition}

\begin{proof} We adopt the preceding notations.

\medskip
\noindent\emph{Step 1: Description of the triangular filtration of $W_0$, see (\ref{D: weak ditalgebra})(1).}
\medskip

Consider the partial order on the set ${\cal P}(\hueca{B}^*_0)=\{(d(\alpha),\nu_0(\alpha))\mid \alpha\in \hueca{B}^*_0\}$ 
defined by $(i,j)<(i',j')$ iff $i<i'$ or (if $i=i'$, then $j<j'$), and its corresponding  height map  $h_0:{\cal P}(\hueca{B}^*_0)\rightmap{}\hueca{N}\cup \{0\}$. Write 
$h_0(\alpha):=h_0(d(\alpha),\nu_0(\alpha))$, for any $\alpha\in \hueca{B}_0^*$. Thus, $(d(\alpha),\nu_0(\alpha)) <(d(\beta),\nu_0(\beta))$ implies that $h_0(\alpha)<h_0(\beta)$.

Consider the filtration 
$0=W_0^0\subseteq W_0^1\subseteq \cdots \subseteq W_0^r=W_0$,  
where $W_0^u$ is the $S$-$S$-subbimodule of $W_0$ generated 
by $\{ \alpha\in \hueca{B}_0^*\mid h_0(\alpha)< u\}$. 

Assume that we have a solid arrow $\alpha\in W_0^u$, thus $h_0(\alpha)< u$ and $u\geq 1$. We want to show that $\delta(\alpha)\in  A_{u-1}\hat{D}(J)A_{u-1}$. We can write 
$$\delta(\alpha)=\delta(\alpha)_2+\delta(\alpha)_3+\cdots+\delta(\alpha)_w,$$
for some $w\geq 2$ and each $\delta(\alpha)_v$ is a linear combination of paths of length $v$ and degree 1.  
Recall that $\alpha$ is a solid arrow, so it can not be a cycle. 

Then, if $\delta(\alpha)_v\not=0$, for $v\in [3,w]$, there is a path (non-cycle)  
$\gamma_2\xi\gamma_1$ in its expression as linear combination of paths, with non-zero coefficient, where $\gamma_2$ and $\gamma_1$ are paths composed by solid arrows and $\xi$ is a dashed arrow. Here, from (\ref{L: coeficientes de delta(x*)}), we know that one of $\gamma_1$ or $\gamma_2$ must be non-trivial. Then, (\ref{L map d det by height of quotient poset}) implies that whenever $\gamma_1$ is not trivial (resp. $\gamma_2$ non-trivial) we have $d(\gamma_1)<d(\gamma_2\xi\gamma_1)=d(\alpha)$ (resp.  $d(\gamma_2)<d(\gamma_2\xi\gamma_1)=d(\alpha)$), then, in any case, the arrows $\beta$ composing $\gamma_1$ or $\gamma_2$ satisfy $d(\beta)<d(\alpha)$, so $h_0(\beta)<h_0(\alpha)$, which implies that $\gamma_2\xi\gamma_1\in A_{u-1}\hat{D}(J)A_{u-1}$. So $\delta(\alpha)_v\in A_{u-1}\hat{D}(J)A_{u-1}$. 

It remains to consider the case $v=2$. If $\delta(\alpha)_2\not=0$, there is a path of the form $\xi\beta$ (or $\beta\xi$), where $\beta$ is a solid arrow and $\xi$ is a dashed arrow, appearing in the expression of $\delta(\alpha)_2$ 
with non-zero coefficient. If $\xi$ is not a preloop, we proceed as before  to 
obtain that 
$\beta\xi\in A_{u-1}\hat{D}(J)A_{u-1}$ (or $\xi\beta\in A_{u-1}\hat{D}(J)A_{u-1}$). Assume that $\xi$ is a preloop and that the path $\xi\beta$ has appeared in $\delta(\alpha)_2$. Here, we have $d(\beta)=d(\alpha)$. Let us examine its coefficient $c^{\alpha}_{\xi\otimes\beta}$ in the explicit description of $\delta(\alpha)$ given in (\ref{L: coeficientes de delta(x*)}).
We have $\alpha=x^*$, $\beta=y^*$ and $\xi=r^*$, where $x,y\in \hueca{B}_0$ and $r\in \hueca{J}$, then 
$$0\not=e_{t(\alpha)}c^{\alpha}_{\xi\otimes \beta}=x^*[b_2(r\otimes y)]=x^*[\sigma(\sigma^{-1}(r)\sigma^{-1}(y))]=x^*(ry).$$ 
We know that $r\in \hueca{J}$, thus $r\in \hueca{J}_s\subseteq J^s$, for some $s\in [1,\ell_0-1]$. Thus, we get $ry\in U_{s+\nu_0(y)}=\bigoplus_{l\geq s+\nu_0(y)}G_l$. Since $x^*(ry)\not=0$, we have $x^*(U_{s+\nu_0(y)})\not=0$, thus  $x^*(G_l)\not=0$, for some $l\geq s+\nu_0(y)$. 
This implies that 
$\nu_0(\beta)=\nu_0(y)<s+\nu_0(y)\leq l=\nu_0(x)=\nu_0(\alpha)$,
so $h_0(\beta)<h_0(\alpha)$ and $\xi\beta\in A_{u-1}\hat{D}(J)A_{u-1}$. 
The case where $\beta\xi$ appears in the expression of $\delta(\alpha)$ with non-zero coefficient is treated similarly. 
Then, we get $\delta(\alpha)_2\in A_{u-1}\hat{D}(J)A_{u-1}$. Adding up, we get that $\delta(\alpha)\in A_{u-1}\hat{D}(J)A_{u-1}$, and we have the triangularity condition for $W_0$.

\medskip
\noindent\emph{Step 2: Description of the triangular filtration of $W_1$, see (\ref{D: weak ditalgebra})(2).}
\medskip

Consider the partial order on the set ${\cal P}(\hueca{J}^*)=\{(d(\xi),\nu_1(\xi))\mid \xi\in \hueca{J}^*\}$ defined by $(i,j)<(i',j')$ iff $i<i'$ or (if $i=i'$, then $j<j'$).  The poset 
${\cal P}(\hueca{J}^*)$ admits a height map $h_1:{\cal P}(\hueca{J}^*)\rightmap{}\hueca{N}\cup \{0\}$. Write $h_1(\xi)=h_1(d(\xi),\nu_1(\xi))$, for any $\xi\in \hueca{J}^*$. Thus, as before, $(d(\xi),\nu_1(\xi))<(d(\zeta),\nu_1(\zeta))$ implies that $h_1(\xi)<h_1(\zeta)$. 

Consider the filtration 
$0=W_1^0\subseteq W_1^1\subseteq \cdots \subseteq W_1^s=W_1$ 
where $W_1^u$ is the $S$-$S$-subbimodule of $W_1$ generated 
by $\{ \xi\in \hueca{J}^*\mid h_1(\xi)< u\}$. 

Assume that we have a dashed arrow $\xi\in W_1^u$, thus $h_1(\xi)< u$ and $u\geq 1$. We want to show that $\delta(\xi)\in  AW_1^{u-1}AW_1^{u-1}A$. We can write 
$$\delta(\xi)=\delta(\xi)_2+\delta(\xi)_3+\cdots+\delta(\xi)_w,$$
for some $w\geq 2$ and each $\delta(\xi)_v$ is a linear combination of paths of length $v$ and degree 2. 
Now, the dashed arrow  $\xi$ may be a loop.  

\medskip
\noindent\emph{Case 1: $\xi$ is not a preloop (or, equivalently, $d(\xi)>0$).}
\medskip 

If $\delta(\xi)_v\not=0$, for $v\in [3,w]$, there is a path 
$\gamma_3\xi_2\gamma_2\xi_1\gamma_1$ in its expression as linear combination of paths, with non-zero coefficient, where $\gamma_1$, $\gamma_2$ and $\gamma_3$ are paths composed by solid arrows and $\xi_1$ and $\xi_2$ are dashed arrows. Moreover, at least one of the paths $\gamma_1$, $\gamma_2$ or $\gamma_3$ is not trivial. 

In each possible case, using (\ref{L map d det by height of quotient poset}), we can show that the dashed arrows  $\xi_1$ and $\xi_2$ satisfy $d(\xi_1),d(\xi_2)<d(\xi)$, so $h_1(\xi_1),h_1(\xi_2)<h_1(\xi)$, which implies that $\gamma_3\xi_2\gamma_2\xi_1\gamma_1\in AW_1^{u-1}AW_1^{u-1}A$.
Thus, we get $\delta(\xi)_v\in AW_1^{u-1}AW_1^{u-1}A$.

It remains to consider the case $v=2$. If $\delta(\xi)_2\not=0$, there is a path of the form $\xi_2\xi_1$, where $\xi_1$ and $\xi_2$ are  dashed arrows, appearing in the expression of $\delta(\xi)_2$ 
with non-zero coefficient. If both $\xi_1$ and $\xi_2$ are not preloops, 
we proceed as before, applying (\ref{L map d det by height of quotient poset}) to $\xi_2\xi_1$, to 
obtain that 
$\xi_2\xi_1\in W_1^{u-1}AW_1^{u-1}$. 
If some of $\xi_1$ or $\xi_2$
is a preloop, so $d(\xi)=d(\xi_1\xi_2)$, we can examine its coefficient $c^{\xi}_{\xi_2\otimes\xi_1}$ in the explicit description of $\delta(\xi)$ given in (\ref{L: coeficientes de delta(x*)}) and deduce, as before, that $\nu_1(\xi_2),\nu_1(\xi_1)<\nu_1(\xi)$. Indeed, we have $\xi=x^*$, $\xi_1=x_1^*$ and $\xi_2=x_2^*$, for $x\in \hueca{J}$, $x_1\in \hueca{J}_{s_1}$ and $x_2\in \hueca{J}_{s_2}$, for some $s_1,s_2\in [1,\ell_0-1]$, and $x^*(x_1x_2)\not=0$.
Thus, $x_1x_2\in J^{s_1+s_2}=\bigoplus_{l\geq s_1+s_2}L_l$, and $x^*(L_l)\not=0$, for some $l\geq s_1+s_2$. Therefore, we get 
$\nu_1(\xi_1),\nu_1(\xi_2)<\nu_1(\xi_1)+\nu_1(\xi_2)=\nu_1(x_1)+\nu_1(x_2)=s_1+s_2\leq l=\nu_1(x)=\nu_1(\xi)$, as claimed. 
 So  $\xi_2\xi_1\in W_1^{u-1}AW_1^{u-1}$. 
Then, we get $\delta(\xi)_2\in W_1^{u-1}AW_1^{u-1}$. From this, we 
 we obtain that $\delta(\xi)\in AW_1^{u-1}AW_1^{u-1}A$.

\medskip
\noindent\emph{Case 2: $\xi$ is a preloop (or, equivalently, $d(\xi)=0$).}
\medskip 
 
 If $\xi$ is a preloop, since the only possible precycles in the bigraph ${\cal B}$ are composed by dashed arrows,
 we obtain that $\delta(\xi)$ is a finite sum of elements of the form $c_{\xi_1\otimes\xi_2}^\xi\xi_2\xi_1$, 
 for some dashed preloops $\xi_1$ and $\xi_2$, and 
 $c_{\xi_1\otimes\xi_2}^\xi\in k$. Here again, examining the coefficient 
  $c^\xi_{\xi_1\otimes\xi_2}$ in case it is not zero, we obtain that $\delta(\xi)\in AW_1^{u-1}AW_1^{u-1}A$.

\medskip
\noindent\emph{Step 3: Description of the triangular filtration of $I$, as in (\ref{D: triangular ideal and balanced ideal})}
\medskip

Consider the quotient poset $\overline{\cal P}$ associated to the preordered set ${\cal P}$ and define the 
following 
relation 
on ${\cal S}:=\overline{\cal P}\times \overline{\cal P}$. We make  $(\overline{i},\overline{j})<(\overline{i'},\overline{j'})$ iff 
$\overline{i'}< \overline{i}$ and [if $\overline{i}=\overline{i'}$, then $\overline{j}< \overline{j'}]$. 
With this relation, the set  ${\cal S}$ becomes a poset. Consider the corresponding 
\emph{height map} $\hat{h}:{\cal  S}\rightmap{}\hueca{N}\cup \{0\}$. 
 Then, given $i,j\in {\cal P}$ and $\eta\in e_jIe_i$, we write $\hat{h}(\eta):=\hat{h}(\overline{i},\overline{j})$. 

Consider the map $\phi:I\rightmap{}\hueca{N}\cup \{0\}$ which associates to an element $\eta \in I$ the minimal length of the paths which appear with non-zero coefficient in the expression of $\eta$ as a linear combination of paths formed by solid arrows of the bigraph ${\cal B}$.  
 Finally, define the map $p:I\rightmap{}\hueca{N}\cup\{0\}$ by the following rule, given $\eta\in I$ with $\phi(\eta)=n$, 
$$p(\eta)=\max\{ h_0(\alpha_{n})+\cdots+h_0(\alpha_{1})\mid   \alpha_{n}\cdots\alpha_{2}\alpha_{1}\in \Gamma(\eta)\},$$
where the elements $\alpha_{n}\cdots\alpha_{2}\alpha_{1}$ run in the set of paths $\Gamma(\eta)$ of length $n$ which appear in the expression of $\eta$ as a linear combination of paths composed by solid arrows with non-zero coefficient. 

Notice that, since the lengths of the paths of solid arrows in ${\cal B}$ 
are bounded, there are bounds for the numbers 
$\hat{h}(\eta)$, $\phi(\eta)$ and $p(\eta)$, for  $\eta\in I$. Consider the finite set 
${\cal P}(I)=\{(\hat{h}(\eta),\phi(\eta),p(\eta))\mid \eta\in I \hbox{ directed} \}$ and the partial order relation on it defined by $(r,s,t)<(r',s',t')$ iff $r<r'$
or [$r=r'$ and [$s>s'$ or   
[$s=s'$ and $t<t']]$.  Then, consider the associated height map $\h:{\cal P}(I)\rightmap{}\hueca{N}\cup\{0\}$, and make $\h(\eta)=\h(\hat{h}(\eta),\phi(\eta),p(\eta))$, for any directed element $\eta\in I$. For each $u\geq 0$, consider the vector subspace $H_u$ of $I$ generated by the directed elements $\eta\in I$ with $\h(\eta)< u$. 
We have a finite vector space filtration of $I$ of the form
$$0= H_0\subseteq H_1\subseteq\cdots\subseteq H_{\ell_I}=I.$$
Take $u\in [1,\ell_I]$, we have to verify that $\delta(H_u)\subseteq
AH_{u-1}V+VH_{u-1}A$. It will be enough to verify that $\delta(\eta)\in AH_{u-1}V+VH_{u-1}A$, for any directed element $\eta\in I$ with $\h(\eta)< u$. Suppose that $0\not=\eta\in e_jIe_i$. The fact that there are no precycles in ${\cal B}$ formed by solid arrows implies that $i\not\sim j$. By (\ref{P: I entrelazado}), we have $\delta(I)\subseteq IV+VI$, hence 
$\delta(\eta)\in e_j(IV+VI)e_i$. Since $IV+VI=IV\oplus VI$, we have the following direct sum decomposition for $ e_j(IV+VI)e_i$
$$ [\bigoplus_{\scriptsize\begin{matrix}r\in {\cal P}\\ \overline{r}\not=\overline{i}\end{matrix}}e_jIe_rVe_i]\oplus 
  [\bigoplus_{\scriptsize\begin{matrix}r\in {\cal P}\\ \overline{j}\not=\overline{r}\end{matrix}}e_jVe_rIe_i]\oplus 
  [\bigoplus_{\scriptsize\begin{matrix}r\in {\cal P}\\ \overline{r}=\overline{i}\end{matrix}}e_jIe_rVe_i]\oplus 
  [\bigoplus_{\scriptsize\begin{matrix}r\in {\cal P}\\ \overline{j}=\overline{r}\end{matrix}}e_jVe_rIe_i].$$

If the component of $\delta(\eta)$ in the direct summand  $e_jIe_rVe_i$, with $\overline{r}\not=\overline{i}$, is of the form $\sum_t\eta_t\zeta_t$, where  $0\not=\zeta_t\in e_rVe_i$ and $\eta_t\in e_jIe_r$. Then, there is a (non-precycle) path in ${\cal B}$ from $i$ to $r$, so $\overline{i}<\overline{r}$. This implies $(\overline{r},\overline{j})<(\overline{i},\overline{j})$, and $\hat{h}(\overline{r},\overline{j})<\hat{h}(\overline{i},\overline{j})$, hence $\hat{h}(\eta_t)<\hat{h}(\eta)$, and $\h(\eta_t)<\h(\eta)< u$. Thus, this component of $\delta(\eta)$ belongs to $H_{u-1}V$.  Similarly, we can show that the component of $\delta(\eta)$ in $e_jVe_rIe_i$ with $\overline{r}\not= \overline{j}$ belongs to $VH_{u-1}$.

Let us show the component of $\delta(\eta)$   in the summand 
  $e_jVe_rIe_i$
  of the preceding decomposition of $e_j(IV+VI)e_i$, with $\overline{r}=\overline{j}$, belongs to $VH_{u-1}$. For this, the following maps will be useful. 
We have the finite-dimensional graded $S$-$S$-bimodule $W$, and its graded right dual bimodule $\check{D}(W)$, and similarly, we have the right graded dual $\check{D}(W_1)$. The directed basis $\hueca{J}$ of $J$, determines a dual basis $\hueca{J}^*$ of $W_1$, and this $\hueca{J}^*$ determines a dual basis of $\check{D}(W_1)$. Thus, any $\xi\in \hueca{J}^*$ determines its dual map $\xi^{\ostar}:W_1\rightmap{}S$, defined by $\xi^{\ostar}(\zeta)=\delta_{\zeta,\xi}e_{s(\zeta)}$, for $\zeta\in \hueca{J}^*$, which can also be considered as an element $\xi^{\ostar}\in \check{D}(W)$ with $\xi^\ostar(W_0)=0$.  Consider  the linear map $\lambda_{\xi}:k({\cal B})\rightmap{}k({\cal B})$, which is  defined on the basis of $k({\cal B})$ formed by the paths of ${\cal B}$ by 
$$\lambda_{\xi}(\alpha_n\alpha_{n-1}\cdots\alpha_1)=\xi^{\ostar}(\alpha_n)\alpha_{n-1}\cdots\alpha_2\alpha_1.$$
Thus,  we get $\xi\lambda_\xi(\xi\alpha_{n-1}\cdots\alpha_1)=\xi\alpha_{n-1}\cdots\alpha_1$, if $\alpha_n=\xi$, and  $\lambda_\xi(\alpha_n\cdots\alpha_1)=0$, whenever $\alpha_n\not=\xi$.
 We have 
 $$e_jVe_rIe_i=
\bigoplus_{\xi} \xi e_rIe_i,$$
where $\xi$ runs in the set of (dashed) preloops from $r$ to $j$ of the bigraph ${\cal B}$.  Let us examine the component, denoted by $\delta(\eta)_\xi$, of $\delta(\eta)$ in the direct summand $\xi e_rIe_i$, for a fixed dashed preloop $\xi$ from $r$  to $j$. 
Assume that $\delta(\eta)_\xi\not=0$ and notice that  
$\delta(\eta)_\xi=\xi\lambda_{\xi}(\delta(\eta))$, so the element $\lambda_{\xi}(\delta(\eta))\in e_rIe_i$ determines the component $\delta(\eta)_\xi$.

Make $m:=\phi(\eta)$ and $n:=\phi(\lambda_\xi(\delta(\eta)))$. So $n\geq m$, because, from (\ref{L: coeficientes de delta(x*)}), 
we have 
$n+1=\phi(\xi\lambda_\xi(\delta(\eta)))=\phi(\delta(\eta)_\xi)\geq \phi(\delta(\eta))>\phi(\eta)=m.$
We have  $\hat{h}(\lambda_{\xi}(\delta(h)))=\hat{h}(\overline{i},\overline{j})=\hat{h}(\eta)$, so if  $n>m$,  we get 
$$(\hat{h}(\lambda_{\xi}[\delta(\eta)]), \phi(\lambda_{\xi}[\delta(\eta)]),p(\lambda_{\xi}[\delta(\eta)]))<(\hat{h}(\eta),\phi(\eta), p(\eta)).$$
This implies that $\h(\lambda_{\xi}(\delta(\eta)))<\h(\eta)$, and $\delta(\eta)_\xi=\xi\lambda_{\xi}(\delta(\eta))\in VH_{u-1}$. 

Now, assume that $n=m$. We can write $\eta=\eta_{m}+\eta_{m+1}+\cdots $, where each 
$\eta_{m+t}$ is a linear combination of paths of length $m+t$. Moreover, we have 
$$\lambda_{\xi}(\delta(\eta))=\lambda_{\xi}(\delta(\eta_m))+\lambda_{\xi}(\delta(\eta_{m+1}))+\cdots$$
Then, we have  $\phi(\lambda_{\xi}(\delta(\eta_s)))=n$, for some $s\geq m$. Therefore, we have 
$n=\phi(\lambda_{\xi}(\delta(\eta_s)))\geq s\geq m,$ which implies that $n=s=m$. 
 
Write $\eta_n=\sum_{v} c_{v}\gamma_v$, where each $\gamma_v$ is a path formed by $n$ solid arrows and $c_v\in k$, thus $\delta(\eta_n)=\sum_vc_v\delta(\gamma_v)$. 
If $\gamma_v=\alpha^v_{n}\alpha^v_{{n-1}}\cdots\alpha^v_2\alpha^v_1$, with $\alpha_n^v,\ldots,\alpha^v_1$ solid arrows, then  
 $$\delta(\gamma_v) =\delta(\alpha^v_{n})\alpha^v_{{n-1}}\cdots\alpha_2^v\alpha^v_{1}
+
 \alpha^v_{n}\delta(\alpha^v_{{n-1}})\cdots\alpha_2^v\alpha^v_{1}
+\cdots+
 \alpha^v_{n}\alpha^v_{{n-1}}\cdots\alpha^v_{2}\delta(\alpha^v_{1}).$$
Notice that $\lambda_{\xi}(\delta(\gamma_v))=\lambda_{\xi}(\delta(\alpha^v_{n})\alpha^v_{{n-1}}\cdots\alpha^v_2\alpha^v_{1})=\lambda_{\xi}(\delta(\alpha^v_{n}))\alpha^v_{{n-1}}\cdots\alpha^v_2\alpha^v_{1}$, because  $\alpha^v_{n}$ is a solid arrow, so $\xi^{\ostar}(\alpha^v_{n})=0$. Now, write 
$\delta(\alpha^v_{n})=\sum_wc^v_w\zeta^v_w\xi^v_w\gamma_w^v$, where each  $\gamma_w^v$ and $\zeta_w^v$ are paths of degree 0, $\xi^v_w$ are dashed arrows, and $c^v_w\in k$.  
Then, $\lambda_{\xi}(\delta(\alpha_n^v))\not=0$ implies that some $\zeta_w^v$ is trivial, again because, otherwise, 
$\xi^{\ostar}$ evaluated at the last arrow of $\zeta_w^v$ would be zero. Then, we have $\lambda_{\xi}(\delta(\alpha_n^v))=\sum_wc_w^v\lambda_{\xi}(\xi_w^v\gamma_w^v)$, where now $w$ runs in the set of indices with $\zeta_w^v$ trivial, and therefore
$$\begin{matrix}
n=\phi[\lambda_{\xi}(\delta(\eta_n))]
&=&
\phi[\sum_vc_v\lambda_{\xi}(\delta(\gamma_v))]\hfill\\
&=&
\phi[\sum_vc_v\lambda_{\xi}(\delta(\alpha^v_n))\alpha^v_{n-1}\cdots\alpha^v_2\alpha^v_1)]\hfill\\
&=&
\phi[\sum_{v,w}c_vc^v_w\lambda_{\xi}(\xi_w^v\gamma_w^v\alpha^v_{n-1}\cdots\alpha^v_2\alpha^v_1)]\hfill\\
&=&
\phi[\sum_{v,w}c_vc^v_w\gamma_w^v\alpha^v_{n-1}\cdots\alpha^v_2\alpha^v_1].\hfill\\
\end{matrix}$$

Then, there is at least one path
 $\gamma_w^v$ which is a solid arrow $\beta_w^v$, and we have  
$\lambda_{\xi}(\delta(\eta_n))=\sum_{v,w}c_vc^v_w\beta_w^v\alpha^v_{n-1}\cdots\alpha^v_2\alpha^v_1$, where all the $\beta^v_w$ are solid arrows.
By the triangularity condition for $W_0$, proved in Step 1, if $h_0(\alpha_n^v)=t$, so $\alpha_n^v\in W_0^{t+1}$, 
we have $\delta(\alpha_n^v)\in A_{t}\hat{D}(J)A_{t}$, which implies that $\beta_w^v\in W_0^{t}$, so 
$h_0(\beta_w^v)< t=h_0(\alpha_n^v)$. 
Then, we obtain 
$$\begin{matrix}p[\lambda_{\xi}(\delta(\eta_n))]&=&\max\{h_0(\beta_w^v)+h_0(\alpha_{n-1}^v)+\cdots+h_0(\alpha_1^v)\mid v,w\}\hfill\\
 &<&
 \max\{h_0(\alpha_n^v)+h_0(\alpha_{n-1}^v)+\cdots+h_0(\alpha_1^v)\mid v\}
 =p(\eta).\hfill\\
  \end{matrix}$$ 
 Since $\phi[\lambda_\xi(\delta(\eta_l))]\geq l$, for all $l> n$, we have $\Gamma(\lambda_{\xi}(\delta(\eta)))=\Gamma(\lambda_{\xi}(\delta(\eta_n)))$. Therefore, we get 
  $p[\lambda_{\xi}(\delta(\eta))]=p[\lambda_{\xi}(\delta(\eta_n))]$, and this implies that $\h(\lambda_{\xi}[\delta(\eta)])<\h(\eta)$. 
 Then, also in this case, we get $\delta(\eta)_\xi=\xi\lambda_{\xi}(\delta(\eta))\in VH_{u-1}$.

We also have that the  component of  $\delta(\eta)$ in $e_jIe_rVe_i$ lies in $H_{u-1}V$. This can be  shown in a similar way. 
We consider the decomposition
$$e_jIe_rVe_i=\bigoplus_\xi e_jIe_r\xi,$$
where $\xi$    runs in the set of (dashed) preloops from $i$ to $r$ of the bigraph ${\cal B}$ and then examine each component  $\delta(\eta)_\xi$ in the summand $e_jIe_r\xi$. For this analysis, we can consider the left dual $S$-$S$-bimodule $\hat{D}(W_1)$ of the $S$-$S$-bimodule $W_1$ and a dual basis for the fixed basis $\hueca{J}$ of $W_1$. Thus, for a fixed basic element $\xi\in \hueca{J}^*$, we have  its dual basic element  $\xi^*:W_1\rightmap{}S$ is given by $\xi^*(\zeta)=\delta_{\xi,\zeta}e_{t(\zeta)}$, and the linear map 
$\rho_{\xi}:k({\cal B})\rightmap{}k({\cal B})$ defined on the basis of $k({\cal B})$ formed by the paths of ${\cal B}$ by 
$$\rho_{\xi}(\alpha_n\alpha_{n-1}\cdots\alpha_1)=\alpha_n\alpha_{n-1}\cdots\alpha_2\xi^*(\alpha_1).$$

Finally, since every component of $\delta(\eta)$ lies in 
$H_{u-1}V+VH_{u-1}$, so does $\delta(\eta)$. This finishes the proof.
\end{proof}

\section{The exact category $({\cal A},I)\g\Mod$}\label{S: exact structure for Roiter dits with ideal}\label{S: Exact cat}

The category of modules $\underline{\cal A}\g\Mod$ for an interlaced weak ditalgebra $\underline{\cal A}=({\cal A},I)$ is defined in general in a similar way than the category $({\cal A}(\Delta),I)\g\Mod$, see (\ref{D: la cat de mod's de cal A(Delta)}) and \cite{bpsqh}\S2.
In this section, we show that, for a Roiter interlaced weak ditalgebra $\underline{\cal A}=({\cal A},I)$, its category of modules $\underline{\cal A}\g\Mod$ admits an exact structure. We only have to adapt \cite{BSZ}\S6 to this context. We will apply this section to $({\cal A}(\Delta),I)$, but we give a more general presentation. Let us first recall some facts from \cite{bpsqh}\S5.

\begin{definition}\label{D: Roiter weak ditalgebra} 
A triangular interlaced weak ditalgebra $({\cal A},I)$, with layer $(R,W)$, is 
called a \emph{Roiter interlaced weak ditalgebra} iff the following 
property is satisfied: for any isomorphism $f^0$ of $R$-modules 
$f^0:M\rightmap{}N$ and any $f^1\in \Hom_{R\g R}(W_1,\Hom_k(M,N))$, if one of 
$M$ or $N$ has a structure of left $A/I$-module, then the other one admits also 
a structure of left $A/I$-module such that $(f^0,f^1)\in \Hom_{({\cal 
A},I)}(M,N)$.
\end{definition}

Notice that, whenever $({\cal A},I)$ is a triangular interlaced weak ditalgebra with layer $(R,W)$ such that $R$ is semisimple, from 
\cite{bpsqh}(5.3), we know that $({\cal A},I)$ is a Roiter  interlaced weak ditalgebra.

 \begin{lemma}\label{L: Ovsienko}  Assume that $({\cal A},I)$ is a Roiter interlaced weak ditalgebra
 with  layer $(R,W)$.
Suppose that $f=(f^0,f^1):M\rightmap{} N$ in $({\cal
A},I)\g\Mod$. Then, 
\begin{enumerate}
\item If $f^0$ is a retraction in $R\g\Mod$, there is a
morphism $h:M'\rightmap{} M$ in $({\cal A},I)\g\Mod$ with
$h^0$ isomorphism and $(fh)^1=0$.
\item If $f^0$ is a section in $R\g\Mod$, there is a
morphism $g:N\rightmap{} N'$ in $({\cal A},I)\g\Mod$ with
$g^0$ isomorphism and $(gf)^1=0$.
\end{enumerate}
\end{lemma}

\begin{proof} We only prove (1). Adopt the notation of (\ref{D: weak ditalgebra}) 
and notice that the proof follows by induction from the following.

\medskip
{\sl Claim:} Given $i\in[1,s]$, if $f^1(W_1^{i-1})=0$ then
there exists a morphism $h=(h^0,h^1):M'\rightmap{} M$ in $({\cal A},I)\g\Mod$ with
$h^0$ isomorphism and $(fh)^1(W_1^i)=0$.
\medskip

 {\sl Proof of the Claim:} Let $t^0\in \Hom_R(N,M)$ be a right inverse
for $f^0$ in $R\g\Mod$. Denote by $M'$ the underlying $R$-module
of $M$. Define, for each $w\in W_1$,  
$h^1(w):=-t^0f^1(w)\in\Hom_k(M',M)$. Then, $h^1\in \Hom_{R\g
R}(W_1,\Hom_k(M',M))$. Since $({\cal A},I)$ is a Roiter interlaced weak ditalgebra, we know that $M'$
admits a left $A/I$-module structure such that $h=(1_M,h^1)$ is a
morphism in $({\cal A},I)\g\Mod$. Since  $f^1(W_1^{i-1})=0$, then
$h^1(W_1^{i-1})=0$. Now, if $w\in W_1^i$, we obtain from the
triangularity condition that $\delta(w)=\sum_ja_jw_jb_jw'_jc_j$,
for some $a_j,b_j,c_j\in A$ and $w_j,w'_j\in W^{i-1}_1$. Then,
$(fh)^1(w)=f^0h^1(w)+f^1(w)1_M+\sum_ja_jf^1(w_j)b_jh^1(w'_j)c_j$.
Therefore $(fh)^1(w)=-f^0t^0f^1(w)+f^1(w)1_M=0$, and the claim  is proved.
\end{proof}

\begin{corollary}\label{L: f iso sii f^0 iso} Assume that  $({\cal A},I)$ is a  Roiter interlaced weak  ditalgebra with  layer $(R,W)$ and
consider any morphism $f=(f^0,f^1):M\rightmap{} N$ in 
$({\cal
A},I)\g\Mod$. Then, $f$ is an isomorphism  in $({\cal A},I)\g\Mod$ iff $f^0$ is an isomorphism in $R\g\Mod$.
\end{corollary}

\begin{proof} It is
clear that if $f$ is an isomorphism  in $({\cal A},I)\g\Mod$ then $f^0$
is an isomorphism in $R\g\Mod$. Conversely, assume that $f^0$ is an isomorphism in $R\g\Mod$. By (\ref{L: Ovsienko}), there is morphism 
$g=(g^0,g^1):N\rightmap{} N'$ with $(gf)^1=0$ and $g^0$ isomorphism in $R\g\Mod$. This implies that $g^0f^0:M\rightmap{}N'$ is an $A/I$-module isomorphism, hence  $gf=(g^0f^0,0)$ is an isomorphism in $({\cal A},I)\g\Mod$, and so $f$ is a section in $({\cal A},I)\g\Mod$. Dually, one shows that  $f$ is a retraction in  $({\cal A},I)\g\Mod$.
\end{proof}

\begin{proposition}\label{P: idempotents split}
   Let $({\cal A},I)$ be a Roiter interlaced weak ditalgebra, then 
idempotents split in $({\cal A},I)\g\Mod$.
   That is, for any idempotent $e\in \End_{({\cal A},I)}(M)$, there is an 
isomorphism $h:M_1\oplus M_2\rightmap{}M$ in 
   $({\cal A},I)\g\Mod$ such that 
   $$h^{-1}eh=\begin{pmatrix}
            1_{M_1}&0\\ 0&0\\
           \end{pmatrix}.$$
\end{proposition}

 \begin{lemma}\label{L: basico de sucs exactas} Let $({\cal A},I)$ be a 
 Roiter interlaced weak ditalgebra. Suppose that 
$M\rightmap{f}E\rightmap{g}N$
is a pair of composable morphisms in $({\cal A},I)\g\Mod$ with $gf=0$
and
$$0\rightmap{}M\rightmap{f^0}E\rightmap{g^0}N\rightmap{}0$$
is a split exact sequence of left $R$-modules. Then, there is an
isomorphism $h:E'\rightmap{}E$ in $({\cal A},I)\g\Mod$ such that
$(gh)^1=0$ and $(h^{-1}f)^1=0$.
\end{lemma}

\begin{proof} By  (\ref{L: Ovsienko}), we may assume that $f^1=0$. Once
again, we adopt the notation of (\ref{D: weak ditalgebra}). The proof will follow by
induction from the following.

\medskip
{\sl Claim:} Given $i\in[1,s]$, if $f^1=0$ and $g^1(W_1^{i-1})=0$ then
there exists an isomorphism $h:E'\rightmap{} E$ in $({\cal A},I)\g\Mod$
with $(h^{-1}f)^1=0$ and $(gh)^1(W_1^i)=0$.
\medskip

{\sl Proof of the Claim:} First follow the argument in the proof of
the claim in (\ref{L: Ovsienko}), where using a right inverse $t^0$ for $g^0$
we constructed an isomorphism $h=(1_E,h^1):E'\rightmap{} E$ in
$({\cal A},I)\g\Mod$, where $E'$ has the same underlying $R$-module
that $E$ has,
$h^1(w)=-t^0g^1(w)$, for $w\in W_1$, and $(gh)^1(W_1^i)=0$. We
will
show that $(h^{-1}f)^1=0$. For this, we first notice that
 $(gf)^1=0$ implies $0=(gf)^1(w)=g^1(w)f^0$, for any
$w\in W_1$. Thus, $h^1(w)f^0=0$, for any $w\in W_1$, too.
Since $f^1=0$, we know that $f^0\in\Hom_A(M,E)$. Then, if $a,b\in
A$, $w\in W_1$ and $m\in M$, we have $\left(h^1(awb)f^0\right)[m]=
\left(ah^1(w)b\right)\left[f^0[m]\right]=ah^1(w)\left[bf^0[m]\right]=
a\left(h^1(w)f^0[bm]\right)=0$. Thus, $h^1(v)f^0=0$ holds for any
$v\in AW_1A=V$.
Now,  if $u:=h^{-1}$, we get that $u^0=1_E$.
Moreover, if $w\in W_1$
and we write $\delta(w)=\sum_jv_jv'_j$, with $v_j,v'_j\in AW_1A$,
then $0=(uh)^1(w)=h^1(w)+u^1(w)+\sum_ju^1(v_j)h^1(v'_j)$.
Then,
$(uf)^1(w)=u^0f^1(w)+u^1(w)f^0+\sum_ju^1(v_j)f^1(v'_j)=u^1(w)f^0=
[-h^1(w)-\sum_ju^1(v_j)h^1(v'_j)]f^0=0$. Which implies that
$(uf)^1=0$, as wanted.\end{proof}

 \begin{definition}\label{D: exact structure} If ${\cal C}$ is an additive $k$-category where
idempotents split, a pair $(s,d)$ of composable morphisms
$M\rightmap{s}E\rightmap{d} N$ in ${\cal C}$ is called  {\sl
exact}\index{Exact pair of composable morphisms} iff $s$ is
the kernel of $d$ and $d$ is the cokernel of $s$.
A {\sl morphism of exact pairs}\index{Morphism!morphism of exact pairs} from
$M\rightmap{s}E\rightmap{d} N$ to $M'\rightmap{s'}E'\rightmap{d'}
N'$  is a triple of morphisms
$(u,v,w)$ in ${\cal C}$ making the
following diagram commutative
$$\begin{matrix}M&\rightmap{s}&E&\rightmap{d}& N\\ 
\lmapdown{u}&&\lmapdown{v}&&\lmapdown{w}\\ 
M'&\rightmap{s'}&E'&\rightmap{d'}& N'.\\ \end{matrix}$$
If $u$, $v$ and $w$ are isomorphisms,  $(u,v,w)$ is an {\sl isomorphism of exact pairs}\index{Isomorphism!isomorphism of exact pairs}.
If there is such an isomorphism with $u=1_M$ and $w=1_N$, the exact pairs
are called {\sl equivalent}\index{Equivalence of exact pairs}.

Once we have fixed a class of exact pairs $\cal E$ closed under isomorphisms,
if $(s,d)\in\cal E$, $s$ is called an {\sl inflation} and $d$ is called a
{\sl deflation}.
\end{definition}

 \begin{lemma}\label{L: pares exactos} Assume that $({\cal A},I)$ is a Roiter interlaced weak ditalgebra. Then,
  whenever  $$0\rightmap{}M\rightmap{f^0}E\rightmap{g^0}N\rightmap{}0$$
is an exact sequence in $A/I\g\Mod$, we have that  $M\rightmap{\ (f^0,0)\ }E\rightmap{\ (g^0,0)\ }N$ is an exact pair in $({\cal A},I)\g\Mod$. 
\end{lemma}

\begin{proof}  Write $f=(f^0,0)$ and $g=(g^0,0)$.
 We only prove that $f=\Ker g$ (the equality $g=\Cok f$
  is proved dually). So suppose that $Z\rightmap{t}E$
satisfies $gt=0$ in $({\cal A},I)\g\Mod$.
Then, $g^0t^0=0$ and $g^0t^1(v)=0$, for all $v\in V$. Then, there
exist a unique $h^0\in \Hom_R(Z,M)$ with $f^0h^0=t^0$ and, for each $v\in
V$, there exists a unique $h^1(v)\in \Hom_k(Z,M)$ with
$f^0h^1(v)=t^1(v)$. Since $f^0$ is an injective morphism of $A$-modules,
we get that $h^1\in \Hom_{A\g A}(V,\Hom_k(Z,M))$.
Moreover, for $a\in A$ and $z\in Z$,
$$f^0(ah^0(z))=at^0(z)=t^0(az)+t^1(\delta(a))(z)=f^0\left[h^0(az)+h^1(\delta(a))(z)\right].$$
Using again that $f^0$ is injective, we get that 
$h=(h^0,h^1)\in\Hom_{({\cal A},I)}(Z,M)$. Clearly, $h$ is unique such that $fh=t$ in $({\cal A},I)\g\Mod$.
\end{proof}

For the sake of completeness, we recall the following definition. 

\begin{definition}\label{D: de estruct exacta} Assume that ${\cal C}$ is an additive $k$-category  and  ${\cal E}$ is a class of exact pairs in ${\cal C}$
closed under isomorphisms.
Then, the class ${\cal E}$ is called \emph{an exact structure
on} ${\cal C}$ iff
\begin{enumerate}
\item[(1)] The composition of deflations is a deflation.
\item[(2)] For each morphism $f:Z'\rightmap{}Z$ and each deflation
$d:Y\rightmap{}Z$ there exists a morphism $f':Y'\rightmap{}Y$ and a
deflation $d':Y'\rightmap{}Z'$ such that $df'=fd'$.
\item[(3)] Identities are deflations. If $gf$ is a deflation, then
so is $g$.
\item[(3)$^{op}$] Identities are inflations. If $gf$ is an inflation, then
so is $f$.
\end{enumerate}
 An \emph{exact category} is a category ${\cal C}$, as above, endowed with an exact structure  ${\cal E}$. 
A functor $F:{\cal C}\rightmap{}{\cal C}'$, where the categories ${\cal C}$ and ${\cal C}'$ are endowed with exact structures ${\cal E}$ and ${\cal E}'$, respectively, is called  \emph{exact} if it maps ${\cal E}$ into ${\cal E}'$.
\end{definition}

\begin{definition}\label{D: S-acceptable y E_A,S}
Assume that $\underline{\cal A}=({\cal A},I)$ is a Roiter interlaced weak ditalgebra
with layer $(R,W)$.  Consider the class 
$\underline{\cal E}=\underline{\cal E}({\underline{\cal A}})$ of composable pairs of morphisms $M\rightmap{f}E\rightmap{g}N$ such that there is a commutative diagram in $\underline{\cal A}\g\Mod$
$$\begin{matrix}
M&\rightmap{f}&E&\rightmap{g}&N\\
\rmapdown{\cong}&&\rmapdown{\cong}&&\rmapdown{\cong}\\
M'&\rightmap{(\varphi^0,0)}&E'&\rightmap{(\psi^0,0)}&N',
  \end{matrix}$$
  such that $0\rightmap{}M'\rightmap{\varphi^0}E'\rightmap{\psi^0}N'\rightmap{}0$
  is a split exact sequence in $R\g\Mod$.
\end{definition}

We will establish that $\underline{\cal E}$ is an exact structure of 
$\underline{\cal A}\g\Mod$, 
for any Roiter interlaced weak ditalgebra $\underline{\cal A}$. The following statement gives a simple description of the class $\underline{\cal E}$.

\begin{lemma}\label{L: caract de conflaciones} Let  $\underline{\cal A}$ be a Roiter interlaced weak ditalgebra. Then, $\underline{\cal E}$ is a class of exact pairs in $\underline{\cal A}\g\Mod$ closed under isomorphisms. Moreover, if we  consider morphisms $f:M\rightmap{}E$ and $g:E\rightmap{}N$ in $\underline{\cal A}\g\Mod$. Then:
\begin{enumerate}
\item $M\rightmap{f}E\rightmap{g}N$ is a conflation iff $gf=0$ and 
$0\rightmap{}M\rightmap{f^0}E\rightmap{g^0}N\rightmap{}0$
is a split exact sequence in $R\g\Mod$.
\item $g:E\rightmap{}N$ is a deflation iff $g^0:E\rightmap{}N$ is a retraction in $R\g\Mod$.
\item $f:M\rightmap{}E$ is an inflation  iff $f^0:M\rightmap{}E$ is a section in $R\g\Mod$.
\end{enumerate}
\end{lemma}

\begin{proof} We will prove (1) and, at the same time, that 
$\underline{\cal E}$ is a class of exact pairs in $\underline{\cal A}\g\Mod$. If $M\rightmap{f}E\rightmap{g}N\in \underline{\cal E}$,  there is a commutative diagram in $\underline{\cal A}\g\Mod$
$$\begin{matrix}
M&\rightmap{f}&E&\rightmap{g}&N\\
\rmapdown{t}&&\rmapdown{s}&&\rmapdown{r}\\
M'&\rightmap{(\varphi^0,0)}&E'&\rightmap{(\psi^0,0)}&N',
  \end{matrix}$$
  such that $r,s,t$ are isomorphisms and  $0\rightmap{}M'\rightmap{\varphi^0}E'\rightmap{\psi^0}N'\rightmap{}0$
  is a split exact sequence in $R\g\Mod$.
  In particular,  we have the commutative diagram in $R$-Mod:
  $$\begin{matrix}
M&\rightmap{f^0}&E&\rightmap{g^0}&N\\
\rmapdown{t^0}&&\rmapdown{s^0}&&\rmapdown{r^0}\\
M'&\rightmap{\varphi^0}&E'&\rightmap{\psi^0}&N',
  \end{matrix}$$
  where $r^0,s^0,t^0$ are isomorphisms, which implies that
 $0\rightmap{}M\rightmap{f^0}E\rightmap{g^0}N\rightmap{}0$
is a split exact sequence in $R\g\Mod$.

Conversely, if $M\rightmap{f}E\rightmap{g}N$ is a pair of composable morphisms in $\underline{\cal A}\g\Mod$ such that $gf=0$ and  $$0\rightmap{}M\rightmap{f^0}E\rightmap{g^0}N\rightmap{}0$$
is a split exact sequence in $R\g\Mod$, then, by (\ref{L: Ovsienko})(1),   there is a commutative square in $\underline{\cal A}\g\Mod$
  $$\begin{matrix}
  E&\rightmap{g}&N\\
  \lmapdown{s}&&\rmapdown{r}\\
  E'&\rightmap{(\psi^0,0)}&N',\\
    \end{matrix}$$
    where $r,s$ are isomorphisms. It follows that $\psi^0$ is a retraction in $R\g\Mod$. Now, consider the kernel $\varphi^0:M'\rightmap{}E'$ of the morphism $\psi^0:E'\rightmap{}N'$ in $A/I\g\Mod$. Since $gf=0$, from (\ref{L: pares exactos}), we know that there is an induced morphism of $\underline{\cal A}$-modules  $t:M\rightmap{}M'$ such that the following diagram commutes:
   $$\begin{matrix}
M&\rightmap{f}&E&\rightmap{g}&N\\
\rmapdown{t}&&\rmapdown{s}&&\rmapdown{r}\\
M'&\rightmap{(\varphi^0,0)}&E'&\rightmap{(\psi^0,0)}&N',
  \end{matrix}$$ 
 where the lower row is an exact pair in $\underline{\cal A}\g\Mod$. In particular, we have a commutative diagram
   $$\begin{matrix}
M&\rightmap{f^0}&E&\rightmap{g^0}&N\\
\rmapdown{t^0}&&\rmapdown{s^0}&&\rmapdown{r^0}\\
M'&\rightmap{\varphi^0}&E'&\rightmap{\psi^0}&N',
  \end{matrix}$$ 
  which implies that $t^0$ is an isomorphism. Since $\underline{\cal A}$ is a Roiter interlaced weak ditalgebra, $t$ is an isomorphism too and 
 $M\rightmap{f}E\rightmap{g}N$ is an exact pair of $\underline{\cal A}\g\Mod$, which belongs to $\underline{\cal E}$.
 
 (2):  If $g:E\rightmap{}N$ is a deflation, clearly, we get that $g^0$ is a retraction in $R\g \Mod$.  On the other hand, if we start from a morphism $g:E\rightmap{}N$ in $\underline{\cal A}\g\Mod$ such that $g^0:E\rightmap{}N$ is a retraction in $R\g\Mod$, since $\underline{\cal A}$ is a Roiter interlaced weak ditalgebra, by (\ref{L: Ovsienko})(1), 
  there is a commutative square in $\underline{\cal A}\g\Mod$
  $$\begin{matrix}
  E&\rightmap{g}&N\\
  \lmapdown{s}&&\rmapdown{r}\\
  E'&\rightmap{(\psi^0,0)}&N',\\
    \end{matrix}$$
    where $r,s$ are  isomorphisms. Then, as before, we consider the kernel $\varphi^0$ of $\psi^0$ in $A/I\g\Mod$ and construct a commutative square in $\underline{\cal A}\g\Mod$ 
     $$\begin{matrix}
M&\rightmap{f}&E&\rightmap{g}&N\\
\rmapdown{t}&&\rmapdown{s}&&\rmapdown{r}\\
M'&\rightmap{(\varphi^0,0)}&E'&\rightmap{(\psi^0,0)}&N',
  \end{matrix}$$ 
    where $t$ is an isomorphism, so $g$ is a deflation.
     The proof of (3) is dual to the proof of (2).
\end{proof}

\begin{proposition}\label{P: E_A,S  estruct exacta}
If $\underline{\cal A}$ is a Roiter interlaced weak ditalgebra, then 
$\underline{\cal E}$ is an exact structure on $\underline{\cal A}\g\Mod$.
\end{proposition}

\begin{proof} From (\ref{L: caract de conflaciones}), items (1), (3) and (3)$^{op}$ of  Definition (\ref{D: exact structure}) are clear. Let us prove the remaining one (2). Take a
 morphism $f:Z'\rightmap{}Z$ and any deflation $d:Y\rightmap{}Z$.
 Consider the morphism $(f,d):Z'\bigoplus Y\rightmap{}Z$ in $({\cal A},I)\g\Mod$. Since $d$ is a deflation, $d^0$ is a retraction in $R\g\Mod$ and, therefore, so is $(f^0,d^0)$. Then, $(f,d):Z'\bigoplus Y\rightmap{}Z$ is a deflation and its kernel has the form
 $(d', -f')^t: Y'\rightmap{} Z'\bigoplus Y$. 
 Therefore, we have $fd'=df'$ and a pull-back diagram in $R\g\Mod$
$$\begin{matrix}
 Y'& \rightmap{(d')^0}& Z'\\ 
 \lmapdown{(f')^0}&&\rmapdown{f^0}\\ 
Y& \rightmap{d^0}& Z.\\ 
\end{matrix}$$
Since $d^0$ is a retraction in $R\g\Mod$, the same is true for $(d')^0$. Therefore, $d'$ is a deflation, proving our statement.
\end{proof}

The preceding result applies to the module category of the Roiter interlaced weak ditalgebra 
$\underline{\cal A}(\Delta)=({\cal A}(\Delta),I)$ constructed before, see (\ref{D: def del bocs A(Delta)}). 
The following theorem combines the equivalence of (\ref{T: Keller-Lefevre})
with the one constructed in (\ref{T: tw(B) equiv (cal A,I)-Mod}). It generalizes the result given in \cite{KKO} for the case of quasi-hereditary algebras. It relies on the explicit description of the equivalence of categories ${\cal F}(\Delta)\simeq H^0(\tw(A))$ mentioned in (\ref{T: Keller-Lefevre}). The detailed proof of this result is beyond the scope of this paper. 

\begin{theorem}\label{T: Keller exactness of equiv}
Consider the equivalence functors $K:{\cal F}(\Delta)\rightmap{}H^0(\tw(A))$ of  (\ref{T: Keller-Lefevre}), $E:H^0(\tw(A))\rightmap{} H^{-1}(\tw(B))$, of  (\ref{L: H0(cal A) equiv H-1(cal B) si cal B es la barra de cal A}), 
 and the equivalence  $M:H^{-1}(\tw(B))\rightmap{}\underline{\cal A}(\Delta)\g \Mod$ constructed in (\ref{T: tw(B) equiv (cal A,I)-Mod}).  
 Then the composition 
 $$G:=MEK:{\cal F}(\Delta)\rightmap{}\underline{\cal A}(\Delta)\g \mod$$
 is an equivalence mapping short exact sequences to conflations and such that $G(\Delta_i)\cong S_i$, for all $i\in {\cal P}$, where $S_i=(ke_i,\rho_{S_i})$ and $\rho_{S_i}=0$, with the notation of (\ref{L: (cal A,I) reformulada}).
\end{theorem}

\section{The right algebra of $({\cal A},I)$}\label{S: Gamma}

In this section we recall some results by W.L. Burt and M.C. Butler on bocses and reformulate them in our context for their application in the next section. 

\begin{remark}\label{L: notation for Burt-Butler context}
In this section we assume that $\underline{\cal A}=({\cal A},I)$ is a \emph{special triangular interlaced weak  ditalgebra}, in the following sense.  The weak ditalgebra  
${\cal A}=(T,\delta)$ admits a triangular layer $(S,W_0\oplus W_1)$, where $S$ is a finite product of fields. We write $S=\bigoplus_{i\in {\cal P}}ke_i$, where $1=\sum_{i\in {\cal P}}e_i$ is a decomposition of the unit of $S$ as a sum of primitive  orthogonal central idempotents. We adopt the notation of (\ref{D: weak ditalgebra})--(\ref{D: triangular layered weak ditalg}) with $R=S$. Thus, we have $T=T_S(W_0\oplus W_1)$, $A=T_S(W_0)$,  $V=A\otimes_SW_1\otimes_SA$, 
and $I$ is an ideal of $A$.  We will assume that $W_0$ and $W_1$ are finite dimensional $S$-$S$-bimodules and we will furthermore assume that the quotient algebra $\underline{A}:=A/I$ is finite-dimensional.  

 Recall, from \cite{bpsqh}(4.3), that $$\underline{V}:=V/(IV+VI)=\frac{A\otimes_S W_1\otimes_S A}{I\otimes W_1\otimes A+A\otimes W_1\otimes I}\cong \underline{A}\otimes_S W_1\otimes_S \underline{A}.$$ 
 We denote by $\pi=p\otimes id_{W_1}\otimes p:A\otimes W_1\otimes A\rightmap{}\underline{A}\otimes W_1\otimes \underline{A}$, where $p:A\rightmap{}\underline{A}$ is the canonical projection. 

The pair $\underline{\cal A}$ is a Roiter interlaced weak ditalgebra, as in (\ref{D: Roiter weak ditalgebra}), and \S\ref{S: exact structure for Roiter dits with ideal} applies to its module category $\underline{\cal A}\g\Mod$. 
In particular, if we make $\underline{A}=A/I$, the embedding functor $$L=L_{\underline{\cal A}}:\underline{A}\g\mod\rightmap{}\underline{\cal A}\g\mod,$$
 which maps each $\underline{A}$-module $M$ onto itself and any morphism 
 $f^0:M\rightmap{}N$ in $\underline{A}\g\mod$ onto $(f^0,0):M\rightmap{}N$ in $\underline{\cal A}\g\mod$,  is an exact functor, when we consider the exact structure $\underline{\cal E}$ on $\underline{\cal A}\g\mod$ introduced in (\ref{D: S-acceptable y E_A,S}), see (\ref{L: caract de conflaciones}). 
\end{remark}

 \begin{definition}\label{D: Gamma}
Assume that $\underline{\cal A}$ is a special interlaced  weak ditalgebra as above. Then, the \emph{right algebra of $\underline{\cal A}$} is the finite-dimensionl $k$-algebra 
$$\Gamma:=\End_{\underline{\cal A}}(\underline{A})^{op}.$$
 There is an embedding of $\underline{A}$ in $\Gamma$  given by the composition 
 $$\begin{matrix}\underline{A}&\rightmap{\cong}&\End_{\underline{A}}(\underline{A})^{op}&\rightmap{L_{\underline{\cal A}}}&\End_{\underline{\cal A}}(\underline{A})^{op}&=&\Gamma\end{matrix},$$
 so $\Gamma$ becomes naturally an $\underline{A}$-algebra.
\end{definition}

The statements in (\ref{P: de A-mod en Gamma-mod y Burt-Butler})  outline important relations between the algebras $\underline{A}$ and $\Gamma$, they translate  results of W.L. Burt and M.C.R. Butler to the terminology we use here, see \cite{BuBt} and \cite{Bu}. For the sake of completeness, we give a sketch-review of  their proofs.  
In the rest of this section, $\underline{\cal A}$ is a special interlaced weak ditalgebra, $\Gamma$ is its right algebra, and we adopt the preceding notation. 

\begin{lemma}\label{L: suc exacta} For any $M\in \underline{A}\g\mod$, there is an exact sequence of left $\underline{A}$-modules
$$0\rightmap{}\Hom_{\underline{A}}(\underline{A},M)\rightmap{L}
\Hom_{\underline{\cal A}}(\underline{A},M)\rightmap{\xi}\Hom_{\underline{A}}(\underline{V},M)\rightmap{}0,$$
where the space $\Hom_{\underline{\cal A}}(\underline{A},M)$ is a left $\underline{A}$-module through $\underline{A}\rightmap{}\End_{\underline{\cal A}}(\underline{A})^{op}$.
\end{lemma}

\begin{proof} Since $M\in \underline{A}\g\mod$, any $f^1\in \Hom_{A\g A}(V,\Hom_k(\underline{A},M))$ determines an $\underline{f}^1\in\Hom_{\underline{A}\g\underline{A}}(\underline{V},\Hom_k(\underline{A},M))$ such that $\underline{f}^1\pi=f^1$. Then, we can  define
$$\begin{matrix}
 \xi&:&\Hom_{\underline{\cal A}}(\underline{A},M)&\rightmap{}&\Hom_{\underline{A}}(\underline{V},M)\\
&&\hbox{ \ \ }f=(f^0,f^1)&\longmapsto&\xi(f):\underline{V}\rightmap{}M\\
&&&&\hbox{\hskip1.5cm}v\mapsto \underline{f}^1(v)[1].\\
\end{matrix}$$

Thus, $\xi$ is a morphism of left $\underline{A}$-modules. Suppose that
 $f\in\Hom_{\underline{\cal A}}(\underline{A},M)$ satisfies that $\xi(f)=0$. Thus, 
$\underline{f}^1(v)[1]=0$, for all $v\in \underline{V}$. Hence, $\underline{f}^1(v)[b]=\underline{f}^1(vb)[1]=0$, for all $b\in \underline{A}$. Thus, $\underline{f}^1(v)=0$, for all $v\in \underline{V}$, and $\underline{f}^1=0$, thus $f^1=0$. This means that $f\in \Im L$. Therefore, 
$\Im L=\Ker \xi$.
 
Take $g\in \Hom_{\underline{A}}(\underline{V},M)$ and consider its image $\underline{g}^1$ under the canonical isomorphism
$$\Hom_{\underline{A}}(\underline{V},M)\cong\Hom_{\underline{A}}(\underline{V}\otimes_{\underline{A}}\underline{A},M)\cong\Hom_{\underline{A}\g \underline{A}}(\underline{V},\Hom_k(\underline{A},M)).$$
Thus, $\underline{g}^1:\underline{V}\rightmap{}\Hom_{k}(\underline{A},M)$ is a morphism of $\underline{A}$-bimodules satisfying $\underline{g}^1(v)[b]=g(vb)$, for $v\in \underline{V}$ and $b\in \underline{A}$. Consider the composition morphism $g^1:=\underline{g}^1\pi\in \Hom_{A\g A}(V,\Hom_k(\underline{A},M))$. 
Then, fix any $m_0\in M$ and define, for $a\in \underline{A}$, 
$$g^0(a):=am_0-g^1(\delta(a))[1].$$
Thus, $g^0\in\Hom_k(\underline{A},M)$. Moreover, $(g^0,g^1)\in\Hom_{\underline{\cal A}}(\underline{A},M)$. Indeed, given $a\in \underline{A}$ and $b\in A$, we have
$$\begin{matrix}
   bg^0(a)&=&bam_0-bg^1(\delta(a))[1]\hfill\\
&=&bam_0-g^1(b\delta(a))[1]\hfill\\
&=&bam_0+g^1(\delta(b)a-\delta(ba))[1]\hfill\\
&=&bam_0-g^1(\delta(ba))[1]+g^1(\delta(b))[a]\hfill\\
&=&g^0(ba)+g^1(\delta(b))[a].\hfill\\
  \end{matrix}$$
Since $\xi(g^0,g^1)(v)=\underline{g}^1(v)[1]=g(v)$, for all $v\in \underline{V}$, $\xi$ is surjective and we are done.
\end{proof}

\begin{remark}\label{R: Gamma es plano} 
Using the same argument in the proof of  \cite{BSZ}(6.20), we obtain that any projective $\underline{A}$-module is an $\underline{\cal E}$-projective $\underline{\cal A}$-module. Moreover, every 
 $\underline{\cal E}$-projective $\underline{\cal A}$-module is isomorphic to 
 a direct summand in $\underline{\cal A}\g\mod$ of some projective $\underline{A}$-module. 
 
 Then, the functor $F=\Hom_{\underline{\cal A}}(\underline{A},-):\underline{\cal A}\g\mod\rightmap{}\Gamma\g\mod$ is exact.  
 The following lemma implies that the right $\underline{A}$-module $\Gamma$ is flat. Indeed, the functor $\Gamma\otimes_{\underline{A}}-$ is exact because  
 $\Gamma\otimes_{\underline{A}}-\cong FL$ and we know that $L$ and $F$ are exact functors, see (\ref{L: pares exactos}). 
 
 It is not hard to see that the functor $F$ is in fact full and faithful. The proof of \cite{BSZ}(7.12) can be adapted to this situation. 
\end{remark}

\begin{lemma}\label{L: el iso sigma de Gamma otimes- a Hom(A,L-)}
For $M\in \underline{A}\g\mod$, the map $\sigma_M:\Gamma\otimes_{\underline{A}}M\rightmap{}\Hom_{\underline{\cal A}}(\underline{A},M)$ given by $\sigma(f\otimes m)=f_m=(f^0_m,f^1_m)$, where 
$f^0_m(a)=f^0(a)m$ and $f^1_m(v)[a]=f^1(v)[a]m$ defines a natural isomorphism $$\sigma_M:\Gamma\otimes_{\underline{A}}M\rightmap{}\Hom_{\underline{\cal A}}(\underline{A},L(M)).$$
\end{lemma}

\begin{proof} Notice that
$\sigma_{\underline{A}}:\Gamma\otimes_{\underline{A}}\underline{A}\rightarrow \Gamma$ is just the isomorphism given
by multiplication. From the additivity of the      
  functors $\Gamma\otimes_{\underline{A}} -$ and
$\Hom_{\underline{\cal A}}(\underline{A},L(-))$, we obtain that
$\sigma_{\underline{A}^n}$ is an isomorphism too. Finally, if we consider a
free resolution $\underline{A}^m\rightarrow \underline{A}^n\rightarrow M\rightarrow 0$ of
$M$ in $\underline{A}\g\mod$, we can apply the exact functors $\Gamma\otimes_{\underline{A}} -$ and
$\Hom_{\underline{\cal A}}(\underline{A},L(-))$ to it, and then $\sigma_M$ is an isomorphism because
the following diagram commutes
$$\begin{matrix}
\Gamma\otimes_{\underline{A}} \underline{A}^m&\rightmap{}&\Gamma\otimes_{\underline{A}}
\underline{A}^n&\rightmap{}&\Gamma\otimes_{\underline{A}}M&\rightmap{}&0\hfill\\ 
\lmapdown{\sigma_{\underline{A}^m}} &&\lmapdown{\sigma_{\underline{A}^n}}
&&\lmapdown{\sigma_{M}} &&\hfill\\ 
\Hom_{\underline{\cal A}}(\underline{A},\underline{A}^m)&\rightmap{}&\Hom_{\underline{\cal
A}}(\underline{A},\underline{A}^n)&\rightmap{}&\Hom_{\underline{\cal A}}(\underline{A},M)&\rightmap{}&0.\hfill\\ \end{matrix}$$
\end{proof}

\begin{lemma}\label{L: adjuncion y alpha mono}
 The functor $T:=\Gamma\otimes_{\underline{A}}-:\underline{A}\g\mod\rightmap{}\Gamma\g\mod$ admits as a right adjoint the restriction functor $S:\Gamma\g\mod\rightmap{}\underline{A}\g\mod$. Consider the adjunction isomorphism
$$\zeta:\Hom_{\Gamma}(TM,N)\rightmap{}\Hom_{\underline{A}}(M,SN),$$
and its unit $\alpha$, that is the natural transformation 
$\alpha:1_{\underline{A}\g\mod}\rightmap{}ST$ defined, for $M\in \underline{A}\g\mod$, by $\alpha_M:=\zeta(1_{TM}):M\rightmap{}STM$. Then, for each $M\in \underline{A}\g\mod$, we have the commutative square
$$\begin{matrix}
M&\rightmap{\alpha_M}&\Gamma\otimes_{\underline{A}}M\\
\lmapdown{\psi_M}&&\rmapdown{\sigma_M}\\
\Hom_{\underline{A}}(\underline{A},M)&\rightmap{L}&\Hom_{\underline{\cal A}}(\underline{A},M),\\
\end{matrix}\\$$
where $\psi_M$ is the canonical isomorphism and $\sigma_M$ is the isomorphism introduced in (\ref{L: el iso sigma de Gamma otimes- a Hom(A,L-)}). In particular, $\alpha_M$ is a monomorphism.
\end{lemma}

\begin{proof} For $M\in \underline{A}\g\mod$ and $N\in \Gamma\g\mod$, we have the adjunction isomorphism 
$$\begin{matrix}\zeta&:&\Hom_{\Gamma}(TM,N)&\rightmap{}&\Hom_{\underline{A}}(M,SN)\\
&&h&\longmapsto&\zeta(h):M\rightmap{}N\\
&&&&\hbox{\hskip1.8cm}m\mapsto h(1\otimes m),\\
 \end{matrix}\\$$
with inverse 
$$\begin{matrix}\zeta'&:&\Hom_{\underline{A}}(M,SN)&\rightmap{}&\Hom_{\Gamma}(TM,N)\\
&&g&\longmapsto&p(1_{\Gamma}\otimes g),\\
 \end{matrix}\\$$
where $p:\Gamma\otimes_{\underline{A}}N\rightmap{}N$ is the product morphism. 
The verification of the naturality of $\zeta$ is straightforward. 

Notice that $\alpha_M(m)=1\otimes m$, for $m\in M$. Make  $\theta:=\sigma_M\alpha_M$. Then, we have 
$\theta(m)=\sigma_M\alpha_M(m)=\sigma_M(id_{\underline{A}}\otimes m)=(\theta(m)^0,\theta(m)^1)$, where $\theta(m)^0[a]=am$ and
$\theta(m)^1(v)[a]=id_{\underline{A}}^1(v)[a]m=0$. Then, $\sigma_M\alpha_M(m)=\theta(m)=(\psi_M(m),0)=L\psi_M(m)$, and we are done.
\end{proof}

\begin{corollary}\label{C: suc ex con alpha} For each $M\in \underline{A}\g\mod$, there is the following exact sequence in $\underline{A}\g\mod$
 $$0\rightmap{}M\rightmap{ \ \ \alpha_M \ \ }\Gamma\otimes_{\underline{A}}M\rightmap{ \ \ }\Hom_{\underline{A}}(\underline{V},M)\rightmap{}0.$$
\end{corollary}

\begin{proof} It follows from (\ref{L: suc exacta}) and (\ref{L: adjuncion y alpha mono}).
\end{proof}

\begin{definition}\label{D: projectivizing} An $\underline{A}$-bimodule $\underline{V}$ is called {\sl projectivizing} iff $\underline{V}\otimes_{\underline{A}} X$ is projective for all $X\in \underline{A}\g\mod$ and $Y\otimes_{\underline{A}}\underline{V}$ is projective for all $Y\in \mod\g \underline{A}$.
\end{definition}

\begin{lemma}\label{R: projectivizing} With our fixed notation, the $S$-$S$-bimodule $W_1$ is freely generated by a set $\hueca{B}_1$. Hence, the $\underline{A}$-bimodule $\underline{V}$ is projectivizing.
\end{lemma}

\begin{proof} Since $S$ is a finite product of fields and $W_1$ is finite-dimensional, there is a directed $k$-basis $\hueca{B}_1$ of $W_1$ and an isomorphism 
$W_1\cong \bigoplus_{w\in\hueca{B}_1}Se_{t(w)}\otimes_ke_{s(w)}S.$
Then,  we have 
$$\begin{matrix}
   \underline{V}&\cong&\underline{A}\otimes_SW_1\otimes_S\underline{A}\hfill\\
&\cong&\bigoplus_{w\in \hueca{B}_1}\underline{A}\otimes_SSe_{t(w)}\otimes_ke_{s(w)}S\otimes_S\underline{A}\hfill\\
&\cong&\bigoplus_{w\in \hueca{B}_1}\underline{A}e_{t(w)}\otimes_ke_{s(w)}\underline{A}\hfill\\
  \end{matrix}$$
Now notice that, given a family $\{Q_i\}_{i=1}^n$ of projective left $\underline{A}$-modules and a family $\{P_i\}_{i=1}^n$ of projective right $\underline{A}$-modules, then the $\underline{A}$-bimodule 
$\underline{V}:=\bigoplus_{i=1}^n{P_i}\otimes_kQ_i$ is projectivizing.
  Thus, the $\underline{A}$-bimodule $\underline{V}$ is projectivizing. 
\end{proof}

\begin{corollary}\label{Gamma es underline(A) modulo proy derecho}
The right $\underline{A}$-module $\Gamma$ is projective. 
\end{corollary}

\begin{proof} Taking $M=\underline{A}$, from (\ref{L: suc exacta}), we have the exact sequence of vector spaces
$$0\rightmap{}\Hom_{\underline{A}}(\underline{A},\underline{A})\rightmap{L}
\Hom_{\underline{\cal A}}(\underline{A},\underline{A})\rightmap{\xi}\Hom_{\underline{A}}(\underline{V},\underline{A})\rightmap{}0.$$
Thus, we have the exact sequence of vector spaces
$$0\rightmap{}\underline{A}\rightmap{s}
\Gamma\rightmap{\xi}\Hom_{\underline{A}}(\underline{V},\underline{A})\rightmap{}0,$$
where $s$ is the canonical embedding. This is, in fact, an exact sequence of right $\underline{A}$-modules. Since the left $\underline{A}$-module $\underline{V}$ is finitely generated projective, its dual right $\underline{A}$-module  
$\Hom_{\underline{A}}(\underline{V},\underline{A})$ is finitely generated  projective, thus the last exact sequence splits, and $\Gamma_{\underline{A}}$ is projective.
\end{proof}

From \cite{BuBt}(3.6) or \cite{Bu}(11.8), we have the following. 

\begin{lemma}\label{L: homs injectivos } If $\underline{V}$ is a projectivizing $\underline{A}$-bimodule, then 
$\Hom_{\underline{A}}(\underline{V},M)$ is an injective $\underline{A}$-module, for all $M\in \underline{A}\g\mod$.
\end{lemma}

\begin{proposition}\label{P: Burt-Butler}
 Consider the tensor functor $T:=\Gamma\otimes_{\underline{A}}-:\underline{A}\g\mod\rightmap{}\Gamma\g\mod$. Then, every extension $e:0\rightmap{}TM\rightmap{}E\rightmap{}TN\rightmap{}0$ in $\Gamma\g\mod$ is equivalent to an extension $Te'$ obtained by applying $T$ to an extension $$e'':0\rightmap{}M\rightmap{}E''\rightmap{}N\rightmap{}0$$ in $\underline{A}\g\mod$.
\end{proposition}

\begin{proof} We have the adjoint pair of exact functors $$\underline{A}\g\mod\dobleflechavieneva{S}{T}\Gamma\g\mod,$$ 
where $S$ is the restriction functor. Moreover, the morphism $\alpha_M:M\rightmap{}STM$ is a monomorphism with injective cokernel. Then, we can follow the argument of \cite{Bu}(10.3) to obtain the wanted extension equivalent to $e$.
\end{proof}

\begin{proposition}\label{P: de A-mod en Gamma-mod y Burt-Butler}
 Given a special interlaced weak ditalgebra $\underline{\cal A}$, we have: 
 \begin{enumerate}
  \item The functor 
 $F=\Hom_{\underline{\cal A}}(\underline{A},-):\underline{\cal A}\g\mod\rightmap{}\Gamma\g\mod$
 is a full and faithful exact functor.
 \item The  $\underline{A}$-algebra $\Gamma$ is a right $\underline{A}$-module by restriction, it  determines an exact functor 
$$\Gamma\otimes_{\underline{A}}-:\underline{A}\g\mod\rightmap{}\Gamma\g\mod.$$ 
and, we have  $\Gamma\otimes_{\underline{A}}-\cong\Hom_{\underline{\cal A}}(\underline{A},L(-))$.
\item The functor $F$ restricts to an equivalence of categories $F:\underline{\cal A}\g\mod\rightmap{}{\cal I}$, where ${\cal I}$ is the full subcategory of $\Gamma\g \mod$ of modules induced from $\underline{A}\g\mod$, that is by the class of $\Gamma$-modules isomorphic to some 
$\Gamma\otimes_{\underline{A}}N$, for some $N\in \underline{A}\g\mod$. 

Moreover, the subcategory ${\cal I}$ of $\Gamma\g\mod$ is closed under extensions. 
\item The functor $\Gamma\otimes_{\underline{A}}-:\underline{A}\g\mod\rightmap{}\Gamma\g\mod$ induces epimorphisms
$$\Ext^n_{\underline{A}}(M,N)\rightmap{}\Ext^n_{\Gamma}(\Gamma\otimes_{\underline{A}}M,\Gamma\otimes_{\underline{A}}N),$$
for all $M,N\in \underline{A}\g\mod$ and $n\geq 1$. They are isomorphisms for $n\geq 2$. 
 \end{enumerate}
 \end{proposition}
 
 \begin{proof} (1) and (2) were remarked in (\ref{R: Gamma es plano}), and (3) is (\ref{P: Burt-Butler}).   
 
\noindent(4): Recall that,  we have an exact sequence $0\rightmap{}N\rightmap{\alpha_N}\Gamma\otimes_{\underline{A}}N\rightmap{}I\rightmap{}0$ in $\underline{A}\g\mod$, with $I$  injective, for any $N\in \underline{A}\g\mod$. The corresponding long exact sequence gives epimorphisms
$$\Ext^n_{\underline{A}}(M,\alpha_N):\Ext^n_{\underline{A}}(M,N)\rightmap{}\Ext^n_{\underline{A}}(M,\Gamma\otimes_{\underline{A}}N),$$
 for $n\geq 1$ and isomorphisms for $n\geq 2$. We have a commutative diagram 
 $$\begin{matrix}\Ext^n_{\underline{A}}(M,N)&\rightmap{\gamma^n}&\Ext^n_{\Gamma}(\Gamma\otimes_{\underline{A}}M,\Gamma\otimes_{\underline{A}}N)\hfill\\
 \lmapdown{\Ext^n_{\underline{A}}(M,\alpha_N)}&&\rmapup{\Ext^n_{\Gamma}(\Gamma\otimes_{\underline{A}}M,\beta_{TN})}\\
 \Ext^n_{\underline{A}}(_{\underline{A}}M,_{\underline{A}}\Gamma\otimes_{\underline{A}}N)&\rightmap{\gamma^n}&\Ext^n_{\Gamma}(\Gamma\otimes_{\underline{A}}M,\Gamma\otimes_{\underline{A}}\Gamma\otimes_{\underline{A}}N),\hfill\\
 \end{matrix}$$
 where $\gamma^n$ denotes the morphisms induced by the application of the exact functor $T=\Gamma\otimes_{\underline{A}}-$, and 
 $\beta_{TN}:\Gamma\otimes_{\underline{A}}\Gamma\otimes_{\underline{A}}N\rightmap{}\Gamma\otimes_{\underline{A}}N$ is the product map. Since the composition $\Ext^n_{\Gamma}(\Gamma\otimes_{\underline{A}}M,\beta_{TN})\gamma^n$ is 
 an isomorphism, see \cite{Bu}(20.9), we are done.  
 \end{proof}

\section{Strict interlaced weak ditalgebras}\label{S: SIWD}

\begin{definition}\label{D: strict weak interlaced ditalgebras}
 Let ${\cal P}=({\cal P},\leq)$ be a finite preordered set. Let $\underline{\cal A}=({\cal A},I)$ be a triangular interlaced weak  ditalgebra, thus ${\cal A}=(T_S(W_0\oplus W_1),\delta)$ is a weak ditalgebra with triangular layer $(S,W_0\oplus W_1)$ and $I$ is an ${\cal A}$-triangular ideal of $A:=T_S(W_0)$. We will say that $\underline{\cal A}$ is ${\cal P}$-\emph{strict  interlaced weak ditalgebra} iff the following holds:
 \begin{enumerate}
  \item The algebra $S$ is a finite product of $\vert {\cal P}\vert$ copies of the field $k$, thus $S=\bigoplus_{i\in {\cal P}}ke_i$, where $1=\sum_{i\in {\cal P}}e_i$ is a decomposition of the unit of $S$ as a sum of primitive orthogonal central idempotents. 
  \item We have that $e_jW_1e_i\not=0$ implies $i\leq j$,
  \item We have that $e_jW_0e_i\not=0$ implies $\overline{i}<\overline{j}$.
  \item $I\subseteq \rad(A)^2$.  
 \end{enumerate}
 We will denote by $\underline{A}$ the quotient $k$-algebra $A/I$. 
\end{definition}

\begin{remark}\label{R: orden de caminos y orden de cal P} If $\underline{\cal A}$ be a ${\cal P}$-strict interlaced weak  ditalgebra as above, 
 we will identify the graded tensor algebra $A=T_S(W_0)$ with the path algebra $k({\cal B}_0)$ of the quiver ${\cal B}_0$ with set of points ${\cal P}$; such that for $i,j\in {\cal P}$, the set of solid arrows from $i$ to $j$ is a fixed basis $\hueca{B}_0(i,j)$ of the vector space $e_jW_0e_i$. 
 
 Notice that the quiver ${\cal B}_0$ coincides with the Gabiel quiver of $\underline{A}$ and has not oriented cycles. Then, we have the  partial order of precedence $\preceq$ in ${\cal P}$, as defined in (\ref{R: on directed algebra}).  
 So, for $i,j\in {\cal P}$, we have $i\prec j$ iff there is a non-trivial path from $i$ to $j$ in the quiver ${\cal B}_0$. In fact, we always have that $i\prec j$ implies $\overline{i}<\overline{j}$. 
 
 Since the quiver ${\cal B}_0$ has no oriented cycle, the algebra $\underline{A}$ is finite-dimensional. Let us write $\underline{P}_i:=\underline{A}e_i$, for all $i\in {\cal P}$. Then, the family  $\{\underline{P}_i\}_{i\in {\cal P}}$ is a complete family of representatives of the indecomposable projective $\underline{A}$-modules (we denote by $\{S_i\}_{i\in {\cal P}}$ the corresponding family of simple $\underline{A}$-modules $S_i=\underline{P}_i/\rad(\underline{P}_i)$, which is a complete family of representatives of the isoclasses of the simple $\underline{A}$-modules).   
\end{remark}

\begin{lemma}\label{L: underline A tiene stict homol system con simples}
 Assume that $\underline{\cal A}$ is a ${\cal P}$-strict interlaced weak ditalgebra, then the algebra $\underline{A}$ is directed. 
 Moreover,  the 
 triple    $\underline{\cal H}=({\cal P},\leq,\{S_i\}_{i\in {\cal P}})$ is a strict homological system for the algebra $\underline{A}$.
\end{lemma}

\begin{proof} From (\ref{L: homological systems on directed algebras}), we already know that ${\cal H}=({\cal P},\preceq,\{S_i\}_{i\in {\cal P}})$ is a strict homological system for $\underline{A}$.  Since $i\prec j$ implies $\overline{i}<\overline{j}$, we immediately obtain that $\underline{\cal H}$ is a strict homological system for $\underline{A}$. 
\end{proof}

\begin{remark}\label{R: indec projs in cal A-mod}
 Let $\underline{\cal A}$ be a ${\cal P}$-strict  interlaced weak ditalgebra and adopt the notation of (\ref{D: strict weak interlaced ditalgebras}). Then, we have a special interlaced weak ditalgebra and the results of sections \S\ref{S: Exact cat} and \S\ref{S: Gamma} hold. 
\end{remark}

\begin{lemma}\label{L: indec projs in cal A-mod}
Let $\underline{\cal A}$ be a ${\cal P}$-strict  interlaced ditalgebra. Then, there is a family $\{Q_i\}_{i\in {\cal P}}$ of non-isomorphic indecomposable ${\underline{\cal E}}$-projective $\underline{\cal A}$-modules, which is a complete family of representatives of the indecomposable $\underline{\cal E}$-projective $\underline{\cal A}$-modules and, moreover, for each $i\in {\cal P}$, we have 
an isomorphism in $\underline{\cal A}\g\mod$
$$\underline{P}_i\cong Q_i\oplus(\bigoplus_{i\prec j}m_{i,j}Q_j), \hbox{ for some } m_{i,j}\geq 0.$$
\end{lemma}

\begin{proof} Given $i\in {\cal P}$, as we remarked in (\ref{R: Gamma es plano}), the $\underline{\cal A}$-module $\underline{P}_i$ is ${\underline{\cal E}}$-projective. Consider its decomposition as a direct sum of indecomposables.   So we have an isomorphism in 
$\underline{\cal A}\g\mod$:  
$$g=(g^0,g^1):\underline{P}_i\rightmap{}\bigoplus_{j=1}^{n_i}Q_{i,j},$$
 a direct sum of indecomposable $\underline{\cal E}$-projectives. 
Then, we have the linear isomorphism $g_i^0:e_i\underline{A}e_i\rightmap{}\bigoplus_{j=1}^{n_i}e_iQ_{i,j}$, where the domain is one dimensional. Then, there is a unique $Q_i:=Q_{i,j}$ with $\dim e_iQ_{i,j}=1$. We shall prove our lemma by induction on the partial order $\prec$ on ${\cal P}$. 

 If $i$ is maximal (that is if $i$ is a sink in ${\cal B}_0$), then $\underline{P}_i\cong Q_i$ is indecomposable in $\underline{\cal A}\g\mod$ and we are done. Assume we have $\underline{P}_s\cong  Q_s\oplus(\bigoplus_{s\prec t}m_{s,t}Q_t)$, for all $s\in {\cal P}$ with $i\prec s$, for a fixed $i\in {\cal P}$. Consider the indecomposable decomposition of $\underline{P}_i\cong  Q_i\oplus(\oplus_{j=1}^{u_i}Q_{i,j})$ of $\underline{P}_i$ in $\underline{\cal A}\g\mod$. Fix any direct summand $Q:=Q_{i,j}$, not isomorphic to $Q_i$, then $Q$ is an indecomposable 
 $\underline{\cal E}$-projective $\underline{\cal A}$-module such that $e_iQ=0$. Moreover, if $e_jQ\not=0$, we have $e_j\underline{A}e_i\not=0$, so $i\prec j$.   Now, consider the projective cover $q:\underline{P}\rightmap{}Q$ of $Q$ in $\underline{A}\g\mod$, then we have the commutative diagram
$$\begin{matrix}
   \underline{P}&\rightmap{q}&Q\\
&\ddmapdown{\overline{q}}&\shortlmapdown{}\\
&&Q/\rad(Q).\\
  \end{matrix}$$
 There, the module  $\underline{P}_j$ is a direct summand of $\underline{P}$ iff $S_j$ is a direct summand of $Q/\rad(Q)$, which is only possible if $e_jQ\not=0$, that is if $i\prec j$. This means that 
$\underline{P}\cong \oplus_{i\prec j}s_j\underline{P}_j$ in $\underline{A}\g\mod$, for some $s_j\geq 0$. Then, if $\sigma:K\rightmap{}\underline{P}$ is the kernel of $q$ in $\underline{A}\g\mod$, we have the conflation 
$$K\rightmap{\ (\sigma,0) \ }\underline{P}\rightmap{ \ (q,0) \ }Q$$
in the exact category $(\underline{\cal A}\g\mod,{\underline{\cal E}})$, where $Q$ is an ${\underline{\cal E}}$-projective module. Hence, $Q$ is a direct summand of $\underline{P}\cong  \oplus_{i\prec j}s_j\underline{P}_j$ in $\underline{\cal A}\g\mod$. Then, by induction hypothesis, $Q\cong Q_t$, for some $i\prec j\prec t$, and this finishes the proof.
\end{proof}

\begin{proposition}\label{P: strict homological system for Gamma}
 Let $\underline{\cal A}$ be a ${\cal P}$-strict  interlaced weak ditalgebra and consider its right algebra $\Gamma$. For $i\in {\cal P}$, make $P'_i:=\Gamma\otimes_{\underline{A}}Q_i$ and $\Delta'_i:=\Gamma\otimes_{\underline{A}}S_i$. Then,  $\{P'_i\}_{i\in {\cal P}}$ is a complete family of representatives of the indecomposable projective $\Gamma$-modules and  
 ${\cal H}':=({\cal P},\leq,\{\Delta'_i\}_{i\in P})$ is a strict homological system for $\Gamma$. 
\end{proposition}

\begin{proof} Since $F$ is full and faithful, each $P'_i=\Gamma\otimes_{\underline{A}}Q_i\cong F(L(Q_i))=F(Q_i)$ and 
each $\Delta'_i=\Gamma\otimes_{\underline{A}}S_i\cong F(L(S_i))=F(S_i)$ are  indecomposable $\Gamma$-modules. 
Moreover, using again that $F$ is full and faithful, and (\ref{L: indec projs in cal A-mod}), we know that $\{P'_i\}_{i\in {\cal P}}$ and $\{\Delta'_i\}_{i\in {\cal P}}$ are families of non-isomorphic $\Gamma$-modules. 

From (\ref{L: indec projs in cal A-mod}), we know that  $\Gamma\cong\Gamma\otimes_{\underline{A}}\underline{A}\cong \Gamma\otimes_{\underline{A}}(\bigoplus_{i\in {\cal P}}\underline{A}e_i)\cong F(\bigoplus_{i\in {\cal P}}\underline{P}_i)\cong \bigoplus_{i\in {\cal P}} F(Q_i)\oplus (\bigoplus_{i\prec j} m_{i,j}F(Q_j))$, so every $P'_i$ is a projective indecomposable $\Gamma$-module, and  $\{P'_i\}_{i\in {\cal P}}$ is a complete family of representatives of the indecomposable projective $\Gamma$-modules.

Assume that, for $i,j\in {\cal P}$, we have $\Hom_\Gamma(\Delta'_i,\Delta'_j)\not=0$. Then, we have that $\Hom_{\underline{\cal A}}(S_i,S_j)\not=0$. If we assume that $i\not=j$, and take $0\not=f=(f^0,f^1)\in \Hom_{\underline{\cal A}}(S_i,S_j)$, we have that $f^0=0$, but $0\not=f^1\in \Hom_{S\g S}(W_1,\Hom_k(S_i,S_j))\cong e_jW_1e_i$, therefore, since $\underline{\cal A}$ is strict, we obtain that $i\leq j$. 

Now, suppose that, for $i,j\in {\cal P}$ and $n\in \hueca{N}$, we have that $\Ext^n_\Gamma(\Delta'_i,\Delta'_j)\not=0$. From (\ref{P: de A-mod en Gamma-mod y Burt-Butler}), we obtain  $\Ext^n_{\underline{A}}(S_i,S_j)\not=0$ and, from (\ref{L: underline A tiene stict homol system con simples}), we get  $\overline{i}<\overline{j}$. 

So, once we show that ${\cal H}'$ is an admissible homological system for the algebra $\Gamma$, we get that it is strict.

Fix $i\in{\cal P}$ and consider the exact sequence $0\rightmap{}\rad(\underline{P}_i)\rightmap{}\underline{P}_i\rightmap{p^0}S_i\rightmap{}0$ in $\underline{A}\g\mod$. Applying $L_{\underline{\cal A}}$, we have a conflation 
$\rad(\underline{P}_i)\rightmap{}\underline{P}_i\rightmap{}S_i$ in $\underline{\cal A}\g\mod$. Consider the following composition of morphisms in $\underline{\cal A}\g\mod$ 
$$\begin{matrix}g_i=(g_i^0,g_i^1)&=&(Q_i&\rightmap{ \ (\sigma_i,0) \ }&Q_i\oplus(\bigoplus_{i\prec j}m_{i,j}Q_j)&
\cong& \underline{P}_i&\rightmap{ \ (p^0,0) \ }&S_i).\end{matrix}$$
Since $g_i^0\not=0$, we get that $g_i^0$ is a retraction, so $g_i$ is a deflation and we have a conflation 
$K_i\rightmap{f_i}Q_i\rightmap{g_i}S_i$ in $\underline{\cal E}$. Here, we know that $e_jK_i\not=0$ implies $e_jQ_i\not=0$, hence that $i\prec j$. 
By the definition of $\underline{\cal E}$ in (\ref{D: S-acceptable y E_A,S}), we know there is a commutative diagram 
$$\begin{matrix}
  K_i&\rightmap{f_i}&Q_i&\rightmap{g_i}&S_i\\
  \rmapdown{\cong}&&\rmapdown{\cong}&&\rmapdown{\cong}\\
  K'_i&\rightmap{ \ (u_i^0,0) \ }&E'_i&\rightmap{\ (v_i^0,0) \ }&S'_i,\\
  \end{matrix}$$
  where $0\rightmap{}K'_i\rightmap{u_i^0}E'_i\rightmap{v_i^0}S'_i\rightmap{}0$ is exact in $\underline{A}\g\mod$. 
  Applying $F$ to the diagram, from (\ref{P: de A-mod en Gamma-mod y Burt-Butler})(2), we get the commutative diagram in $\Gamma\g\mod$ with exact lower row
 $$\begin{matrix}
  F(K_i)&\rightmap{F(f_i)}&F(Q_i)&\rightmap{F(g_i)}&F(S_i)\\
  \rmapdown{\cong}&&\rmapdown{\cong}&&\rmapdown{\cong}\\
  F(K'_i)&\rightmap{ \ F(u_i^0,0) \ }&F(E'_i)&\rightmap{\ F(v_i^0,0) \ }&F(S'_i).\\
  \end{matrix}$$
  Hence, we obtain an exact sequence of $\Gamma$-modules
  $$0\rightmap{}\Gamma\otimes_{\underline{A}}K_i\rightmap{}\Gamma\otimes_{\underline{A}}Q_i\rightmap{}\Gamma\otimes_{\underline{A}}S_i\rightmap{}0,$$
  where $P'_i=\Gamma\otimes_{\underline{A}}Q_i$ and $\Delta'_i=\Gamma\otimes_{\underline{A}}S_i$. 
  
  Now, we show that $\Gamma\otimes_{\underline{A}}K_i$ is filtered by $\Gamma$-modules $\Delta'_j$ with $\overline{i}<\overline{j}$. 
  Recall that the $\underline{A}$-module $K_i$ admits a simple $\underline{A}$-module $S_j$ as a composition factor iff $\Hom_{\underline{A}}(\underline{A}e_j,K_i)\not=0$, that is iff $e_jK_i\not=0$, thus $i\prec j$. Consider a composition series $0\subseteq E_\ell\subseteq \cdots\subseteq E_1\subseteq E_0=K_i$ of the $\underline{A}$-module $K_i$, then we have exact sequences 
  $$\begin{matrix}
     0\rightmap{}E_\ell\rightmap{}E_{\ell-1}\rightmap{}S_{j_2}\rightmap{}0\\
     0\rightmap{}E_{\ell-1}\rightmap{}E_{\ell-2}\rightmap{}S_{j_3}\rightmap{}0\\
     \cdots \\
     0\rightmap{}E_1\rightmap{}E_0\rightmap{}S_{j_\ell}\rightmap{}0\\
    \end{matrix}$$
    with $i\prec j_1,\ldots,j_\ell$ and $E_\ell\cong S_{j_1}$. Applying the exact functor $\Gamma\otimes_{\underline{A}}-$ we obtain exact sequences 
     $$\begin{matrix}
     0\rightmap{}\Gamma\otimes_{\underline{A}}S_{j_1}\rightmap{}\Gamma\otimes_{\underline{A}}E_{\ell-1}\rightmap{}\Gamma\otimes_{\underline{A}}S_{j_2}\rightmap{}0\\
     0\rightmap{}\Gamma\otimes_{\underline{A}}E_{\ell-1}\rightmap{}\Gamma\otimes_{\underline{A}}E_{\ell-2}\rightmap{}\Gamma\otimes_{\underline{A}}S_{j_3}\rightmap{}0\\
     \cdots \\
     0\rightmap{}\Gamma\otimes_{\underline{A}}E_1\rightmap{}\Gamma\otimes_{\underline{A}}E_0\rightmap{}\Gamma\otimes_{\underline{A}}S_{j_\ell}\rightmap{}0.\\
    \end{matrix}$$
     Then, the module $\Gamma\otimes_{\underline{A}}K_i$ is filtered by $\Gamma$-modules $\Delta'_j$ with $i\prec j_1,\ldots,j_\ell$, thus with $\overline{i}< \overline{j_1},\ldots,\overline{j_\ell}$. Thus,  $P'_i\in {\cal F}(\Delta')$, for all $i\in {\cal P}$, and ${\cal H}'$ is  a strict homological system. 
\end{proof}

\begin{lemma}\label{L: cal F(Delta')=Inducidos y consecuencia} Adopt the notation of (\ref{P: de A-mod en Gamma-mod y Burt-Butler}). Then, we have ${\cal I}={\cal F}(\Delta')$.  Moreover, the family $\{P'_i\}_{i\in {\cal P}}$ is a complete family of representatives of the indecomposable ${\cal F}(\Delta')$-projective objects.
\end{lemma}

\begin{proof} Notice first that since ${\cal I}$ is closed under extensions and $\Delta'_i=\Gamma\otimes_{\underline{A}}S_i\in {\cal I}$, an easy induction shows that ${\cal F}(\Delta')\subseteq{\cal I}$. Conversely, given $\Gamma\otimes_{\underline{A}}M\in {\cal I}$, a composition series for the $\underline{A}$-module $M$ is transformed by the application of $\Gamma\otimes_{\underline{A}}-$ into a $\Delta'$-filtration of $\Gamma\otimes_{\underline{A}}M$, so it belongs to ${\cal F}(\Delta')$. 
So, indeed we have ${\cal I}={\cal F}(\Delta')$.

From (\ref{P: strict homological system for Gamma}), we know that 
 each $P'_i\in {\cal F}(\Delta')$, and it is an indecomposable  projective $\Gamma$-module, so, it is an indecomposable ${\cal F}(\Delta')$-projective. Assume that $Q'\in {\cal F}(\Delta')$ is an indecomposable ${\cal F}(\Delta')$-projective. Since ${\cal F}(\Delta')$ consists of induced $\Gamma$-modules, we know that $Q'\cong \Gamma\otimes_{\underline{A}}Q$, for some $Q\in \underline{A}\g\mod$. Consider any exact sequence in $\underline{A}\g\mod$ 
$$0\rightmap{}K\rightmap{f^0}\underline{P}\rightmap{g^0}Q\rightmap{}0$$
with $\underline{P}$ a projective $\underline{A}$-module. 
Then, we have the conflation 
$K\rightmap{(f^0,0)}\underline{P}\rightmap{(g^0,0)}Q$ in $\underline{\cal A}\g\mod$ and an exact sequence 
$$0\rightmap{}\Gamma\otimes_{\underline{A}}K\rightmap{}\Gamma\otimes_{\underline{A}}\underline{P}\rightmap{}Q'
\rightmap{}0,$$
in $\Gamma\g\mod$, which lies in ${\cal I}={\cal F}(\Delta')$. Moreover, we know that 
the projective $\underline{A}$-module $\underline{P}$ is $\underline{\cal E}$-projective and, so, it has the form $\underline{P}\cong\bigoplus_{i\in {\cal P}}m_iQ_i$ in $\underline{\cal A}\g\mod$. Then, 
$\Gamma\otimes_{\underline{A}}\underline{P}\cong \bigoplus_{i\in {\cal P}}m_i\Gamma\otimes_{\underline{A}}Q_i=\bigoplus_{i\in {\cal P}}m_iP'_i$. Since $Q'$ is an indecomposable ${\cal F}(\Delta')$-projective, the preceding exact sequence splits and we obtain that $Q'\cong P'_i$, for some $i\in {\cal P}$
\end{proof}

\eject

\begin{theorem}\label{T: charact of strict Yoneda algebras}
Assume that ${\cal H}=({\cal P},\leq,\{\Delta_i\}_{i\in {\cal P}})$ is an admissible homological system for $\Lambda$. Then, the following are equivalent.
\begin{enumerate}
 \item The Yoneda algebra $A(\Delta)=\bigoplus_{n\geq 0}\Ext^n_\Lambda(\Delta,\Delta)$ admits a strict structure of $A_\infty$-algebra over $S=\bigoplus_{i\in {\cal P}}ke_i$.
 \item The homological system ${\cal H}$ is strict. 
\end{enumerate}
Moreover, in this case the algebra $\Lambda$ and the right algebra $\Gamma$ of the interlaced weak ditalgebra $\underline{\cal A}(\Delta)$ of $\Delta$ 
are Morita equivalent. 
\end{theorem}

\begin{proof}$(2)$ implies $(1)$: If ${\cal H}$ is a strict homological system,  from (\ref{T: Yoneda A-infinito algebra es estricta}), we know that the Yoneda algebra of $\Delta$ admits a strict structure of $A_\infty$-algebra over $S$. 

$(1)$ implies $(2)$: The construction of the triangular interlaced weak ditalgebra $\underline{\cal A}(\Delta)=({\cal A}(\Delta),I)$ accomplished in (\ref{D: def del bocs A(Delta)}), only required an admissible homological system ${\cal H}$ for $\Lambda$ and a strict structure of $A_\infty$-algebra on the Yoneda algebra of $\Delta$, as we assumed at the begining of \S4. It is clear that the resulting interlaced weak  ditalgebra  $\underline{\cal A}(\Delta)$ is ${\cal P}$-strict. Then, we have the strict homological system ${\cal H}'=({\cal P},\leq,\{\Delta'_i\}_{i\in {\cal P}})$ for its right algebra $\Gamma$ described in (\ref{P: strict homological system for Gamma}). 

We have  the exact functor $F:\underline{\cal A}(\Delta)\g\mod\rightmap{}\Gamma\g\mod$ such that $F(S_i)=\Delta'_i$, for all $i\in {\cal P}$. From (\ref{P: de A-mod en Gamma-mod y Burt-Butler}) and (\ref{L: cal F(Delta')=Inducidos y consecuencia}), we know that this functor restricts to an equivalence of categories $F':\underline{\cal A}(\Delta)\g\mod\rightmap{}{\cal F}(\Delta')$, which maps conflations onto short exact sequences.   
  
  From (\ref{T: Keller exactness of equiv}), we know  that there is an equivalence of  categories  
  $$G:{\cal F}(\Delta)\rightmap{}\underline{\cal A}(\Delta)\g\mod, \hbox{ with } G(\Delta_i)\cong S_i, \hbox{ for } i\in {\cal P},$$
  and such that it maps short exact sequences onto conflations. 
  The  equivalence $\Theta=F'G:{\cal F}(\Delta)\rightmap{}{\cal F}(\Delta')$,  satisfies $\Theta(\Delta_i)\cong\Delta'_i$, for all $i\in {\cal P}$, and it maps short exact sequences with terms in ${\cal F}(\Delta)$ onto short exact sequences with terms in ${\cal F}(\Delta')$. Its quasi-inverse $\Theta':{\cal F}(\Delta')\rightmap{}{\cal F}(\Delta)$ satisfies $\Theta'(\Delta'_i)\cong\Delta_i$, for $i\in {\cal P}$.
  
  We clearly have that $\{P_i\}_{i\in {\cal P}}$ is a family of non-isomorphic indecomposable ${\cal F}(\Delta)$-projectives. Since $\Theta'$ is an equivalence, from (\ref{L: cal F(Delta')=Inducidos y consecuencia}), we know that $\{\Theta'(P'_i)\}_{i\in {\cal P}}$ is a complete family of non-isomorphic indecomposable ${\cal F}(\Delta)$-projectives. Suppose that $P_i\cong \Theta'(P'_j)$, thus $\Theta(P_i)\cong \Theta\Theta'(P'_j)\cong P'_j$. From (\ref{P: admis implies H2}), we have the exact sequence in ${\cal F}(\Delta)$
 $$\begin{matrix}
   0&\rightmap{}&Q_i&\rightmap{}&P_i&\rightmap{\nu_i}&\Delta_i&\rightmap{}&0.
   \end{matrix}$$ 
 Thus, we have an exact sequence in ${\cal F}(\Delta')$
 $$\begin{matrix}
   0&\rightmap{}&\Theta(Q_i)&\rightmap{}&\Theta(P_i)&\rightmap{}&\Theta(\Delta_i)&\rightmap{}&0.
   \end{matrix}$$ 
This gives a surjective morphism $P'_j\rightmap{}\Delta'_i$ in $\Gamma\g\mod$, which has to factor through the projective cover $\nu_i':P'_i\rightmap{}\Delta'_i$, see again (\ref{P: admis implies H2}). Then, there is a surjective morphism $P'_i\rightmap{}P'_j$ and, so, we get $P'_i\cong P'_j$, and $i=j$. So we have that $\Theta'(P'_i)\cong P_i$, for all $i\in {\cal P}$. 

 Thus the equivalence $\Theta:{\cal F}(\Delta)\rightmap{}{\cal F}(\Delta')$ is such that $\Theta(\Delta_i)\cong \Delta'_i$ and $\Theta(P_i)\cong P'_i$, for all $i\in {\cal P}$. After adjusting $\Theta$, if necessary, we can assume that $\Theta(\Delta_i)= \Delta'_i$ and $\Theta(P_i)=P'_i$, for all $i\in {\cal P}$. Make $P:=\bigoplus_{i\in {\cal P}}P_i$ and $P':=\bigoplus_{i\in {\cal P}}P'_i$, then $\Theta$ determines an isomorphism of $k$-algebras
 $$\Lambda':=\End_\Lambda(P)^{op}=
 \End_{\Lambda}(\bigoplus_{i\in {\cal P}}P_i)^{op}
 \rightmap{\theta}
 \End_{\Gamma}(\bigoplus_{i\in {\cal P}}P'_i)^{op}=\End_{\Gamma}(P')^{op}=:\Gamma'.$$
 We know that $\Gamma$ is Morita equivalent to $\Gamma'$ and that $\Lambda$ is Morita quivalent to $\Lambda'$. 
 Consider the following series of equivalences
$$\Gamma\g\mod\rightmap{\Phi}\Gamma'\g\mod\rightmap{F_\theta}\Lambda'\g\mod\leftmap{\Psi}\Lambda\g\mod$$
where $\Phi=\Hom_\Gamma(P',-)$, $\Psi=\Hom_\Lambda(P,-)$, and $F_\theta$ is the restriction functor determined by the isomorphism  $\theta$. Once the equivalence $\Theta$ has been fixed as above, we can choose wisely a quasi-inverse $\Theta':{\cal F}(\Delta')\rightmap{}{\cal F}(\Delta)$ of $\Theta$ such that  $\Theta'(\Delta'_i)=\Delta_i$, $\Theta'(P'_i)=P_i$, for $i\in {\cal P}$ and, furthermore, such that $\Theta'\Theta(\lambda')=\lambda'$, for any $\lambda'\in\Lambda'$. 

We claim that $F_\theta\Phi(\Delta'_i)\cong \Psi(\Delta_i)$, for all $i\in {\cal P}$. Indeed, we have 
$\Phi(\Delta'_i)=\Hom_\Gamma(P',\Delta'_i)=\Hom_{{\cal F}(\Delta')}(P',\Delta'_i)$. Similarly, we have  $\Psi(\Delta_i)=\Hom_{\Lambda}(P,\Delta_i)=\Hom_{{\cal F}(\Delta)}(P,\Delta_i)$.  By the preceding considerations, we have an isomorphism of vector spaces
$$F_\theta\Phi(\Delta'_i)=\,_\theta\Hom_{{\cal F}(\Delta')}(P',\Delta'_i)\rightmap{\Theta'}\,_{\Lambda'}\Hom_{{\cal F}(\Delta)}(P,\Delta_i)\cong\Psi(\Delta_i).$$
So we only have to verify that $\Theta'$ is an isomorphism of $\Lambda'$-modules. Take any morphism $f:P'\rightmap{}\Delta'_i$ in $\Gamma\g\mod$, $\lambda'\in \Lambda'$ and $p\in P$, then we have 
$$\begin{matrix}
\Theta'(\lambda' f)(p)
&=&
\Theta'(\theta(\lambda') f)(p)\hfill\\
&=&
 \Theta'(f\circ\Theta(\lambda'))(p)\hfill\\
 &=&
 (\Theta'(f)\circ\lambda')(p)\hfill\\
 &=&
 \Theta'(f)(p\lambda')=(\lambda'\Theta'(f))(p).\hfill\\
  \end{matrix}$$
Therefore, if $\Psi':\Lambda'\g\mod\rightmap{}\Lambda\g\mod$ is any quasi-inverse of $\Psi$, we have that the composition $\Omega:=\Psi'F_\theta\Phi:\Gamma\g\mod\rightmap{}\Lambda\g\mod$ is an equivalence of categories with $\Omega(\Delta'_i)=\Psi'F_\theta\Phi(\Delta'_i)\cong \Delta_i$, for all $i\in {\cal P}$. 
From this we obtain the strictness condition for the homological system ${\cal H}$ from the corresponding condition for the homological system ${\cal H}'$, see (\ref{P: strict homological system for Gamma}). 
\end{proof}

Now, we can derive the following. 

\begin{theorem}\label{T: Borel subalgebra}
 Let $k$ be an algebraically closed field and let $\Lambda$ be a finite dimensional $k$-algebra with a strict homological system ${\cal H}=({\cal P},\leq,\{\Delta_i\}_{i\in {\cal P}})$. Then, the algebra $\Lambda$ is Morita equivalent to a finite-dimensional $k$-algebra $\Gamma$ which admits a strict homological system   
  ${\cal H}'=({\cal P},\leq,\{\Delta'\}_{i\in P})$ such that $\Gamma$ has a regular homological exact Borel subalgebra $\underline{A}$, as in (\ref{D: exact Borel subalgebra}). Moreover, there is an equivalence $\Omega:\Gamma\g\mod\rightmap{}\Lambda\g\mod$ such that $\Omega(\Delta'_i)\cong \Delta_i$, for all $i\in {\cal P}$.
\end{theorem}

\begin{proof} Consider the ${\cal P}$-strict interlaced weak ditalgebra $\underline{\cal A}(\Delta)$, adopt the notations of this section.  
 From (\ref{P: strict homological system for Gamma}) and (\ref{Gamma es underline(A) modulo proy derecho}), the algebra $\underline{A}$ is an exact Borel subalgebra of $\Gamma$, the right algebra of $\underline{\cal A}(\Delta)$. From (\ref{P: de A-mod en Gamma-mod y Burt-Butler})(4), the Borel subalgebra $\underline{A}$ is homological. In order to show that $\underline{A}$ is regular, consider the epimorphisms 
 $$\Ext_{\underline{A}}^1(S_i,S_j)\rightmap{}\Ext^1_\Gamma(\Gamma\otimes_{\underline{A}}S_i,\Gamma\otimes_{\underline{A}}S_j)=\Ext_\Gamma^1(\Delta'_i,\Delta'_j)\cong \Ext_\Lambda^1(\Delta_i,\Delta_j),$$
 for $i,j\in {\cal P}$, where the last isomorphism is obtained from the equivalence $\Omega:\Gamma\g\mod\rightmap{}\Lambda\g\mod$ constructed in the  proof of (\ref{T: charact of strict Yoneda algebras}). Then, since 
 $e_j\hat{D}(B)_0e_i\cong e_j\Ext^1_\Lambda(\Delta,\Delta)e_i\cong\Ext^1_\Lambda(\Delta_i,\Delta_j)$, we get 
 $\dim_k\Ext_{\underline{A}}^1(S_i,S_j)=\dim_ke_j\hat{D}(B)_0e_i=\dim_k\Ext^1_\Lambda(\Delta_i,\Delta_j)$.  
\end{proof}

\begin{example}\label{E: Drozds ditalgebra}
Let $\Lambda$ be a finite-dimensional $k$-algebra and ${\cal D}=(T,\delta)$ its Drozd's ditalgebra, see \cite{BSZ}\S19 and (23.25). We recall some of its features. Here, we have a splitting $\Lambda=R\oplus J$ of the algebra $\Lambda$ over its radical  $J$, so the algebra $R$ is a finite product of copies of the field $k$, say $R=\bigoplus_{i\in F}ke_i$, where $1_R=\sum_{i\in F}ke_i$ is a decomposition of the unit element as a sum of primitive orthogonal idempotents of $R$. 
Consider the  product algebra $S=R\times R$, so we have that $S$ is again a finite product of fields and we can write  $S=\bigoplus_{i\in {\cal P}}ke_i$, with ${\cal P}:={\cal P}'\biguplus{\cal P}''$, where 
${\cal P}'=\{i'\mid i\in F\}$ and ${\cal P}''=\{i''\mid i\in F\}$ are disjoint copies  of the set $F$, and $1=\sum_{i\in {\cal P}}e_i$ is a decomposition of the unit  as a sum of primitive central ortogonal idempotents of $S$. Let us fix a basis $\hueca{B}(i,j)$ of the space $e_jJe_i$, for each $i,i\in F$. 
Then, the tensor algebra $T$ mentioned above can be identified with the graded path algebra $k({\cal B})$ of the bigraph ${\cal B}$ defined by the following.
The set of points of ${\cal B}$ is ${\cal P}$; and, for each $\alpha\in \hueca{B}(i,j)$, there is a solid arrow $\alpha:i'\rightarrow j''$ and two dashed arrows $\alpha':i'\rightarrow j'$ and $\alpha'':i''\rightarrow j''$ in ${\cal B}$.  

The ditalgebra ${\cal D}$ has no relations, that is we take $I=0$. Consider the preorder $\leq $ in ${\cal P}$ defined by $i'\leq j'$ and $i''\leq j''$, for all $i,j\in F$ and $i<j$ iff $\hueca{B}(i,j)\not=\emptyset$. Thus, the poset $\overline{\cal P}$ has only two elements. Then, the triangular interlaced weak ditalgebra $\underline{\cal D}=({\cal D},0)$ is strict with preordered set $({\cal P},\leq)$ and this section applies to it. Its right algebra is not quasi-hereditary in general.

Similarly, given any finite partially ordered sets $\hueca{S}$ and $\hueca{T}$, the associated ditalgebra ${\cal A}(\hueca{S},\hueca{T})$, see \cite{BSZ}(34.1), determines a triangular interlaced weak ditalgebra $\underline{{\cal A}(\hueca{S},\hueca{T}})=({\cal A}(\hueca{S},\hueca{T}),0)$, with a preordered set $({\cal P},\leq)$ defined similarly. 
\end{example}

\noindent{\bf  Acknowledgements.}  The authors acknowledge partial financial  support by  CONACyT, sabbatical grant 710306, and UNAM-IES mobility academic program 2020.

\hskip2cm

\vbox{\noindent R. Bautista\\
Centro de Ciencias Matem\'aticas\\
Universidad Nacional Aut\'onoma de M\'exico\\
Morelia, M\'exico\\
raymundo@matmor.unam.mx\\}

\vbox{\noindent E. P\'erez\\
Facultad de Matem\'aticas\\
Universidad Aut\'onoma de Yucat\'an\\
M\'erida, M\'exico\\
jperezt@correo.uady.mx\\}

\vbox{\noindent L. Salmer\'on\\
Centro de Ciencias  Matem\'aticas\\
Universidad Nacional Aut\'onoma de M\'exico\\
Morelia, M\'exico\\
salmeron@matmor.unam.mx\\}

 \end{document}